\documentclass[12pt,a4paper]{article}

\usepackage{amsfonts,amsmath,amsthm}
\usepackage{graphicx}
\usepackage{amssymb}
\usepackage{mathrsfs}
\usepackage[top=1.5cm, bottom=1.8cm, left=1.5cm, right=1.5cm]{geometry}
\usepackage[all]{xy}
\usepackage{color}

\newenvironment{changemargin}[2]{\begin{list}{}{%
\setlength{\topsep}{0pt}%
\setlength{\leftmargin}{0pt}%
\setlength{\rightmargin}{0pt}%
\setlength{\listparindent}{\parindent}%
\setlength{\itemindent}{\parindent}%
\setlength{\parsep}{0pt plus 1pt}%
\addtolength{\leftmargin}{#1}%
\addtolength{\rightmargin}{#2}%
}\item }{\end{list}}

\newtheorem{theorem}{Theorem}[section]
\newtheorem{proposition}[theorem]{Proposition}
\newtheorem{corollary}[theorem]{Corollary}
\newtheorem{definition}[theorem]{Definition}
\newtheorem{lemma}[theorem]{Lemma}

\theoremstyle{remark}
\newtheorem{remark}[theorem]{Remark}

\newcommand{\flecheIso}[4]{                     % fonction
            \begin{array}{rcl} #1 & \rightarrow & #2 \\   %
                         #3 &\mapsto & #4          %
            \end{array}}

\newcommand{\fonc}[5]{                     % fonction
            \begin{array}{crll}#1 :& #2 & \rightarrow & #3 \\   %
                         &#4 &\mapsto & #5          %
            \end{array}}

\newcommand{\foncIso}[5]{                     % fonction
            \begin{array}{crll}#1 :& #2 & \overset{\sim}{\rightarrow} & #3 \\   %
                         &#4 &\mapsto & #5          %
            \end{array}}

\newcommand{\exposantGauche}[2]{{\vphantom{#2}}^{#1}#2}

\title{\bf Holonomy and (stated) skein algebras \\in combinatorial quantization}
\author{Matthieu \textsc{Faitg}}
\date{}

\begin{document}
\maketitle

\vspace{-2em}

\begin{center}
\noindent Fachbereich Mathematik\\ 
\noindent Universit\"{a}t Hamburg\\ 
\noindent Bereich Algebra und Zahlentheorie\\ 
\noindent Bundesstra{\ss}e 55,\\ 
\noindent D -- 20 146 Hamburg
\end{center}

\begin{changemargin}{1cm}{0cm}
E-mail address: \texttt{matthieu.faitg@uni-hamburg.de}

\vspace{2em}

\begin{changemargin}{2cm}{2cm}
{\small
\noindent \textsc{Abstract}.~ The algebra $\mathcal{L}_{g,n}(H)$ was introduced by Alekseev--Grosse--Schomerus and Buffenoir--Roche and quantizes the character variety of the Riemann surface $\Sigma_{g,n}\!\setminus\! D$ ($D$ is an open disk). In this article we define a holonomy map in that quantized setting, which associates a tensor with components in $\mathcal{L}_{g,n}(H)$ to tangles in $(\Sigma_{g,n}\!\setminus\!D) \times [0,1]$, generalizing previous works of Buffenoir--Roche and Bullock--Frohman--Kania-Bartoszynska. We show that holonomy behaves well for the stack product and the action of the mapping class group; then we specialize this notion to links in order to define a generalized Wilson loop map. Thanks to the holonomy map, we give a geometric interpretation of the vacuum representation of $\mathcal{L}_{g,0}(H)$ on $\mathcal{L}_{0,g}(H)$. Finally, the general results are applied to the case $H=U_{q^2}(\mathfrak{sl}_2)$ in relation to skein theory and the most important consequence is that the stated skein algebra of a compact oriented surface with just one boundary edge is isomorphic to $\mathcal{L}_{g,n}\big( U_{q^2}(\mathfrak{sl}_2) \big)$. Throughout the paper we use a graphical calculus for tensors with coefficients in $\mathcal{L}_{g,n}(H)$ which makes the computations and definitions very intuitive.
}
\end{changemargin}
\vspace{0.5em}
\section{Introduction}

\indent Let $G$ be a Lie group, usually assumed connected and simply-connected (the most studied example is $G = \mathrm{SL}_2(\mathbb{C})$) and let $\Sigma_{g,n}$ be the compact oriented surface of genus $g$ with $n$ punctures (or equivalently with $n$ open disks removed). It is well-known that the moduli space of flat $G$-connections on $\Sigma_{g,n}$ can be identified with the $G$-character variety $\mathrm{Hom}\big(\pi_1(\Sigma_{g,n}),G\big)/G$, thanks to the identification between a connection and its holonomy on closed curves. The character variety carries a Poisson structure due to Atiyah--Bott \cite{AB} and Goldman \cite{goldman}. A good survey of these topics is \cite{labourie}.

\smallskip

\indent There exist several quantizations of the character variety. In this paper we will use the combinatorial quantization, developped in the mid-90s by Alekseev \cite{alekseev}, Alekseev--Grosse--Schomerus \cite{AGS, AGS2} and Buffenoir--Roche \cite{BR}. The starting point of this approach is the combinatorial description of the Atiyah--Bott--Goldman Poisson structure by Fock--Rosly \cite{FockRosly}, based on classical $r$-matrices, and the definition of an algebra $\mathcal{L}_{g,n}$ which quantizes $\mathrm{Hom}\big(\pi_1(\Sigma_{g,n} \!\setminus\! D),G\big)$,  where $D$ is an open disk.
\end{changemargin}

\smallskip

\indent More precisely let $\Sigma_{g,n}^{\mathrm{o}} = \Sigma_{g,n} \!\setminus\! D$, $\mathcal{A}_{g,n} = \mathrm{Hom}\big(\pi_1(\Sigma_{g,n}^{\mathrm{o}}),G\big)$ and $\mathbb{C}[\mathcal{A}_{g,n}]$ be the corresponding algebra of functions. Since $\pi_1(\Sigma_{g,n}^{\mathrm{o}})$ is a free group, any $\nabla \in \mathcal{A}_{g,n} \cong G^{2g+n}$ can be identified with a collection of $2g+n$ elements of $G$
\[ \nabla = \big( h_{b_1}, h_{a_1}, \ldots, h_{b_g}, h_{a_g}, h_{m_{g+1}}, \ldots, h_{m_{g+n}} \big) \in G^{2g+n} \]
which are the holonomies of the generators $b_i, a_i, m_i$ of $\pi_1(\Sigma_{g,n}^{\mathrm{o}})$. Let $\overset{V}{T} : G \to \mathrm{End}_{\mathbb{C}}(V)$ be a finite dimensional representation of $G$. Then we define matrices
\[ \overset{V}{B}(1), \, \overset{V}{A}(1), \, \ldots, \, \overset{V}{B}(g), \, \overset{V}{A}(g), \, \overset{V}{M}(g+1), \, \ldots, \, \overset{V}{M}(g+n) \in \mathbb{C}[\mathcal{A}_{g,n}] \otimes \mathrm{End}_{\mathbb{C}}(V) \]
with coefficients in $\mathbb{C}[\mathcal{A}_{g,n}]$ as follows
\[ \big(\overset{V}{B}(k)^i_j\big)(\nabla) = \overset{V}{T}{^i_j}(h_{b_k}), \:\:\: \big(\overset{V}{A}(k)^i_j\big)(\nabla) = \overset{V}{T}{^i_j}(h_{a_k}), \:\:\:  \big(\overset{V}{M}(k)^i_j\big)(\nabla) = \overset{V}{T}{^i_j}(h_{m_{k}}) \]
where $\overset{V}{T}{^i_j}$ are the matrix coefficients of $\overset{V}{T}$ in some basis. Since $\mathbb{C}[\mathcal{A}_{g,n}] \cong \mathbb{C}[G]^{\otimes (2g+n)}$, the coefficients of the matrices $\overset{V}{B}(k), \overset{V}{A}(k), \overset{V}{M}(k)$ (with $V$ running in the set of finite dimensional $G$-modules) generate $\mathbb{C}[\mathcal{A}_{g,n}]$ as an algebra. The Fock-Rosly description expresses the Poisson brackets $\{ \overset{V}{B}(k)^i_j, \overset{W}{B}(k')^{i'}_{j'} \}$, $\{ \overset{V}{B}(k)^i_j, \overset{W}{A}(k')^{i'}_{j'} \}$ \textit{etc} in a matrix form thanks to a classical $r$-matrix. The idea of combinatorial quantization is to replace the Lie group $G$ by the quantum group $U_q(\mathfrak{g})$ ($q=e^h$) and the classical $r$-matrix $r \in \mathfrak{g}^{\otimes 2}$ by the quantum $R$-matrix $R \in U_q(\mathfrak{g})^{\otimes 2}$, modulo non-trivial commutation relations governed by $R$. This defines an associative non-commutative algebra $\mathcal{L}_{g,n}$, generated by coefficients $\overset{V}{B}(k)^i_j, \overset{V}{A}(k)^i_j, \overset{V}{M}(k)^i_j$ (with $V$ now running in the set of finite dimensional $U_q(\mathfrak{g})$-modules), and such that the commutator of two such matrix coefficients contains the Fock--Rosly Poisson bracket at the order $1$ in $h$. An important feature of $\mathcal{L}_{g,n}$ is that it is endowed with a $U_q(\mathfrak{g})$-module-algebra structure, which is the quantized version of the action of $G$ on $\mathbb{C}[\mathcal{A}_{g,n}]$ by conjugation; in particular we have a subalgebra $\mathcal{L}_{g,n}^{\mathrm{inv}}$ of invariant elements, which is the quantized version of the observables $\mathbb{C}[\mathcal{A}_{g,n}]^G$. The definition of $\mathcal{L}_{g,n}$ is purely algebraic and works for any ribbon Hopf algebra $H$ instead of $U_q(\mathfrak{g})$, hence defining an algebra $\mathcal{L}_{g,n}(H)$ (see Definition \ref{defLgn}). In this paper there are no further assumptions on $H$.

\smallskip

\indent Holonomy and Wilson loops (trace of the holonomy in some representation) are very natural notions in classical gauge theory. Thus an immediate problem is to define the corresponding notions in the quantized setting. This is has been achieved first in \cite{BR2}, where they defined a Wilson loop operation for framed links in the thickened surface with values in $\mathcal{L}_{g,n}^{\mathrm{inv}}(H)$; the element associated to a framed link is written as the quantum trace of some matrix with coefficients in $\mathcal{L}_{g,n}(H)$ which can be interpreted as the holonomy of the link. The definition is however quite complicated and it is difficult to explicitly compute examples. These notions have been redefined in a more conceptual way in \cite{BFKB2} (also see \cite{BFKB}), but in a formalism which is dual to the one used in \cite{BR2} and here (the bridge between the two formalisms is discussed in detail in \cite[\S 5.1.3]{these}). More precisely, in \cite{BFKB2}, they use a description of the surface as a fat graph and encode operations on this graph by objects called multitangles; then holonomy is defined as a particular multitangle. Computing explicit examples requires to use several rules and is still a bit tricky. Finally, in \cite[\S 8]{BaR}, a Wilson loop functor for $H$-colored tangles in a thickened punctured disk (\textit{i.e.} $g=0$, $n\geq 0$) is defined as a generalization of the Reshetikhin--Turaev functor in a manner similar to ours. Let us also mention that a holonomy functor is defined in \cite[\S 5]{MW}, but only for paths in a non-thickened ribbon graph.

\smallskip

\indent Let us now describe the content of the paper. We view the surface $\Sigma_{g,n}^{\mathrm{o}}$ as the fat graph defined by the generators of $\pi_1(\Sigma_{g,n}^{\mathrm{o}})$ (see Figures \ref{surfaceGN} and \ref{grapheSurSurface}). Our approach is based on a graphical calculus introduced in \S \ref{sectionDiagramCalculus}, which transforms complicated computations for tensors with coefficients in $\mathcal{L}_{g,n}(H)$ into simple manipulations of diagrams. This graphical calculus is an extension of the graphical calculus in ribbon categories; namely we add a new diagram, called a handle, which corresponds to a matrix with coefficients in $\mathcal{L}_{g,n}(H)$. With this tool at hand, the definition of holonomy is almost obvious and it is transparent that it is a generalization of the Reshetikhin--Turaev functor. Moreover, this formulation allows us to define holonomy for tangles (and not just links) in $\Sigma_{g,n}^{\mathrm{o}} \times [0,1]$; this is especially useful to establish the relation with stated skein algebras.

\smallskip

\indent Section \ref{sectionPreliminaries} is devoted to preliminaries about surfaces, Hopf algebras, the Reshetikhin--Turaev functor and the algebra $\mathcal{L}_{g,n}(H)$. In section \ref{sectionDiagramCalculus} we introduce the graphical calculus for tensors with coefficients in $\mathcal{L}_{g,n}(H)$, which is our main technical tool in this paper, and we reformulate the defining relations of $\mathcal{L}_{g,n}(H)$ in terms of diagram equalities. 

\smallskip

\indent In section \ref{sectionHolonomyWilson} we define holonomy for tangles in $\Sigma_{g,n}^{\mathrm{o}} \times [0,1]$ (Definition \ref{defHol}). The reason why  this operation is called holonomy is explained after Proposition \ref{holonomySimpleLoop}. We show that holonomy is compatible with the stack product (Theorem \ref{wilsonStack}, this result was known for links \cite[Thm. 1]{BR2}, \cite[Thm. 8]{BFKB2} and is similar to \cite[Th 8.1]{BaR} for $g=0$) and that it behaves well under the action of the mapping class group (Theorem \ref{MCGcommuteW}, with $n=0$ for simplicity); the proofs are entirely diagrammatic. We then specialize this definition to framed based links and provide a Hennings-type description of the holonomy for such links; this can be compared with the rules given in \cite[\S 6.2]{BFKB2}. Finally, we define generalized Wilson loops for framed links colored by symmetric linear forms (as in the Hennings setting \cite{hennings}) and with values in $\mathcal{L}_{g,n}^{\mathrm{inv}}(H)$, namely the subalgebra of invariant elements. Enlarging the set of colors is important when $H$ is non-semisimple because in that case we have symmetric linear forms which are not characters of some $H$-module and hence if we color the links only by $H$-modules (\textit{i.e.} by characters) we miss a lot of interesting invariant elements.

\smallskip

\indent Let $\Sigma_{g,n}^{\mathrm{o}, \bullet}$ be $\Sigma_{g,n}^{\mathrm{o}}$ with one point removed on its boundary. In section \ref{sectionSkeinAlg} we take $H = U_{q^2}(\mathfrak{sl}_2)$ (quantum group of $\mathfrak{sl}_2$ at a formal parameter $q^2$) and we show that the stated skein algebra $\mathcal{S}_q^{\mathrm{s}}(\Sigma_{g,n}^{\mathrm{o}, \bullet})$ \cite{Le, CL} is isomorphic to $\mathcal{L}_{g,n}\big( U_{q^2}(\mathfrak{sl}_2) \big)$ (Theorem \ref{thIsoSkeinLgn}), answering the question of \cite{CL} about the relation between the two theories. In particular, this gives a presentation by generators and relations of $\mathcal{S}_q^{\mathrm{s}}(\Sigma_{g,n}^{\mathrm{o}, \bullet})$ (Corollary \ref{presentationStatedSkein}). The isomorphism is defined as follows. For $H = U_{q^2}(\mathfrak{sl}_2)$, all the strands can be colored without loss of generality by the fundamental representation $V_2$, with basis $(v_-, v_+)$. Thanks to the isomorphism $V_2^* \cong V_2$, the holonomy $\mathrm{hol}(\mathbf{T}^{\mathrm{s}})$ of a stated tangle $\mathbf{T}^{\mathrm{s}} \in \mathcal{S}_q^{\mathrm{s}}(\Sigma_{g,n}^{\mathrm{o}, \bullet})$ does not depend on the orientation of its strands and is an element of $\mathcal{L}_{g,n}\big( U_{q^2}(\mathfrak{sl}_2) \big) \otimes V_2^{\otimes k}$, where $k$ is the number of boundary points of $\mathbf{T}^{\mathrm{s}}$. Then the isomorphism between $\mathcal{S}_q^{\mathrm{s}}(\Sigma_{g,n}^{\mathrm{o}, \bullet})$ and $\mathcal{L}_{g,n}\big( U_{q^2}(\mathfrak{sl}_2) \big)$ is simply given by reading the component of the tensor $\mathrm{hol}(\mathbf{T}^{\mathrm{s}})$ indexed by the state of $\mathbf{T}^{\mathrm{s}}$. Our result shows that in the case of the surfaces $\Sigma_{g,n}^{\mathrm{o}, \bullet}$, the stated skein algebra is a particular case of the algebra $\mathcal{L}_{g,n}(H)$. So it might be possible to generalize stated skein algebras to other Hopf algebras than $U_{q^2}(\mathfrak{sl}_2)$; in particular the states would be vectors in some representations (for $U_{q^2}(\mathfrak{sl}_2)$ the states $-,+$ correspond to the basis vectors $v_-, v_+$ of $V_2$).

\smallskip

\indent Finally, in section \ref{sectionVacuum} we explain that the stacking representation of $\Sigma_{g,0}^{\mathrm{o}} \times [0,1]$ on $\Sigma_{0,g}^{\mathrm{o}} \times [0,1]$ corresponds through holonomy to the vacuum representation of $\mathcal{L}_{g,0}(H)$ on $\mathcal{L}_{0,g}(H)$ (Theorem \ref{theoremRepVacuum}). As a result, when $H = U_{q^2}(\mathfrak{sl}_2)$, the representation of the (usual) skein algebra $\mathcal{S}_q(\Sigma_{g,0})$ obtained via the Wilson loop map and the representation of $\mathcal{L}_{g,0}^{\mathrm{inv}}\big( U_{q^2}(\mathfrak{sl}_2) \big)$ on $\mathcal{L}_{0,g}^{\mathrm{inv}}\big( U_{q^2}(\mathfrak{sl}_2) \big)$ is equivalent to the obvious representation of $\mathcal{S}_q(\Sigma_{g,0})$ on $\mathcal{S}_q(\Sigma_{0,g+1})$ defined by stacking (Corollary \ref{corollaireRepSkein}). When $q$ is a root of unity $\epsilon$, we obtain more interesting (and finite dimensional) representations thanks to the use of the restricted quantum group $\overline{U}_{\! \epsilon^2}(\mathfrak{sl}_2)$, as discussed in section \ref{sectionRepSkeinAlg}. If one wants to study explicitly these representations of skein algebras at roots of unity in higher genus (here we only describe explicitly the case of the torus $g=1$), one problem is to describe $\mathcal{L}^{\mathrm{inv}}_{0,g}\big( \overline{U}_{\! \epsilon^2}(\mathfrak{sl}_2) \big)$, which for $g>1$ is complicated both as a vector space and as an algebra. We hope that the generalized Wilson loops defined in this paper may help to understand this space (for instance to find an interesting set of generators or even better a basis in which the multiplication has a suitable form).

\smallskip

\begin{remark}
During the preparation of this manuscript, similar results as in Theorem \ref{thIsoSkeinLgn} and Theorem \ref{theoremRepVacuum} were proved independently in \cite{GJS} (see Remark 2.21 and Theorem 6.4 of that paper, respectively, for the corresponding statements). That work used the framework of factorization homology \cite{BZBJ} and its relationship to skein theory established in \cite{cooke, cooke2}, and so our proofs are very different.
\end{remark}

\medskip

\noindent \textbf{Acknowledgments.} ~ The author is partially supported  by the Deutsche Forschungsgemeinschaft (DFG, German Research Foundation) under Germany's Excellence Strategy - EXC 2121 ``Quantum Universe'' - 390833306. A significant part of this work was done when I was a PhD student at the University of Montpellier; I thank my past advisors, S. Baseilhac and P. Roche, for their useful remarks.

\section{Preliminaries}\label{sectionPreliminaries}

\subsection{Surfaces}\label{sectionDefSurfaces}
\indent We denote by $\Sigma_{g,n}$ the compact oriented surface of genus $g$ with $n$ punctures, except in \S \ref{sectionVacuum} and \S \ref{sectionRepSkeinAlg} where $\Sigma_{g,n}$ will be the compact oriented surface of genus $g$ with $n$ open disks removed (\textit{i.e.} we remove open neighborhoods of the punctures). Let $D \subset \Sigma_{g,n}$ be an open disk disjoint from the punctures, then we denote $\Sigma_{g,n}^{\mathrm{o}} = \Sigma_{g,n} \! \setminus \! D$. We put an orientation on the boundary curve induced by the deletion of $D$, as indicated in Figures \ref{surfaceGN} and \ref{grapheSurSurface}.
\begin{figure}[h]
\centering
%% Accompanies image file 'surfaceGN.pdf' (pdf, eps, ps)
\begingroup%
  \makeatletter%
  \providecommand\color[2][]{%
    \errmessage{(Inkscape) Color is used for the text in Inkscape, but the package 'color.sty' is not loaded}%
    \renewcommand\color[2][]{}%
  }%
  \providecommand\transparent[1]{%
    \errmessage{(Inkscape) Transparency is used (non-zero) for the text in Inkscape, but the package 'transparent.sty' is not loaded}%
    \renewcommand\transparent[1]{}%
  }%
  \providecommand\rotatebox[2]{#2}%
  \newcommand*\fsize{\dimexpr\f@size pt\relax}%
  \newcommand*\lineheight[1]{\fontsize{\fsize}{#1\fsize}\selectfont}%
  \ifx\svgwidth\undefined%
    \setlength{\unitlength}{437.43283681bp}%
    \ifx\svgscale\undefined%
      \relax%
    \else%
      \setlength{\unitlength}{\unitlength * \real{\svgscale}}%
    \fi%
  \else%
    \setlength{\unitlength}{\svgwidth}%
  \fi%
  \global\let\svgwidth\undefined%
  \global\let\svgscale\undefined%
  \makeatother%
  \begin{picture}(1,0.22965054)%
    \lineheight{1}%
    \setlength\tabcolsep{0pt}%
    \put(0,0){\includegraphics[width=\unitlength,page=1]{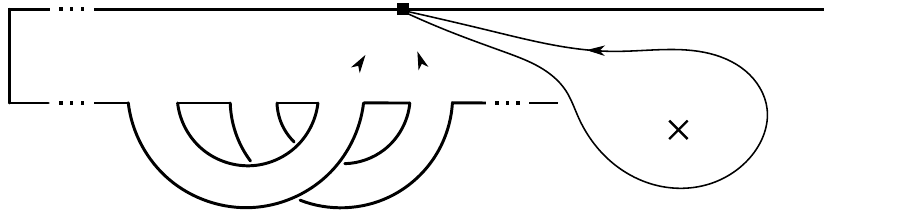}}%
    \put(0.15562339,0.13493008){\color[rgb]{0,0,0}\makebox(0,0)[lt]{\lineheight{1.25}\smash{\begin{tabular}[t]{l}$b_i$\end{tabular}}}}%
    \put(0.47775014,0.1342305){\color[rgb]{0,0,0}\makebox(0,0)[lt]{\lineheight{1.25}\smash{\begin{tabular}[t]{l}$a_i$\end{tabular}}}}%
    \put(0.83188493,0.14938513){\color[rgb]{0,0,0}\makebox(0,0)[lt]{\lineheight{1.25}\smash{\begin{tabular}[t]{l}$m_j$\end{tabular}}}}%
    \put(0,0){\includegraphics[width=\unitlength,page=2]{surfaceGN.pdf}}%
  \end{picture}%
\endgroup%

\caption{Surface $\Sigma_{g,n}^{\mathrm{o}, \bullet}$ and generators of the fundamental group.}
\label{surfaceGN}
\end{figure}

\begin{figure}[h]
\centering
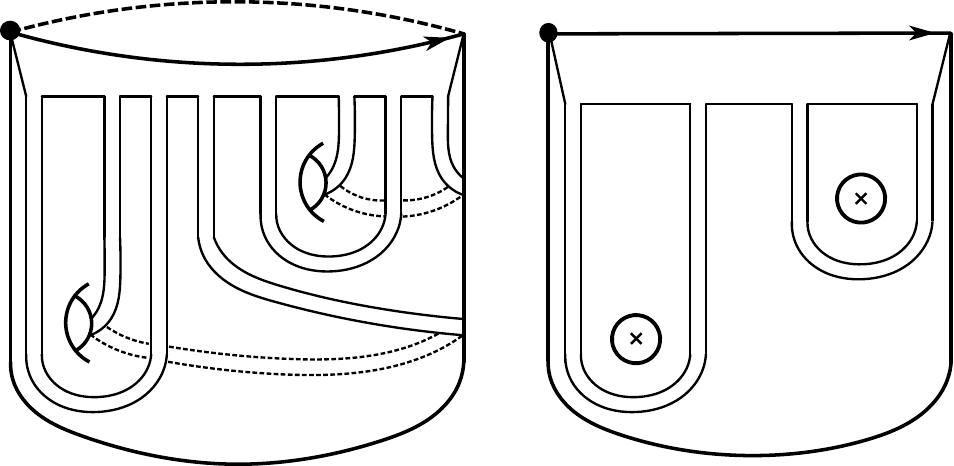
\caption{Surfaces $\Sigma_{2,0}^{\mathrm{o},\bullet}$ and $\Sigma_{0,2}^{\mathrm{o},\bullet}$ with canonical fat graphs embedded on them.}
\label{grapheSurSurface}
\end{figure}

\smallskip

\indent The surface $\Sigma_{g,n}^{\mathrm{o}}$ is represented in Figure \ref{surfaceGN} under a form suitable for our purposes (for the moment do not consider the black dot on the boundary in the figures); we represent both the situations with punctures and with open disks removed. This is not the usual view of $\Sigma_{g,n}^{\mathrm{o}}$. For instance, $\Sigma_{2,0}^{\mathrm{o}}$ and $\Sigma_{0,2}^{\mathrm{o}}$ are represented in a more familiar way in Figure \ref{grapheSurSurface}. To pass from one view to another, retract $\Sigma_{g,n}^{\mathrm{o}}$ to a neighborhood of the canonical loops generating the fundamental group or conversely embed the fat graph of Figure \ref{surfaceGN} in the corresponding surface like in Figure \ref{grapheSurSurface}.

\smallskip

\indent The fundamental group $\pi_1(\Sigma_{g,n}^{\mathrm{o}})$ is a free group generated by the loops $b_1, a_1, \ldots, b_g, a_g, m_{g+1}, \ldots,$ $m_{g+n}$ displayed in Figure \ref{surfaceGN}; moreover, the basepoint of $\pi_1(\Sigma_{g,n}^{\mathrm{o}})$ is represented by a black square dot.

\smallskip

\indent Let $\mathrm{MCG}(\Sigma_{g,n}^{\mathrm{o}})$ be the mapping class group of $\Sigma_{g,n}^{\mathrm{o}}$. For simplicity, we will restrict to $g \geq 1, n= 0$ for the results using or involving the mapping class group. So we recall that $\mathrm{MCG}(\Sigma_{g,0}^{\mathrm{o}})$ is the group of isotopy classes of orientation-preserving diffeomorphisms $\Sigma_{g,0}^{\mathrm{o}} \to \Sigma_{g,0}^{\mathrm{o}}$ which fix the boundary pointwise. A useful feature is that this group is generated by $2g+1$ Dehn twists, called the Humphries generators (see \cite{FM}).

\smallskip

\indent Finally, we let $\Sigma_{g,n}^{\mathrm{o},\bullet}$ be the surface $\Sigma_{g,n}^{\mathrm{o}}$ with one point removed from the boundary curve induced by the deletion of the open disc $D$. This is represented by a big black dot in Figures \ref{surfaceGN}, \ref{grapheSurSurface}. Note that $\pi_1(\Sigma_{g,n}^{\mathrm{o}, \bullet}) = \pi_1(\Sigma_{g,n}^{\mathrm{o}})$ and $\mathrm{MCG}(\Sigma_{g,n}^{\mathrm{o}, \bullet}) = \mathrm{MCG}(\Sigma_{g,n}^{\mathrm{o}})$.

\subsection{Hopf algebras, modules, Reshetikhin-Turaev functor}
\indent Let $H$ be a braided Hopf algebra with coproduct $\Delta$, counit $\varepsilon$, antipode $S$ and universal $R$-matrix $R$ (see \cite[Chap. VIII]{kassel}). We will often write $R = a_i \otimes b_i$ (with implicit summation on $i$) and we set $R' = b_i \otimes a_i$. The Drinfeld element and its inverse are
\begin{equation}\label{elementDrinfeld}
u = S(b_i)a_i, \:\:\:\:\:\:\:\:\: u^{-1} = S^{-2}(b_i)a_i
\end{equation}
It implements $S^2$ by conjugation: $S^2(x) = uxu^{-1}$. In all this paper we assume that $H$ is ribbon, which means that there exists a central and invertible element $v \in H$ such that 
\[ \Delta(v) = (R'R)^{-1} v \otimes v, \:\:\:\:\:\:\: S(v) = v, \:\:\:\:\:\:\:(\text{and thus } \:\:\: \varepsilon(v) =1, \:\:\:\:\:\:\: v^2 = uS(u)\:\text{).} \]
It follows that $H$ is pivotal; indeed, the element $g = uv^{-1}$ satisfies
\begin{equation}\label{pivot}
\Delta(g) = g \otimes g, \:\:\:\:\:\:\:\:\: \forall \, x \in H, \: S^2(x) = gxg^{-1}. 
\end{equation}
For any finite dimensional $H$-module $I$, $g$ provides an isomorphism of $H$-modules
\begin{equation}\label{identificationBidual}
\foncIso{e_I}{I^{**}}{I}{\mathrm{ev}_x}{g^{-1}x}
\end{equation}
where $\mathrm{ev}_x(\varphi) = \varphi(x)$ for $\varphi \in I^*$ and $x \in I$. If $I$ is a finite dimensional $H$-module, we denote by 
\[ \overset{I}{T} : H \to \mathrm{End}_{\mathbb{C}}(I) \cong \mathrm{Mat}_{\dim(I)}(\mathbb{C}) \]
the representation morphism, and by $\overset{I}{h} = \overset{I}{T}(h)$ the representation of $h \in H$ on $I$. The coefficients $\overset{I}{T}{^i_j} \in H^*$ of $\overset{I}{T}$ in some basis are called the matrix coefficients of the $H$-module $I$. The restricted dual of $H$, denoted $H^{\circ}$, is the vector subspace of $H^*$  generated by all the matrix coefficients of all the finite dimensional $H$-modules. We assume that $H^{\circ}$ separates the points, which means that if an equality is true in all the finite dimensional representations then it is true in $H$. We denote by $\mathcal{O}(H)$ the vector space $H^{\circ}$ endowed with the canonical Hopf algebra structure dual to that of $H$.

\smallskip

\indent Let $\mathrm{mod}_l(H)$ be the category of finite dimensional left $H$-modules. The braiding $c_{I,J} : I \otimes J \to J \otimes I$, its inverse $c_{I,J}^{-1} : J \otimes I \to I \otimes J$ and the twist $\theta_I : I \to I$ are defined by
\[ c_{I,J}(x \otimes y) = b_iy \otimes a_ix, \:\:\:\:\:\:\:\:\: c_{I,J}^{-1}(y \otimes x) = S(a_i)x \otimes b_iy, \:\:\:\:\:\:\:\:\: \theta_I(x) = v^{-1}x. \]
The $H$-action on the dual $I^*$ is $(h\varphi)(x) = \varphi\big( S(h)x \big)$ and the duality morphisms $b_I : \mathbb{C} \to I \otimes I^*$, $d_I : I^* \otimes I \to \mathbb{C}$, $b'_I : \mathbb{C} \to I^* \otimes I$, $d'_I : I \otimes I^* \to \mathbb{C}$ are defined by
\[ b_I(1) = v_i \otimes v^i, \:\:\:\:\:\: d_I(\varphi \otimes x) = \varphi(x), \:\:\:\:\:\: b'_I(1) = v^i \otimes g^{-1}v_i, \:\:\:\:\:\: d'_I(x \otimes \varphi) = \varphi(gx) \]
where $(v_i)$ is a basis of $I$, $(v^i)$ is the dual basis and we use Einstein's convention for pairs of indices (implicit summation).

\smallskip

\indent It is well-known (\cite{RT}, see also \cite{KM} and \cite[XIV.5.1]{kassel}) that there is a tensor functor $F_{\mathrm{RT}} : \mathcal{R}\mathcal{G}_H \to \mathrm{mod}_l(H)$, where $\mathcal{R}\mathcal{G}_H$ is the category of $H$-colored ribbon graphs; $F_{\mathrm{RT}}$ is called the Reshetikhin-Turaev functor and takes the following values:
\begin{center}
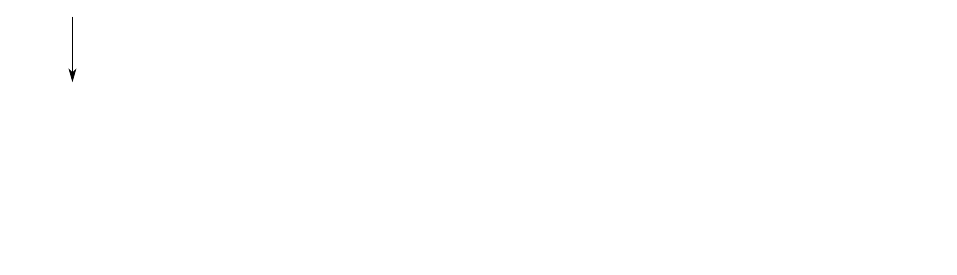
\end{center}

\noindent In the sequel we identify a ribbon graph with its evaluation through $F_{\mathrm{RT}}$. Note that we read diagrams from bottom to top.

\subsection{Algebra $\mathcal{L}_{g,n}(H)$ and related notions}\label{sectionDefLgnH}
\indent Consider $2g+n$ copies of the vector space $H^{\circ}$ (restricted dual, see above) indexed by the simple closed loops $b_i, a_i, m_j$ which generate $\pi_1(\Sigma_{g,n}^{\mathrm{o}})$ (see Figure \ref{surfaceGN}), and denote by $V_{g,n}$ their direct sum:
\[ V_{g,n} = H^{\circ}_{b_1} \oplus H^{\circ}_{a_1} \oplus \ldots \oplus H^{\circ}_{b_g} \oplus H^{\circ}_{a_g} \oplus H^{\circ}_{m_{g+1}} \oplus \ldots \oplus H^{\circ}_{m_{g+n}}. \]
Let $\mathfrak{i}_{b_1}, \mathfrak{i}_{a_1}, \ldots, \mathfrak{i}_{b_g}, \mathfrak{i}_{a_g}, \mathfrak{i}_{m_{g+1}}, \ldots, \mathfrak{i}_{m_{g+n}} : H^{\circ} \to V_{g,n}$ be the canonical injections in each copy. We define
\begin{equation}\label{coeffMatriciels}
\begin{split}
&\overset{I}{B}(1)^i_j = \mathfrak{i}_{b_1}\!\bigl( \overset{I}{T}{^i_j} \bigr), \, \overset{I}{A}(1)^i_j = \mathfrak{i}_{a_1}\!\bigl( \overset{I}{T}{^i_j} \bigr), \ldots, \overset{I}{B}(g)^i_j = \mathfrak{i}_{b_g}\!\bigl( \overset{I}{T}{^i_j} \bigr), \, \overset{I}{A}(g)^i_j = \mathfrak{i}_{a_g}\!\bigl( \overset{I}{T}{^i_j} \bigr), \\
&\overset{I}{M}(g+1)^i_j = \mathfrak{i}_{m_{g+1}}\!\bigl( \overset{I}{T}{^i_j} \bigr), \ldots, \overset{I}{M}(g+n)^i_j = \mathfrak{i}_{m_{g+n}}\!\bigl( \overset{I}{T}{^i_j} \bigr)
\end{split} 
\end{equation}
for all $I,i,j$. The vector space $V_{g,n}$ is generated by these elements. Now consider the tensor algebra $\mathbb{T}(V_{g,n})$; by definition a generic element in $\mathbb{T}(V_{g,n})$ is a linear combination of formal products of the elements introduced in \eqref{coeffMatriciels}. We identify the elements of $V_{g,n}$ with their obvious embedding in $\mathbb{T}(V_{g,n})$. Then $\mathbb{T}(V_{g,n})$ is generated as an algebra by the coefficients of the matrices
\[ \overset{I}{B}(1), \overset{I}{A}(1), \ldots, \overset{I}{B}(g), \overset{I}{A}(g), \overset{I}{M}(g+1), \ldots, \overset{I}{M}(g+n) \in \mathbb{T}(V_{g,n}) \otimes \mathrm{End}_{\mathbb{C}}(I) \]
for all the finite dimensional $H$-modules $I$.

\begin{definition}\label{defLgn}
The algebra $\mathcal{L}_{g,n}(H)$ is the quotient of $\mathbb{T}(V_{g,n})$ by the relations given by the following matrix equalities:
\begin{align}
& \overset{I \otimes J}{X}\!\!(i) = \overset{I}{X}(i)_1\,\overset{IJ}{R'}\,\overset{J}{X}(i)_2\, \overset{IJ}{R'}{^{-1}} \:\:\:\:\:\:\:\:\:\:\:\:\:\:\:\:\:\:\:\:\:\:\:\:\:\:\:\:\:\: \text{ for } 1 \leq i \leq g+n \label{relationFusion}\\
& \overset{IJ}{R}\,\overset{I}{B}(i)_1\, \overset{IJ}{R'}\,\overset{J}{A}(i)_2 = \overset{J}{A}(i)_2\, \overset{IJ}{R}\,\overset{I}{B}(i)_1\, \overset{IJ}{R}{^{-1}} \:\:\:\:\:\:\:\:\:\:\:\:\, \text{ for } 1 \leq i \leq g \label{echangeL10}\\
& \overset{IJ}{R}\,\overset{I}{X}(i)_1\, \overset{IJ}{R}{^{-1}} \,\overset{J}{Y}(j)_2 = \overset{J}{Y}(j)_2\, \overset{IJ}{R}\,\overset{I}{X}(i)_1\, \overset{IJ}{R}{^{-1}} \:\:\:\:\:\:\:\: \text{ for } 1 \leq i < j \leq g+n \label{echangeLgn}
\end{align}
where $X(i)$ is $A(i)$ or $B(i)$ if $1 \leq i \leq g$ and is $M(i)$ if $g+1 \leq i \leq g+n$, and the same applies to $Y(j)$.
\end{definition}
\noindent Let us explain the notations used in this definition. If we have a matrix $\overset{I}{U} \in \mathcal{L}_{g,n}(H) \otimes \mathrm{End}_{\mathbb{C}}(I)$ (resp. $\overset{J}{V} \in \mathcal{L}_{g,n}(H) \otimes \mathrm{End}_{\mathbb{C}}(J)$) then we denote by $\overset{I}{U_1}$ (resp. $\overset{J}{V_2}$) its canonical embedding in $\mathcal{L}_{g,n}(H) \otimes \mathrm{End}_{\mathbb{C}}(I) \otimes \mathrm{End}_{\mathbb{C}}(J)$. In components, this reads $\bigl(\overset{I}{U_1}\bigr)^{ik}_{jl} = \overset{I}{U}{^i_j}\delta^k_l$ (resp. $\bigl(\overset{J}{V_2}\bigr)^{ik}_{jl} = \delta^i_j\overset{J}{V}{^k_l}$). $\overset{IJ}{R} \in \mathrm{End}_{\mathbb{C}}(I) \otimes \mathrm{End}_{\mathbb{C}}(J)$ is the representation of $R \in H \otimes H$ on $I \otimes J$, and we implicitly identify it with $1 \otimes \overset{IJ}{R} \in \mathcal{L}_{g,n}(H) \otimes \mathrm{End}_{\mathbb{C}}(I) \otimes \mathrm{End}_{\mathbb{C}}(J)$. Similarly for $R', R^{-1}, R'^{-1}$ (recall that $R'$ is the flip of $R$). Hence these relations are equalities in $\mathcal{L}_{g,n}(H) \otimes \mathrm{End}_{\mathbb{C}}(I) \otimes \mathrm{End}_{\mathbb{C}}(J)$. To obtain the defining relations of $\mathcal{L}_{g,n}(H)$, one computes the matrix products in each side of these equalities and identifies the coefficients of the resulting matrices. Note that the defining relations may differ slightly from one paper to another, depending on the conventions choosen by the authors. This algebra first appeared in this form in \cite{alekseev, AGS2, AS}; in \cite{AGS, BR} to each fat graph $\Gamma$ describing the surface and satisfying certain properties was associated an algebra $\mathcal{L}_{\Gamma}(H)$.

\smallskip

\indent By construction we have the naturality of the (families of) matrices $\overset{I}{X}(i)$, namely for any $H$-morphism $f : I \to J$ it holds
\begin{equation}\label{fonctorialite}
f \overset{I}{X}(i) = \overset{J}{X}(i) f
\end{equation}
where $X(i)$ is $A(i)$ or $B(i)$ if $1 \leq i \leq g$ and is $M(i)$ if $g+1 \leq i \leq g+n$ (and we identify $f$ with its matrix).

\smallskip

\indent We define a right action $\cdot$ of $H$ on $\mathcal{L}_{g,n}(H)$ by
\begin{equation}\label{actionH}
\overset{I}{X}(i) \cdot h = \overset{I}{h'} \, \overset{I}{X}(i) \, \overset{I}{S(h'')} \:\:\:\:\:\:\:\: (\:\text{\it i.e.} \:\: \overset{I}{X}(i){^ j_k} \cdot h = \overset{I}{h'}{^j_l} \, \overset{I}{X}(i){^l_m} \, \overset{I}{S(h'')}{^m _k}\:)
\end{equation}
where we use Sweedler's notation for the coproduct $\Delta(h) = h' \otimes h''$, Einstein's convention for pairs of indices and we recall that $\overset{I}{x}$ is the representation of $x \in H$ on $I$. Then it is straightforward to check that $\cdot$ endows $\mathcal{L}_{g,n}(H)$ with a structure of (right) $H$-module-algebra. Equivalently, one can work with the associated left $\mathcal{O}(H)$-coaction $\Omega : \mathcal{L}_{g,n}(H) \to \mathcal{O}(H) \otimes \mathcal{L}_{g,n}(H)$:
\begin{equation}\label{coactionH}
\Omega\bigl(\overset{I}{X}(i)\bigr) = \overset{I}{T} \, \overset{I}{X}(i) \, S\bigl(\overset{I}{T}\bigr) \:\:\:\:\:\:\:\: (\:\text{\it i.e.}  \:\: \Omega\bigl(\overset{I}{X}(i)^j_k\bigr) = \overset{I}{T}{^j_l} S\bigl( \overset{I}{T}{^m_k} \bigr) \otimes \overset{I}{X}(i)^l_m \: )
\end{equation}
where in the right-hand side of the first equality we identify $\mathcal{O}(H)$ (resp. $\mathcal{L}_{g,n}(H)$) with the subalgebra $\mathcal{O}(H) \otimes 1$ (resp. $1 \otimes \mathcal{L}_{g,n}(H)$). The subalgebra of invariant elements is
\begin{equation}\label{defInvariants}
\mathcal{L}_{g,n}^{\mathrm{inv}}(H) = \big\{ x \in \mathcal{L}_{g,n}(H) \, \big| \, \forall \, h \in H, \: x \cdot h = \varepsilon(h)x \big\}.
\end{equation}
\noindent Equivalently, an element $x$ is invariant if $\Omega(x) = \varepsilon \otimes x$.

\smallskip

\indent Let $x \in \pi_1(\Sigma_{g,n}^{\mathrm{o}})$ be a simple loop (\textit{i.e.} without self-crossings). We say that $x$ is positively oriented if its orientation looks as follows (this represents a neighborhood of the basepoint in Figure \ref{surfaceGN}):
\begin{equation}\label{positivelyOriented}
%% Accompanies image file 'positivelyOriented.pdf' (pdf, eps, ps)
\begingroup%
  \makeatletter%
  \providecommand\color[2][]{%
    \errmessage{(Inkscape) Color is used for the text in Inkscape, but the package 'color.sty' is not loaded}%
    \renewcommand\color[2][]{}%
  }%
  \providecommand\transparent[1]{%
    \errmessage{(Inkscape) Transparency is used (non-zero) for the text in Inkscape, but the package 'transparent.sty' is not loaded}%
    \renewcommand\transparent[1]{}%
  }%
  \providecommand\rotatebox[2]{#2}%
  \newcommand*\fsize{\dimexpr\f@size pt\relax}%
  \newcommand*\lineheight[1]{\fontsize{\fsize}{#1\fsize}\selectfont}%
  \ifx\svgwidth\undefined%
    \setlength{\unitlength}{78.74999915bp}%
    \ifx\svgscale\undefined%
      \relax%
    \else%
      \setlength{\unitlength}{\unitlength * \real{\svgscale}}%
    \fi%
  \else%
    \setlength{\unitlength}{\svgwidth}%
  \fi%
  \global\let\svgwidth\undefined%
  \global\let\svgscale\undefined%
  \makeatother%
  \begin{picture}(1,0.37273743)%
    \lineheight{1}%
    \setlength\tabcolsep{0pt}%
    \put(0,0){\includegraphics[width=\unitlength,page=1]{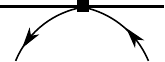}}%
  \end{picture}%
\endgroup%

\end{equation}
The lift $\overset{I}{\widetilde{x}} \in \mathcal{L}_{g,n}(H) \otimes \mathrm{End}_{\mathbb{C}}(I)$ of a positively oriented simple loop $x$ in the representation $I$ is defined as follows: first express $x$ in terms of the generators of $\pi_1(\Sigma_{g,n}^{\mathrm{o}})$, then replace each generator by the corresponding matrix in the representation $I$: $b_1 \mapsto \overset{I}{B}(1), a_1 \mapsto \overset{I}{A}(1), \ldots$ and finally multiply the resulting matrix by $\overset{I}{v}{^{N(x)}}$, where $v \in H$ is the ribbon element and $N(x) \in \mathbb{Z}$ is defined in \cite[\S 5.3.2]{these} (the definition of $N(x)$ is not important for the sequel). For instance, $\overset{I}{\widetilde{b_1}} = \overset{I}{B}(1), \overset{I}{\widetilde{a_1}} = \overset{I}{A}(1), \ldots$. If $x$ is negatively oriented, then $x^{-1}$ is positively oriented and the lift of $x$ is defined as $\overset{I}{\widetilde{x}} = \bigl( \overset{I}{\widetilde{x^{-1}}} \bigr)^{-1}$. Note that in general $\overset{I}{\widetilde{xy}} \neq \overset{I}{\widetilde{x}}\overset{I}{\widetilde{y}}$, due to the fact that $N(xy) \neq N(x) + N(y)$. A remarkable property is that the lifts of a positively-oriented simple loop $x$ satisfy the fusion relation \eqref{relationFusion} \cite[Prop. 5.3.14]{these}:
\[ \overset{I \otimes J}{\widetilde{x}} = \overset{I}{\widetilde{x}}_1\,\overset{IJ}{R'}\,\overset{J}{\widetilde{x}}_2\, \overset{IJ}{R'}{^{-1}}. \]
Without the normalization $\overset{I}{v}{^{N(x)}}$ this last equality would not be true.

\smallskip

\indent As already said, for results involving the mapping class group we will restrict to $g \geq 1, n=0$ for simplicity (the generalization to any $g,n$ is possible but would require more discussion). The lift of $f \in \mathrm{MCG}(\Sigma_{g,0}^{\mathrm{o}})$ is an automorphism $\widetilde{f} : \mathcal{L}_{g,0}(H) \to \mathcal{L}_{g,0}(H)$ defined on generators by
\begin{equation}\label{defLiftMappingClass}
 \widetilde{f}\bigl( \overset{I}{B}(i) \bigr) = \overset{I}{\widetilde{f(b_i)}}, \:\:\: \widetilde{f}\bigl( \overset{I}{A}(i) \bigr) = \overset{I}{\widetilde{f(a_i)}}
\end{equation}
which obviously means $\widetilde{f}\bigl( \overset{I}{B}(i)^k_l \bigr) = \bigl(\overset{I}{\widetilde{f(b_i)}}\bigr)^k_l, \: \widetilde{f}\bigl( \overset{I}{A}(i)^k_l \bigr) = \bigl(\overset{I}{\widetilde{f(a_i)}}\bigr)^k_l$. It holds $\widetilde{f\circ g} = \widetilde{f} \circ \widetilde{g}$. Moreover, for any simple loop $x \in \pi_1(\Sigma_{g,0}^{\mathrm{o}})$:
\begin{equation}\label{pasSurprenant}
\widetilde{f}\bigl( \overset{I}{\widetilde{x}} \bigr) = \overset{I}{\widetilde{f(x)}}.
\end{equation}
\indent Recall that in \cite{Fai18c} we assumed $H$ to be finite dimensional and factorizable in order to obtain a projective representation of $\mathrm{MCG}(\Sigma_{g,0})$ on some subspace $\mathrm{Inv}\big((H^*)^{\otimes g}\big)$ of $(H^*)^{\otimes g}$. Without these assumptions, this representation is \textit{a priori} not defined. Nevertheless the lift of a mapping class can be defined without these assumptions on $H$, as we just did above.

\section{Diagrammatic description of $\mathcal{L}_{g,n}(H)$}\label{sectionDiagramCalculus}

\indent We will define a graphical calculus for the algebras $\mathcal{L}_{g,n}(H)$. The basic observation is that these algebras are defined in a matrix way, and that we have matrices labelled by finite dimensional $H$-modules. Hence we will extend the usual evaluation of $H$-colored ribbon graphs in the sense of \cite{RT} by adding a diagram corresponding to such matrices.

\smallskip

We denote by $\overset{I}{X}$ an element of $\mathcal{L}_{g,n}(H) \otimes \mathrm{End}_{\mathbb{C}}(I)$ (in other words a matrix with coefficients in $\mathcal{L}_{g,n}(H)$ of size $\dim(I)$), where $I$ is a finite dimensional $H$-module. A typical example is a product of the matrices of generators $\overset{I}{A}(i), \overset{I}{B}(j), \overset{I}{M}(k)$, namely:
\begin{equation}\label{matriceGenerale}
\overset{I}{X} = \overset{I}{A}(i_1)^{l_1} \overset{I}{B}(j_1)^{m_1} \overset{I}{M}(k_1)^{n_1} \ldots \overset{I}{A}(i_s)^{l_s} \overset{I}{B}(j_s)^{m_s} \overset{I}{M}(k_s)^{n_s}
\end{equation}
with $l_{\alpha}, m_{\alpha}, n_{\alpha} \in \mathbb{Z}$ and $1 \leq i_{\alpha}, j_{\alpha} \leq g, \: g+1 \leq k_{\alpha} \leq g+n$; for instance $\overset{I}{X} = \overset{I}{M}(3)^{-2}\overset{I}{B}(1){^{-1}}\overset{I}{A}(2)$ (it is not obvious that the matrices $\overset{I}{A}(i), \overset{I}{B}(j), \overset{I}{M}(k)$ are invertible; a simple proof is given in Proposition \ref{propEqRefEtInverse} below). In particular a matrix of the form \eqref{matriceGenerale} is defined for any $I$ and by \eqref{fonctorialite} it satisfies naturality:
\begin{equation}\label{fonctorialiteX}
f \overset{I}{X} = \overset{J}{X} f
\end{equation}
Note that the lift $\overset{I}{\widetilde{x}}$ of a simple loop $x \in \pi_1(\Sigma_{g,n}^{\circ})$ is of the form \eqref{matriceGenerale} except that there is the normalization $\overset{I}{v}{^{N(x)}}$, but since $v$ is central the lifts also satisfies naturality.

\smallskip

\indent Let $(v_i)$ be a basis of $I$ and $(v^j)$ be the dual basis. Then we have an isomorphism of vector spaces
\begin{equation}\label{decMatrice}
\begin{array}{rcl} 
\mathcal{L}_{g,n}(H) \otimes \mathrm{End}_{\mathbb{C}}(I) & \overset{\sim}{\rightarrow} & \mathcal{L}_{g,n}(H) \otimes I \otimes I^* \\
\overset{I}{X} & \mapsto & \overset{I}{X}{^i_j}  \otimes v_i \otimes v^j
\end{array} 
\end{equation}
where we use Einstein's convention for pairs of indices. In this paper, we systematically identify a matrix $\overset{I}{X}$ with $\overset{I}{X}{^i_j} \otimes v_i \otimes v^j$. Such an element will be represented by the following diagram.

\begin{equation}\label{anse}
%% Accompanies image file 'anse.pdf' (pdf, eps, ps)
\begingroup%
  \makeatletter%
  \providecommand\color[2][]{%
    \errmessage{(Inkscape) Color is used for the text in Inkscape, but the package 'color.sty' is not loaded}%
    \renewcommand\color[2][]{}%
  }%
  \providecommand\transparent[1]{%
    \errmessage{(Inkscape) Transparency is used (non-zero) for the text in Inkscape, but the package 'transparent.sty' is not loaded}%
    \renewcommand\transparent[1]{}%
  }%
  \providecommand\rotatebox[2]{#2}%
  \newcommand*\fsize{\dimexpr\f@size pt\relax}%
  \newcommand*\lineheight[1]{\fontsize{\fsize}{#1\fsize}\selectfont}%
  \ifx\svgwidth\undefined%
    \setlength{\unitlength}{128.04744093bp}%
    \ifx\svgscale\undefined%
      \relax%
    \else%
      \setlength{\unitlength}{\unitlength * \real{\svgscale}}%
    \fi%
  \else%
    \setlength{\unitlength}{\svgwidth}%
  \fi%
  \global\let\svgwidth\undefined%
  \global\let\svgscale\undefined%
  \makeatother%
  \begin{picture}(1,0.41064086)%
    \lineheight{1}%
    \setlength\tabcolsep{0pt}%
    \put(0,0){\includegraphics[width=\unitlength,page=1]{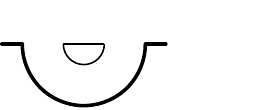}}%
    \put(0.5166977,0.011114){\color[rgb]{0,0,0}\makebox(0,0)[lt]{\lineheight{1.25}\smash{\begin{tabular}[t]{l}$\overset{I}{X}$\end{tabular}}}}%
    \put(0,0){\includegraphics[width=\unitlength,page=2]{anse.pdf}}%
    \put(0.06324935,0.35840362){\color[rgb]{0,0,0}\makebox(0,0)[lt]{\lineheight{1.25}\smash{\begin{tabular}[t]{l}$I$\end{tabular}}}}%
  \end{picture}%
\endgroup%

\end{equation}

\noindent The module $I$ colors the strand while the matrix $\overset{I}{X}$ colors the handle. We define a graphical element corresponding to the other possible orientation of the strand:
\begin{equation}\label{sensOppose}
%% Accompanies image file 'sens_oppose.pdf' (pdf, eps, ps)
\begingroup%
  \makeatletter%
  \providecommand\color[2][]{%
    \errmessage{(Inkscape) Color is used for the text in Inkscape, but the package 'color.sty' is not loaded}%
    \renewcommand\color[2][]{}%
  }%
  \providecommand\transparent[1]{%
    \errmessage{(Inkscape) Transparency is used (non-zero) for the text in Inkscape, but the package 'transparent.sty' is not loaded}%
    \renewcommand\transparent[1]{}%
  }%
  \providecommand\rotatebox[2]{#2}%
  \newcommand*\fsize{\dimexpr\f@size pt\relax}%
  \newcommand*\lineheight[1]{\fontsize{\fsize}{#1\fsize}\selectfont}%
  \ifx\svgwidth\undefined%
    \setlength{\unitlength}{269.69181505bp}%
    \ifx\svgscale\undefined%
      \relax%
    \else%
      \setlength{\unitlength}{\unitlength * \real{\svgscale}}%
    \fi%
  \else%
    \setlength{\unitlength}{\svgwidth}%
  \fi%
  \global\let\svgwidth\undefined%
  \global\let\svgscale\undefined%
  \makeatother%
  \begin{picture}(1,0.30218191)%
    \lineheight{1}%
    \setlength\tabcolsep{0pt}%
    \put(0.32057457,0.10997133){\color[rgb]{0,0,0}\makebox(0,0)[lt]{\lineheight{1.25}\smash{\begin{tabular}[t]{l}=\end{tabular}}}}%
    \put(0,0){\includegraphics[width=\unitlength,page=1]{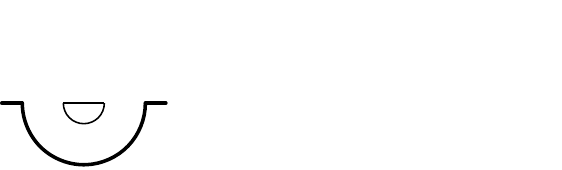}}%
    \put(0.24532379,0.0073012){\color[rgb]{0,0,0}\makebox(0,0)[lt]{\lineheight{1.25}\smash{\begin{tabular}[t]{l}$\overset{I}{X}$\end{tabular}}}}%
    \put(0,0){\includegraphics[width=\unitlength,page=2]{sens_oppose.pdf}}%
    \put(0.24193197,0.17263799){\color[rgb]{0,0,0}\makebox(0,0)[lt]{\lineheight{1.25}\smash{\begin{tabular}[t]{l}$I$\end{tabular}}}}%
    \put(0,0){\includegraphics[width=\unitlength,page=3]{sens_oppose.pdf}}%
    \put(0.62820906,0.00527683){\color[rgb]{0,0,0}\makebox(0,0)[lt]{\lineheight{1.25}\smash{\begin{tabular}[t]{l}$\overset{I^*}{X}$\end{tabular}}}}%
    \put(0,0){\includegraphics[width=\unitlength,page=4]{sens_oppose.pdf}}%
    \put(0.58648591,0.2014189){\color[rgb]{0,0,0}\makebox(0,0)[lt]{\lineheight{1.25}\smash{\begin{tabular}[t]{l}$e_I$\end{tabular}}}}%
    \put(0,0){\includegraphics[width=\unitlength,page=5]{sens_oppose.pdf}}%
    \put(0.41079139,0.20040922){\color[rgb]{0,0,0}\makebox(0,0)[lt]{\lineheight{1.25}\smash{\begin{tabular}[t]{l}$\mathrm{id}_{I^*}$\end{tabular}}}}%
    \put(0.40507797,0.26828752){\color[rgb]{0,0,0}\makebox(0,0)[lt]{\lineheight{1.25}\smash{\begin{tabular}[t]{l}$I$\end{tabular}}}}%
    \put(0.61778801,0.27036015){\color[rgb]{0,0,0}\makebox(0,0)[lt]{\lineheight{1.25}\smash{\begin{tabular}[t]{l}$I$\end{tabular}}}}%
    \put(0.39513994,0.14369681){\color[rgb]{0,0,0}\makebox(0,0)[lt]{\lineheight{1.25}\smash{\begin{tabular}[t]{l}$I^*$\end{tabular}}}}%
    \put(0.69044334,0.10844228){\color[rgb]{0,0,0}\makebox(0,0)[lt]{\lineheight{1.25}\smash{\begin{tabular}[t]{l}$=\overset{I^*}{X}{^i_j} \otimes v^i \otimes g^{-1}v_j$\end{tabular}}}}%
  \end{picture}%
\endgroup%

\end{equation}

\noindent where $e_I: I^{**} \to I$ is the isomorphism \eqref{identificationBidual}. Let us explain this. To define the graphical element on the left, we put a ribbon graph atop the one defined in \eqref{anse}. This ribbon graph represents a morphism $I^* \otimes I^{**} \to I^* \otimes I$ in $\mathrm{mod}_l(H)$, which can be applied to $\overset{I^*}{X}{^i_j} \otimes v^i \otimes \langle ?, v_j \rangle$ and gives the element in $\mathcal{L}_{g,n}(H) \otimes I^* \otimes I$ written at the right. The converse to \eqref{sensOppose} is
\begin{equation}\label{sensOpposeBis}
%% Accompanies image file 'sens_oppose_bis.pdf' (pdf, eps, ps)
\begingroup%
  \makeatletter%
  \providecommand\color[2][]{%
    \errmessage{(Inkscape) Color is used for the text in Inkscape, but the package 'color.sty' is not loaded}%
    \renewcommand\color[2][]{}%
  }%
  \providecommand\transparent[1]{%
    \errmessage{(Inkscape) Transparency is used (non-zero) for the text in Inkscape, but the package 'transparent.sty' is not loaded}%
    \renewcommand\transparent[1]{}%
  }%
  \providecommand\rotatebox[2]{#2}%
  \newcommand*\fsize{\dimexpr\f@size pt\relax}%
  \newcommand*\lineheight[1]{\fontsize{\fsize}{#1\fsize}\selectfont}%
  \ifx\svgwidth\undefined%
    \setlength{\unitlength}{221.76558341bp}%
    \ifx\svgscale\undefined%
      \relax%
    \else%
      \setlength{\unitlength}{\unitlength * \real{\svgscale}}%
    \fi%
  \else%
    \setlength{\unitlength}{\svgwidth}%
  \fi%
  \global\let\svgwidth\undefined%
  \global\let\svgscale\undefined%
  \makeatother%
  \begin{picture}(1,0.37012559)%
    \lineheight{1}%
    \setlength\tabcolsep{0pt}%
    \put(0.38985462,0.13373741){\color[rgb]{0,0,0}\makebox(0,0)[lt]{\lineheight{1.25}\smash{\begin{tabular}[t]{l}=\end{tabular}}}}%
    \put(0,0){\includegraphics[width=\unitlength,page=1]{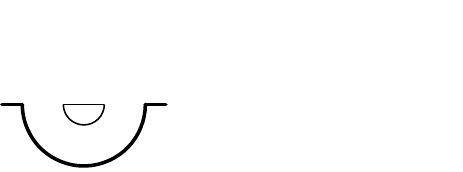}}%
    \put(0.29834121,0.00887905){\color[rgb]{0,0,0}\makebox(0,0)[lt]{\lineheight{1.25}\smash{\begin{tabular}[t]{l}$\overset{I}{X}$\end{tabular}}}}%
    \put(0,0){\includegraphics[width=\unitlength,page=2]{sens_oppose_bis.pdf}}%
    \put(0.76397268,0.00641722){\color[rgb]{0,0,0}\makebox(0,0)[lt]{\lineheight{1.25}\smash{\begin{tabular}[t]{l}$\overset{I^*}{X}$\end{tabular}}}}%
    \put(0,0){\includegraphics[width=\unitlength,page=3]{sens_oppose_bis.pdf}}%
    \put(0.03652018,0.20940395){\color[rgb]{0,0,0}\makebox(0,0)[lt]{\lineheight{1.25}\smash{\begin{tabular}[t]{l}$I$\end{tabular}}}}%
    \put(0,0){\includegraphics[width=\unitlength,page=4]{sens_oppose_bis.pdf}}%
    \put(0.51017843,0.2475864){\color[rgb]{0,0,0}\makebox(0,0)[lt]{\lineheight{1.25}\smash{\begin{tabular}[t]{l}$e_I$\end{tabular}}}}%
    \put(0,0){\includegraphics[width=\unitlength,page=5]{sens_oppose_bis.pdf}}%
    \put(0.70234823,0.24635847){\color[rgb]{0,0,0}\makebox(0,0)[lt]{\lineheight{1.25}\smash{\begin{tabular}[t]{l}$\mathrm{id}_{I^*}$\end{tabular}}}}%
    \put(0.69540021,0.32890613){\color[rgb]{0,0,0}\makebox(0,0)[lt]{\lineheight{1.25}\smash{\begin{tabular}[t]{l}$I$\end{tabular}}}}%
    \put(0.54824536,0.33142667){\color[rgb]{0,0,0}\makebox(0,0)[lt]{\lineheight{1.25}\smash{\begin{tabular}[t]{l}$I$\end{tabular}}}}%
    \put(0,0){\includegraphics[width=\unitlength,page=6]{sens_oppose_bis.pdf}}%
    \put(0.74585067,0.17551393){\color[rgb]{0,0,0}\makebox(0,0)[lt]{\lineheight{1.25}\smash{\begin{tabular}[t]{l}$I^*$\end{tabular}}}}%
  \end{picture}%
\endgroup%

\end{equation}
\noindent Note that the two previous diagrammatic identities are completely parallel to the relation between the values assigned by the Reshetikhin-Turaev functor to cups with different orientations.

\smallskip

\indent The Kronecker product of two matrices $\overset{I}{X}, \overset{J}{Y}$, defined by
\[ \overset{I}{X} \odot \overset{J}{Y} = \overset{I}{X}{^i_j} \overset{J}{Y}{^k_l} \otimes v_i \otimes v^j \otimes w_k \otimes w^l \]
(where $(v_i)$ is a basis of $I$, $(w_k)$ is a basis of $J$ and $(v^j), (w^l)$ are their respective dual bases), is represented by the gluing of the corresponding graphical elements:
\begin{equation}\label{tensorProd}
%% Accompanies image file 'produit_Kronecker.pdf' (pdf, eps, ps)
\begingroup%
  \makeatletter%
  \providecommand\color[2][]{%
    \errmessage{(Inkscape) Color is used for the text in Inkscape, but the package 'color.sty' is not loaded}%
    \renewcommand\color[2][]{}%
  }%
  \providecommand\transparent[1]{%
    \errmessage{(Inkscape) Transparency is used (non-zero) for the text in Inkscape, but the package 'transparent.sty' is not loaded}%
    \renewcommand\transparent[1]{}%
  }%
  \providecommand\rotatebox[2]{#2}%
  \newcommand*\fsize{\dimexpr\f@size pt\relax}%
  \newcommand*\lineheight[1]{\fontsize{\fsize}{#1\fsize}\selectfont}%
  \ifx\svgwidth\undefined%
    \setlength{\unitlength}{203.93260141bp}%
    \ifx\svgscale\undefined%
      \relax%
    \else%
      \setlength{\unitlength}{\unitlength * \real{\svgscale}}%
    \fi%
  \else%
    \setlength{\unitlength}{\svgwidth}%
  \fi%
  \global\let\svgwidth\undefined%
  \global\let\svgscale\undefined%
  \makeatother%
  \begin{picture}(1,0.25942581)%
    \lineheight{1}%
    \setlength\tabcolsep{0pt}%
    \put(0,0){\includegraphics[width=\unitlength,page=1]{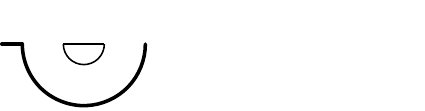}}%
    \put(0.32442983,0.00856639){\color[rgb]{0,0,0}\makebox(0,0)[lt]{\lineheight{1.25}\smash{\begin{tabular}[t]{l}$\overset{I}{X}$\end{tabular}}}}%
    \put(0,0){\includegraphics[width=\unitlength,page=2]{produit_Kronecker.pdf}}%
    \put(0.03971368,0.22662652){\color[rgb]{0,0,0}\makebox(0,0)[lt]{\lineheight{1.25}\smash{\begin{tabular}[t]{l}$I$\end{tabular}}}}%
    \put(0,0){\includegraphics[width=\unitlength,page=3]{produit_Kronecker.pdf}}%
    \put(0.69909389,0.00856627){\color[rgb]{0,0,0}\makebox(0,0)[lt]{\lineheight{1.25}\smash{\begin{tabular}[t]{l}$\overset{J}{Y}$\end{tabular}}}}%
    \put(0,0){\includegraphics[width=\unitlength,page=4]{produit_Kronecker.pdf}}%
    \put(0.41437777,0.22662639){\color[rgb]{0,0,0}\makebox(0,0)[lt]{\lineheight{1.25}\smash{\begin{tabular}[t]{l}$J$\end{tabular}}}}%
  \end{picture}%
\endgroup%

\end{equation}

\begin{definition}\label{defEvaluation}
A general diagram is obtained by gluing (as in \eqref{tensorProd}) several copies of the handle diagrams introduced in \eqref{anse} and \eqref{sensOppose}, and by putting atop an oriented and colored compatible ribbon graph $G$. The evaluation of diagrams is a map $\widetilde{F}_{\mathrm{RT}}$ which consists of applying $\mathrm{id}_{\mathcal{L}_{g,n}(H)} \otimes F_{\mathrm{RT}}(G)$ to the tensors associated to the handle diagrams introduced previously, where $F_{\mathrm{RT}}$ is the Reshetikhin-Turaev functor.  The evaluation of a diagram is an element of $\mathcal{L}_{g,n}(H) \otimes J_1 \otimes \ldots \otimes J_l$, where $J_1, \ldots, J_l$ are $H$-modules. This is depicted as follows:
\begin{center}
%% Accompanies image file 'defEvaluation.pdf' (pdf, eps, ps)
\begingroup%
  \makeatletter%
  \providecommand\color[2][]{%
    \errmessage{(Inkscape) Color is used for the text in Inkscape, but the package 'color.sty' is not loaded}%
    \renewcommand\color[2][]{}%
  }%
  \providecommand\transparent[1]{%
    \errmessage{(Inkscape) Transparency is used (non-zero) for the text in Inkscape, but the package 'transparent.sty' is not loaded}%
    \renewcommand\transparent[1]{}%
  }%
  \providecommand\rotatebox[2]{#2}%
  \newcommand*\fsize{\dimexpr\f@size pt\relax}%
  \newcommand*\lineheight[1]{\fontsize{\fsize}{#1\fsize}\selectfont}%
  \ifx\svgwidth\undefined%
    \setlength{\unitlength}{816.1716278bp}%
    \ifx\svgscale\undefined%
      \relax%
    \else%
      \setlength{\unitlength}{\unitlength * \real{\svgscale}}%
    \fi%
  \else%
    \setlength{\unitlength}{\svgwidth}%
  \fi%
  \global\let\svgwidth\undefined%
  \global\let\svgscale\undefined%
  \makeatother%
  \begin{picture}(1,0.11069484)%
    \lineheight{1}%
    \setlength\tabcolsep{0pt}%
    \put(0,0){\includegraphics[width=\unitlength,page=1]{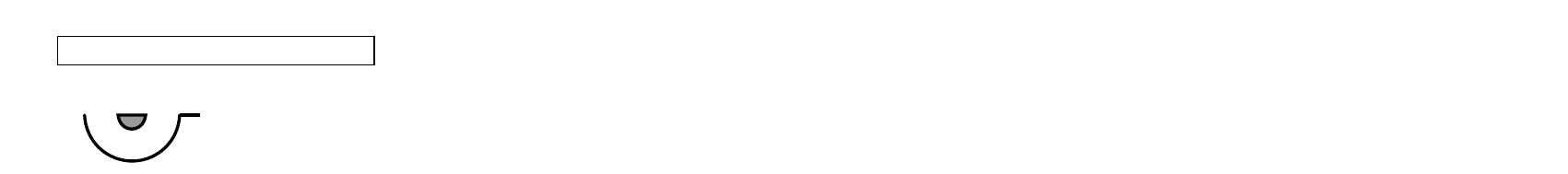}}%
    \put(0.13050593,0.07377892){\color[rgb]{0,0,0}\makebox(0,0)[lt]{\lineheight{1.25}\smash{\begin{tabular}[t]{l}$G$\end{tabular}}}}%
    \put(0.10866731,0.003366){\color[rgb]{0,0,0}\makebox(0,0)[lt]{\lineheight{1.25}\smash{\begin{tabular}[t]{l}$\overset{I_1}{X_1}$\end{tabular}}}}%
    \put(0,0){\includegraphics[width=\unitlength,page=2]{defEvaluation.pdf}}%
    \put(0.21595155,0.00339386){\color[rgb]{0,0,0}\makebox(0,0)[lt]{\lineheight{1.25}\smash{\begin{tabular}[t]{l}$\overset{I_k}{X_k}$\end{tabular}}}}%
    \put(0,0){\includegraphics[width=\unitlength,page=3]{defEvaluation.pdf}}%
    \put(0.484132,0.01581776){\color[rgb]{0,0,0}\makebox(0,0)[lt]{\lineheight{1.25}\smash{\begin{tabular}[t]{l}$\overset{I_1}{X_1}$\end{tabular}}}}%
    \put(0,0){\includegraphics[width=\unitlength,page=4]{defEvaluation.pdf}}%
    \put(0.25499525,0.05106647){\color[rgb]{0,0,0}\makebox(0,0)[lt]{\lineheight{1.25}\smash{\begin{tabular}[t]{l}$= \bigl( \mathrm{id}_{\mathcal{L}_{g,n}(H)} \otimes F_{\mathrm{RT}}(G) \bigr) \circ \widetilde{F}_{\mathrm{RT}}$\end{tabular}}}}%
    \put(0,0){\includegraphics[width=\unitlength,page=5]{defEvaluation.pdf}}%
    \put(0.58473943,0.01581771){\color[rgb]{0,0,0}\makebox(0,0)[lt]{\lineheight{1.25}\smash{\begin{tabular}[t]{l}$\overset{I_k}{X_k}$\end{tabular}}}}%
    \put(0,0){\includegraphics[width=\unitlength,page=6]{defEvaluation.pdf}}%
    \put(-0.00101704,0.0508484){\color[rgb]{0,0,0}\makebox(0,0)[lt]{\lineheight{1.25}\smash{\begin{tabular}[t]{l}$\widetilde{F}_{\mathrm{RT}}$\end{tabular}}}}%
    \put(0,0){\includegraphics[width=\unitlength,page=7]{defEvaluation.pdf}}%
    \put(0.04421411,0.04653505){\color[rgb]{0,0,0}\makebox(0,0)[lt]{\lineheight{1.25}\smash{\begin{tabular}[t]{l}$I_1$\end{tabular}}}}%
    \put(0.15236849,0.04723192){\color[rgb]{0,0,0}\makebox(0,0)[lt]{\lineheight{1.25}\smash{\begin{tabular}[t]{l}$I_k$\end{tabular}}}}%
    \put(0.45196953,0.07175101){\color[rgb]{0,0,0}\makebox(0,0)[lt]{\lineheight{1.25}\smash{\begin{tabular}[t]{l}$I_1$\end{tabular}}}}%
    \put(0.5525955,0.07155787){\color[rgb]{0,0,0}\makebox(0,0)[lt]{\lineheight{1.25}\smash{\begin{tabular}[t]{l}$I_k$\end{tabular}}}}%
    \put(0,0){\includegraphics[width=\unitlength,page=8]{defEvaluation.pdf}}%
    \put(0.06595286,0.10125663){\color[rgb]{0,0,0}\makebox(0,0)[lt]{\lineheight{1.25}\smash{\begin{tabular}[t]{l}$J_1$\end{tabular}}}}%
    \put(0.17879516,0.10116884){\color[rgb]{0,0,0}\makebox(0,0)[lt]{\lineheight{1.25}\smash{\begin{tabular}[t]{l}$J_l$\end{tabular}}}}%
    \put(0,0){\includegraphics[width=\unitlength,page=9]{defEvaluation.pdf}}%
  \end{picture}%
\endgroup%

\end{center}
where the double arrows mean that the strands can carry any orientation.
\end{definition}
\noindent In the sequel, we always identify a diagram with its evaluation through $\widetilde{F}_{\mathrm{RT}}$.

\smallskip

\indent For instance
\begin{center}
%% Accompanies image file 'produit_matriciel.pdf' (pdf, eps, ps)
\begingroup%
  \makeatletter%
  \providecommand\color[2][]{%
    \errmessage{(Inkscape) Color is used for the text in Inkscape, but the package 'color.sty' is not loaded}%
    \renewcommand\color[2][]{}%
  }%
  \providecommand\transparent[1]{%
    \errmessage{(Inkscape) Transparency is used (non-zero) for the text in Inkscape, but the package 'transparent.sty' is not loaded}%
    \renewcommand\transparent[1]{}%
  }%
  \providecommand\rotatebox[2]{#2}%
  \newcommand*\fsize{\dimexpr\f@size pt\relax}%
  \newcommand*\lineheight[1]{\fontsize{\fsize}{#1\fsize}\selectfont}%
  \ifx\svgwidth\undefined%
    \setlength{\unitlength}{388.46909214bp}%
    \ifx\svgscale\undefined%
      \relax%
    \else%
      \setlength{\unitlength}{\unitlength * \real{\svgscale}}%
    \fi%
  \else%
    \setlength{\unitlength}{\svgwidth}%
  \fi%
  \global\let\svgwidth\undefined%
  \global\let\svgscale\undefined%
  \makeatother%
  \begin{picture}(1,0.14828019)%
    \lineheight{1}%
    \setlength\tabcolsep{0pt}%
    \put(0,0){\includegraphics[width=\unitlength,page=1]{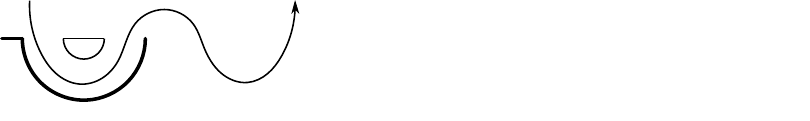}}%
    \put(0.17031425,0.02337738){\color[rgb]{0,0,0}\makebox(0,0)[lt]{\lineheight{1.25}\smash{\begin{tabular}[t]{l}$\overset{I}{X}$\end{tabular}}}}%
    \put(0,0){\includegraphics[width=\unitlength,page=2]{produit_matriciel.pdf}}%
    \put(0.0131945,0.12618147){\color[rgb]{0,0,0}\makebox(0,0)[lt]{\lineheight{1.25}\smash{\begin{tabular}[t]{l}$I$\end{tabular}}}}%
    \put(0,0){\includegraphics[width=\unitlength,page=3]{produit_matriciel.pdf}}%
    \put(0.36699971,0.02337732){\color[rgb]{0,0,0}\makebox(0,0)[lt]{\lineheight{1.25}\smash{\begin{tabular}[t]{l}$\overset{I}{Y}$\end{tabular}}}}%
    \put(0,0){\includegraphics[width=\unitlength,page=4]{produit_matriciel.pdf}}%
    \put(0.41591505,0.09519433){\color[rgb]{0,0,0}\makebox(0,0)[lt]{\lineheight{1.25}\smash{\begin{tabular}[t]{l}$= \overset{I}{X}{^i_j}\overset{I}{Y}{^k_l}\otimes F_{\mathrm{RT}}\Bigg($\end{tabular}}}}%
    \put(0.77009708,0.09677631){\color[rgb]{0,0,0}\makebox(0,0)[lt]{\lineheight{1.25}\smash{\begin{tabular}[t]{l}$\Bigg)\bigl( v_i \otimes v^j \otimes v_k \otimes v^l \bigr)$\end{tabular}}}}%
    \put(0,0){\includegraphics[width=\unitlength,page=5]{produit_matriciel.pdf}}%
    \put(0.63108523,0.10866333){\color[rgb]{0,0,0}\makebox(0,0)[lt]{\lineheight{1.25}\smash{\begin{tabular}[t]{l}$I$\end{tabular}}}}%
    \put(0.7097752,0.11710648){\color[rgb]{0,0,0}\makebox(0,0)[lt]{\lineheight{1.25}\smash{\begin{tabular}[t]{l}$I$\end{tabular}}}}%
    \put(0.75165315,0.07657947){\color[rgb]{0,0,0}\makebox(0,0)[lt]{\lineheight{1.25}\smash{\begin{tabular}[t]{l}$I$\end{tabular}}}}%
    \put(0.42004622,0.00336179){\color[rgb]{0,0,0}\makebox(0,0)[lt]{\lineheight{1.25}\smash{\begin{tabular}[t]{l}$= \overset{I}{X}{^i_j}\overset{I}{Y}{^j_l} \otimes v_i \otimes v^l = \bigl( \overset{I}{X}\overset{I}{Y} \big)^i_l \otimes v_i \otimes v^l$\end{tabular}}}}%
  \end{picture}%
\endgroup%

\end{center}
and we see that this diagram represents the matrix product $\overset{I}{X} \overset{I}{Y}$. Similarly
\begin{equation}\label{traceQuantique}
%% Accompanies image file 'trace_quantique.pdf' (pdf, eps, ps)
\begingroup%
  \makeatletter%
  \providecommand\color[2][]{%
    \errmessage{(Inkscape) Color is used for the text in Inkscape, but the package 'color.sty' is not loaded}%
    \renewcommand\color[2][]{}%
  }%
  \providecommand\transparent[1]{%
    \errmessage{(Inkscape) Transparency is used (non-zero) for the text in Inkscape, but the package 'transparent.sty' is not loaded}%
    \renewcommand\transparent[1]{}%
  }%
  \providecommand\rotatebox[2]{#2}%
  \newcommand*\fsize{\dimexpr\f@size pt\relax}%
  \newcommand*\lineheight[1]{\fontsize{\fsize}{#1\fsize}\selectfont}%
  \ifx\svgwidth\undefined%
    \setlength{\unitlength}{471.89152586bp}%
    \ifx\svgscale\undefined%
      \relax%
    \else%
      \setlength{\unitlength}{\unitlength * \real{\svgscale}}%
    \fi%
  \else%
    \setlength{\unitlength}{\svgwidth}%
  \fi%
  \global\let\svgwidth\undefined%
  \global\let\svgscale\undefined%
  \makeatother%
  \begin{picture}(1,0.11752761)%
    \lineheight{1}%
    \setlength\tabcolsep{0pt}%
    \put(0,0){\includegraphics[width=\unitlength,page=1]{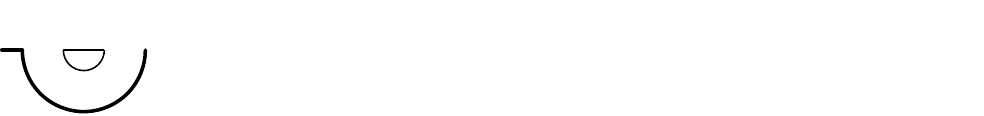}}%
    \put(0.14020555,0.00301578){\color[rgb]{0,0,0}\makebox(0,0)[lt]{\lineheight{1.25}\smash{\begin{tabular}[t]{l}$\overset{I}{X}$\end{tabular}}}}%
    \put(0,0){\includegraphics[width=\unitlength,page=2]{trace_quantique.pdf}}%
    \put(0.01086193,0.09019183){\color[rgb]{0,0,0}\makebox(0,0)[lt]{\lineheight{1.25}\smash{\begin{tabular}[t]{l}$I$\end{tabular}}}}%
    \put(0,0){\includegraphics[width=\unitlength,page=3]{trace_quantique.pdf}}%
    \put(0.18078562,0.05882696){\color[rgb]{0,0,0}\makebox(0,0)[lt]{\lineheight{1.25}\smash{\begin{tabular}[t]{l}$= \overset{I}{X}{^i_j}\otimes F_{\mathrm{RT}}\bigg($\end{tabular}}}}%
    \put(0,0){\includegraphics[width=\unitlength,page=4]{trace_quantique.pdf}}%
    \put(0.36173004,0.07758364){\color[rgb]{0,0,0}\makebox(0,0)[lt]{\lineheight{1.25}\smash{\begin{tabular}[t]{l}$I$\end{tabular}}}}%
    \put(0.37899821,0.05912639){\color[rgb]{0,0,0}\makebox(0,0)[lt]{\lineheight{1.25}\smash{\begin{tabular}[t]{l}$\bigg)\bigl( v_i \otimes v^j \bigr) = \overset{I}{X}{^i_j} \, v^j(gv_i) = \overset{I}{X}{^i_j} \, \overset{I}{g}{^j_i}= \mathrm{tr}\bigl(\overset{I}{g} \overset{I}{X}\bigr) = \mathrm{tr}_q\bigl(\overset{I}{X}\bigr)$\end{tabular}}}}%
  \end{picture}%
\endgroup%

\end{equation}
represents the quantum trace of $\overset{I}{X}$.

\begin{proposition}\label{propDefRelDiagramme}
1. The defining relations \eqref{relationFusion}, \eqref{echangeL10}, \eqref{echangeLgn} of $\mathcal{L}_{g,n}(H)$ are respectively equivalent to the diagrammatic relations \eqref{dessinRelationFusion}, \eqref{dessinEchangeL10}, \eqref{dessinEchangeLgn} below.
\begin{equation}\label{dessinRelationFusion}
%% Accompanies image file 'relation_fusion.pdf' (pdf, eps, ps)
\begingroup%
  \makeatletter%
  \providecommand\color[2][]{%
    \errmessage{(Inkscape) Color is used for the text in Inkscape, but the package 'color.sty' is not loaded}%
    \renewcommand\color[2][]{}%
  }%
  \providecommand\transparent[1]{%
    \errmessage{(Inkscape) Transparency is used (non-zero) for the text in Inkscape, but the package 'transparent.sty' is not loaded}%
    \renewcommand\transparent[1]{}%
  }%
  \providecommand\rotatebox[2]{#2}%
  \newcommand*\fsize{\dimexpr\f@size pt\relax}%
  \newcommand*\lineheight[1]{\fontsize{\fsize}{#1\fsize}\selectfont}%
  \ifx\svgwidth\undefined%
    \setlength{\unitlength}{324.26536176bp}%
    \ifx\svgscale\undefined%
      \relax%
    \else%
      \setlength{\unitlength}{\unitlength * \real{\svgscale}}%
    \fi%
  \else%
    \setlength{\unitlength}{\svgwidth}%
  \fi%
  \global\let\svgwidth\undefined%
  \global\let\svgscale\undefined%
  \makeatother%
  \begin{picture}(1,0.27784231)%
    \lineheight{1}%
    \setlength\tabcolsep{0pt}%
    \put(0,0){\includegraphics[width=\unitlength,page=1]{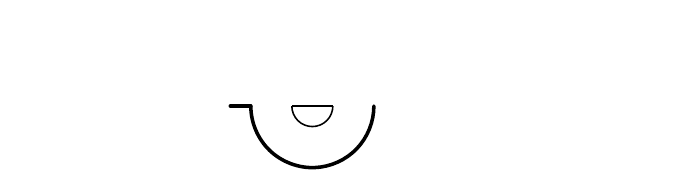}}%
    \put(0.5423006,0.02833682){\color[rgb]{0,0,0}\makebox(0,0)[lt]{\lineheight{1.25}\smash{\begin{tabular}[t]{l}$\overset{I}{X}(i)$\end{tabular}}}}%
    \put(0,0){\includegraphics[width=\unitlength,page=2]{relation_fusion.pdf}}%
    \put(0.37487312,0.24323134){\color[rgb]{0,0,0}\makebox(0,0)[lt]{\lineheight{1.25}\smash{\begin{tabular}[t]{l}$I$\end{tabular}}}}%
    \put(0.4801692,0.24323134){\color[rgb]{0,0,0}\makebox(0,0)[lt]{\lineheight{1.25}\smash{\begin{tabular}[t]{l}$J$\end{tabular}}}}%
    \put(0.29464792,0.11081663){\color[rgb]{0,0,0}\makebox(0,0)[lt]{\lineheight{1.25}\smash{\begin{tabular}[t]{l}$=$\end{tabular}}}}%
    \put(0,0){\includegraphics[width=\unitlength,page=3]{relation_fusion.pdf}}%
    \put(0.22381484,0.00615701){\color[rgb]{0,0,0}\makebox(0,0)[lt]{\lineheight{1.25}\smash{\begin{tabular}[t]{l}$\overset{I \otimes J}{X}\!\!(i)$\end{tabular}}}}%
    \put(0,0){\includegraphics[width=\unitlength,page=4]{relation_fusion.pdf}}%
    \put(0.07532627,0.16008484){\color[rgb]{0,0,0}\makebox(0,0)[lt]{\lineheight{1.25}\smash{\begin{tabular}[t]{l}$J$\end{tabular}}}}%
    \put(0,0){\includegraphics[width=\unitlength,page=5]{relation_fusion.pdf}}%
    \put(0.02265235,0.15818254){\color[rgb]{0,0,0}\makebox(0,0)[lt]{\lineheight{1.25}\smash{\begin{tabular}[t]{l}$I$\end{tabular}}}}%
    \put(0,0){\includegraphics[width=\unitlength,page=6]{relation_fusion.pdf}}%
    \put(0.78623211,0.02851531){\color[rgb]{0,0,0}\makebox(0,0)[lt]{\lineheight{1.25}\smash{\begin{tabular}[t]{l}$\overset{J}{X}(i)$\end{tabular}}}}%
    \put(0.90727,0.11221024){\color[rgb]{0,0,0}\makebox(0,0)[lt]{\lineheight{1.25}\smash{\begin{tabular}[t]{l}$(\forall \, i)$\end{tabular}}}}%
  \end{picture}%
\endgroup%

\end{equation}
\begin{equation}\label{dessinEchangeL10}
%% Accompanies image file 'echange_L10.pdf' (pdf, eps, ps)
\begingroup%
  \makeatletter%
  \providecommand\color[2][]{%
    \errmessage{(Inkscape) Color is used for the text in Inkscape, but the package 'color.sty' is not loaded}%
    \renewcommand\color[2][]{}%
  }%
  \providecommand\transparent[1]{%
    \errmessage{(Inkscape) Transparency is used (non-zero) for the text in Inkscape, but the package 'transparent.sty' is not loaded}%
    \renewcommand\transparent[1]{}%
  }%
  \providecommand\rotatebox[2]{#2}%
  \newcommand*\fsize{\dimexpr\f@size pt\relax}%
  \newcommand*\lineheight[1]{\fontsize{\fsize}{#1\fsize}\selectfont}%
  \ifx\svgwidth\undefined%
    \setlength{\unitlength}{396.55940089bp}%
    \ifx\svgscale\undefined%
      \relax%
    \else%
      \setlength{\unitlength}{\unitlength * \real{\svgscale}}%
    \fi%
  \else%
    \setlength{\unitlength}{\svgwidth}%
  \fi%
  \global\let\svgwidth\undefined%
  \global\let\svgscale\undefined%
  \makeatother%
  \begin{picture}(1,0.24815658)%
    \lineheight{1}%
    \setlength\tabcolsep{0pt}%
    \put(0,0){\includegraphics[width=\unitlength,page=1]{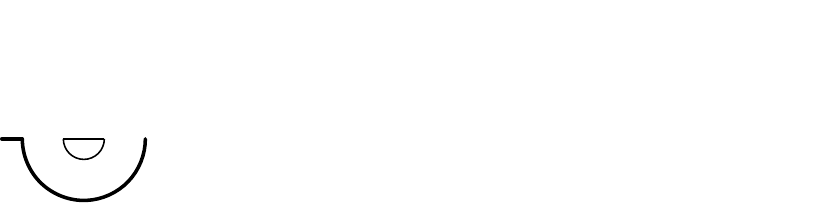}}%
    \put(0.16683967,0.00440525){\color[rgb]{0,0,0}\makebox(0,0)[lt]{\lineheight{1.25}\smash{\begin{tabular}[t]{l}$\overset{J}{A}(i)$\end{tabular}}}}%
    \put(0,0){\includegraphics[width=\unitlength,page=2]{echange_L10.pdf}}%
    \put(0.36542284,0.00440525){\color[rgb]{0,0,0}\makebox(0,0)[lt]{\lineheight{1.25}\smash{\begin{tabular}[t]{l}$\overset{I}{B}(i)$\end{tabular}}}}%
    \put(0,0){\includegraphics[width=\unitlength,page=3]{echange_L10.pdf}}%
    \put(0.62665421,0.00440525){\color[rgb]{0,0,0}\makebox(0,0)[lt]{\lineheight{1.25}\smash{\begin{tabular}[t]{l}$\overset{I}{B}(i)$\end{tabular}}}}%
    \put(0,0){\includegraphics[width=\unitlength,page=4]{echange_L10.pdf}}%
    \put(0.82523726,0.00440525){\color[rgb]{0,0,0}\makebox(0,0)[lt]{\lineheight{1.25}\smash{\begin{tabular}[t]{l}$\overset{J}{A}(i)$\end{tabular}}}}%
    \put(0,0){\includegraphics[width=\unitlength,page=5]{echange_L10.pdf}}%
    \put(0.48408267,0.21986887){\color[rgb]{0,0,0}\makebox(0,0)[lt]{\lineheight{1.25}\smash{\begin{tabular}[t]{l}$J$\end{tabular}}}}%
    \put(0.69428752,0.22725321){\color[rgb]{0,0,0}\makebox(0,0)[lt]{\lineheight{1.25}\smash{\begin{tabular}[t]{l}$I$\end{tabular}}}}%
    \put(0,0){\includegraphics[width=\unitlength,page=6]{echange_L10.pdf}}%
    \put(0.41997176,0.07192132){\color[rgb]{0,0,0}\makebox(0,0)[lt]{\lineheight{1.25}\smash{\begin{tabular}[t]{l}$=$\end{tabular}}}}%
    \put(0.92843372,0.07298846){\color[rgb]{0,0,0}\makebox(0,0)[lt]{\lineheight{1.25}\smash{\begin{tabular}[t]{l}$(\forall \, i)$\end{tabular}}}}%
    \put(0,0){\includegraphics[width=\unitlength,page=7]{echange_L10.pdf}}%
    \put(0.22009753,0.10973775){\color[rgb]{0,0,0}\makebox(0,0)[lt]{\lineheight{1.25}\smash{\begin{tabular}[t]{l}$I$\end{tabular}}}}%
    \put(0,0){\includegraphics[width=\unitlength,page=8]{echange_L10.pdf}}%
    \put(0.01953613,0.10866717){\color[rgb]{0,0,0}\makebox(0,0)[lt]{\lineheight{1.25}\smash{\begin{tabular}[t]{l}$J$\end{tabular}}}}%
  \end{picture}%
\endgroup%

\end{equation}
\begin{equation}\label{dessinEchangeLgn}
%% Accompanies image file 'echange_Lgn.pdf' (pdf, eps, ps)
\begingroup%
  \makeatletter%
  \providecommand\color[2][]{%
    \errmessage{(Inkscape) Color is used for the text in Inkscape, but the package 'color.sty' is not loaded}%
    \renewcommand\color[2][]{}%
  }%
  \providecommand\transparent[1]{%
    \errmessage{(Inkscape) Transparency is used (non-zero) for the text in Inkscape, but the package 'transparent.sty' is not loaded}%
    \renewcommand\transparent[1]{}%
  }%
  \providecommand\rotatebox[2]{#2}%
  \newcommand*\fsize{\dimexpr\f@size pt\relax}%
  \newcommand*\lineheight[1]{\fontsize{\fsize}{#1\fsize}\selectfont}%
  \ifx\svgwidth\undefined%
    \setlength{\unitlength}{396.15031199bp}%
    \ifx\svgscale\undefined%
      \relax%
    \else%
      \setlength{\unitlength}{\unitlength * \real{\svgscale}}%
    \fi%
  \else%
    \setlength{\unitlength}{\svgwidth}%
  \fi%
  \global\let\svgwidth\undefined%
  \global\let\svgscale\undefined%
  \makeatother%
  \begin{picture}(1,0.25659102)%
    \lineheight{1}%
    \setlength\tabcolsep{0pt}%
    \put(0.48961829,0.22877856){\color[rgb]{0,0,0}\makebox(0,0)[lt]{\lineheight{1.25}\smash{\begin{tabular}[t]{l}$J$\end{tabular}}}}%
    \put(0.69799171,0.2347202){\color[rgb]{0,0,0}\makebox(0,0)[lt]{\lineheight{1.25}\smash{\begin{tabular}[t]{l}$I$\end{tabular}}}}%
    \put(0,0){\includegraphics[width=\unitlength,page=1]{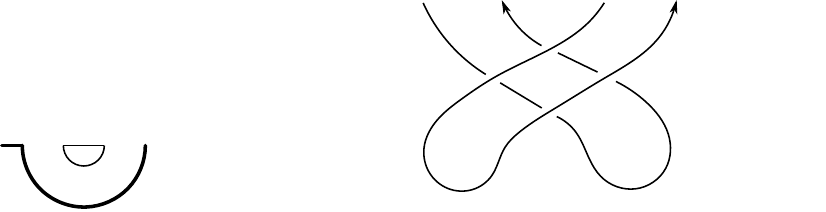}}%
    \put(0.16701199,0.00458189){\color[rgb]{0,0,0}\makebox(0,0)[lt]{\lineheight{1.25}\smash{\begin{tabular}[t]{l}$\overset{J}{Y}(j)$\end{tabular}}}}%
    \put(0,0){\includegraphics[width=\unitlength,page=2]{echange_Lgn.pdf}}%
    \put(0.36580015,0.00458189){\color[rgb]{0,0,0}\makebox(0,0)[lt]{\lineheight{1.25}\smash{\begin{tabular}[t]{l}$\overset{I}{X}(i)$\end{tabular}}}}%
    \put(0,0){\includegraphics[width=\unitlength,page=3]{echange_Lgn.pdf}}%
    \put(0.62730134,0.00458189){\color[rgb]{0,0,0}\makebox(0,0)[lt]{\lineheight{1.25}\smash{\begin{tabular}[t]{l}$\overset{I}{X}(i)$\end{tabular}}}}%
    \put(0,0){\includegraphics[width=\unitlength,page=4]{echange_Lgn.pdf}}%
    \put(0.82608974,0.00458189){\color[rgb]{0,0,0}\makebox(0,0)[lt]{\lineheight{1.25}\smash{\begin{tabular}[t]{l}$\overset{J}{Y}(j)$\end{tabular}}}}%
    \put(0,0){\includegraphics[width=\unitlength,page=5]{echange_Lgn.pdf}}%
    \put(0.42288591,0.07216772){\color[rgb]{0,0,0}\makebox(0,0)[lt]{\lineheight{1.25}\smash{\begin{tabular}[t]{l}=\end{tabular}}}}%
    \put(0.92910374,0.07520578){\color[rgb]{0,0,0}\makebox(0,0)[lt]{\lineheight{1.25}\smash{\begin{tabular}[t]{l}$(i < j)$\end{tabular}}}}%
    \put(0,0){\includegraphics[width=\unitlength,page=6]{echange_Lgn.pdf}}%
    \put(0.21955993,0.10449532){\color[rgb]{0,0,0}\makebox(0,0)[lt]{\lineheight{1.25}\smash{\begin{tabular}[t]{l}$I$\end{tabular}}}}%
    \put(0,0){\includegraphics[width=\unitlength,page=7]{echange_Lgn.pdf}}%
    \put(0.01882439,0.10545148){\color[rgb]{0,0,0}\makebox(0,0)[lt]{\lineheight{1.25}\smash{\begin{tabular}[t]{l}$J$\end{tabular}}}}%
  \end{picture}%
\endgroup%

\end{equation}
We recall that $X(i)$ is $A(i)$ or $B(i)$ if $1 \leq i \leq g$ and is $M(i)$ if $g+1 \leq i \leq g+n$, and the same applies to $Y(j)$.
\\
2. The naturality relation \eqref{fonctorialiteX} is equivalent to each one of the following diagrammatic relations:
\begin{equation}\label{naturalite}
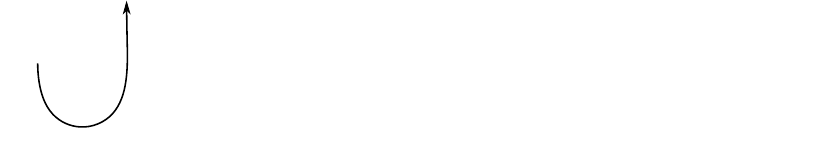
\end{equation}
where $\overset{I}{X}$ is any matrix of the form \eqref{matriceGenerale}.
\end{proposition}
\begin{proof}
1. First note that for any $a,b \in H$, it holds
\begin{equation}\label{trucEvidentdAlgebreLineaire}
\bigl(\overset{I}{a} \overset{I}{X}(i) \overset{I}{b}\bigr)^j_k \otimes v_j \otimes v^k = \overset{I}{X}(i){^j_k} \otimes av_j \otimes S^{-1}(b)v^k.
\end{equation}
To show \eqref{dessinRelationFusion}, write $R = a_{\alpha} \otimes b_{\alpha}$ (with implicit summation on $\alpha$), so that \eqref{relationFusion} becomes
 \[ \overset{I \otimes J}{X}\!(i) =  \left(\overset{I}{X}(i) \, \overset{I}{b_{\alpha}} \, \overset{I}{b_{\beta}}\right)_1 \left(\overset{J}{a_{\alpha}} \, \overset{J}{X}(i) \, \overset{J}{S(a_{\beta})}\right)_2. \]
Hence, by \eqref{trucEvidentdAlgebreLineaire} and the definition of the braiding in $\mathrm{mod}_l(H)$ we have
\begin{align*}
&\overset{I \otimes J}{X}\!(i){^{km}_{ln}} \otimes v_k \otimes w_m \otimes w^n \otimes v^l\\
&= \left(\overset{I}{X}(i) \, \overset{I}{b_{\alpha}} \, \overset{I}{b_{\beta}}\right)^k_l \left(\overset{J}{a_{\alpha}} \, \overset{J}{X}(i) \, \overset{J}{S(a_{\beta})}\right)^m_n \otimes v_k \otimes w_m \otimes w^n \otimes v^l \\
&= \overset{I}{X}(i){^k_l} \overset{J}{X}(i){^m_n} \otimes v_k \otimes a_{\alpha}w_m \otimes a_{\beta} w^n \otimes S^{-1}(b_{\alpha} b_{\beta}) v^l\\
&= \overset{I}{X}(i){^k_l} \overset{J}{X}(i){^m_n} \otimes v_k \otimes S(a_{\alpha}) w_m \otimes S(a_{\beta}) w^n \otimes b_{\beta} b_{\alpha} v^l\\
&= \overset{I}{X}(i){^k_l} \overset{J}{X}(i){^m_n} \otimes \left(\mathrm{id}_I \otimes \mathrm{id}_J \otimes c_{J^*, I^*}^{-1}\right)\!\left(v_k \otimes S(a_{\alpha}) w_m \otimes b_{\alpha} v^l \otimes w^n \right)\\
&= \overset{I}{X}(i){^k_l} \overset{J}{X}(i){^m_n} \otimes \left(\mathrm{id}_I \otimes \mathrm{id}_J \otimes c_{J^*, I^*}^{-1} \right) \circ \left( \mathrm{id}_I \otimes c_{J,I^*}^{-1} \otimes \mathrm{id}_{J^*}\right)\!\left(v_k \otimes v^l \otimes w_m \otimes w^n \right).
\end{align*}
This corresponds to the desired diagrammatic relation. To show \eqref{dessinEchangeL10} we first have to rewrite \eqref{echangeL10}:
\[ \overset{IJ}{R} \, \overset{I}{B}(i)_1 \, \overset{IJ}{R'} \, \overset{J}{A}(i)_2 \overset{IJ}{R} = \overset{J}{A}(i)_2 \, \overset{IJ}{R}  \, \overset{I}{B}(i)_1 = \overset{J}{A}(i)_2 \, (\overset{I}{a_{\alpha}})_1 \, (\overset{J}{b_{\alpha}})_2 \, \overset{I}{B}(i)_1 = (\overset{I}{a_{\alpha}})_1 \, \overset{J}{A}(i)_2 \, \overset{I}{B}(i)_1 \, (\overset{J}{b_{\alpha}})_2. \]
Using the properties of the antipode $S$ and of $R$ (see \cite{kassel}), we see that $a_{\alpha} a_{\beta} \otimes b_{\beta} S(b_{\alpha}) = 1 \otimes 1$; hence we get
\[ \overset{J}{A}(i)_2 \, \overset{I}{B}(i)_1 = (\overset{I}{a_{\alpha}})_1 \, \overset{IJ}{R} \, \overset{I}{B}(i)_1 \, \overset{IJ}{R'} \, \overset{J}{A}(i)_2 \overset{IJ}{R} \, \overset{J}{S(b_{\alpha})}_2 = \left( \overset{I}{a_{\alpha}} \, \overset{I}{a_{\beta}} \, \overset{I}{B}(i) \, \overset{I}{b_{\gamma}} \, \overset{I}{a_{\delta}} \right)_1 \left( \overset{J}{b_{\beta}} \, \overset{J}{a_{\gamma}} \, \overset{J}{A}(i) \, \overset{J}{b_{\delta}} \, \overset{J}{S(b_{\alpha})} \right)_2. \]
By \eqref{trucEvidentdAlgebreLineaire}, the fact that $(S\otimes S)(R) = R$ and the definition of the braiding in $\mathrm{mod}_l(H)$ we get
\begin{align*}
& \: \overset{J}{A}(i)^k_l \overset{I}{B}(i)^m_n \otimes w_k \otimes w^l \otimes v_m \otimes v^n \\
=& \: \left( \overset{I}{a_{\alpha}} \, \overset{I}{a_{\beta}} \, \overset{I}{B}(i) \, \overset{I}{b_{\gamma}} \, \overset{I}{a_{\delta}} \right)^m_n \left( \overset{J}{b_{\beta}} \, \overset{J}{a_{\gamma}} \, \overset{J}{A}(i) \, \overset{J}{b_{\delta}} \, \overset{J}{S(b_{\alpha})} \right)^k_l \otimes w_k \otimes w^l \otimes v_m \otimes v^n\\
=& \: \overset{I}{B}(i){^m_n} \overset{J}{A}(i){^k_l} \otimes b_{\beta} a_{\gamma} w_k \otimes S^{-1}\!\left(b_{\delta}S(b_{\alpha})\right)w^l \otimes a_{\alpha}a_{\beta} v_m \otimes S^{-1}(b_{\gamma} a_{\delta}) v_n\\
=& \: \overset{I}{B}(i){^m_n} \overset{J}{A}(i){^k_l} \otimes b_{\beta}S(a_{\gamma})w_k \otimes b_{\alpha} b_{\delta}w^l \otimes a_{\alpha}a_{\beta} v_m \otimes a_{\delta} b_{\gamma} v^n\\
=& \: \overset{I}{B}(i){^m_n} \overset{J}{A}(i){^k_l} \otimes (\mathrm{id}_I \otimes c_{J,I^*} \otimes \mathrm{id}_{J^*})\bigl(b_{\beta}S(a_{\gamma})w_k \otimes a_{\beta} v_m \otimes b_{\delta}w^l \otimes a_{\delta} b_{\gamma} v^n\bigr)\\
=& \: \overset{I}{B}(i){^m_n} \overset{J}{A}(i){^k_l} \otimes (\mathrm{id}_I \otimes c_{J,I^*} \otimes \mathrm{id}_{J^*}) \circ (c_{J,I} \otimes c_{J^*, I^*})\bigl(v_m \otimes S(a_{\gamma})w_k \otimes b_{\gamma} v^n \otimes w^l\bigr)\\
=& \: \overset{I}{B}(i){^m_n} \overset{J}{A}(i){^k_l} \otimes (\mathrm{id}_I \otimes c_{J,I^*} \otimes \mathrm{id}_{J^*}) \circ (c_{J,I} \otimes c_{J^*, I^*}) \circ (\mathrm{id}_J \otimes c_{I,J^*}^{-1} \otimes \mathrm{id}_{I^*})\bigl(v_m \otimes v^n \otimes w_k \otimes w^l\bigr)
\end{align*}
which corresponds to the desired diagrammatic relation. \eqref{dessinEchangeLgn} is proven in a similar way.
\\2. By \eqref{fonctorialiteX}, we have
\begin{align*}
\overset{I}{X}{^i_j} \otimes (f \otimes \mathrm{id}_{I^*})(v_i \otimes v^j)  &= \overset{I}{X}{^i_j} \otimes f_i^kv_k \otimes v^j = \bigl(f\overset{I}{X}\bigr)^k_j \otimes v_k \otimes v^j = \bigl(\overset{J}{X}f\bigr)^k_j \otimes w_k \otimes w^j\\
& = \overset{J}{X}{^k_l} \otimes w_k \otimes f^l_jw^j = \overset{J}{X}{^k_l} \otimes (\mathrm{id}_I \otimes f^*)(w_k \otimes w^l)
\end{align*}
where $f^* : J^* \to I^*$ is the transpose of $f$. This gives the first diagram below. The second diagram is equivalent to the first thanks to \eqref{sensOppose} and the equality $f \circ e_I = e_J \circ f^{**}$.
\end{proof}

\indent This diagrammatic reformulation of the defining relations of $\mathcal{L}_{g,n}(H)$ easily implies two basic well-known properties of the matrices $\overset{I}{B}(i), \overset{I}{A}(i), \overset{I}{M}(i)$:

\begin{proposition}\label{propEqRefEtInverse}
Let $X(i)$ be $A(i)$ or $B(i)$ if $1 \leq i \leq g$ and be $M(i)$ if $g+1 \leq i \leq g+n$. 
\\1. The reflection equation $\overset{IJ}{R}_{12} \overset{I}{X}(i)_1 (\overset{IJ}{R'})_{12} \overset{J}{X}(i)_2 = \overset{J}{X}(i)_2 \overset{IJ}{R}_{12} \overset{I}{X}(i)_1 (\overset{IJ}{R'})_{12}$ holds and is equivalent to the following diagrammatic equality:
\begin{equation}\label{equationReflexion}
%% Accompanies image file 'equation_reflexion.pdf' (pdf, eps, ps)
\begingroup%
  \makeatletter%
  \providecommand\color[2][]{%
    \errmessage{(Inkscape) Color is used for the text in Inkscape, but the package 'color.sty' is not loaded}%
    \renewcommand\color[2][]{}%
  }%
  \providecommand\transparent[1]{%
    \errmessage{(Inkscape) Transparency is used (non-zero) for the text in Inkscape, but the package 'transparent.sty' is not loaded}%
    \renewcommand\transparent[1]{}%
  }%
  \providecommand\rotatebox[2]{#2}%
  \newcommand*\fsize{\dimexpr\f@size pt\relax}%
  \newcommand*\lineheight[1]{\fontsize{\fsize}{#1\fsize}\selectfont}%
  \ifx\svgwidth\undefined%
    \setlength{\unitlength}{401.72939106bp}%
    \ifx\svgscale\undefined%
      \relax%
    \else%
      \setlength{\unitlength}{\unitlength * \real{\svgscale}}%
    \fi%
  \else%
    \setlength{\unitlength}{\svgwidth}%
  \fi%
  \global\let\svgwidth\undefined%
  \global\let\svgscale\undefined%
  \makeatother%
  \begin{picture}(1,0.24244395)%
    \lineheight{1}%
    \setlength\tabcolsep{0pt}%
    \put(0,0){\includegraphics[width=\unitlength,page=1]{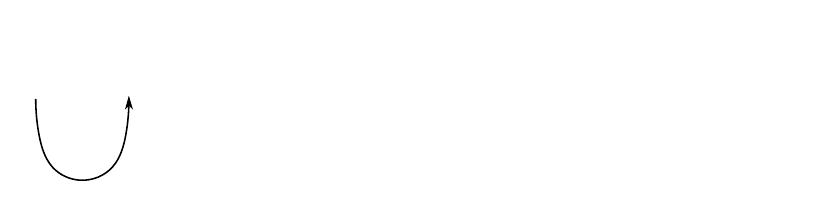}}%
    \put(0.01781626,0.10216918){\color[rgb]{0,0,0}\makebox(0,0)[lt]{\lineheight{1.25}\smash{\begin{tabular}[t]{l}$I$\end{tabular}}}}%
    \put(0,0){\includegraphics[width=\unitlength,page=2]{equation_reflexion.pdf}}%
    \put(0.16469254,0.00434855){\color[rgb]{0,0,0}\makebox(0,0)[lt]{\lineheight{1.25}\smash{\begin{tabular}[t]{l}$\overset{I}{X}(i)$\end{tabular}}}}%
    \put(0,0){\includegraphics[width=\unitlength,page=3]{equation_reflexion.pdf}}%
    \put(0.36072011,0.00434855){\color[rgb]{0,0,0}\makebox(0,0)[lt]{\lineheight{1.25}\smash{\begin{tabular}[t]{l}$\overset{J}{X}(i)$\end{tabular}}}}%
    \put(0,0){\includegraphics[width=\unitlength,page=4]{equation_reflexion.pdf}}%
    \put(0.21331162,0.1020592){\color[rgb]{0,0,0}\makebox(0,0)[lt]{\lineheight{1.25}\smash{\begin{tabular}[t]{l}$J$\end{tabular}}}}%
    \put(0.41456705,0.07099574){\color[rgb]{0,0,0}\makebox(0,0)[lt]{\lineheight{1.25}\smash{\begin{tabular}[t]{l}$=$\end{tabular}}}}%
    \put(0.50351096,0.222354){\color[rgb]{0,0,0}\makebox(0,0)[lt]{\lineheight{1.25}\smash{\begin{tabular}[t]{l}$I$\end{tabular}}}}%
    \put(0.69651129,0.22611963){\color[rgb]{0,0,0}\makebox(0,0)[lt]{\lineheight{1.25}\smash{\begin{tabular}[t]{l}$J$\end{tabular}}}}%
    \put(0,0){\includegraphics[width=\unitlength,page=5]{equation_reflexion.pdf}}%
    \put(0.63142479,0.00434855){\color[rgb]{0,0,0}\makebox(0,0)[lt]{\lineheight{1.25}\smash{\begin{tabular}[t]{l}$\overset{J}{X}(i)$\end{tabular}}}}%
    \put(0,0){\includegraphics[width=\unitlength,page=6]{equation_reflexion.pdf}}%
    \put(0.8274522,0.00434855){\color[rgb]{0,0,0}\makebox(0,0)[lt]{\lineheight{1.25}\smash{\begin{tabular}[t]{l}$\overset{I}{X}(i)$\end{tabular}}}}%
    \put(0,0){\includegraphics[width=\unitlength,page=7]{equation_reflexion.pdf}}%
  \end{picture}%
\endgroup%

\end{equation}
2. The matrix $\overset{I}{X}(i)$ is invertible with inverse $\overset{I}{X}(i){^{-1}} = \overset{I}{u}{^{-1}} \overset{I}{S(b_i)} \, {^{t\!}}\overset{I^*}{X}(i) \, \overset{I}{a_i}$ where $u$ is the Drinfeld element \eqref{elementDrinfeld}, $^t$ is the transpose and $R = a_i \otimes b_i$. Diagramatically:
\begin{equation}\label{dessinMatriceInverse}
%% Accompanies image file 'matrice_inverse.pdf' (pdf, eps, ps)
\begingroup%
  \makeatletter%
  \providecommand\color[2][]{%
    \errmessage{(Inkscape) Color is used for the text in Inkscape, but the package 'color.sty' is not loaded}%
    \renewcommand\color[2][]{}%
  }%
  \providecommand\transparent[1]{%
    \errmessage{(Inkscape) Transparency is used (non-zero) for the text in Inkscape, but the package 'transparent.sty' is not loaded}%
    \renewcommand\transparent[1]{}%
  }%
  \providecommand\rotatebox[2]{#2}%
  \newcommand*\fsize{\dimexpr\f@size pt\relax}%
  \newcommand*\lineheight[1]{\fontsize{\fsize}{#1\fsize}\selectfont}%
  \ifx\svgwidth\undefined%
    \setlength{\unitlength}{240.47936655bp}%
    \ifx\svgscale\undefined%
      \relax%
    \else%
      \setlength{\unitlength}{\unitlength * \real{\svgscale}}%
    \fi%
  \else%
    \setlength{\unitlength}{\svgwidth}%
  \fi%
  \global\let\svgwidth\undefined%
  \global\let\svgscale\undefined%
  \makeatother%
  \begin{picture}(1,0.29427232)%
    \lineheight{1}%
    \setlength\tabcolsep{0pt}%
    \put(0,0){\includegraphics[width=\unitlength,page=1]{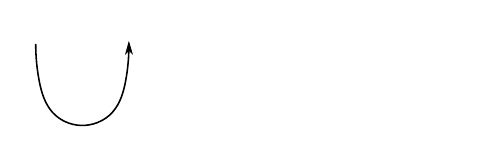}}%
    \put(0.02976273,0.16933077){\color[rgb]{0,0,0}\makebox(0,0)[lt]{\lineheight{1.25}\smash{\begin{tabular}[t]{l}$I$\end{tabular}}}}%
    \put(0,0){\includegraphics[width=\unitlength,page=2]{matrice_inverse.pdf}}%
    \put(0.27512482,0.00591784){\color[rgb]{0,0,0}\makebox(0,0)[lt]{\lineheight{1.25}\smash{\begin{tabular}[t]{l}$\overset{I}{X}(i)^{-1}$\end{tabular}}}}%
    \put(0,0){\includegraphics[width=\unitlength,page=3]{matrice_inverse.pdf}}%
    \put(0.46444731,0.25472812){\color[rgb]{0,0,0}\makebox(0,0)[lt]{\lineheight{1.25}\smash{\begin{tabular}[t]{l}$I$\end{tabular}}}}%
    \put(0,0){\includegraphics[width=\unitlength,page=4]{matrice_inverse.pdf}}%
    \put(0.71175272,0.00591784){\color[rgb]{0,0,0}\makebox(0,0)[lt]{\lineheight{1.25}\smash{\begin{tabular}[t]{l}$\overset{I}{X}(i)$\end{tabular}}}}%
    \put(0,0){\includegraphics[width=\unitlength,page=5]{matrice_inverse.pdf}}%
    \put(0.3601648,0.1173894){\color[rgb]{0,0,0}\makebox(0,0)[lt]{\lineheight{1.25}\smash{\begin{tabular}[t]{l}$=$\end{tabular}}}}%
  \end{picture}%
\endgroup%

\end{equation}
\end{proposition}
\begin{proof}
1. This is a consequence of \eqref{naturalite} and \eqref{dessinRelationFusion}:
\begin{center}
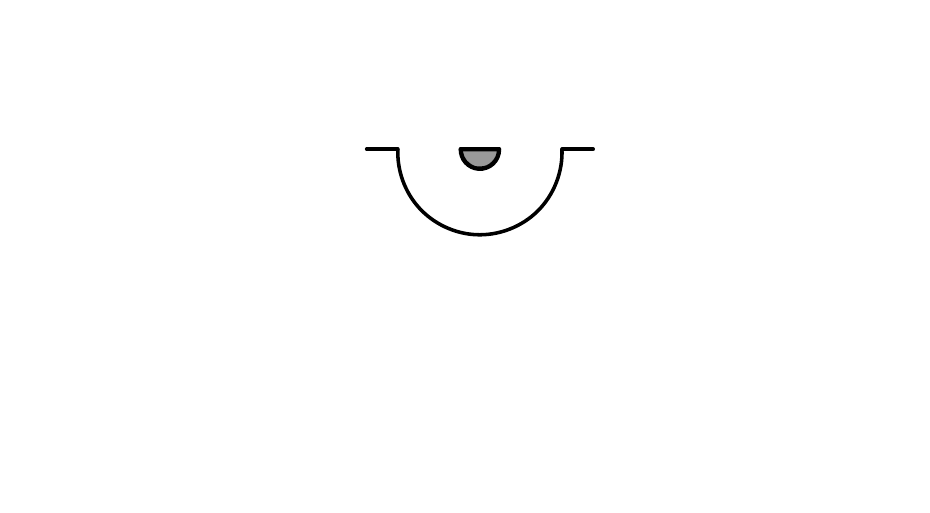
\end{center}
It is left to the reader to check that the diagrammatic equality is equivalent to the matrix equality.
\\2. This is also a consequence of \eqref{naturalite} and \eqref{dessinRelationFusion}:
\begin{center}
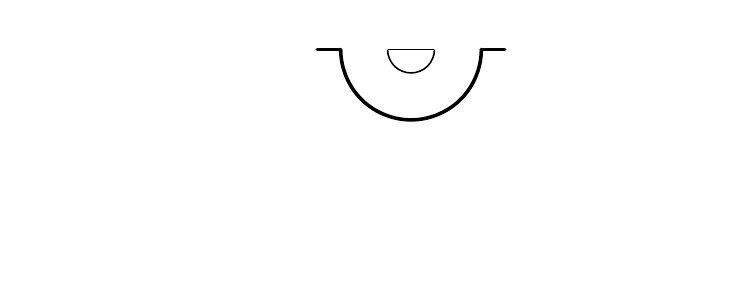
\end{center}
Again it is left to the reader to check using \eqref{sensOppose} that the diagrammatic and matrix equalities are equivalent.
\end{proof}

\begin{remark}\label{remarkInverseLift}
The proof of the previous proposition uses only the naturality and the fusion relation \eqref{relationFusion}, hence the result applies to any family of matrices $(\overset{I}{X})_{I \in \mathrm{mod}_l(H)}$ satisfying these properties and in particular to the lifts $(\overset{I}{\widetilde{x}})_{I \in \mathrm{mod}_l(H)}$ of a positively-oriented simple loop $x \in \pi_1(\Sigma_{g,n}^{\circ})$ (see section \ref{sectionDefLgnH}).
\end{remark}

\section{Holonomy and Wilson loops}\label{sectionHolonomyWilson}
\subsection{General definition for tangles}
\indent Recall that the surface $\Sigma_{g,n}^{\mathrm{o},\bullet}$ is defined in \S \ref{sectionDefSurfaces}; it is a punctured bordered surface in the sense of \cite[\S 2]{Le}, \cite[\S 2.2]{CL}.

\smallskip

\indent Let $\mathbf{\Sigma} = \Sigma_{g,n}^{\mathrm{o},\bullet} \times [0,1]$ be the thickening of $\Sigma_{g,n}^{\mathrm{o},\bullet}$. For a point $p = (x,t) \in \mathbf{\Sigma}$ we call $t \in [0,1]$ the height of $p$. The following definition is taken from \cite[\S 2.2]{Le}, except that we add an orientation and a coloring on the tangles. Note however that we restrict to the surfaces $\Sigma_{g,n}^{\mathrm{o},\bullet}$ (we do not consider more general punctured bordered surfaces).

\begin{definition}\label{defBordTangle}
1. A $H$-colored $\partial \mathbf{\Sigma}$-tangle is an oriented, framed, compact, properly embedded $1$-dimensional submanifold $\mathbf{T} \subset \mathbf{\Sigma}$ such that
\begin{itemize}
\item at every point of $\partial\mathbf{T} = \mathbf{T} \cap \partial\mathbf{\Sigma}$ the framing is vertical
\item the points of $\partial\mathbf{T}$ have distinct heights
\item each connected component of $\mathbf{T}$ is colored (\textit{i.e.} labelled) by a finite dimensional $H$-module.
\end{itemize}
We denote by $\mathscr{T}(\Sigma_{g,n}^{\mathrm{o},\bullet})$ the set of isotopy classes of such tangles.\\
2. An isotopy of $H$-colored $\partial \mathbf{\Sigma}$-tangles is an isotopy preserving these defining properties.\\
3. A $H$-colored $\partial \mathbf{\Sigma}$-tangle diagram is the projection of $\mathbf{T} \subset \Sigma_{g,n}^{\mathrm{o},\bullet} \times [0,1]$ on $\Sigma_{g,n}^{\mathrm{o},\bullet}$, together with the colors, the over/underpassing information at each double point and the height information at each boundary point. We assume as usual that $\mathbf{T}$ is in general position before doing the projection.\\
4. Let $\mathfrak{o}$ be an orientation of the boundary curve $\partial\Sigma_{g,n}^{\mathrm{o},\bullet}$. A $\partial \mathbf{\Sigma}$-tangle diagram is $\mathfrak{o}$-ordered if the heights of the boundary points are increasing when one goes along $\partial\Sigma_{g,n}^{\mathrm{o},\bullet}$ according to $\mathfrak{o}$.
\end{definition}
Up to isotopy, a $\partial \mathbf{\Sigma}$-tangle $\mathbf{T}$ can always be presented by a $\mathfrak{o}$-ordered diagram. More precisely, up to isotopy, a $H$-colored $\partial \mathbf{\Sigma}$-tangle diagram can always be presented as follows:
\begin{equation}\label{formeStandard}
%% Accompanies image file 'formeStandard.pdf' (pdf, eps, ps)
\begingroup%
  \makeatletter%
  \providecommand\color[2][]{%
    \errmessage{(Inkscape) Color is used for the text in Inkscape, but the package 'color.sty' is not loaded}%
    \renewcommand\color[2][]{}%
  }%
  \providecommand\transparent[1]{%
    \errmessage{(Inkscape) Transparency is used (non-zero) for the text in Inkscape, but the package 'transparent.sty' is not loaded}%
    \renewcommand\transparent[1]{}%
  }%
  \providecommand\rotatebox[2]{#2}%
  \newcommand*\fsize{\dimexpr\f@size pt\relax}%
  \newcommand*\lineheight[1]{\fontsize{\fsize}{#1\fsize}\selectfont}%
  \ifx\svgwidth\undefined%
    \setlength{\unitlength}{424.14031937bp}%
    \ifx\svgscale\undefined%
      \relax%
    \else%
      \setlength{\unitlength}{\unitlength * \real{\svgscale}}%
    \fi%
  \else%
    \setlength{\unitlength}{\svgwidth}%
  \fi%
  \global\let\svgwidth\undefined%
  \global\let\svgscale\undefined%
  \makeatother%
  \begin{picture}(1,0.27098175)%
    \lineheight{1}%
    \setlength\tabcolsep{0pt}%
    \put(0,0){\includegraphics[width=\unitlength,page=1]{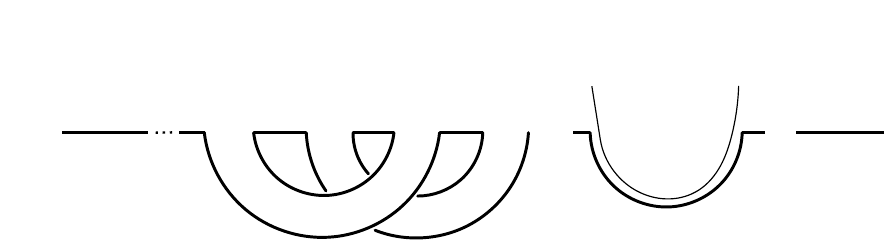}}%
    \put(0.50867249,0.1885787){\color[rgb]{0,0,0}\makebox(0,0)[lt]{\lineheight{1.25}\smash{\begin{tabular}[t]{l}$T$\end{tabular}}}}%
    \put(0,0){\includegraphics[width=\unitlength,page=2]{formeStandard.pdf}}%
    \put(-0.00335231,-0.35651493){\color[rgb]{0,0,0}\makebox(0,0)[lt]{\begin{minipage}{0.03486159\unitlength}\raggedright \end{minipage}}}%
    \put(0,0){\includegraphics[width=\unitlength,page=3]{formeStandard.pdf}}%
    \put(0.63192396,0.14208154){\color[rgb]{0,0,0}\makebox(0,0)[lt]{\lineheight{1.25}\smash{\begin{tabular}[t]{l}$K_1$\end{tabular}}}}%
    \put(0.70813048,0.14272569){\color[rgb]{0,0,0}\makebox(0,0)[lt]{\lineheight{1.25}\smash{\begin{tabular}[t]{l}$K_r$\end{tabular}}}}%
    \put(0.77041803,0.23137388){\color[rgb]{0,0,0}\makebox(0,0)[lt]{\lineheight{1.25}\smash{\begin{tabular}[t]{l}$V_k$\end{tabular}}}}%
    \put(0.26199619,0.23131375){\color[rgb]{0,0,0}\makebox(0,0)[lt]{\lineheight{1.25}\smash{\begin{tabular}[t]{l}$V_1$\end{tabular}}}}%
    \put(0,0){\includegraphics[width=\unitlength,page=4]{formeStandard.pdf}}%
    \put(-0.00142315,0.14803053){\color[rgb]{0,0,0}\makebox(0,0)[lt]{\lineheight{1.25}\smash{\begin{tabular}[t]{l}$\mathbf{T} =$\end{tabular}}}}%
    \put(0,0){\includegraphics[width=\unitlength,page=5]{formeStandard.pdf}}%
    \put(0.20778507,0.1429347){\color[rgb]{0,0,0}\makebox(0,0)[lt]{\lineheight{1.25}\smash{\begin{tabular}[t]{l}$I_1$\end{tabular}}}}%
    \put(0.27213777,0.14249262){\color[rgb]{0,0,0}\makebox(0,0)[lt]{\lineheight{1.25}\smash{\begin{tabular}[t]{l}$I_l$\end{tabular}}}}%
    \put(0,0){\includegraphics[width=\unitlength,page=6]{formeStandard.pdf}}%
    \put(0.32170569,0.14335409){\color[rgb]{0,0,0}\makebox(0,0)[lt]{\lineheight{1.25}\smash{\begin{tabular}[t]{l}$J_1$\end{tabular}}}}%
    \put(0.39037021,0.14251518){\color[rgb]{0,0,0}\makebox(0,0)[lt]{\lineheight{1.25}\smash{\begin{tabular}[t]{l}$J_m$\end{tabular}}}}%
  \end{picture}%
\endgroup%

\end{equation}
where the handles contain only bunches of parallel strands, $T$ is a (non-unique) tangle, the $H$-modules $I_i, J_i, K_i, V_i$ color the strands, the double arrows mean that each strand is oriented, the boundary points are all on the line at the top of the rectangle and the big dot ($\bullet$) is the point that we have removed from $\partial\Sigma_{g,n}^{\mathrm{o}}$. Last but not least, the boundary points of $\mathbf{T}$ have increasing heights with respect to the orientation of the boundary represented by the arrow. Note that $k$ may be equal to $0$, which means that $\mathbf{T}$ is a link (no boundary points).

\begin{definition}\label{defHol}
Let $\mathbf{T} \in \mathscr{T}(\Sigma_{g,n}^{\mathrm{o},\bullet})$ be represented by a $H$-colored $\partial \mathbf{\Sigma}$-tangle diagram as in \eqref{formeStandard}. The holonomy of $\mathbf{T}$ is an element $\mathrm{hol}(\mathbf{T}) \in \mathcal{L}_{g,n}(H) \otimes V_1^{a_1} \otimes \ldots \otimes V_k^{a_k}$ defined as the evaluation (see Definition \ref{defEvaluation}) of the following diagram:
\begin{center}
%% Accompanies image file 'link_tangle.pdf' (pdf, eps, ps)
\begingroup%
  \makeatletter%
  \providecommand\color[2][]{%
    \errmessage{(Inkscape) Color is used for the text in Inkscape, but the package 'color.sty' is not loaded}%
    \renewcommand\color[2][]{}%
  }%
  \providecommand\transparent[1]{%
    \errmessage{(Inkscape) Transparency is used (non-zero) for the text in Inkscape, but the package 'transparent.sty' is not loaded}%
    \renewcommand\transparent[1]{}%
  }%
  \providecommand\rotatebox[2]{#2}%
  \newcommand*\fsize{\dimexpr\f@size pt\relax}%
  \newcommand*\lineheight[1]{\fontsize{\fsize}{#1\fsize}\selectfont}%
  \ifx\svgwidth\undefined%
    \setlength{\unitlength}{435.74998059bp}%
    \ifx\svgscale\undefined%
      \relax%
    \else%
      \setlength{\unitlength}{\unitlength * \real{\svgscale}}%
    \fi%
  \else%
    \setlength{\unitlength}{\svgwidth}%
  \fi%
  \global\let\svgwidth\undefined%
  \global\let\svgscale\undefined%
  \makeatother%
  \begin{picture}(1,0.33185671)%
    \lineheight{1}%
    \setlength\tabcolsep{0pt}%
    \put(0,0){\includegraphics[width=\unitlength,page=1]{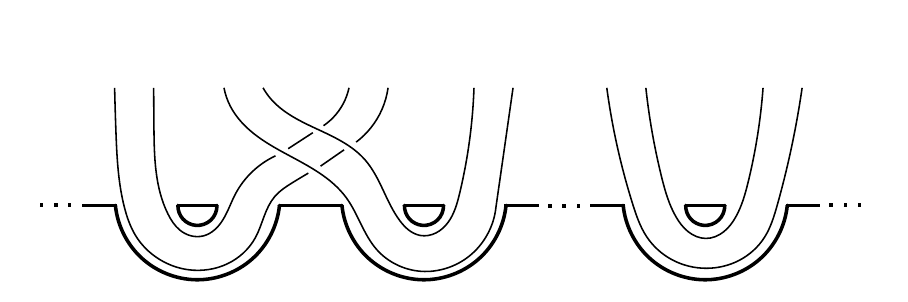}}%
    \put(0.10132467,0.2072859){\color[rgb]{0,0,0}\makebox(0,0)[lt]{\lineheight{1.25}\smash{\begin{tabular}[t]{l}$I_1$\end{tabular}}}}%
    \put(0.17476744,0.20746386){\color[rgb]{0,0,0}\makebox(0,0)[lt]{\lineheight{1.25}\smash{\begin{tabular}[t]{l}$I_l$\end{tabular}}}}%
    \put(0.22410838,0.20816811){\color[rgb]{0,0,0}\makebox(0,0)[lt]{\lineheight{1.25}\smash{\begin{tabular}[t]{l}$J_1$\end{tabular}}}}%
    \put(0.31793237,0.21144673){\color[rgb]{0,0,0}\makebox(0,0)[lt]{\lineheight{1.25}\smash{\begin{tabular}[t]{l}$J_m$\end{tabular}}}}%
    \put(0.63421108,0.20912114){\color[rgb]{0,0,0}\makebox(0,0)[lt]{\lineheight{1.25}\smash{\begin{tabular}[t]{l}$K_1$\end{tabular}}}}%
    \put(0.72188803,0.20905166){\color[rgb]{0,0,0}\makebox(0,0)[lt]{\lineheight{1.25}\smash{\begin{tabular}[t]{l}$K_r$\end{tabular}}}}%
    \put(0,0){\includegraphics[width=\unitlength,page=2]{link_tangle.pdf}}%
    \put(0.28416726,0.0033755){\color[rgb]{0,0,0}\makebox(0,0)[lt]{\lineheight{1.25}\smash{\begin{tabular}[t]{l}$\overset{I_1 \otimes \ldots \otimes I_l}{B(i)}$\end{tabular}}}}%
    \put(0.53396078,0.00607667){\color[rgb]{0,0,0}\makebox(0,0)[lt]{\lineheight{1.25}\smash{\begin{tabular}[t]{l}$\overset{J_1 \otimes \ldots \otimes J_m}{A(i)}$\end{tabular}}}}%
    \put(0.84791176,0.00641319){\color[rgb]{0,0,0}\makebox(0,0)[lt]{\lineheight{1.25}\smash{\begin{tabular}[t]{l}$\overset{K_1 \otimes \ldots \otimes K_r}{M(j)}$\end{tabular}}}}%
    \put(-0.01310965,-0.26437538){\color[rgb]{0,0,0}\makebox(0,0)[lt]{\begin{minipage}{0.03393278\unitlength}\raggedright \end{minipage}}}%
    \put(0.48027698,0.2471247){\color[rgb]{0,0,0}\makebox(0,0)[lt]{\lineheight{1.25}\smash{\begin{tabular}[t]{l}$T$\end{tabular}}}}%
    \put(0,0){\includegraphics[width=\unitlength,page=3]{link_tangle.pdf}}%
    \put(0.18964084,0.29248309){\color[rgb]{0,0,0}\makebox(0,0)[lt]{\lineheight{1.25}\smash{\begin{tabular}[t]{l}$V_1$\end{tabular}}}}%
    \put(0.78389819,0.29286843){\color[rgb]{0,0,0}\makebox(0,0)[lt]{\lineheight{1.25}\smash{\begin{tabular}[t]{l}$V_k$\end{tabular}}}}%
    \put(0,0){\includegraphics[width=\unitlength,page=4]{link_tangle.pdf}}%
  \end{picture}%
\endgroup%

\end{center}
The symbol $a_i \in \{ \downarrow, \uparrow \}$ depends on the orientation of the strands ($a_i = \downarrow$ if the strand is oriented downwards and $a_i = \uparrow$ else) and we define $V_i^{\downarrow} = V_i$ and $V_i^{\uparrow} = V_i^*$. If $\mathbf{T}$ does not have boundary points then $\mathrm{hol}(\mathbf{T})$ is just an element of $\mathcal{L}_{g,n}(H)$.
\end{definition}

\indent We note that $\mathrm{hol}(\mathbf{T})$ does not depend on the choice of the tangle $T$. Indeed, since the Reshetikhin-Turaev functor is an isotopy invariant, the evaluation of the diagram in Definition \ref{defHol} depends only on the isotopy class of $T$. Moreover, if we drag certain crossings, cups and caps along the handles in order to obtain another tangle $T'$ representing $\mathbf{T}$, then this does not change the value of $\mathrm{hol}(\mathbf{T})$ thanks to naturality \eqref{naturalite}:
\begin{center}
%% Accompanies image file 'Wilson_bien_def.pdf' (pdf, eps, ps)
\begingroup%
  \makeatletter%
  \providecommand\color[2][]{%
    \errmessage{(Inkscape) Color is used for the text in Inkscape, but the package 'color.sty' is not loaded}%
    \renewcommand\color[2][]{}%
  }%
  \providecommand\transparent[1]{%
    \errmessage{(Inkscape) Transparency is used (non-zero) for the text in Inkscape, but the package 'transparent.sty' is not loaded}%
    \renewcommand\transparent[1]{}%
  }%
  \providecommand\rotatebox[2]{#2}%
  \newcommand*\fsize{\dimexpr\f@size pt\relax}%
  \newcommand*\lineheight[1]{\fontsize{\fsize}{#1\fsize}\selectfont}%
  \ifx\svgwidth\undefined%
    \setlength{\unitlength}{454.75533014bp}%
    \ifx\svgscale\undefined%
      \relax%
    \else%
      \setlength{\unitlength}{\unitlength * \real{\svgscale}}%
    \fi%
  \else%
    \setlength{\unitlength}{\svgwidth}%
  \fi%
  \global\let\svgwidth\undefined%
  \global\let\svgscale\undefined%
  \makeatother%
  \begin{picture}(1,0.13710413)%
    \lineheight{1}%
    \setlength\tabcolsep{0pt}%
    \put(0,0){\includegraphics[width=\unitlength,page=1]{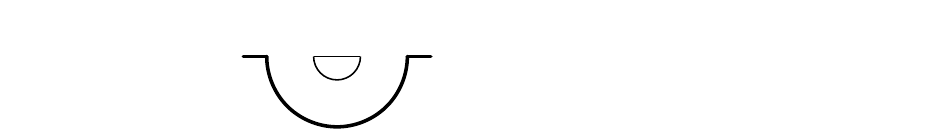}}%
    \put(0.27142095,0.10796273){\color[rgb]{0,0,0}\makebox(0,0)[lt]{\lineheight{1.25}\smash{\begin{tabular}[t]{l}$I$\end{tabular}}}}%
    \put(0.4278767,0.00918257){\color[rgb]{0,0,0}\makebox(0,0)[lt]{\lineheight{1.25}\smash{\begin{tabular}[t]{l}$\overset{I \otimes I}{X}\!\!(i)$\end{tabular}}}}%
    \put(0,0){\includegraphics[width=\unitlength,page=2]{Wilson_bien_def.pdf}}%
    \put(0.01651063,0.10263953){\color[rgb]{0,0,0}\makebox(0,0)[lt]{\lineheight{1.25}\smash{\begin{tabular}[t]{l}$I$\end{tabular}}}}%
    \put(0,0){\includegraphics[width=\unitlength,page=3]{Wilson_bien_def.pdf}}%
    \put(0.21415393,0.07065556){\color[rgb]{0,0,0}\makebox(0,0)[lt]{\lineheight{1.25}\smash{\begin{tabular}[t]{l}$=$\end{tabular}}}}%
    \put(0,0){\includegraphics[width=\unitlength,page=4]{Wilson_bien_def.pdf}}%
    \put(0.17003162,0.00794811){\color[rgb]{0,0,0}\makebox(0,0)[lt]{\lineheight{1.25}\smash{\begin{tabular}[t]{l}$\overset{\mathbb{C}}{X}(i)$\end{tabular}}}}%
    \put(0,0){\includegraphics[width=\unitlength,page=5]{Wilson_bien_def.pdf}}%
    \put(0.51897404,0.11780528){\color[rgb]{0,0,0}\makebox(0,0)[lt]{\lineheight{1.25}\smash{\begin{tabular}[t]{l}$J$\end{tabular}}}}%
    \put(0.5813655,0.11741866){\color[rgb]{0,0,0}\makebox(0,0)[lt]{\lineheight{1.25}\smash{\begin{tabular}[t]{l}$I$\end{tabular}}}}%
    \put(0.72541782,0.07065558){\color[rgb]{0,0,0}\makebox(0,0)[lt]{\lineheight{1.25}\smash{\begin{tabular}[t]{l}$=$\end{tabular}}}}%
    \put(0,0){\includegraphics[width=\unitlength,page=6]{Wilson_bien_def.pdf}}%
    \put(0.77170179,0.11701361){\color[rgb]{0,0,0}\makebox(0,0)[lt]{\lineheight{1.25}\smash{\begin{tabular}[t]{l}$J$\end{tabular}}}}%
    \put(0.82420152,0.11727055){\color[rgb]{0,0,0}\makebox(0,0)[lt]{\lineheight{1.25}\smash{\begin{tabular}[t]{l}$I$\end{tabular}}}}%
    \put(0,0){\includegraphics[width=\unitlength,page=7]{Wilson_bien_def.pdf}}%
    \put(0.46666947,0.07436717){\color[rgb]{0,0,0}\makebox(0,0)[lt]{\lineheight{1.25}\smash{\begin{tabular}[t]{l},\end{tabular}}}}%
    \put(0.68355502,0.00786871){\color[rgb]{0,0,0}\makebox(0,0)[lt]{\lineheight{1.25}\smash{\begin{tabular}[t]{l}$\overset{I \otimes J}{X}\!\!(i)$\end{tabular}}}}%
    \put(0.93910766,0.00741032){\color[rgb]{0,0,0}\makebox(0,0)[lt]{\lineheight{1.25}\smash{\begin{tabular}[t]{l}$\overset{J \otimes I}{X}\!\!(i)$\end{tabular}}}}%
  \end{picture}%
\endgroup%

\end{center}

\indent The diagrammatic relations introduced previously allow us to compute the value of $\mathrm{hol}$ in a purely diagrammatic way, and sometimes to obtain a simple expression of the result. For instance, here is a computation:
\begin{center}
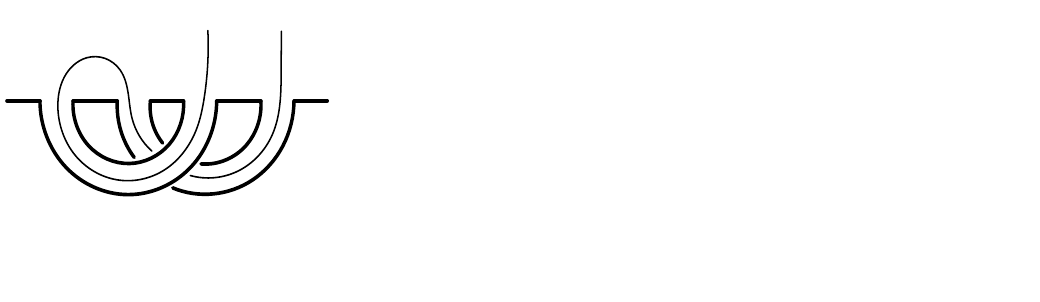
\end{center}

\indent A natural operation for $H$-colored $\partial \mathbf{\Sigma}$-tangles is the stack product, defined in the usual way:
\begin{definition}\label{stackProduct}
Let $\mathbf{T}_1, \mathbf{T}_2 \in \mathscr{T}(\Sigma_{g,n}^{\mathrm{o},\bullet})$ and let $\mathbf{T}^-_1 \in \Sigma_{g,n}^{\mathrm{o},\bullet} \times [0, \frac{1}{2}[$ be isotopic to $\mathbf{T}_1$ and $\mathbf{T}^+_2 \in \Sigma_{g,n}^{\mathrm{o},\bullet} \times ]\frac{1}{2}, 1]$ be isotopic to $\mathbf{T}_2$. The stack product of $\mathbf{T}_1$ and $\mathbf{T}_2$ is $\mathbf{T}_1 \ast \mathbf{T}_2 = \mathbf{T}^-_1 \cup \mathbf{T}^+_2 \in \mathscr{T}(\Sigma_{g,n}^{\mathrm{o},\bullet})$. 
\end{definition}
\noindent Note that our convention for $\ast$ is the opposite to the one of \cite{Le, CL}: here $\mathbf{T}_1 \ast \mathbf{T}_2$ means that we put $\mathbf{T}_1$ {\em below} $\mathbf{T}_2$ in the thickened surface.

\smallskip

\indent For $x \otimes v \in \mathcal{L}_{g,n}(H) \otimes V_1 \otimes \ldots \otimes V_k$ and $y \otimes w \in \mathcal{L}_{g,n}(H) \otimes W_1 \otimes \ldots \otimes W_l$ (with $x,y \in \mathcal{L}_{g,n}(H)$, $v \in V_1 \otimes \ldots \otimes V_k$ and $w \in W_1 \otimes \ldots \otimes W_l$), we define
\[ (x \otimes v) \odot (y \otimes w) = xy \otimes v \otimes w \in \mathcal{L}_{g,n}(H) \otimes V_1 \otimes \ldots \otimes V_k \otimes W_1 \otimes \ldots \otimes W_{l}. \]

\begin{theorem}\label{wilsonStack}
Let $\mathbf{T}_1, \mathbf{T}_2 \in \mathscr{T}(\Sigma_{g,n}^{\mathrm{o},\bullet})$. It holds
\[ \mathrm{hol}(\mathbf{T}_1 \ast \mathbf{T}_2) = \mathrm{hol}(\mathbf{T}_1) \odot \mathrm{hol}(\mathbf{T}_2). \]
\end{theorem}
\begin{proof}
See the appendix at the end of the paper (page \pageref{appendixPreuveStack}).
\end{proof}

\indent The holonomy behaves well under the action of the mapping class group. As said in \S \ref{sectionDefLgnH}, we will prove this for $n=0$. Note that $\mathrm{MCG}(\Sigma_{g,0}^{\mathrm{o}})$ acts on $\Sigma_{g,0}^{\mathrm{o},\bullet} \times [0,1]$ by $f(x,t) = (f(x),t)$ and recall the lift $\widetilde{f}$ of a mapping class $f$ defined in \eqref{defLiftMappingClass}.
\begin{theorem}\label{MCGcommuteW}
Let $f \in \mathrm{MCG}(\Sigma_{g,0}^{\mathrm{o}})$ and let $\mathbf{T} \in \mathscr{T}(\Sigma_{g,0}^{\mathrm{o},\bullet})$, then
\[ \mathrm{hol}\bigl( f(\mathbf{T}) \bigr) = \widetilde{f} \otimes \mathrm{id}_{V^{a_1}_1 \otimes \ldots \otimes V^{a_k}_k}\bigl( \mathrm{hol}(\mathbf{T}) \bigr) \]
where the $V_i^{a_i}$ are the colors of the boundary strands of $\mathbf{T}$ (see Definition \ref{defHol}).
\end{theorem}
\begin{proof}
The proof is purely diagrammatic. We can assume that $f$ is one of the Humphries generators $\tau_{\gamma}$ \cite[\S 4.4.3]{FM} where $\gamma$ is one of the simple closed curves depicted in \cite[Fig. 4.5]{FM}. Let $U \subset \Sigma_{1,0}^{\mathrm{o},\bullet}$ be a connected subset containing all the handles and a very small part of the bottom of the rectangle (see Figure \ref{surfaceGN}). By isotopy, we can assume that $\gamma \subset U$.  Also by isotopy we can assume that $\mathbf{T} \cap (U \times [0,1])$ contains only bunches of parallel strands and then by another isotopy that $\mathbf{T} \cap (U \times [0,1]) \subset U \times \{0\}$. Then $\tau_{\gamma}(\mathbf{T})$ is obtained by computing $\tau_{\gamma}\bigl( \mathbf{T} \cap (U \times \{0\} \bigr)$ in the usual way in the surface $\Sigma_{g,0}^{\mathrm{o},\bullet} \times \{0\}$. Moreover, as in the proof of Theorem \ref{wilsonStack}, we can assume up to introducing coupons that each handle contains only one positively oriented strand (as in \eqref{anse}). Now, if $f = \tau_{a_1}$, which is one of the Humphries generators ($a_1$ is $m_1$ in the notations of \cite[Fig. 4.5]{FM}), we can restrict to $\Sigma_{1,0}^{\mathrm{o},\bullet}$ (since $a_1$ is contained in this subsurface) and then we have the graphical computation in $\mathcal{L}_{1,0}(H)$ represented in Figure \ref{twistAetMCG}. We used the fusion relation, the inverse of the reflection equation \eqref{equationReflexion} and the definition of the lift $\widetilde{\tau}_a = \widetilde{\tau}_{a_1}$ (see section \ref{sectionDefLgnH}):
\[ \widetilde{\tau}_a\bigl( \overset{I}{A} \bigr) = \overset{I}{A}, \:\:\:\: \widetilde{\tau}_a\bigl( \overset{I}{B} \bigr) = \overset{I}{v}{^{-1}} \overset{I}{B} \overset{I}{A}. \]
The equalities for the others Humphries generators are shown similarly, although the diagrams are more cumbersome.
\begin{figure}[p]
\centering
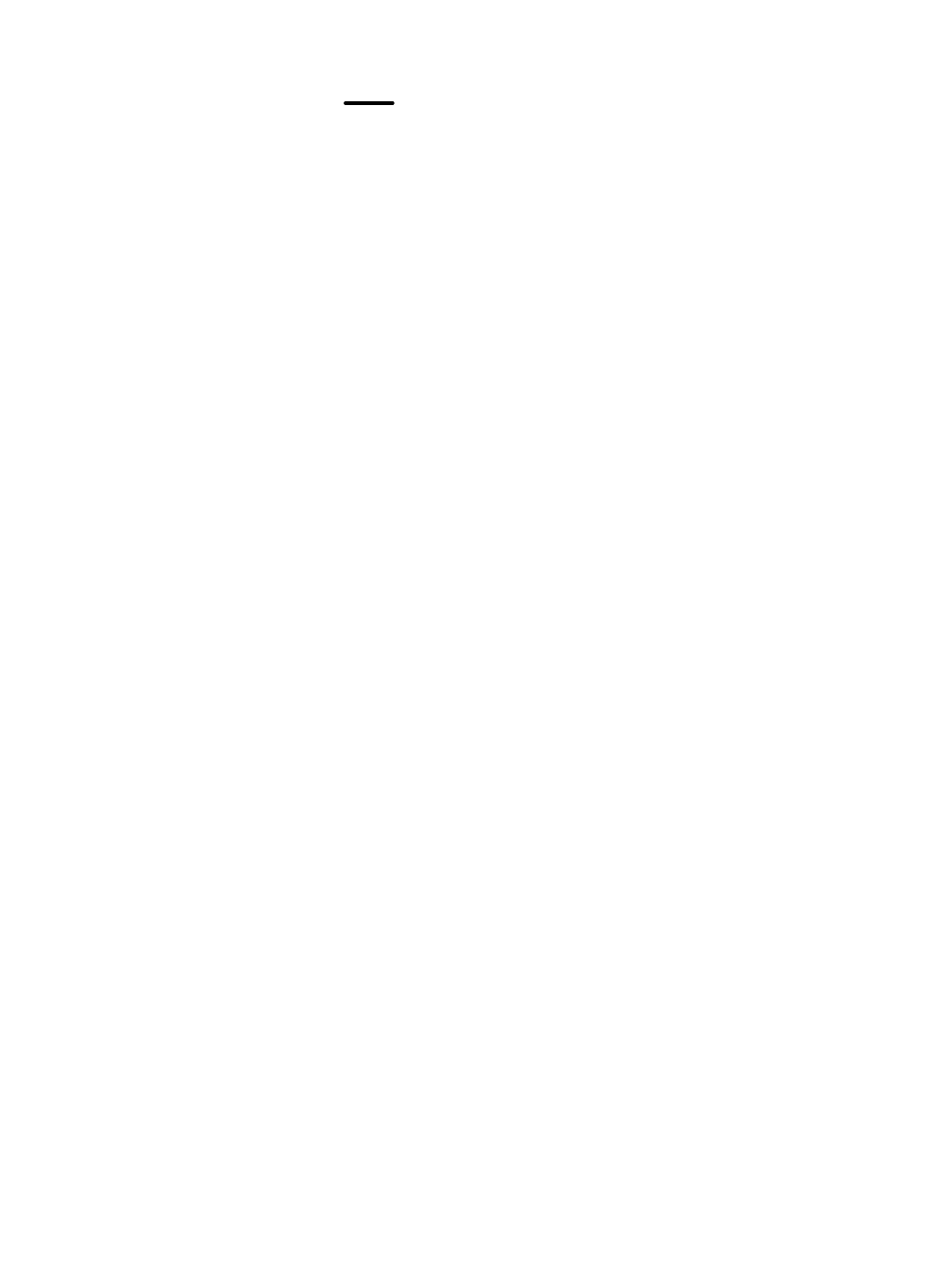
\caption{Proof of the equality $\mathrm{hol} \circ \tau_a(\mathbf{T}) = (\widetilde{\tau_a} \otimes \mathrm{id}) \circ \mathrm{hol}(\mathbf{T})$.}
\label{twistAetMCG}
\end{figure}
\end{proof}

\indent For further use let us describe the behaviour of $\mathrm{hol}$ when we change the orientation of a strand. Consider (not colored) $\partial \mathbf{\Sigma}$-tangles $\mathbf{T}^+$, $\mathbf{T}^-$ with one connected component, which are equal as unoriented tangles and which are oriented as follows :
\begin{center}
%% Accompanies image file 'holChangeOrientation.pdf' (pdf, eps, ps)
\begingroup%
  \makeatletter%
  \providecommand\color[2][]{%
    \errmessage{(Inkscape) Color is used for the text in Inkscape, but the package 'color.sty' is not loaded}%
    \renewcommand\color[2][]{}%
  }%
  \providecommand\transparent[1]{%
    \errmessage{(Inkscape) Transparency is used (non-zero) for the text in Inkscape, but the package 'transparent.sty' is not loaded}%
    \renewcommand\transparent[1]{}%
  }%
  \providecommand\rotatebox[2]{#2}%
  \newcommand*\fsize{\dimexpr\f@size pt\relax}%
  \newcommand*\lineheight[1]{\fontsize{\fsize}{#1\fsize}\selectfont}%
  \ifx\svgwidth\undefined%
    \setlength{\unitlength}{282.82892561bp}%
    \ifx\svgscale\undefined%
      \relax%
    \else%
      \setlength{\unitlength}{\unitlength * \real{\svgscale}}%
    \fi%
  \else%
    \setlength{\unitlength}{\svgwidth}%
  \fi%
  \global\let\svgwidth\undefined%
  \global\let\svgscale\undefined%
  \makeatother%
  \begin{picture}(1,0.12312914)%
    \lineheight{1}%
    \setlength\tabcolsep{0pt}%
    \put(0,0){\includegraphics[width=\unitlength,page=1]{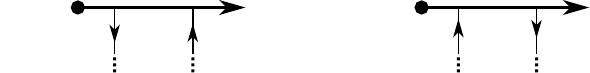}}%
    \put(-0.00143905,0.05351815){\color[rgb]{0,0,0}\makebox(0,0)[lt]{\lineheight{1.25}\smash{\begin{tabular}[t]{l}$\mathbf{T}^+ = $\end{tabular}}}}%
    \put(0.58824875,0.05351468){\color[rgb]{0,0,0}\makebox(0,0)[lt]{\lineheight{1.25}\smash{\begin{tabular}[t]{l}$\mathbf{T}^- = $\end{tabular}}}}%
  \end{picture}%
\endgroup%

\end{center}
For $I$ a finite dimensional $H$-module, we denote by $\mathrm{hol}^I(\mathbf{T}^{\pm})$ the value of $\mathrm{hol}$ on the $\partial \mathbf{\Sigma}$-tangle $\mathbf{T}^{\pm}$ colored by $I$. Then thanks to \eqref{sensOppose}, \eqref{sensOpposeBis}, \eqref{naturalite}, it is easy to see that
\begin{equation}\label{changeOrientationTangle}
\begin{split}
&\mathrm{hol}^I(\mathbf{T}^-) = (\mathrm{id}_{I^*} \otimes e_I)\bigl( \mathrm{hol}^{I^*}(\mathbf{T}^+) \bigr)\\
&\mathrm{hol}^I(\mathbf{T}^+) = (e_I \otimes \mathrm{id}_{I^*})\bigl( \mathrm{hol}^{I^*}(\mathbf{T}^-) \bigr).
\end{split}
\end{equation}
This is obviously generalized to arbitrary $\partial \mathbf{\Sigma}$-tangles.

\subsection{Holonomy of a link}
\indent We will now specialize Definition \ref{defHol} to the case of $H$-colored framed links in $\Sigma_{g,n}^{\mathrm{o}} \times [0,1]$ with basepoints, so that it will become more clear why we call this operation the ``holonomy''.

\begin{definition}
A based $H$-colored link $L \subset \boldsymbol{\Sigma} = \Sigma_{g,n}^{\mathrm{o},\bullet} \times [0,1]$ is an oriented and framed embedding of circles such that:
\begin{itemize}
\item for each connected component $L_i$ of $L$, we have $L_i \cap \partial\boldsymbol{\Sigma} = \{p_i\}$; $p_i$ is called a basepoint
\item the basepoints have distinct heights
\item each connected component of $L$ is colored (\textit{i.e.} labelled) by a finite dimensional $H$-module
\end{itemize}
\end{definition}
An isotopy of based $H$-colored links  is an isotopy compatible with the defining properties and preserving globally the boundary $\partial \boldsymbol{\Sigma}$ (so that basepoints can be moved along the boundary but cannot be suppressed). It is clear that using isotopy any based $H$-colored  link can be presented as follows, where the basepoints $p_1, \ldots, p_k$ have strictly increasing heights:
\begin{equation}\label{basedLink}
%% Accompanies image file 'basedLink.pdf' (pdf, eps, ps)
\begingroup%
  \makeatletter%
  \providecommand\color[2][]{%
    \errmessage{(Inkscape) Color is used for the text in Inkscape, but the package 'color.sty' is not loaded}%
    \renewcommand\color[2][]{}%
  }%
  \providecommand\transparent[1]{%
    \errmessage{(Inkscape) Transparency is used (non-zero) for the text in Inkscape, but the package 'transparent.sty' is not loaded}%
    \renewcommand\transparent[1]{}%
  }%
  \providecommand\rotatebox[2]{#2}%
  \newcommand*\fsize{\dimexpr\f@size pt\relax}%
  \newcommand*\lineheight[1]{\fontsize{\fsize}{#1\fsize}\selectfont}%
  \ifx\svgwidth\undefined%
    \setlength{\unitlength}{397.92420221bp}%
    \ifx\svgscale\undefined%
      \relax%
    \else%
      \setlength{\unitlength}{\unitlength * \real{\svgscale}}%
    \fi%
  \else%
    \setlength{\unitlength}{\svgwidth}%
  \fi%
  \global\let\svgwidth\undefined%
  \global\let\svgscale\undefined%
  \makeatother%
  \begin{picture}(1,0.31296841)%
    \lineheight{1}%
    \setlength\tabcolsep{0pt}%
    \put(0,0){\includegraphics[width=\unitlength,page=1]{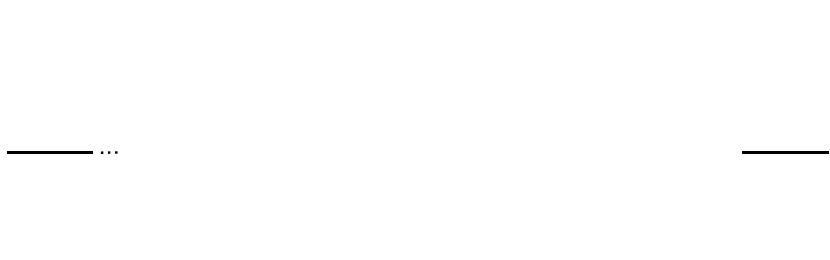}}%
    \put(0.4763027,0.2010027){\color[rgb]{0,0,0}\makebox(0,0)[lt]{\lineheight{1.25}\smash{\begin{tabular}[t]{l}$T$\end{tabular}}}}%
    \put(0,0){\includegraphics[width=\unitlength,page=2]{basedLink.pdf}}%
    \put(0.68422306,0.24739335){\color[rgb]{0,0,0}\makebox(0,0)[lt]{\lineheight{1.25}\smash{\begin{tabular}[t]{l}$V_k$\end{tabular}}}}%
    \put(0.16893149,0.24771535){\color[rgb]{0,0,0}\makebox(0,0)[lt]{\lineheight{1.25}\smash{\begin{tabular}[t]{l}$V_1$\end{tabular}}}}%
    \put(0,0){\includegraphics[width=\unitlength,page=3]{basedLink.pdf}}%
    \put(0.25747753,0.30072637){\color[rgb]{0,0,0}\makebox(0,0)[lt]{\lineheight{1.25}\smash{\begin{tabular}[t]{l}$p_1$\end{tabular}}}}%
    \put(0.77610542,0.30093457){\color[rgb]{0,0,0}\makebox(0,0)[lt]{\lineheight{1.25}\smash{\begin{tabular}[t]{l}$p_k$\end{tabular}}}}%
    \put(0,0){\includegraphics[width=\unitlength,page=4]{basedLink.pdf}}%
    \put(0.60767437,0.15144219){\color[rgb]{0,0,0}\makebox(0,0)[lt]{\lineheight{1.25}\smash{\begin{tabular}[t]{l}$K_1$\end{tabular}}}}%
    \put(0.68890149,0.15212878){\color[rgb]{0,0,0}\makebox(0,0)[lt]{\lineheight{1.25}\smash{\begin{tabular}[t]{l}$K_r$\end{tabular}}}}%
    \put(0,0){\includegraphics[width=\unitlength,page=5]{basedLink.pdf}}%
    \put(0.15559223,0.15235157){\color[rgb]{0,0,0}\makebox(0,0)[lt]{\lineheight{1.25}\smash{\begin{tabular}[t]{l}$I_1$\end{tabular}}}}%
    \put(0.22418458,0.15188034){\color[rgb]{0,0,0}\makebox(0,0)[lt]{\lineheight{1.25}\smash{\begin{tabular}[t]{l}$I_l$\end{tabular}}}}%
    \put(0,0){\includegraphics[width=\unitlength,page=6]{basedLink.pdf}}%
    \put(0.27701819,0.15279858){\color[rgb]{0,0,0}\makebox(0,0)[lt]{\lineheight{1.25}\smash{\begin{tabular}[t]{l}$J_1$\end{tabular}}}}%
    \put(0.35020648,0.1519044){\color[rgb]{0,0,0}\makebox(0,0)[lt]{\lineheight{1.25}\smash{\begin{tabular}[t]{l}$J_m$\end{tabular}}}}%
    \put(0,0){\includegraphics[width=\unitlength,page=7]{basedLink.pdf}}%
  \end{picture}%
\endgroup%

\end{equation}
We will always use this presentation in the sequel.

\smallskip

\indent To a based $H$-colored link $L$ we associate a $H$-colored $\partial\boldsymbol{\Sigma}$-tangle $\mathbf{T}(L)$ as follows. First using isotopy, represent $L$ as in \eqref{basedLink}. Then at each basepoint $p_i$ apply one of the following substitutions, depending on the orientation of the connected component $L_i$ attached to that basepoint:
\begin{equation}\label{openLoop}
%% Accompanies image file 'openLoop.pdf' (pdf, eps, ps)
\begingroup%
  \makeatletter%
  \providecommand\color[2][]{%
    \errmessage{(Inkscape) Color is used for the text in Inkscape, but the package 'color.sty' is not loaded}%
    \renewcommand\color[2][]{}%
  }%
  \providecommand\transparent[1]{%
    \errmessage{(Inkscape) Transparency is used (non-zero) for the text in Inkscape, but the package 'transparent.sty' is not loaded}%
    \renewcommand\transparent[1]{}%
  }%
  \providecommand\rotatebox[2]{#2}%
  \newcommand*\fsize{\dimexpr\f@size pt\relax}%
  \newcommand*\lineheight[1]{\fontsize{\fsize}{#1\fsize}\selectfont}%
  \ifx\svgwidth\undefined%
    \setlength{\unitlength}{421.49997071bp}%
    \ifx\svgscale\undefined%
      \relax%
    \else%
      \setlength{\unitlength}{\unitlength * \real{\svgscale}}%
    \fi%
  \else%
    \setlength{\unitlength}{\svgwidth}%
  \fi%
  \global\let\svgwidth\undefined%
  \global\let\svgscale\undefined%
  \makeatother%
  \begin{picture}(1,0.12457394)%
    \lineheight{1}%
    \setlength\tabcolsep{0pt}%
    \put(0,0){\includegraphics[width=\unitlength,page=1]{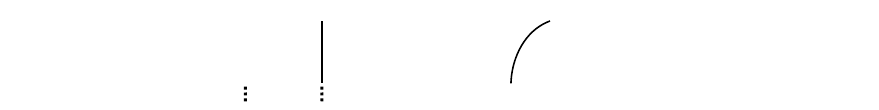}}%
    \put(0.80149571,0.01817934){\color[rgb]{0,0,0}\makebox(0,0)[lt]{\lineheight{1.25}\smash{\begin{tabular}[t]{l}$V_i$\end{tabular}}}}%
    \put(0,0){\includegraphics[width=\unitlength,page=2]{openLoop.pdf}}%
    \put(0.55344323,0.02554131){\color[rgb]{0,0,0}\makebox(0,0)[lt]{\lineheight{1.25}\smash{\begin{tabular}[t]{l}$V_i$\end{tabular}}}}%
    \put(0,0){\includegraphics[width=\unitlength,page=3]{openLoop.pdf}}%
    \put(0.00758293,0.02746101){\color[rgb]{0,0,0}\makebox(0,0)[lt]{\lineheight{1.25}\smash{\begin{tabular}[t]{l}$V_i$\end{tabular}}}}%
    \put(0,0){\includegraphics[width=\unitlength,page=4]{openLoop.pdf}}%
    \put(0.24957584,0.02691851){\color[rgb]{0,0,0}\makebox(0,0)[lt]{\lineheight{1.25}\smash{\begin{tabular}[t]{l}$V_i$\end{tabular}}}}%
    \put(0,0){\includegraphics[width=\unitlength,page=5]{openLoop.pdf}}%
    \put(0.19735778,0.04895647){\color[rgb]{0,0,0}\makebox(0,0)[lt]{\lineheight{1.25}\smash{\begin{tabular}[t]{l}$\mapsto$\end{tabular}}}}%
    \put(0.74735825,0.05829805){\color[rgb]{0,0,0}\makebox(0,0)[lt]{\lineheight{1.25}\smash{\begin{tabular}[t]{l}$\mapsto$\end{tabular}}}}%
    \put(0.09383504,0.11228939){\color[rgb]{0,0,0}\makebox(0,0)[lt]{\lineheight{1.25}\smash{\begin{tabular}[t]{l}$p_i$\end{tabular}}}}%
    \put(0.27292033,0.11275842){\color[rgb]{0,0,0}\makebox(0,0)[lt]{\lineheight{1.25}\smash{\begin{tabular}[t]{l}$p_i^-$\end{tabular}}}}%
    \put(0.36134569,0.11266355){\color[rgb]{0,0,0}\makebox(0,0)[lt]{\lineheight{1.25}\smash{\begin{tabular}[t]{l}$p_i^+$\end{tabular}}}}%
    \put(0,0){\includegraphics[width=\unitlength,page=6]{openLoop.pdf}}%
    \put(0.82201649,0.11271645){\color[rgb]{0,0,0}\makebox(0,0)[lt]{\lineheight{1.25}\smash{\begin{tabular}[t]{l}$p_i^-$\end{tabular}}}}%
    \put(0.92065016,0.11226709){\color[rgb]{0,0,0}\makebox(0,0)[lt]{\lineheight{1.25}\smash{\begin{tabular}[t]{l}$p_i^+$\end{tabular}}}}%
    \put(0,0){\includegraphics[width=\unitlength,page=7]{openLoop.pdf}}%
    \put(0.63831903,0.1122894){\color[rgb]{0,0,0}\makebox(0,0)[lt]{\lineheight{1.25}\smash{\begin{tabular}[t]{l}$p_i$\end{tabular}}}}%
    \put(0,0){\includegraphics[width=\unitlength,page=8]{openLoop.pdf}}%
  \end{picture}%
\endgroup%

\end{equation}
where the resulting boundary points are arranged so that $p_1^-, p_1^+, \ldots, p_k^-, p_k^+$ (in this order) have strictly increasing heights (to demistify the reason of the right hand-side assignment, look at \eqref{dessinMatriceInverse} as well as at the proof of Proposition \ref{holonomySimpleLoop} and at Remark \ref{remarqueAvecCaracteres} below). Using this transformation, we define the holonomy of a $H$-colored based link as
\[ \mathrm{hol}(L) = \mathrm{hol}\bigl( \mathbf{T}(L) \bigr). \]
By definition, we have 
\[
\mathrm{hol}(L) \in \mathcal{L}_{g,n}(H) \otimes V_1 \otimes V_1^* \otimes \ldots \otimes V_k \otimes V_k^* = \mathcal{L}_{g,n}(H) \otimes \mathrm{End}_{\mathbb{C}}(V_1) \otimes \ldots \otimes \mathrm{End}_{\mathbb{C}}(V_k). \]
Hence $\mathrm{hol}(L)$ can be written as $x \otimes N_1 \otimes \ldots \otimes N_k$ ($N_i \in \mathrm{End}_{\mathbb{C}}(V_i)$), with implicit summation to simplify notations.

\begin{proposition}\label{actionSurHol}
The right action \eqref{actionH} of $h \in H$ on the coefficients of the tensor $\mathrm{hol}(L)$ is given by
\[ \mathrm{hol}(L) \cdot h = x \otimes \overset{V_1}{h^{(1)}} N_{1} \overset{V_1}{S(h^{(2)})} \otimes \ldots \otimes \overset{V_k}{h^{(2k-1)}} N_{k} \overset{V_k}{S(h^{(2k)})} \]
where $\Delta^{(2k-1)}(h) = h^{(1)} \otimes \ldots \otimes h^{(2k)}$ is the coproduct iterated $2k-1$ times.
\end{proposition}
\begin{proof}
This is straightforward. For instance, consider the following link $L \subset \Sigma_{1,0}^{\mathrm{o},\bullet}$ with two basepoints:
\begin{center}
%% Accompanies image file 'tanglePreuveInv.pdf' (pdf, eps, ps)
\begingroup%
  \makeatletter%
  \providecommand\color[2][]{%
    \errmessage{(Inkscape) Color is used for the text in Inkscape, but the package 'color.sty' is not loaded}%
    \renewcommand\color[2][]{}%
  }%
  \providecommand\transparent[1]{%
    \errmessage{(Inkscape) Transparency is used (non-zero) for the text in Inkscape, but the package 'transparent.sty' is not loaded}%
    \renewcommand\transparent[1]{}%
  }%
  \providecommand\rotatebox[2]{#2}%
  \newcommand*\fsize{\dimexpr\f@size pt\relax}%
  \newcommand*\lineheight[1]{\fontsize{\fsize}{#1\fsize}\selectfont}%
  \ifx\svgwidth\undefined%
    \setlength{\unitlength}{466.24977592bp}%
    \ifx\svgscale\undefined%
      \relax%
    \else%
      \setlength{\unitlength}{\unitlength * \real{\svgscale}}%
    \fi%
  \else%
    \setlength{\unitlength}{\svgwidth}%
  \fi%
  \global\let\svgwidth\undefined%
  \global\let\svgscale\undefined%
  \makeatother%
  \begin{picture}(1,0.24052189)%
    \lineheight{1}%
    \setlength\tabcolsep{0pt}%
    \put(0,0){\includegraphics[width=\unitlength,page=1]{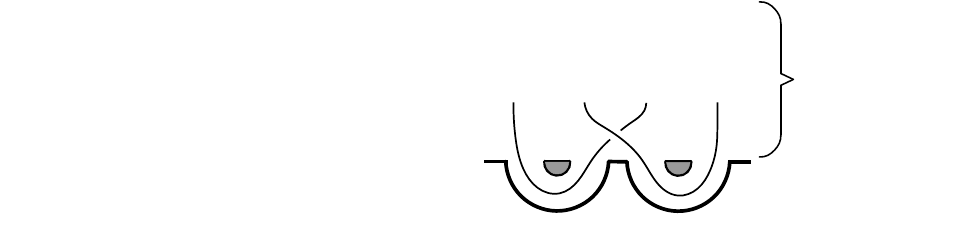}}%
    \put(0.82495609,0.15226195){\color[rgb]{0,0,0}\makebox(0,0)[lt]{\lineheight{1.25}\smash{\begin{tabular}[t]{l}$T_{\mathrm{tot}}$\end{tabular}}}}%
    \put(0.61580084,0.01293417){\color[rgb]{0,0,0}\makebox(0,0)[lt]{\lineheight{1.25}\smash{\begin{tabular}[t]{l}$\overset{I}{B}$\end{tabular}}}}%
    \put(0.7427796,0.0120121){\color[rgb]{0,0,0}\makebox(0,0)[lt]{\lineheight{1.25}\smash{\begin{tabular}[t]{l}$\overset{J}{A}$\end{tabular}}}}%
    \put(0.51105762,0.09987802){\color[rgb]{0,0,0}\makebox(0,0)[lt]{\lineheight{1.25}\smash{\begin{tabular}[t]{l}$I$\end{tabular}}}}%
    \put(0.74217268,0.09982683){\color[rgb]{0,0,0}\makebox(0,0)[lt]{\lineheight{1.25}\smash{\begin{tabular}[t]{l}$J$\end{tabular}}}}%
    \put(0,0){\includegraphics[width=\unitlength,page=2]{tanglePreuveInv.pdf}}%
    \put(0.02683856,0.10424863){\color[rgb]{0,0,0}\makebox(0,0)[lt]{\lineheight{1.25}\smash{\begin{tabular}[t]{l}$I$\end{tabular}}}}%
    \put(0,0){\includegraphics[width=\unitlength,page=3]{tanglePreuveInv.pdf}}%
    \put(0.14511477,0.13755837){\color[rgb]{0,0,0}\makebox(0,0)[lt]{\lineheight{1.25}\smash{\begin{tabular}[t]{l}$T$\end{tabular}}}}%
    \put(0.10193394,0.10450161){\color[rgb]{0,0,0}\makebox(0,0)[lt]{\lineheight{1.25}\smash{\begin{tabular}[t]{l}$J$\end{tabular}}}}%
    \put(0.12831212,0.18100642){\color[rgb]{0,0,0}\makebox(0,0)[lt]{\lineheight{1.25}\smash{\begin{tabular}[t]{l}$V_1$\end{tabular}}}}%
    \put(0,0){\includegraphics[width=\unitlength,page=4]{tanglePreuveInv.pdf}}%
    \put(0.61942446,0.14229389){\color[rgb]{0,0,0}\makebox(0,0)[lt]{\lineheight{1.25}\smash{\begin{tabular}[t]{l}$T$\end{tabular}}}}%
    \put(0,0){\includegraphics[width=\unitlength,page=5]{tanglePreuveInv.pdf}}%
    \put(0.50013147,0.2183839){\color[rgb]{0,0,0}\makebox(0,0)[lt]{\lineheight{1.25}\smash{\begin{tabular}[t]{l}$V_1$\end{tabular}}}}%
    \put(0,0){\includegraphics[width=\unitlength,page=6]{tanglePreuveInv.pdf}}%
    \put(0.3897259,0.12279121){\color[rgb]{0,0,0}\makebox(0,0)[lt]{\lineheight{1.25}\smash{\begin{tabular}[t]{l}$\mathrm{hol}$\end{tabular}}}}%
    \put(0,0){\includegraphics[width=\unitlength,page=7]{tanglePreuveInv.pdf}}%
    \put(0.26504139,0.18100642){\color[rgb]{0,0,0}\makebox(0,0)[lt]{\lineheight{1.25}\smash{\begin{tabular}[t]{l}$V_2$\end{tabular}}}}%
    \put(0,0){\includegraphics[width=\unitlength,page=8]{tanglePreuveInv.pdf}}%
    \put(0.66524204,0.21971784){\color[rgb]{0,0,0}\makebox(0,0)[lt]{\lineheight{1.25}\smash{\begin{tabular}[t]{l}$V_2$\end{tabular}}}}%
  \end{picture}%
\endgroup%

\end{center}
Then using the identification \eqref{decMatrice} together with \eqref{actionH}, \eqref{trucEvidentdAlgebreLineaire} and the fact that $F_{\mathrm{RT}}(T_{\mathrm{tot}}) : I \otimes I^* \otimes J \otimes J^* \to V_1 \otimes V_1^* \otimes V_2 \otimes V_2^*$ is $H$-linear, we obtain
\begin{align*}
\mathrm{hol}(L) \cdot h &= \bigl(\overset{I}{B}{^i_j} \overset{J}{A}{^k_l}\bigr) \cdot h \otimes F_{\mathrm{RT}}(T_{\mathrm{tot}})\bigl( x_i \otimes x^j \otimes y_k \otimes y^l \bigr)\\
& = \bigl( \overset{I}{h^{(1)}} \, \overset{I}{B} \, \overset{I}{S(h^{(2)})} \bigr)^i_j \, \bigl( \overset{J}{h^{(3)}} \, \overset{J}{A} \, \overset{J}{S(h^{(4)})} \bigr)^k_l \otimes F_{\mathrm{RT}}(T_{\mathrm{tot}})\bigl( x_i \otimes x^j \otimes y_k \otimes y^l \bigr)\\
& = \overset{I}{B}{^i_j}\overset{J}{A}{^k_l} \otimes F_{\mathrm{RT}}(T_{\mathrm{tot}})\bigl( h^{(1)}x_i \otimes h^{(2)}x^j \otimes h^{(3)}y_k \otimes h^{(4)}y^l \bigr)\\
& = \overset{I}{B}{^i_j}\overset{J}{A}{^k_l} \otimes \bigg((h^{(1)} \otimes h^{(2)} \otimes h^{(3)} \otimes h^{(4)}) \, F_{\mathrm{RT}}(T_{\mathrm{tot}})\bigl( x_i \otimes x^j \otimes y_k \otimes y^l \bigr)\bigg)\\
& = x \, (N_1)^m_n \, (N_2)^o_p \otimes h^{(1)}v_m \otimes h^{(2)}v^n \otimes h^{(3)}w_o \otimes h^{(4)}w^p\\
& = x \, \big(\overset{V_1}{h^{(1)}} N_1 \overset{V_1}{S(h^{(2)})} \big)^m_n \, \big(\overset{V_2}{h^{(3)}} N_2 \overset{V_2}{S(h^{(4)})} \big)^o_p  \otimes v_m \otimes v^n \otimes w_o \otimes w^p\\
& = x \otimes \overset{V_1}{h^{(1)}} N_{1} \overset{V_1}{S(h^{(2)})} \otimes \overset{V_2}{h^{(3)}} N_2 \overset{V_2}{S(h^{(4)})}
\end{align*}
where $(x_i), (y_i), (v_i), (w_i)$ are basis of $I,J,V_1,V_2$ respectively.
\end{proof}

\smallskip

\indent A simple loop $x \in \pi_1(\Sigma_{g,n}^{\mathrm{o}})$ can naturally be viewed as a based link $x \subset \Sigma_{g,n}^{\mathrm{o},\bullet} \times \{0\}$. For the next result we will use the mapping class group, so we restrict to $n=0$ for simplicity. We denote by $\mathrm{hol}^I(x)$ the value of $\mathrm{hol}$ on the simple loop $x$ colored by $I$. Thanks to Theorem \ref{MCGcommuteW}, we have for all $f \in \mathrm{MCG}(\Sigma_{g,0}^{\mathrm{o}})$
\begin{equation}\label{coroHolLinksMCG}
\widetilde{f}\bigl( \mathrm{hol}^I(x) \bigr) = \mathrm{hol}^I\bigl( f(x) \bigr)
\end{equation}
where the matrix at the left-hand side is obviously defined by $\widetilde{f}\bigl( \mathrm{hol}^I(x) \bigr)^i_j = \widetilde{f}\bigl( \mathrm{hol}^I(x)^i_j\bigr)$. Also recall the lift of a simple loop defined in \S \ref{sectionDefLgnH}.

\begin{proposition}\label{holonomySimpleLoop}
For any simple loop $x \in \pi_1(\Sigma_{g,0}^{\mathrm{o}})$ and finite dimensional $H$-module $I$ it holds
\[ \mathrm{hol}^I(x) = \overset{I}{\widetilde{x}}.\]
\end{proposition}
\begin{proof}
Observe first that if the result is true for some simple loop $x$, then it is true for every $f(x)$, where $f \in \mathrm{MCG}(\Sigma_{g,0}^{\mathrm{o},\bullet})$. Indeed, thanks to \eqref{coroHolLinksMCG} and \eqref{pasSurprenant}:
\[ \mathrm{hol}^I\bigl(f(x)\bigr) = \widetilde{f}\bigl( \mathrm{hol}^I(x) \bigr) = \widetilde{f}\bigl(\overset{I}{\widetilde{x}}\bigr) =  \overset{I}{\widetilde{f(x)}}. \]
Now, one can check directly that the result is true for the loops 
\begin{equation}\label{canonicalSimpleLoops}
a_1, \: s_1 = b_1a_1^{-1}b_1^{-1}a_1, \: s_2 = b_1a_1^{-1}b_1^{-1}a_1b_2a_2^{-1}b_2^{-1}a_2, \: \ldots, \: s_g =  b_1a_1^{-1}b_1^{-1}a_1 \ldots b_ga_g^{-1}b_g^{-1}a_g.
\end{equation}
These are positively oriented simple loops (see \eqref{positivelyOriented}). $a_1$ is a non-separating loop while $s_1, \ldots, s_g$ are separating loops exhausting all the possible topological types of separating loops (see \cite[\S 1.3.1]{FM}). Hence, if $y$ is a positively oriented simple loop, there exists a homeomorphism $f \in \mathrm{MCG}(\Sigma_{g,0}^{\mathrm{o}})$ such that $y = f(x)$ where $x$ is one of the loops in \eqref{canonicalSimpleLoops}. It follows that the result is true for any positively oriented simple loop. The result for negatively oriented simple loops is deduced as follows:
\begin{center}
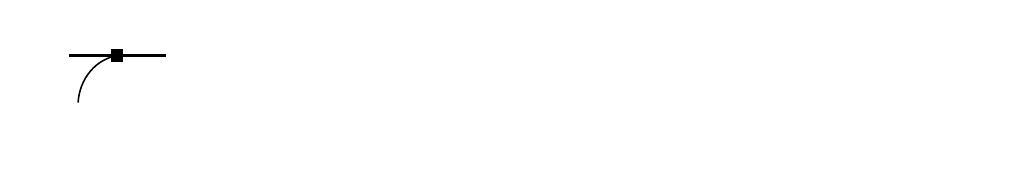
\end{center}
where $T_x$ is the part of $x$ outside a small neighborhood of the basepoint and $D_{x^{-1}}$ is the corresponding part of diagram (but with opposite orientation) obtained when applying $\mathrm{hol}^I$. For the first equality we used the definition of $\mathrm{hol}$ together with \eqref{changeOrientationTangle}, for the second we applied the previously established result for the positively oriented loop $x^{-1}$ and also the naturality of braiding and twist, for the third we used \eqref{sensOppose} and for the fourth we used Remark \ref{remarkInverseLift}. We are done since by definition $\bigl(\overset{I}{\widetilde{x^{-1}}}\bigr)^{-1} = \overset{I}{\widetilde{x}}$.
\end{proof}
\indent Recall what is holonomy in the classical case. For a discrete $G$-connection $\nabla \in \mathcal{A}_{g,n} = \mathrm{Hom}\bigl( \pi_1(\Sigma_{g,n}^{\mathrm{o}}), G \bigr)$ and a simple loop $x \in \pi_1(\Sigma_{g,n}^{\mathrm{o}})$, the holonomy is simply $\mathrm{hol}(x)(\nabla) = \nabla(x) \in G$. Now, if $V$ is a finite dimensional representation of $G$, we define $\mathrm{hol}^V(x)(\nabla)$ to be $\overset{V}{\nabla(x)}$, namely the representation of $\nabla(x)$ on $V$. By taking matrix coefficients, we obtain functions on $\mathcal{A}_{g,n}$:
\[ \mathrm{hol}^V(x)^i_j : \mathcal{A}_{g,n} \to \mathbb{C}, \:\:\:\:\: \nabla \mapsto \overset{V}{\nabla(x)}{^i_j} \]
Hence $\mathrm{hol}^V(x) \in \mathbb{C}[\mathcal{A}_{g,n}] \otimes \mathrm{End}_{\mathbb{C}}(V)$. Let $\overset{V}{B}(1) = \mathrm{hol}^V(b_1), \, \overset{V}{A}(1) = \mathrm{hol}^V(a_1), \ldots$ Now express $x$ in terms of the generators of $\pi_1(\Sigma_{g,n}^{\mathrm{o}})$ and replace each generator by the corresponding matrix in the representation $V$: $b_1 \mapsto \overset{V}{B}(1), a_1 \mapsto \overset{V}{A}(1), \ldots$; this defines the lift $\overset{V}{\widetilde{x}} \in \mathbb{C}[\mathcal{A}_{g,n}] \otimes \mathrm{End}_{\mathbb{C}}(V)$ of $x$ in that representation. It is clear that $\mathrm{hol}^V(x) = \overset{V}{\widetilde{x}}$, which is exactly the same formula as in Proposition \ref{holonomySimpleLoop} and then justifies the name ``holonomy'' for Definition \ref{defHol}.

\subsection{Hennings formulation of the holonomy of a based link}\label{Sectionhennings}

\indent We now give a description \textit{\`a la} Hennings \cite{hennings} of the holonomy which will be useful in the next section. This description applies only to based links.

\smallskip

\indent We will use universal elements. These are elements $\mathcal{B}(i), \mathcal{A}(i), \mathcal{M}(i) \in \mathcal{L}_{g,n}(H) \otimes H$ such that
\begin{equation}\label{elementUniversel}
\overset{I}{B}(i) = (\mathrm{id} \otimes \overset{I}{T})\bigl( \mathcal{B}(i) \bigr), \:\:\:\:\: \overset{I}{A}(i) = (\mathrm{id} \otimes \overset{I}{T})\bigl( \mathcal{A}(i) \bigr), \:\:\:\:\: \overset{I}{M}(i) = (\mathrm{id} \otimes \overset{I}{T})\bigl( \mathcal{M}(i) \bigr) 
\end{equation}
for all finite dimensional $H$-module $I$ and all $i$ (we recall that $\overset{I}{T}$ is the representation morphism of $H$ on $I$). In the sequel we assume that $H$ is such that these elements exist; this assumption is fulfilled by relevant classes of Hopf algebras:
\begin{lemma}
If $H$ is finite dimensional or if $H = U_q(\mathfrak{g})$ with $q$ generic, then there exist universal elements satisfying \eqref{elementUniversel}.
\end{lemma}
\begin{proof}
We can restrict to $\mathcal{L}_{0,1}(H)$, and we let $\overset{I}{M} = \overset{I}{M}(1)$ for all $I$. Assume first that $H$ is finite dimensional; then we can use the regular representation $H$. Write $\overset{H}{M} = x_i \otimes f_i \in \mathcal{L}_{0,1}(H)  \otimes \mathrm{End}_{\mathbb{C}}(H)$ and define $\mathcal{M} = x_i \otimes f_i(1) \in \mathcal{L}_{0,1}(H)  \otimes H$. Let $\rho_a \in \mathrm{End}_H(H)$ defined by $\rho_a(h) = ha$. By  \eqref{fonctorialite} it holds $\overset{H}{M} \rho_a = \rho_a \overset{H}{M}$, thus we have for all $a \in H$:
\[ x_i \otimes f_i(a) = x_i \otimes f_i \circ \rho_a(1) = x_i \otimes \rho_a \circ f_i(1) = x_i \otimes f_i(1)a = x_i \otimes \overset{H}{T}\big(f_i(1)\big)(a) \]
which shows that $\overset{H}{M} = (\mathrm{id} \otimes \overset{H}{T})(\mathcal{M})$. Next consider a direct sum $H^{\oplus N}$, and let $j_{\alpha} : H \to H^{\oplus N}$ and $p_{\alpha} : H^{\oplus N} \to H$ be the associated canonical injections and projections. By \eqref{fonctorialite}, we have
\[ \overset{H^{\oplus N}}{M} = \sum_{\alpha} j_{\alpha} \overset{H}{M} p_{\alpha} = \sum_{\alpha} (\mathrm{id} \otimes j_{\alpha}\overset{H}{T} p_{\alpha})(\mathcal{M}) = (\mathrm{id} \otimes \overset{H^{\oplus N}}{T})(\mathcal{M}). \]
Finally, let $I$ be any finite dimensional $H$-module. Since $I$ is finitely generated there is a surjective morphism $\pi : H^{\oplus N} \to I$ for some $N$ and hence $\overset{I}{M} \pi = \pi \overset{H^{\oplus N}}{M} = (\mathrm{id} \otimes \pi\overset{H^{\oplus N}}{T})(\mathcal{M}) = (\mathrm{id} \otimes \overset{I}{T}\pi)(\mathcal{M})$. But since $\pi$ is surjective, it is right invertible and thus $\overset{I}{M} = (\mathrm{id} \otimes \overset{I}{T})(\mathcal{M})$.
\\Now take $H = U_q(\mathfrak{g})$. The map $\Phi : \mathcal{L}_{0,1}(H) \to H$ defined by $\overset{I}{M} \to (\mathrm{id}_H \otimes \overset{I}{T})(R'R)$ is an injective morphism of algebras (see \cite[Th. 4.3]{BaR} and the references therein). Then under the identification $\mathcal{L}_{0,1}(H) \cong \mathrm{im}(\Phi)$ we have $\overset{I}{M} = (\mathrm{id} \otimes \overset{I}{T})(\mathcal{M})$, where $\mathcal{M} = R'R \in \mathrm{im}(\Phi) \otimes H$. Note that $\mathcal{M}$ actually belongs to some completion but we will not discuss this detail here.
\end{proof}
We write 
\[ \mathcal{B}(i) = \mathcal{B}(i)_0 \otimes \mathcal{B}(i)_1, \:\:\:\:\:\: \mathcal{A}(i) = \mathcal{A}(i)_0 \otimes \mathcal{A}(i)_1, \:\:\:\:\:\: \mathcal{M}(i) = \mathcal{M}(i)_0 \otimes \mathcal{M}(i)_1 \]
with implicit summation to simplify notations. By definition, the right action \eqref{actionH} of $H$ on $\mathcal{L}_{g,n}(H)$ satisfies
\[ \mathcal{X}_0\cdot h \otimes \mathcal{X}_1 = \mathcal{X}_0 \otimes h' \mathcal{X}_1 S(h'') \]
where $\mathcal{X}$ is $\mathcal{B}(i), \mathcal{A}(i)$ or $\mathcal{M}(i)$

\smallskip

\indent Let $L$ be a based link presented as in \eqref{basedLink} but without coloring. First, for each basepoint $p_i$ we define a coupon $\mathrm{Or}_i$ depending on the orientation of the strand and which is the analogue of \eqref{openLoop}:
\begin{center}
%% Accompanies image file 'couponOr.pdf' (pdf, eps, ps)
\begingroup%
  \makeatletter%
  \providecommand\color[2][]{%
    \errmessage{(Inkscape) Color is used for the text in Inkscape, but the package 'color.sty' is not loaded}%
    \renewcommand\color[2][]{}%
  }%
  \providecommand\transparent[1]{%
    \errmessage{(Inkscape) Transparency is used (non-zero) for the text in Inkscape, but the package 'transparent.sty' is not loaded}%
    \renewcommand\transparent[1]{}%
  }%
  \providecommand\rotatebox[2]{#2}%
  \newcommand*\fsize{\dimexpr\f@size pt\relax}%
  \newcommand*\lineheight[1]{\fontsize{\fsize}{#1\fsize}\selectfont}%
  \ifx\svgwidth\undefined%
    \setlength{\unitlength}{399.72259014bp}%
    \ifx\svgscale\undefined%
      \relax%
    \else%
      \setlength{\unitlength}{\unitlength * \real{\svgscale}}%
    \fi%
  \else%
    \setlength{\unitlength}{\svgwidth}%
  \fi%
  \global\let\svgwidth\undefined%
  \global\let\svgscale\undefined%
  \makeatother%
  \begin{picture}(1,0.12148484)%
    \lineheight{1}%
    \setlength\tabcolsep{0pt}%
    \put(0,0){\includegraphics[width=\unitlength,page=1]{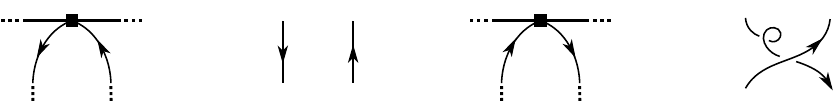}}%
    \put(0.19716468,0.04224217){\color[rgb]{0,0,0}\makebox(0,0)[lt]{\lineheight{1.25}\smash{\begin{tabular}[t]{l}$\Rightarrow \:\: \mathrm{Or}_i =$\end{tabular}}}}%
    \put(0.09894727,0.10902559){\color[rgb]{0,0,0}\makebox(0,0)[lt]{\lineheight{1.25}\smash{\begin{tabular}[t]{l}$p_i$\end{tabular}}}}%
    \put(0.66216248,0.10876511){\color[rgb]{0,0,0}\makebox(0,0)[lt]{\lineheight{1.25}\smash{\begin{tabular}[t]{l}$p_i$\end{tabular}}}}%
    \put(0.76100262,0.0515332){\color[rgb]{0,0,0}\makebox(0,0)[lt]{\lineheight{1.25}\smash{\begin{tabular}[t]{l}$\Rightarrow \:\: \mathrm{Or}_i =$\end{tabular}}}}%
  \end{picture}%
\endgroup%

\end{center}
\noindent Then we form the following diagram:
\begin{equation}\label{defHolHennings}
%% Accompanies image file 'defHolHennings.pdf' (pdf, eps, ps)
\begingroup%
  \makeatletter%
  \providecommand\color[2][]{%
    \errmessage{(Inkscape) Color is used for the text in Inkscape, but the package 'color.sty' is not loaded}%
    \renewcommand\color[2][]{}%
  }%
  \providecommand\transparent[1]{%
    \errmessage{(Inkscape) Transparency is used (non-zero) for the text in Inkscape, but the package 'transparent.sty' is not loaded}%
    \renewcommand\transparent[1]{}%
  }%
  \providecommand\rotatebox[2]{#2}%
  \newcommand*\fsize{\dimexpr\f@size pt\relax}%
  \newcommand*\lineheight[1]{\fontsize{\fsize}{#1\fsize}\selectfont}%
  \ifx\svgwidth\undefined%
    \setlength{\unitlength}{389.99999055bp}%
    \ifx\svgscale\undefined%
      \relax%
    \else%
      \setlength{\unitlength}{\unitlength * \real{\svgscale}}%
    \fi%
  \else%
    \setlength{\unitlength}{\svgwidth}%
  \fi%
  \global\let\svgwidth\undefined%
  \global\let\svgscale\undefined%
  \makeatother%
  \begin{picture}(1,0.40551137)%
    \lineheight{1}%
    \setlength\tabcolsep{0pt}%
    \put(0,0){\includegraphics[width=\unitlength,page=1]{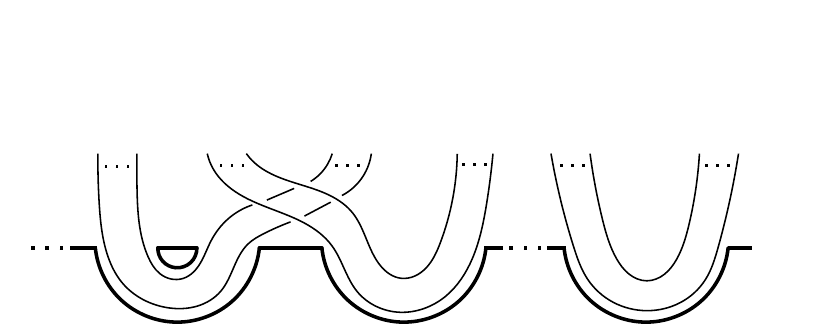}}%
    \put(0.29083938,0.00260801){\color[rgb]{0,0,0}\makebox(0,0)[lt]{\lineheight{1.25}\smash{\begin{tabular}[t]{l}$\mathcal{B}(i)$\end{tabular}}}}%
    \put(0.56993552,0.0056261){\color[rgb]{0,0,0}\makebox(0,0)[lt]{\lineheight{1.25}\smash{\begin{tabular}[t]{l}$\mathcal{A}(i)$\end{tabular}}}}%
    \put(0.8726386,0.00600205){\color[rgb]{0,0,0}\makebox(0,0)[lt]{\lineheight{1.25}\smash{\begin{tabular}[t]{l}$\mathcal{M}(j)$\end{tabular}}}}%
    \put(0.51420018,0.22980582){\color[rgb]{0,0,0}\makebox(0,0)[lt]{\lineheight{1.25}\smash{\begin{tabular}[t]{l}$T$\end{tabular}}}}%
    \put(0,0){\includegraphics[width=\unitlength,page=2]{defHolHennings.pdf}}%
    \put(0.16892534,0.29849659){\color[rgb]{0,0,0}\makebox(0,0)[lt]{\lineheight{1.25}\smash{\begin{tabular}[t]{l}$\mathrm{Or}_1$\end{tabular}}}}%
    \put(0,0){\includegraphics[width=\unitlength,page=3]{defHolHennings.pdf}}%
    \put(0.22958602,0.38757469){\color[rgb]{0,0,0}\makebox(0,0)[lt]{\lineheight{1.25}\smash{\begin{tabular}[t]{l}$1$\end{tabular}}}}%
    \put(0,0){\includegraphics[width=\unitlength,page=4]{defHolHennings.pdf}}%
    \put(0.78623302,0.2984966){\color[rgb]{0,0,0}\makebox(0,0)[lt]{\lineheight{1.25}\smash{\begin{tabular}[t]{l}$\mathrm{Or}_k$\end{tabular}}}}%
    \put(0,0){\includegraphics[width=\unitlength,page=5]{defHolHennings.pdf}}%
    \put(0.84498106,0.38903164){\color[rgb]{0,0,0}\makebox(0,0)[lt]{\lineheight{1.25}\smash{\begin{tabular}[t]{l}$k$\end{tabular}}}}%
    \put(0,0){\includegraphics[width=\unitlength,page=6]{defHolHennings.pdf}}%
  \end{picture}%
\endgroup%

\end{equation}
We put the dotted line to stress that the diagram can be deformed by isotopy, but with the restriction that all the points below (resp. above) the line must remain below (resp. above) the line. The evaluation of this diagram is an element $\mathrm{Hen}(L) \in \mathcal{L}_{g,n}(H) \otimes H^{\otimes k}$ computed according to the following rules, which generalize those defining the Hennings invariant. Starting from one of the basepoints  labelled $1, \ldots, k$ and following the strand according to its orientation, we multiply together from left to right the values associated to each graphical element that we encounter until we come back to the basepoint:
\begin{center}
%% Accompanies image file 'anseHennings.pdf' (pdf, eps, ps)
\begingroup%
  \makeatletter%
  \providecommand\color[2][]{%
    \errmessage{(Inkscape) Color is used for the text in Inkscape, but the package 'color.sty' is not loaded}%
    \renewcommand\color[2][]{}%
  }%
  \providecommand\transparent[1]{%
    \errmessage{(Inkscape) Transparency is used (non-zero) for the text in Inkscape, but the package 'transparent.sty' is not loaded}%
    \renewcommand\transparent[1]{}%
  }%
  \providecommand\rotatebox[2]{#2}%
  \newcommand*\fsize{\dimexpr\f@size pt\relax}%
  \newcommand*\lineheight[1]{\fontsize{\fsize}{#1\fsize}\selectfont}%
  \ifx\svgwidth\undefined%
    \setlength{\unitlength}{419.15262059bp}%
    \ifx\svgscale\undefined%
      \relax%
    \else%
      \setlength{\unitlength}{\unitlength * \real{\svgscale}}%
    \fi%
  \else%
    \setlength{\unitlength}{\svgwidth}%
  \fi%
  \global\let\svgwidth\undefined%
  \global\let\svgscale\undefined%
  \makeatother%
  \begin{picture}(1,0.1517084)%
    \lineheight{1}%
    \setlength\tabcolsep{0pt}%
    \put(0,0){\includegraphics[width=\unitlength,page=1]{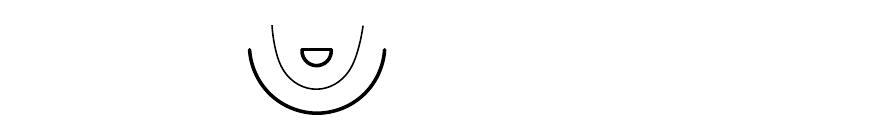}}%
    \put(0.42977248,0.03011199){\color[rgb]{0,0,0}\makebox(0,0)[lt]{\lineheight{1.25}\smash{\begin{tabular}[t]{l}$\mathcal{X}$\end{tabular}}}}%
    \put(0,0){\includegraphics[width=\unitlength,page=2]{anseHennings.pdf}}%
    \put(0.25280342,0.13336376){\color[rgb]{0,0,0}\makebox(0,0)[lt]{\lineheight{1.25}\smash{\begin{tabular}[t]{l}$g^{-1}S(\mathcal{X}_1^{(l)})$\end{tabular}}}}%
    \put(0,0){\includegraphics[width=\unitlength,page=3]{anseHennings.pdf}}%
    \put(0.16923491,0.02920834){\color[rgb]{0,0,0}\makebox(0,0)[lt]{\lineheight{1.25}\smash{\begin{tabular}[t]{l}$\mathcal{X}$\end{tabular}}}}%
    \put(0,0){\includegraphics[width=\unitlength,page=4]{anseHennings.pdf}}%
    \put(0.14515849,0.13607979){\color[rgb]{0,0,0}\makebox(0,0)[lt]{\lineheight{1.25}\smash{\begin{tabular}[t]{l}$\mathcal{X}_1^{(l)}$\end{tabular}}}}%
    \put(0,0){\includegraphics[width=\unitlength,page=5]{anseHennings.pdf}}%
    \put(0.80830713,0.14205441){\color[rgb]{0,0,0}\makebox(0,0)[lt]{\lineheight{1.25}\smash{\begin{tabular}[t]{l}$g^{-1}$\end{tabular}}}}%
    \put(0,0){\includegraphics[width=\unitlength,page=6]{anseHennings.pdf}}%
    \put(0.99138461,0.13776162){\color[rgb]{0,0,0}\makebox(0,0)[lt]{\lineheight{1.25}\smash{\begin{tabular}[t]{l}$1$\end{tabular}}}}%
    \put(0,0){\includegraphics[width=\unitlength,page=7]{anseHennings.pdf}}%
    \put(0.87811397,0.00264305){\color[rgb]{0,0,0}\makebox(0,0)[lt]{\lineheight{1.25}\smash{\begin{tabular}[t]{l}$1$\end{tabular}}}}%
    \put(0,0){\includegraphics[width=\unitlength,page=8]{anseHennings.pdf}}%
    \put(0.92876661,0.00264306){\color[rgb]{0,0,0}\makebox(0,0)[lt]{\lineheight{1.25}\smash{\begin{tabular}[t]{l}$g$\end{tabular}}}}%
    \put(0,0){\includegraphics[width=\unitlength,page=9]{anseHennings.pdf}}%
    \put(0.651626,0.02111792){\color[rgb]{0,0,0}\makebox(0,0)[lt]{\lineheight{1.25}\smash{\begin{tabular}[t]{l}$b_i$\end{tabular}}}}%
    \put(0.72641156,0.02068914){\color[rgb]{0,0,0}\makebox(0,0)[lt]{\lineheight{1.25}\smash{\begin{tabular}[t]{l}$S(a_i)$\end{tabular}}}}%
    \put(0,0){\includegraphics[width=\unitlength,page=10]{anseHennings.pdf}}%
    \put(0.51488345,0.02377918){\color[rgb]{0,0,0}\makebox(0,0)[lt]{\lineheight{1.25}\smash{\begin{tabular}[t]{l}$a_i$\end{tabular}}}}%
    \put(0.5951129,0.02356788){\color[rgb]{0,0,0}\makebox(0,0)[lt]{\lineheight{1.25}\smash{\begin{tabular}[t]{l}$b_i$\end{tabular}}}}%
    \put(0,0){\includegraphics[width=\unitlength,page=11]{anseHennings.pdf}}%
  \end{picture}%
\endgroup%

\end{center}
Here $\mathcal{X} = \mathcal{X}_0 \otimes \mathcal{X}_1$ is one of the $\mathcal{B}(i), \mathcal{A}(i), \mathcal{M}(i)$ and $\Delta^{(m)}(\mathcal{X}_1) = \mathcal{X}_1^{(1)} \otimes \ldots \otimes \mathcal{X}_1^{(m+1)}$ is the iterated coproduct. Note that the handle contains a bunch of parallel strands and the picture represents the value associated to the $l$-th strand. Then by applying this rule for each starting point, we obtain elements $Z_1, \ldots, Z_k \in H$, and we define
\[ \mathrm{Hen}(L) = \bigl(\mathcal{B}(1)_0 \, \mathcal{A}(1)_0 \ldots \mathcal{B}(g)_0 \, \mathcal{A}(g)_0 \, \mathcal{M}(g+1)_0 \ldots \mathcal{M}(g+n)_0\bigr) \otimes Z_1 \otimes \ldots \otimes Z_k.  \]
\noindent Here is an example in $\Sigma_{0,1}^{\mathrm{o}} \times [0,1]$:
\begin{center}
%% Accompanies image file 'exempleHolHennings.pdf' (pdf, eps, ps)
\begingroup%
  \makeatletter%
  \providecommand\color[2][]{%
    \errmessage{(Inkscape) Color is used for the text in Inkscape, but the package 'color.sty' is not loaded}%
    \renewcommand\color[2][]{}%
  }%
  \providecommand\transparent[1]{%
    \errmessage{(Inkscape) Transparency is used (non-zero) for the text in Inkscape, but the package 'transparent.sty' is not loaded}%
    \renewcommand\transparent[1]{}%
  }%
  \providecommand\rotatebox[2]{#2}%
  \newcommand*\fsize{\dimexpr\f@size pt\relax}%
  \newcommand*\lineheight[1]{\fontsize{\fsize}{#1\fsize}\selectfont}%
  \ifx\svgwidth\undefined%
    \setlength{\unitlength}{428.31073855bp}%
    \ifx\svgscale\undefined%
      \relax%
    \else%
      \setlength{\unitlength}{\unitlength * \real{\svgscale}}%
    \fi%
  \else%
    \setlength{\unitlength}{\svgwidth}%
  \fi%
  \global\let\svgwidth\undefined%
  \global\let\svgscale\undefined%
  \makeatother%
  \begin{picture}(1,0.20143505)%
    \lineheight{1}%
    \setlength\tabcolsep{0pt}%
    \put(0,0){\includegraphics[width=\unitlength,page=1]{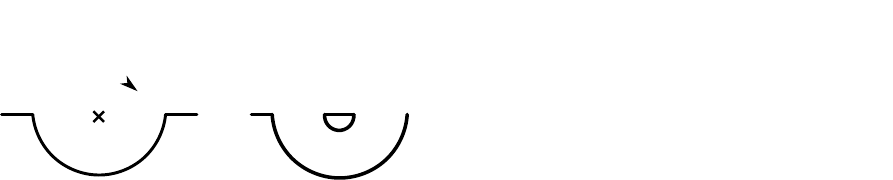}}%
    \put(0.44545119,0.00895719){\color[rgb]{0,0,0}\makebox(0,0)[lt]{\lineheight{1.25}\smash{\begin{tabular}[t]{l}$\mathcal{M}$\end{tabular}}}}%
    \put(0,0){\includegraphics[width=\unitlength,page=2]{exempleHolHennings.pdf}}%
    \put(0.23710315,0.06838286){\color[rgb]{0,0,0}\makebox(0,0)[lt]{\lineheight{1.25}\smash{\begin{tabular}[t]{l}$\mapsto$\end{tabular}}}}%
    \put(0.50034701,0.06869562){\color[rgb]{0,0,0}\makebox(0,0)[lt]{\lineheight{1.25}\smash{\begin{tabular}[t]{l}$= \: \mathcal{M}_0 \otimes gv^{-1}S(a_i)g^{-1}S(\mathcal{M}_1')a_jg^{-1}S(\mathcal{M}_1'')b_jb_i$\end{tabular}}}}%
    \put(0,0){\includegraphics[width=\unitlength,page=3]{exempleHolHennings.pdf}}%
  \end{picture}%
\endgroup%

\end{center}
where we used that $b_iga_i = v^{-1}$.

\smallskip

\indent The evaluation of $\mathrm{Hen}(L)$ on representations almost gives back the holonomy:
\begin{equation}\label{HenUniversal}
(\mathrm{id} \otimes \overset{V_1}{T} \otimes \ldots \otimes \overset{V_k}{T})\big( \mathrm{Hen}(L) \big) = \big(\overset{V_1}{g} \otimes \ldots \otimes \overset{V_k}{g}\big) \mathrm{hol}^{V_1, \ldots, V_k}(L) 
\end{equation}
where $\mathrm{hol}^{V_1, \ldots, V_k}$ means the holonomy of the based link whose strands are colored by $V_1, \ldots, V_k$. This equality is the generalization of the corresponding fact for the Hennings invariant and the proof is left to the reader. Note that the pivotal elements in the right-hand side come from the caps at the top of \eqref{defHolHennings}. Let us write $\mathrm{Hen}(L) = Z_0 \otimes Z_1 \otimes \ldots \otimes Z_k \in \mathcal{L}_{g,n}(H) \otimes H^{\otimes k}$ with implicit summation. Then according to Proposition \ref{actionSurHol} and \eqref{pivot}, the right action of $H$ on $\mathcal{L}_{g,n}(H)$ satisfies
\begin{equation}\label{actionSurHen}
(Z_0 \cdot h) \otimes Z_1 \otimes \ldots \otimes Z_k = Z_0 \otimes S^2(h^{(1)}) Z_1 S(h^{(2)}) \otimes \ldots \otimes S^2(h^{(2k-1)}) Z_k S(h^{(2k)}).
\end{equation}

\subsection{Generalized Wilson loops}
\indent In the classical case, a Wilson loop around a simple curve $\gamma$ assigns to a connection the trace of the holonomy of $\gamma$ in some representation; this quantity is gauge invariant. We will straightforwardly generalize this to the quantized case, by replacing the simple curve by a based link and by replacing the trace by a quantum trace (or more generally by a symmetric linear form shifted by $g$).

\smallskip

\indent Let $L$ be an uncolored based link with $k$ connected components. Let $f \in (H^{\circ})^{\otimes k}$ (recall that $H^{\circ}$ is the restricted dual of $H$). By definition $f$ is a linear combination of the form 
\[ f = \sum_{I_1, \ldots, I_k} (\Lambda_{I_1, \ldots, I_k})^{i_1, \ldots, i_k}_{j_1, \ldots, j_k} \, \overset{I_1}{T}{^{\,i_1}_{\,j_1}} \otimes \ldots \otimes \overset{I_k}{T}{^{\,i_k}_{\,j_k}} \]
with implicit summation on $i_1, \ldots, i_k, j_1, \ldots, j_k$ and where the $I_l$ are isomorphism classes of finite dimensional $H$-modules and the coefficients $(\Lambda_{I_1, \ldots, I_k})^{i_1, \ldots, i_k}_{j_1, \ldots, j_k} \in \mathbb{C}$ are all zero except a finite number of them. This can be written more conveniently as $f = \sum_{I_1, \ldots, I_k} \mathrm{tr}\bigr( \Lambda_{I_1, \ldots, I_k} \, \overset{I_1}{T} \otimes \ldots \otimes \overset{I_k}{T} \, \bigl)$ where the matrices $\Lambda_{I_1, \ldots, I_k} \in \mathrm{End}_{\mathbb{C}}(I_1 \otimes \ldots \otimes I_k)$ are all zero except a finite number of them. We define 
\[ W^f(L)  = \sum_{I_1, \ldots, I_k} \mathrm{tr}\big( \Lambda_{I_1, \ldots, I_k} \, \big(\overset{I_1}{g} \otimes \ldots \otimes \overset{I_k}{g}\big) \, \mathrm{hol}^{I_1, \ldots, I_k}(L) \big)  \in \mathcal{L}_{g,n}(H) \]
where $\mathrm{hol}^{I_1, \ldots, I_k}(L)$ is the holonomy of the based link whose components are colored by $I_1, \ldots, I_k$. 
In terms of the Hennings formulation, we see thanks to \eqref{HenUniversal} that it holds
\[ W^f(L) = (\mathrm{id} \otimes f)\big( \mathrm{Hen}(L) \big). \]

\smallskip

\indent We denote by $\mathrm{Inv}_k(H)$ the subspace of multilinear forms which are invariant under an iterated coadjoint action:
\[ \mathrm{Inv}_k(H) = \big\{ f \in (H^{\circ})^{\otimes k} \, \big| \, \forall \, h \in H, \: f\big( S(h^{(2k)}) ? h^{(2k-1)}, \ldots, S(h^{(2)}) ? h^{(1)} \big) = \varepsilon(h)f \big\} \]
where the ``$?$'' are the positions of the variables. In particular $\mathrm{Inv}_1(H)$ is the subspace of symmetric linear forms:
\[ \mathrm{Inv}_1(H) = \mathrm{SLF}(H) = \big\{ \varphi \in H^{\circ} \: \big| \: \forall \, x,y \in H, \: \varphi(xy) = \varphi(yx) \big\}. \]

\begin{proposition}\label{propertiesWf}
We have the following properties
\begin{enumerate}
\item $f \in \mathrm{Inv}_k(H) \: \Rightarrow \: W^f(L) \in \mathcal{L}_{g,n}^{\mathrm{inv}}(H)$, where $L$ is a based link with $k$ basepoints.
\item $W^f(L)W^{f'}(L') = W^{f \otimes f'}(L \ast L') $.
\item For $n=0$ and any $j \in \mathrm{MCG}(\Sigma_{g,0}^{\mathrm{o}})$, $\widetilde{j}\bigl( W^f(L) \bigr) = W^f\bigl( j(L) \bigr)$.
\item Any element of $\mathcal{L}_{g,n}(H)$ can be written $W^f(L)$ for some $f \in (H^{\circ})^{\otimes k}$ and some based link $L$ with $k$ basepoints ($k,f,L$ are non-unique).
\item Any element of $\mathcal{L}_{g,n}^{\mathrm{inv}}(H)$ can be written $W^f(L)$ for some $f \in \mathrm{Inv}_k(H)$ and some based link $L$ with $k$ basepoints ($k,f,L$ are non-unique).
\end{enumerate}
\end{proposition}
\begin{proof}
1. Write $\mathrm{Hen}(L) = Z_0 \otimes Z_1 \otimes \ldots \otimes Z_k$. Then thanks to \eqref{actionSurHen}:
\begin{align*}
W^f(L) \cdot h &= (Z_0 \cdot h) \, f(Z_1, \ldots, Z_k) = Z_0 \, f\big(S^2(h^{(1)}) Z_1 S(h^{(2)}), \ldots, S^2(h^{(2k-1)}) Z_k S(h^{(2k)})\big)\\
&= Z_0 \, f\big(S\big( S(h)^{(2k)}\big) Z_1 S(h)^{(2k-1)}, \ldots, S\big(S(h)^{(2)}\big) Z_k S(h)^{(1)}\big)\\
&= Z_0 \, \varepsilon\big(S(h)\big)f\big(Z_1, \ldots, Z_k \big) = \varepsilon(h)W^f(L).
\end{align*}
2. Follows immediately from Theorem \ref{wilsonStack}.\\
3. Follows immediately from Theorem \ref{MCGcommuteW}.\\
4.  Let $x \in \mathcal{L}_{g,n}(H)$. Thanks to the defining relations of $\mathcal{L}_{g,n}(H)$ (Definition \ref{defLgn}), it is clear that any element of $\mathcal{L}_{g,n}(H)$ can be written as a linear combination of some products of the form
\[ \overset{I_1}{B}(1)^{i_1}_{j_1} \, \overset{J_1}{A}(1)^{k_1}_{l_1} \, \ldots \, \overset{I_g}{B}(g)^{i_g}_{j_g} \, \overset{J_g}{A}(g)^{k_g}_{l_g} \, \overset{K_1}{M}(g+1)^{m_1}_{o_1} \, \ldots \, \overset{K_n}{M}(g+n)^{m_n}_{o_n}. \]
To avoid cumbersome notations, take for instance $(g,n)=(1,1)$. Then we can write
\[ x = \sum_{I,J,K} \big(\Lambda_{I,J,K}\overset{I\otimes J \otimes K}{g}\big)_{ikl}^{jlo} \, \overset{I}{B}{^i_j} \, \overset{J}{A}{^k_l} \,  \overset{K}{M}{^l_o} = \sum_{I,J,K} \mathrm{tr}\big(\Lambda_{I,J,K}\overset{I\otimes J \otimes K}{g} \overset{I}{B} \, \overset{J}{A} \,  \overset{K}{M}\big) \]
where only a finite number of the matrices $\Lambda_{I,J,K} \in \mathrm{End}_{\mathbb{C}}(I \otimes J \otimes K)$ are non-zero, and inserting $g$ is not an issue since it is invertible. Take $f=\sum_{I,J,K} \mathrm{tr}\big(\Lambda_{I,J,K} \, \overset{I}{T} \otimes \overset{J}{T} \otimes  \overset{K}{T}\big) \in (H^{\circ})^{\otimes 3}$ and 
\begin{center}
%% Accompanies image file 'linkWfSurjectif.pdf' (pdf, eps, ps)
\begingroup%
  \makeatletter%
  \providecommand\color[2][]{%
    \errmessage{(Inkscape) Color is used for the text in Inkscape, but the package 'color.sty' is not loaded}%
    \renewcommand\color[2][]{}%
  }%
  \providecommand\transparent[1]{%
    \errmessage{(Inkscape) Transparency is used (non-zero) for the text in Inkscape, but the package 'transparent.sty' is not loaded}%
    \renewcommand\transparent[1]{}%
  }%
  \providecommand\rotatebox[2]{#2}%
  \newcommand*\fsize{\dimexpr\f@size pt\relax}%
  \newcommand*\lineheight[1]{\fontsize{\fsize}{#1\fsize}\selectfont}%
  \ifx\svgwidth\undefined%
    \setlength{\unitlength}{272.24285498bp}%
    \ifx\svgscale\undefined%
      \relax%
    \else%
      \setlength{\unitlength}{\unitlength * \real{\svgscale}}%
    \fi%
  \else%
    \setlength{\unitlength}{\svgwidth}%
  \fi%
  \global\let\svgwidth\undefined%
  \global\let\svgscale\undefined%
  \makeatother%
  \begin{picture}(1,0.37449927)%
    \lineheight{1}%
    \setlength\tabcolsep{0pt}%
    \put(0,0){\includegraphics[width=\unitlength,page=1]{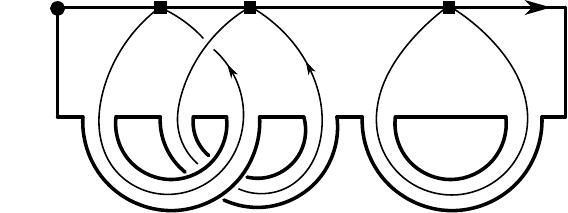}}%
    \put(-0.00370206,0.19907216){\color[rgb]{0,0,0}\makebox(0,0)[lt]{\lineheight{1.25}\smash{\begin{tabular}[t]{l}$L =$\end{tabular}}}}%
  \end{picture}%
\endgroup%

\end{center}
By definition we have $W^f(L) = x$.\\
5. Follows by combining the proofs of the properties 1 and 4.
\end{proof}

\begin{remark}\label{baseInvariants}
Let $L$ be the obvious generalization to any $g,n$ of the based link used in the previous proof. Then we have shown that
\begin{align*}
(H^{\circ})^{\otimes (2g+n)} \twoheadrightarrow \mathcal{L}_{g,n}(H), &\:\:\:\: f \mapsto W^f(L),\\
\mathrm{Inv}_{2g+n}(H) \twoheadrightarrow \mathcal{L}_{g,n}^{\mathrm{inv}}(H), &\:\:\:\: f \mapsto W^f(L)
\end{align*}
are surjective linear maps. When $H$ is finite dimensional and factorizable, we have $\dim\big( \mathcal{L}_{g,n}(H) \big) = \dim(H)^{2g+n}$ (see \cite[\S 3.3]{Fai18c}) and by comparison of dimensions these maps are isomorphisms of vector spaces.
\end{remark}

\indent For a based link $L$, we define its free isotopy class $[L]$ to be the result of the following operation (applied to all the basepoints):
\begin{center}
%% Accompanies image file 'unbasing.pdf' (pdf, eps, ps)
\begingroup%
  \makeatletter%
  \providecommand\color[2][]{%
    \errmessage{(Inkscape) Color is used for the text in Inkscape, but the package 'color.sty' is not loaded}%
    \renewcommand\color[2][]{}%
  }%
  \providecommand\transparent[1]{%
    \errmessage{(Inkscape) Transparency is used (non-zero) for the text in Inkscape, but the package 'transparent.sty' is not loaded}%
    \renewcommand\transparent[1]{}%
  }%
  \providecommand\rotatebox[2]{#2}%
  \newcommand*\fsize{\dimexpr\f@size pt\relax}%
  \newcommand*\lineheight[1]{\fontsize{\fsize}{#1\fsize}\selectfont}%
  \ifx\svgwidth\undefined%
    \setlength{\unitlength}{416.24999745bp}%
    \ifx\svgscale\undefined%
      \relax%
    \else%
      \setlength{\unitlength}{\unitlength * \real{\svgscale}}%
    \fi%
  \else%
    \setlength{\unitlength}{\svgwidth}%
  \fi%
  \global\let\svgwidth\undefined%
  \global\let\svgscale\undefined%
  \makeatother%
  \begin{picture}(1,0.10159872)%
    \lineheight{1}%
    \setlength\tabcolsep{0pt}%
    \put(0,0){\includegraphics[width=\unitlength,page=1]{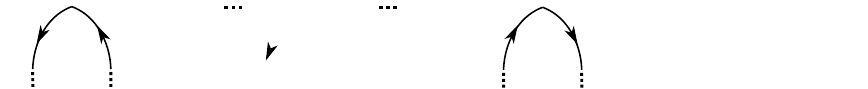}}%
    \put(0.21065774,0.04056496){\color[rgb]{0,0,0}\makebox(0,0)[lt]{\lineheight{1.25}\smash{\begin{tabular}[t]{l}$\mapsto$\end{tabular}}}}%
    \put(0.75137892,0.05002437){\color[rgb]{0,0,0}\makebox(0,0)[lt]{\lineheight{1.25}\smash{\begin{tabular}[t]{l}$\mapsto$\end{tabular}}}}%
    \put(0,0){\includegraphics[width=\unitlength,page=2]{unbasing.pdf}}%
  \end{picture}%
\endgroup%

\end{center}
By definition, $[L] \subset \Sigma_{g,n}^{\mathrm{o}} \times [0,1]$ is a link in the usual sense. Recall that in the classical case, a Wilson loop does not depend on the basepoint, but only on the free isotopy class of the curve. In general, $W^f(L)$ depends on the basepoint. If we want to remove this dependence, we must restrict $f$ to be in a suitable subspace:

\begin{proposition}Let $L_1, L_2$ be based links with $k$ basepoints. If $f \in \mathrm{SLF}(H)^{\otimes k}$, then $W^f$ depends only on the free isotopy class:
\[ [L_1] = [L_2] \:\: \implies \:\: W^f(L_1) = W^f(L_2). \]
\end{proposition}
\begin{proof}
Thanks to the cyclicity of symmetric linear forms, it is clear that the basepoints in the diagram \eqref{defHolHennings} can be moved at any generic point (\textit{i.e.} not at a maximal, minimal or crossing point, nor in a handle) without changing the value of $(\mathrm{id} \otimes \varphi_1 \otimes \ldots \otimes \varphi_k)\bigl( \mathrm{Hen}(L_i)\bigr)$. Moreover, that the evaluation of the diagram \eqref{defHolHennings} is unchanged under isotopy is a generalization of the corresponding fact for the Hennings invariant.
\end{proof}
\noindent It follows that $W^f$ makes sense for (usual) links if $f \in \mathrm{SLF}(H)$. More precisely, let $\mathscr{L}(\Sigma_{g,n}^{\mathrm{o}})$ be the set of isotopy classes of framed links in $\Sigma_{g,n}^{\mathrm{o}} \times [0,1]$ which are $\mathrm{SLF}(H)$-colored (\textit{i.e.} each connected component is labelled by a symmetric linear form). Let $L \in \mathscr{L}(\Sigma_{g,n}^{\mathrm{o}})$ be a link with $k$ components and let $L_b$ be any uncolored based link such that $[L_b] = L$, then we define
\begin{equation}\label{defWilsonGen}
W(L) = W^{\varphi_1 \otimes \ldots \otimes \varphi_k}(L_b)
\end{equation}
where $\varphi_1, \ldots, \varphi_k$ are the colors of the components of $L$.
\begin{definition}\label{defWilsonLoopMap}
We call $W : \mathbb{C}\mathscr{L}(\Sigma_{g,n}^{\mathrm{o}}) \to \mathcal{L}_{g,n}^{\mathrm{inv}}(H)$ the generalized Wilson loop map, and an element $W(L)$ is called a generalized Wilson loop.
\end{definition}
\noindent By the second property in Proposition \ref{propertiesWf}, the generalized Wilson loops form a subalgebra in $\mathcal{L}_{g,n}^{\mathrm{inv}}(H)$.

\begin{remark}\label{remarqueAvecCaracteres}
Recall that the character of a finite dimensional $H$-module $I$ is $\chi^I = \mathrm{tr}\big( \overset{I}{T} \big)$. Let $L^{\chi}$ be an unbased link whose all strands are colored by characters $\chi^{I_1}, \ldots, \chi^{I_k}$ and $L$ be the same link without coloring. Then 
\[ W(L^{\chi}) = \mathrm{hol}^{I_1, \ldots, I_k}(L) \]
and we recover the Wilson loop map defined in \cite[Chap. 6]{these}. Indeed, it suffices to consider the case where there is only one strand. Let $L_b$ be a  based link such that $[L_b] = L$. By definition $W(L^{\chi}) = \mathrm{tr}_q\bigl( \mathrm{hol}^I(L_b) \bigr)$. Hence according to \eqref{openLoop} and \eqref{traceQuantique}, we have:
\begin{center}
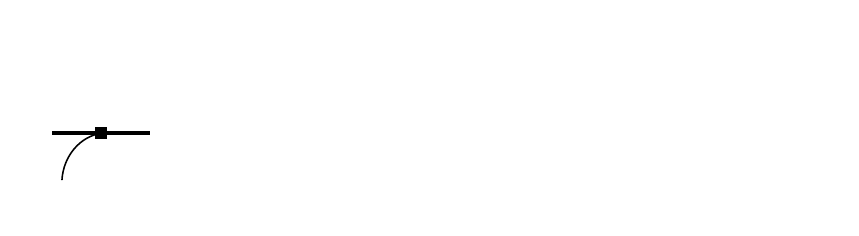
\end{center}
where $T_{L_b}$ is the part of $L_b$ outside a small neighborhood of the basepoint, $D_{L_b}$ is the corresponding diagram obtained when applying $\mathrm{hol}^I$ and $D_L$ is the free isotopy diagram of $D_{L_b}$.
\end{remark}

\section{Relationship with stated skein algebras}\label{sectionSkeinAlg}
\indent Recall the definition \ref{defBordTangle} of $\partial \boldsymbol{\Sigma}$-tangles, which is just a slight generalization of \cite[\S 2.3]{Le}. In this section and as in \cite{Le, CL}, we will consider $\partial \boldsymbol{\Sigma}$-tangles which are uncolored and unoriented. Recall from \cite[\S 2.5]{Le} that the {\em stated skein algebra} $\mathcal{S}^{\mathrm{s}}_q(\boldsymbol{\Sigma})$ of a punctured bordered surface $\boldsymbol{\Sigma}$ is the $\mathbb{Z}[q^{1/2}, q^{-1/2}]$-module freely generated by the (isotopy classes of) $ \partial \boldsymbol{\Sigma}$-tangles on that surface, modulo the Kauffman relations:
\begin{equation}\label{kauffmanUq2}
%% Accompanies image file 'kauffmanUq2.pdf' (pdf, eps, ps)
\begingroup%
  \makeatletter%
  \providecommand\color[2][]{%
    \errmessage{(Inkscape) Color is used for the text in Inkscape, but the package 'color.sty' is not loaded}%
    \renewcommand\color[2][]{}%
  }%
  \providecommand\transparent[1]{%
    \errmessage{(Inkscape) Transparency is used (non-zero) for the text in Inkscape, but the package 'transparent.sty' is not loaded}%
    \renewcommand\transparent[1]{}%
  }%
  \providecommand\rotatebox[2]{#2}%
  \newcommand*\fsize{\dimexpr\f@size pt\relax}%
  \newcommand*\lineheight[1]{\fontsize{\fsize}{#1\fsize}\selectfont}%
  \ifx\svgwidth\undefined%
    \setlength{\unitlength}{301.80688023bp}%
    \ifx\svgscale\undefined%
      \relax%
    \else%
      \setlength{\unitlength}{\unitlength * \real{\svgscale}}%
    \fi%
  \else%
    \setlength{\unitlength}{\svgwidth}%
  \fi%
  \global\let\svgwidth\undefined%
  \global\let\svgscale\undefined%
  \makeatother%
  \begin{picture}(1,0.101312)%
    \lineheight{1}%
    \setlength\tabcolsep{0pt}%
    \put(0,0){\includegraphics[width=\unitlength,page=1]{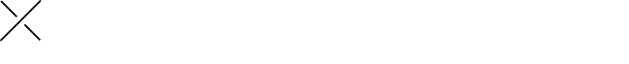}}%
    \put(0.08300571,0.06273809){\color[rgb]{0,0,0}\makebox(0,0)[lt]{\lineheight{1.25}\smash{\begin{tabular}[t]{l}$= \: q$\end{tabular}}}}%
    \put(0.2414581,0.06121874){\color[rgb]{0,0,0}\makebox(0,0)[lt]{\lineheight{1.25}\smash{\begin{tabular}[t]{l}$+ \:\, q^{-1}$\end{tabular}}}}%
    \put(0,0){\includegraphics[width=\unitlength,page=2]{kauffmanUq2.pdf}}%
    \put(0.6148487,0.05516126){\color[rgb]{0,0,0}\makebox(0,0)[lt]{\lineheight{1.25}\smash{\begin{tabular}[t]{l}$= \: -(q^2+q^{-2})\varnothing$\end{tabular}}}}%
    \put(0,0){\includegraphics[width=\unitlength,page=3]{kauffmanUq2.pdf}}%
  \end{picture}%
\endgroup%

\end{equation}
and the boundary relations:
\begin{equation}\label{boundaryRels}
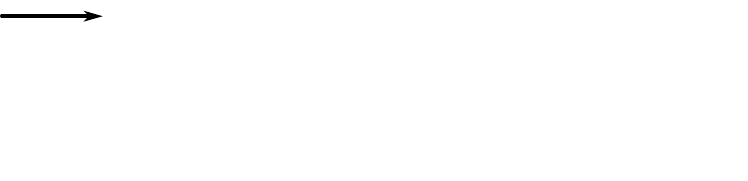
\end{equation}
The conventions for the figures are the same as in \cite{Le, CL}. However, we stress that our convention for the stack product (Definition \ref{stackProduct}) is {\em opposite} to that of \cite{Le, CL}; in other words, the stated skein algebras considered here have {\em opposite multiplication} to those in these papers.

\smallskip

\indent In this section we work with $H = U_{q^2} = U_{q^2}(\mathfrak{sl}_2)$, the quantum group associated to $\mathfrak{sl}_2(\mathbb{C})$ with ground ring $\mathbb{Z}[q^{1/2}, q^{-1/2}]$, where $q$ is a formal parameter or a complex number which is not a root of unity. Let $\mathcal{O}_{q^2} = \mathcal{O}_{q^2}(\mathfrak{sl}_2)$ be the restricted dual of $U_{q^2}$ with its canonical Hopf algebra structure. The goal of this section is to show that the holonomy map provides an isomorphism of $\mathcal{O}_{q^2}$-comodule-algebras between $\mathcal{S}^{\mathrm{s}}_q(\Sigma_{g,n}^{\mathrm{o},\bullet})$ and $\mathcal{L}_{g,n}(U_{q^2})$. We recall that $\Sigma_{g,n}^{\mathrm{o},\bullet}$ is the surface $\Sigma_{g,n}\!\setminus\! D$ (where $D$ is an open disk) with one point removed on its boundary (Figure \ref{surfaceGN}), where $\Sigma_{g,n}$ is the compact oriented surface of genus $g$ with $n$ punctures. We begin with useful preliminary remarks.

\smallskip

\indent Let $V_2$ be a 2-dimensional $\mathbb{C}$-vector space with basis $\bigl(v_-, v_+\bigr)$. $V_2$ can be endowed with an $U_{q^2}$-module structure, called the fundamental representation of $U_{q^2}$. The action is given by:
\[ \begin{array}{lll}
Ev_- = 0, \:\:\:& Fv_- = v_+, \:\:\:& Kv_- = q^2v_-,\\
Ev_+ = v_-, \:\:\:& Fv_+ = 0, \:\:\:& Kv_+ = q^{-2}v_+.
\end{array} \]
It is well-known that every finite dimensional simple $U_{q^2} $-module is a direct summand of some tensor power $V_2^{\otimes N}$. It follows that $\mathcal{O}_{q^2}$ is generated as an algebra by the matrix coefficients $\overset{V_2}{T}{^s_t}$ of $V_2$ (with $s,t \in \{\pm\}$). This also applies to the algebra $\mathcal{L}_{g,n}(U_{q^2})$; indeed, let $B(i) = \overset{V_2}{B}(i), A(i) = \overset{V_2}{A}(i), M(i) = \overset{V_2}{M}(i)$. For any matrix $X  \in \mathcal{L}_{g,n}(U_{q^2}) \otimes \mathrm{End}_{\mathbb{C}}(V_2)$, we denote by $X^s_t$ ($s,t \in \{ \pm \}$) its components in the basis $\bigl(v_-, v_+\bigr)$:
\[
X =
\begin{pmatrix}
X^-_- & X^-_+\\
X^+_- & X^+_+
\end{pmatrix}.
\]
Finally, let $R \in \mathrm{End}_{\mathbb{C}}(V_2)^{\otimes 2}$ be the $R$-matrix of $U_{q^2}$ evaluated on $V_2 \otimes V_2$, and let $R_{21}$ be the matrix defined by $(R_{21})^{s_1s_2}_{t_1t_2} = R^{s_2s_1}_{t_2t_1}$. Explicitly:
\begin{equation}\label{RMatriceV2}
R=q^{-1}
\begin{pmatrix}
q^2 & 0 & 0 & 0 \\
0 & 1 & q^2-q^{-2} & 0\\
0 & 0 & 1 & 0\\
0 & 0 & 0 & q^2
\end{pmatrix}, \:\:\:\:\:\:\:\: 
R_{21}=q^{-1}
\begin{pmatrix}
q^2 & 0 & 0 & 0 \\
0 & 1 & 0 & 0\\
0 & q^2-q^{-2} & 1 & 0\\
0 & 0 & 0 & q^2
\end{pmatrix}.
\end{equation}
\begin{lemma}\label{GensRelsLgnUq2}
The algebra $\mathcal{L}_{g,n}(U_{q^2})$ is generated by the coefficients of the matrices $B(i), A(i)$ for $1 \leq i \leq g$ and $M(i)$ for $g+1\leq i \leq g+n$, modulo the relations given by the following matrix identities:
\[
\begin{array}{ll}
R \, X(i)_1\, R_{21} \, X(i)_2 = X(i)_2 \, R \, X(i)_1\, R_{21}, &\:\:\: \text{for } 1 \leq i \leq g+n,\\[2pt]
X(i)^-_-X(i)^+_+ - q^4X(i)^-_+X(i)^+_- = 1 &\:\:\: \text{for } 1 \leq i \leq g+n,\\[2pt]
R \, B(i)_1\, R_{21} \, A(i)_2 = A(i)_2\, R \, B(i)_1 \, R^{-1} &\:\:\: \text{for } 1 \leq i \leq g,\\[2pt]
R \, X(i)_1\, R^{-1} \, Y(j)_2 = Y(j)_2 \, R \, X(i)_1 \, R^{-1} &\:\:\: \text{for } 1 \leq i < j \leq g+n.
\end{array}
\]
where $X(i)$ is $A(i)$ or $B(i)$ if $1 \leq i \leq g$ and is $M(i)$ if $g+1 \leq i \leq g+n$, and the same applies to $Y(j)$. Recall that the meaning of the subscripts $1,2$ is explained after Definition \ref{defLgn}.
\end{lemma}
\begin{proof}
Since any finite dimensional $U_{q^2}$-module is a direct summand of some $V_2^{\otimes N}$, $N \in \mathbb{N}$, the fusion relation \eqref{relationFusion} implies the claim about the set of generators. It is well-known (see e.g. \cite[Lemma 5.1]{BaR}) that the fusion relation \eqref{relationFusion} with $I=J=V_2$ is equivalent to the two first lines of relations. The third and fourth lines of relations are just \eqref{echangeL10} and \eqref{echangeLgn} with $I=J=V_2$.
\end{proof}
\indent As a result, we can without loss of generality restrict the color of all the strands in diagrams to be $V_2$, since any color can be recovered by cabling and inserting a Jones-Wenzl idempotent. Hence in this section, all the strands in diagrams are implicitly colored by $V_2$. 

\smallskip

\indent Another important property is that $V_2$ is self-dual. For our purposes we fix a particular isomorphism as follows:
\begin{equation}\label{isoDualUq2}
\begin{array}{rcrcl} 
D& : & V_2^* & \overset{\sim}{\rightarrow} & V_2\\
& & v^- & \mapsto & -q^{5/2} v_+\\
& & v^+ & \mapsto & q^{1/2}v_-
\end{array}
\:\:\:\:\:\:\:\:\:\:\: \text{\it i.e.} \:\:\: D=
\begin{pmatrix}
0 & q^{1/2}\\
-q^{5/2} & 0
\end{pmatrix}
\end{equation}
where $(v^-, v^+)$ is the dual basis of $(v_-, v_+)$. Recall that by definition of the $H$-module structure on $I^*$, it holds $\overset{I^*}{h} = \exposantGauche{t}{\!\bigl(}\overset{I}{S(h)}\bigr)$, where $^t$ denotes the transpose and we recall that $\overset{I}{x}$ means the representation of $x \in H$ on the finite dimensional $H$-module $I$. If we take $I=V_2$, we have $\overset{V_2}{h}D = D \overset{V_2^*}{h}$ and hence we deduce the following equality, which will be useful later:
\begin{equation}\label{ruseFondaDual}
\exposantGauche{t}{\!\bigl(} \overset{V_2}{h} \bigr) = D^{-1} \, \overset{V_2}{S^{-1}(h)} \, D.
\end{equation}

\smallskip

\indent There exists a particular pivotal element $g$ lying in some completion of $U_{q^2}$ (see \cite{tingley}) and satisfying
\[ gv_- = -q^2 v_-, \:\:\:\:\: gv_+ = -q^{-2}v_+ \]
(of course $g \neq -K$). With this choice one can check that it holds
\begin{equation}\label{propKD}
e \circ D^* = D \:\:\:\:\:\:\:\:\:\:\: \text{\it i.e.} \:\:\: ^t\!D = \overset{V_2}{g} D
\end{equation}
where $e = e_{V_2} : V_2^{**} \overset{\sim}{\to} V_2$ is the identification with the bidual \eqref{identificationBidual}. This implies that the value of a diagram does not depend of the orientation of its closed components. Indeed:
\begin{center}
%% Accompanies image file 'changeOrientation.pdf' (pdf, eps, ps)
\begingroup%
  \makeatletter%
  \providecommand\color[2][]{%
    \errmessage{(Inkscape) Color is used for the text in Inkscape, but the package 'color.sty' is not loaded}%
    \renewcommand\color[2][]{}%
  }%
  \providecommand\transparent[1]{%
    \errmessage{(Inkscape) Transparency is used (non-zero) for the text in Inkscape, but the package 'transparent.sty' is not loaded}%
    \renewcommand\transparent[1]{}%
  }%
  \providecommand\rotatebox[2]{#2}%
  \newcommand*\fsize{\dimexpr\f@size pt\relax}%
  \newcommand*\lineheight[1]{\fontsize{\fsize}{#1\fsize}\selectfont}%
  \ifx\svgwidth\undefined%
    \setlength{\unitlength}{499.83395212bp}%
    \ifx\svgscale\undefined%
      \relax%
    \else%
      \setlength{\unitlength}{\unitlength * \real{\svgscale}}%
    \fi%
  \else%
    \setlength{\unitlength}{\svgwidth}%
  \fi%
  \global\let\svgwidth\undefined%
  \global\let\svgscale\undefined%
  \makeatother%
  \begin{picture}(1,0.14487761)%
    \lineheight{1}%
    \setlength\tabcolsep{0pt}%
    \put(0,0){\includegraphics[width=\unitlength,page=1]{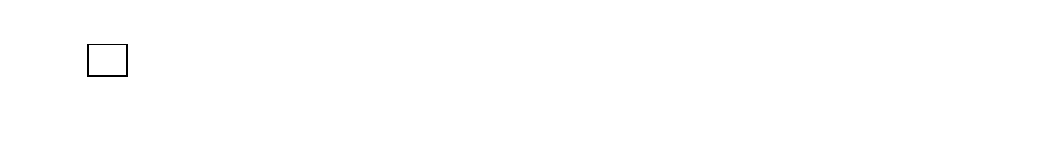}}%
    \put(0.09245059,0.07913032){\color[rgb]{0,0,0}\makebox(0,0)[lt]{\lineheight{1.25}\smash{\begin{tabular}[t]{l}$D$\end{tabular}}}}%
    \put(0,0){\includegraphics[width=\unitlength,page=2]{changeOrientation.pdf}}%
    \put(0.11722515,0.00292178){\color[rgb]{0,0,0}\makebox(0,0)[lt]{\lineheight{1.25}\smash{\begin{tabular}[t]{l}$\overset{V_2}{X}$\end{tabular}}}}%
    \put(0,0){\includegraphics[width=\unitlength,page=3]{changeOrientation.pdf}}%
    \put(0.312204,0.0019448){\color[rgb]{0,0,0}\makebox(0,0)[lt]{\lineheight{1.25}\smash{\begin{tabular}[t]{l}$\overset{V_2}{X}$\end{tabular}}}}%
    \put(0,0){\includegraphics[width=\unitlength,page=4]{changeOrientation.pdf}}%
    \put(0.15959362,0.05102385){\color[rgb]{0,0,0}\makebox(0,0)[lt]{\lineheight{1.25}\smash{\begin{tabular}[t]{l}$=$\end{tabular}}}}%
    \put(0,0){\includegraphics[width=\unitlength,page=5]{changeOrientation.pdf}}%
    \put(0.21659025,0.07960781){\color[rgb]{0,0,0}\makebox(0,0)[lt]{\lineheight{1.25}\smash{\begin{tabular}[t]{l}$D$\end{tabular}}}}%
    \put(0,0){\includegraphics[width=\unitlength,page=6]{changeOrientation.pdf}}%
    \put(0.52436023,0.06663191){\color[rgb]{0,0,0}\makebox(0,0)[lt]{\lineheight{1.25}\smash{\begin{tabular}[t]{l}$=$\end{tabular}}}}%
    \put(0,0){\includegraphics[width=\unitlength,page=7]{changeOrientation.pdf}}%
    \put(0.4739469,0.08005616){\color[rgb]{0,0,0}\makebox(0,0)[lt]{\lineheight{1.25}\smash{\begin{tabular}[t]{l}$D$\end{tabular}}}}%
    \put(0.57219283,0.07967783){\color[rgb]{0,0,0}\makebox(0,0)[lt]{\lineheight{1.25}\smash{\begin{tabular}[t]{l}$D$\end{tabular}}}}%
    \put(0,0){\includegraphics[width=\unitlength,page=8]{changeOrientation.pdf}}%
    \put(0.86215168,0.07843771){\color[rgb]{0,0,0}\makebox(0,0)[lt]{\lineheight{1.25}\smash{\begin{tabular}[t]{l}$=$\end{tabular}}}}%
    \put(0,0){\includegraphics[width=\unitlength,page=9]{changeOrientation.pdf}}%
    \put(0.81227255,0.06467286){\color[rgb]{0,0,0}\makebox(0,0)[lt]{\lineheight{1.25}\smash{\begin{tabular}[t]{l}$D$\end{tabular}}}}%
    \put(0,0){\includegraphics[width=\unitlength,page=10]{changeOrientation.pdf}}%
    \put(0.90980501,0.06467286){\color[rgb]{0,0,0}\makebox(0,0)[lt]{\lineheight{1.25}\smash{\begin{tabular}[t]{l}$D$\end{tabular}}}}%
    \put(0,0){\includegraphics[width=\unitlength,page=11]{changeOrientation.pdf}}%
  \end{picture}%
\endgroup%

\end{center}
(we use \eqref{sensOppose} and the similar fact for cups and caps) and hence by inserting coupons $D$ and $D^{-1}$ on a strand we can change its orientation wihout changing the value of the diagram.
\begin{remark}
If one chooses the usual pivotal element $K$ instead of $g$, then \eqref{propKD} becomes $e \circ D^* = -D$ and the value of a diagram depends up to a sign of the orientation on its closed components, see \cite[Lem 3.18]{KM}, \cite{tingley}. This explains the choice of the less natural-looking pivotal element $g$.
\end{remark}

According to these remarks, we define a non-oriented graphical calculus for $\mathcal{L}_{g,n}(U_{q^2})$, with unoriented cups, caps, crossings and handles:
\begin{center}
%% Accompanies image file 'anse_non_orientee.pdf' (pdf, eps, ps)
\begingroup%
  \makeatletter%
  \providecommand\color[2][]{%
    \errmessage{(Inkscape) Color is used for the text in Inkscape, but the package 'color.sty' is not loaded}%
    \renewcommand\color[2][]{}%
  }%
  \providecommand\transparent[1]{%
    \errmessage{(Inkscape) Transparency is used (non-zero) for the text in Inkscape, but the package 'transparent.sty' is not loaded}%
    \renewcommand\transparent[1]{}%
  }%
  \providecommand\rotatebox[2]{#2}%
  \newcommand*\fsize{\dimexpr\f@size pt\relax}%
  \newcommand*\lineheight[1]{\fontsize{\fsize}{#1\fsize}\selectfont}%
  \ifx\svgwidth\undefined%
    \setlength{\unitlength}{431.26421829bp}%
    \ifx\svgscale\undefined%
      \relax%
    \else%
      \setlength{\unitlength}{\unitlength * \real{\svgscale}}%
    \fi%
  \else%
    \setlength{\unitlength}{\svgwidth}%
  \fi%
  \global\let\svgwidth\undefined%
  \global\let\svgscale\undefined%
  \makeatother%
  \begin{picture}(1,0.18314553)%
    \lineheight{1}%
    \setlength\tabcolsep{0pt}%
    \put(0.13590691,0.02422788){\color[rgb]{0,0,0}\makebox(0,0)[lt]{\lineheight{1.25}\smash{\begin{tabular}[t]{l}$X$\end{tabular}}}}%
    \put(0.36113747,0.02353071){\color[rgb]{0,0,0}\makebox(0,0)[lt]{\lineheight{1.25}\smash{\begin{tabular}[t]{l}$\overset{V_2}{X}$\end{tabular}}}}%
    \put(0,0){\includegraphics[width=\unitlength,page=1]{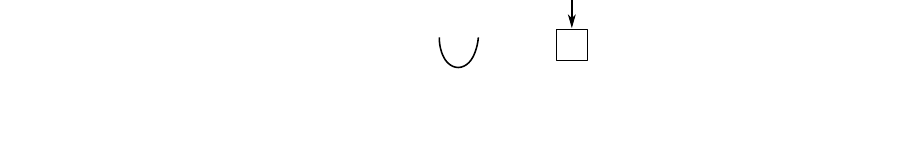}}%
    \put(0.62403796,0.1223501){\color[rgb]{0,0,0}\makebox(0,0)[lt]{\lineheight{1.25}\smash{\begin{tabular}[t]{l}$D$\end{tabular}}}}%
    \put(0,0){\includegraphics[width=\unitlength,page=2]{anse_non_orientee.pdf}}%
    \put(0.32988392,0.10830546){\color[rgb]{0,0,0}\makebox(0,0)[lt]{\lineheight{1.25}\smash{\begin{tabular}[t]{l}$D$\end{tabular}}}}%
    \put(0,0){\includegraphics[width=\unitlength,page=3]{anse_non_orientee.pdf}}%
    \put(0.18099622,0.07680305){\color[rgb]{0,0,0}\makebox(0,0)[lt]{\lineheight{1.25}\smash{\begin{tabular}[t]{l}=\end{tabular}}}}%
    \put(0.55053862,0.11951511){\color[rgb]{0,0,0}\makebox(0,0)[lt]{\lineheight{1.25}\smash{\begin{tabular}[t]{l}=\end{tabular}}}}%
    \put(0,0){\includegraphics[width=\unitlength,page=4]{anse_non_orientee.pdf}}%
    \put(0.82970295,0.12777001){\color[rgb]{0,0,0}\makebox(0,0)[lt]{\lineheight{1.25}\smash{\begin{tabular}[t]{l}=\end{tabular}}}}%
    \put(0,0){\includegraphics[width=\unitlength,page=5]{anse_non_orientee.pdf}}%
    \put(0.89519043,0.12292786){\color[rgb]{0,0,0}\makebox(0,0)[lt]{\lineheight{1.25}\smash{\begin{tabular}[t]{l}$D^{-1}$\end{tabular}}}}%
    \put(0,0){\includegraphics[width=\unitlength,page=6]{anse_non_orientee.pdf}}%
    \put(0.55010761,0.01844094){\color[rgb]{0,0,0}\makebox(0,0)[lt]{\lineheight{1.25}\smash{\begin{tabular}[t]{l}=\end{tabular}}}}%
    \put(0.82970295,0.0193281){\color[rgb]{0,0,0}\makebox(0,0)[lt]{\lineheight{1.25}\smash{\begin{tabular}[t]{l}=\end{tabular}}}}%
  \end{picture}%
\endgroup%

\end{center}
The explicit values of these unoriented graphical elements are:
\begin{equation}\label{evalNonOriente}
%% Accompanies image file 'valeurRmatrice.pdf' (pdf, eps, ps)
\begingroup%
  \makeatletter%
  \providecommand\color[2][]{%
    \errmessage{(Inkscape) Color is used for the text in Inkscape, but the package 'color.sty' is not loaded}%
    \renewcommand\color[2][]{}%
  }%
  \providecommand\transparent[1]{%
    \errmessage{(Inkscape) Transparency is used (non-zero) for the text in Inkscape, but the package 'transparent.sty' is not loaded}%
    \renewcommand\transparent[1]{}%
  }%
  \providecommand\rotatebox[2]{#2}%
  \newcommand*\fsize{\dimexpr\f@size pt\relax}%
  \newcommand*\lineheight[1]{\fontsize{\fsize}{#1\fsize}\selectfont}%
  \ifx\svgwidth\undefined%
    \setlength{\unitlength}{413.79383869bp}%
    \ifx\svgscale\undefined%
      \relax%
    \else%
      \setlength{\unitlength}{\unitlength * \real{\svgscale}}%
    \fi%
  \else%
    \setlength{\unitlength}{\svgwidth}%
  \fi%
  \global\let\svgwidth\undefined%
  \global\let\svgscale\undefined%
  \makeatother%
  \begin{picture}(1,0.23715289)%
    \lineheight{1}%
    \setlength\tabcolsep{0pt}%
    \put(0.13779256,0.08509071){\color[rgb]{0,0,0}\makebox(0,0)[lt]{\lineheight{1.25}\smash{\begin{tabular}[t]{l}$X$\end{tabular}}}}%
    \put(0,0){\includegraphics[width=\unitlength,page=1]{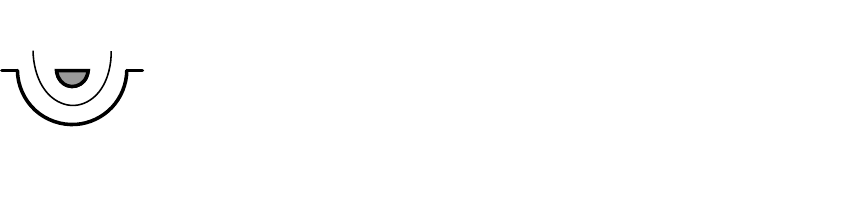}}%
    \put(0.18331141,0.14744469){\color[rgb]{0,0,0}\makebox(0,0)[lt]{\lineheight{1.25}\smash{\begin{tabular}[t]{l}$= \begin{pmatrix} q^{1/2}X^-_+\\-q^{5/2}X^-_-\\q^{1/2}X^+_+\\-q^{5/2}X^+_- \end{pmatrix}$\end{tabular}}}}%
    \put(0.44310404,0.14571){\color[rgb]{0,0,0}\makebox(0,0)[lt]{\lineheight{1.25}\smash{\begin{tabular}[t]{l}$\cup = \begin{pmatrix}0\\-q^{5/2}\\q^{1/2}\\0\end{pmatrix} \:\:\:\:\:\:\:\:\cap = \begin{pmatrix} 0 & q^{-1/2} &  -q^{-5/2} & 0\end{pmatrix}$\end{tabular}}}}%
    \put(0,0){\includegraphics[width=\unitlength,page=2]{valeurRmatrice.pdf}}%
    \put(0.35875549,0.01742025){\color[rgb]{0,0,0}\makebox(0,0)[lt]{\lineheight{1.25}\smash{\begin{tabular}[t]{l}$=PR$\end{tabular}}}}%
    \put(0.57606268,0.01703534){\color[rgb]{0,0,0}\makebox(0,0)[lt]{\lineheight{1.25}\smash{\begin{tabular}[t]{l}$=(PR)^{-1}$\end{tabular}}}}%
  \end{picture}%
\endgroup%

\end{equation}
where $P$ is the flip tensor $P^{ik}_{jl} = \delta^i_l\delta^k_j$, the $X^s_t$ (with $s,t \in \{\pm\}$) are the components of the matrix $X = \overset{V_2}{X}$ and $R$ is the $R$-matrix \eqref{RMatriceV2}. More generally, a handle containing several unoriented strands is obviously defined by putting a coupon $D$ on each strand of the oriented version, and the associated explicit value is computed thanks to the fusion relation \eqref{dessinRelationFusion}. The values \eqref{evalNonOriente} extend  in the usual way to an evaluation map $Z$ of diagrams. When $Z$ is restricted to tangles (diagrams without handles), it is just an unoriented version of the Reshetikhin-Turaev functor for $U_{q^2}$-colored framed tangles. It is well-known that $Z$ satisfies the Kauffman bracket skein relation \eqref{kauffmanUq2}. Let us record for further use that by definition
\begin{center}
%% Accompanies image file 'valeurCroisementComposantes.pdf' (pdf, eps, ps)
\begingroup%
  \makeatletter%
  \providecommand\color[2][]{%
    \errmessage{(Inkscape) Color is used for the text in Inkscape, but the package 'color.sty' is not loaded}%
    \renewcommand\color[2][]{}%
  }%
  \providecommand\transparent[1]{%
    \errmessage{(Inkscape) Transparency is used (non-zero) for the text in Inkscape, but the package 'transparent.sty' is not loaded}%
    \renewcommand\transparent[1]{}%
  }%
  \providecommand\rotatebox[2]{#2}%
  \newcommand*\fsize{\dimexpr\f@size pt\relax}%
  \newcommand*\lineheight[1]{\fontsize{\fsize}{#1\fsize}\selectfont}%
  \ifx\svgwidth\undefined%
    \setlength{\unitlength}{216.10190825bp}%
    \ifx\svgscale\undefined%
      \relax%
    \else%
      \setlength{\unitlength}{\unitlength * \real{\svgscale}}%
    \fi%
  \else%
    \setlength{\unitlength}{\svgwidth}%
  \fi%
  \global\let\svgwidth\undefined%
  \global\let\svgscale\undefined%
  \makeatother%
  \begin{picture}(1,0.11477773)%
    \lineheight{1}%
    \setlength\tabcolsep{0pt}%
    \put(0,0){\includegraphics[width=\unitlength,page=1]{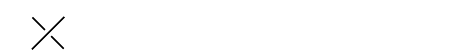}}%
    \put(-0.00308193,0.02466968){\color[rgb]{0,0,0}\makebox(0,0)[lt]{\lineheight{1.25}\smash{\begin{tabular}[t]{l}$Z\Big($\end{tabular}}}}%
    \put(0.14531502,0.02473324){\color[rgb]{0,0,0}\makebox(0,0)[lt]{\lineheight{1.25}\smash{\begin{tabular}[t]{l}$\Big)^{ik}_{jl}=R^{ki}_{jl},$\end{tabular}}}}%
    \put(0,0){\includegraphics[width=\unitlength,page=2]{valeurCroisementComposantes.pdf}}%
    \put(0.48279996,0.02537544){\color[rgb]{0,0,0}\makebox(0,0)[lt]{\lineheight{1.25}\smash{\begin{tabular}[t]{l}$Z\Big($\end{tabular}}}}%
    \put(0.63271808,0.02473323){\color[rgb]{0,0,0}\makebox(0,0)[lt]{\lineheight{1.25}\smash{\begin{tabular}[t]{l}$\Big)^{ik}_{jl}=(R^{-1})^{ik}_{lj}$\end{tabular}}}}%
  \end{picture}%
\endgroup%

\end{center}
where $i,j,k,l \in \{ \pm \}$.

\smallskip

\indent Let us now proceed with preliminary remarks about stated skein algebras. By \eqref{boundaryRels} we have
\begin{center}
%% Accompanies image file 'cupCapBoundary.pdf' (pdf, eps, ps)
\begingroup%
  \makeatletter%
  \providecommand\color[2][]{%
    \errmessage{(Inkscape) Color is used for the text in Inkscape, but the package 'color.sty' is not loaded}%
    \renewcommand\color[2][]{}%
  }%
  \providecommand\transparent[1]{%
    \errmessage{(Inkscape) Transparency is used (non-zero) for the text in Inkscape, but the package 'transparent.sty' is not loaded}%
    \renewcommand\transparent[1]{}%
  }%
  \providecommand\rotatebox[2]{#2}%
  \newcommand*\fsize{\dimexpr\f@size pt\relax}%
  \newcommand*\lineheight[1]{\fontsize{\fsize}{#1\fsize}\selectfont}%
  \ifx\svgwidth\undefined%
    \setlength{\unitlength}{334.30169833bp}%
    \ifx\svgscale\undefined%
      \relax%
    \else%
      \setlength{\unitlength}{\unitlength * \real{\svgscale}}%
    \fi%
  \else%
    \setlength{\unitlength}{\svgwidth}%
  \fi%
  \global\let\svgwidth\undefined%
  \global\let\svgscale\undefined%
  \makeatother%
  \begin{picture}(1,0.12164344)%
    \lineheight{1}%
    \setlength\tabcolsep{0pt}%
    \put(0,0){\includegraphics[width=\unitlength,page=1]{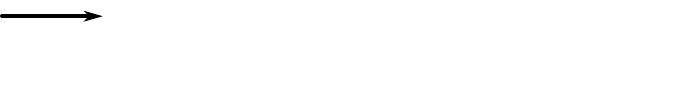}}%
    \put(0.0161684,0.10745079){\color[rgb]{0,0,0}\makebox(0,0)[lt]{\lineheight{1.25}\smash{\begin{tabular}[t]{l}$s$\end{tabular}}}}%
    \put(0,0){\includegraphics[width=\unitlength,page=2]{cupCapBoundary.pdf}}%
    \put(0.09353848,0.107422){\color[rgb]{0,0,0}\makebox(0,0)[lt]{\lineheight{1.25}\smash{\begin{tabular}[t]{l}$t$\end{tabular}}}}%
    \put(0,0){\includegraphics[width=\unitlength,page=3]{cupCapBoundary.pdf}}%
    \put(0.1520195,0.05897559){\color[rgb]{0,0,0}\makebox(0,0)[lt]{\lineheight{1.25}\smash{\begin{tabular}[t]{l}$=\: \cup^{st}$\end{tabular}}}}%
    \put(0,0){\includegraphics[width=\unitlength,page=4]{cupCapBoundary.pdf}}%
    \put(0.6213421,0.00274549){\color[rgb]{0,0,0}\makebox(0,0)[lt]{\lineheight{1.25}\smash{\begin{tabular}[t]{l}$s$\end{tabular}}}}%
    \put(0,0){\includegraphics[width=\unitlength,page=5]{cupCapBoundary.pdf}}%
    \put(0.69871212,0.00277428){\color[rgb]{0,0,0}\makebox(0,0)[lt]{\lineheight{1.25}\smash{\begin{tabular}[t]{l}$t$\end{tabular}}}}%
    \put(0,0){\includegraphics[width=\unitlength,page=6]{cupCapBoundary.pdf}}%
    \put(0.75685434,0.05915947){\color[rgb]{0,0,0}\makebox(0,0)[lt]{\lineheight{1.25}\smash{\begin{tabular}[t]{l}$=\: \cap_{st}$\end{tabular}}}}%
    \put(0,0){\includegraphics[width=\unitlength,page=7]{cupCapBoundary.pdf}}%
  \end{picture}%
\endgroup%

\end{center}
More generally, it is shown in \cite[\S 5]{CL} that 
\begin{equation}\label{RTbigone}
%% Accompanies image file 'counitRT.pdf' (pdf, eps, ps)
\begingroup%
  \makeatletter%
  \providecommand\color[2][]{%
    \errmessage{(Inkscape) Color is used for the text in Inkscape, but the package 'color.sty' is not loaded}%
    \renewcommand\color[2][]{}%
  }%
  \providecommand\transparent[1]{%
    \errmessage{(Inkscape) Transparency is used (non-zero) for the text in Inkscape, but the package 'transparent.sty' is not loaded}%
    \renewcommand\transparent[1]{}%
  }%
  \providecommand\rotatebox[2]{#2}%
  \newcommand*\fsize{\dimexpr\f@size pt\relax}%
  \newcommand*\lineheight[1]{\fontsize{\fsize}{#1\fsize}\selectfont}%
  \ifx\svgwidth\undefined%
    \setlength{\unitlength}{171.0583589bp}%
    \ifx\svgscale\undefined%
      \relax%
    \else%
      \setlength{\unitlength}{\unitlength * \real{\svgscale}}%
    \fi%
  \else%
    \setlength{\unitlength}{\svgwidth}%
  \fi%
  \global\let\svgwidth\undefined%
  \global\let\svgscale\undefined%
  \makeatother%
  \begin{picture}(1,0.4043123)%
    \lineheight{1}%
    \setlength\tabcolsep{0pt}%
    \put(0,0){\includegraphics[width=\unitlength,page=1]{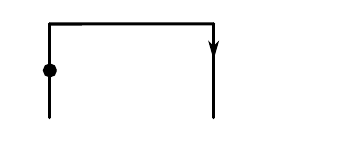}}%
    \put(0.20857457,0.360694){\color[rgb]{0,0,0}\makebox(0,0)[lt]{\lineheight{1.25}\smash{\begin{tabular}[t]{l}$s_1$\end{tabular}}}}%
    \put(0.21047637,0.00860985){\color[rgb]{0,0,0}\makebox(0,0)[lt]{\lineheight{1.25}\smash{\begin{tabular}[t]{l}$t_1$\end{tabular}}}}%
    \put(0,0){\includegraphics[width=\unitlength,page=2]{counitRT.pdf}}%
    \put(0.4935343,0.36289875){\color[rgb]{0,0,0}\makebox(0,0)[lt]{\lineheight{1.25}\smash{\begin{tabular}[t]{l}$s_k$\end{tabular}}}}%
    \put(0.49432141,0.00979249){\color[rgb]{0,0,0}\makebox(0,0)[lt]{\lineheight{1.25}\smash{\begin{tabular}[t]{l}$t_l$\end{tabular}}}}%
    \put(0.72162363,0.19432132){\color[rgb]{0,0,0}\makebox(0,0)[lt]{\lineheight{1.25}\smash{\begin{tabular}[t]{l}$= Z(T)^{s_1 \ldots s_k}_{t_1 \ldots t_l}.$\end{tabular}}}}%
    \put(0,0){\includegraphics[width=\unitlength,page=3]{counitRT.pdf}}%
    \put(0.33675799,0.1872295){\color[rgb]{0,0,0}\makebox(0,0)[lt]{\lineheight{1.25}\smash{\begin{tabular}[t]{l}$T$\end{tabular}}}}%
    \put(0,0){\includegraphics[width=\unitlength,page=4]{counitRT.pdf}}%
    \put(-0.00303038,0.19480983){\color[rgb]{0,0,0}\makebox(0,0)[lt]{\lineheight{1.25}\smash{\begin{tabular}[t]{l}$\varepsilon$\end{tabular}}}}%
    \put(0,0){\includegraphics[width=\unitlength,page=5]{counitRT.pdf}}%
    \put(0.24962337,0.29125141){\color[rgb]{0,0,0}\makebox(0,0)[lt]{\lineheight{1.25}\smash{\begin{tabular}[t]{l}...\end{tabular}}}}%
    \put(0.43048273,0.29125143){\color[rgb]{0,0,0}\makebox(0,0)[lt]{\lineheight{1.25}\smash{\begin{tabular}[t]{l}...\end{tabular}}}}%
    \put(0.24962337,0.11587266){\color[rgb]{0,0,0}\makebox(0,0)[lt]{\lineheight{1.25}\smash{\begin{tabular}[t]{l}...\end{tabular}}}}%
    \put(0.43048273,0.11587269){\color[rgb]{0,0,0}\makebox(0,0)[lt]{\lineheight{1.25}\smash{\begin{tabular}[t]{l}...\end{tabular}}}}%
  \end{picture}%
\endgroup%

\end{equation}
where $s_i, t_i \in \{\pm\}$ are states. The left-hand side of this figure represents an element of the stated skein algebra of the bigon $\mathfrak{B}$ (with $T$ any tangle), $\varepsilon : \mathcal{S}^{\mathrm{s}}_q(\mathfrak{B}) \to \mathbb{Z}[q^{\pm 1/2}]$ denotes the counit of $\mathcal{S}_q^{\mathrm{s}}(\mathfrak{B})$ \cite[\S 3.4]{CL} and $Z$ is the evaluation map defined by \eqref{evalNonOriente}.

\smallskip

\indent As explained in \cite[\S 3.2]{CL}, cutting out a bigon $\mathfrak{B}$ from a boundary edge  $e$ of a punctured bordered surface $\mathfrak{S}$ and applying the splitting theorem \cite[Th. 3.1]{Le}  yields a morphism $\Delta_e : \mathcal{S}_q^{\mathrm{s}}(\mathfrak{S}) \to \mathcal{S}_q^{\mathrm{s}}(\mathfrak{B}) \otimes \mathcal{S}_q^{\mathrm{s}}(\mathfrak{S})$ which is a $\mathcal{S}_q^{\mathrm{s}}(\mathfrak{B})$-comodule-algebra structure on $\mathcal{S}_q^{\mathrm{s}}(\mathfrak{S})$ (see \cite[Fig. 7]{CL}). An important fact for our purposes is that the equality $(\varepsilon \otimes \mathrm{id}) \circ \Delta_e = \mathrm{id}$ (which is one of the axioms of a comodule) together with \eqref{RTbigone} allows us to prove relations in $\mathcal{S}_q^{\mathrm{s}}(\mathfrak{S})$ in a graphical way.

\smallskip

\indent Let us illustrate this in the case of $\mathcal{S}_q^{\mathrm{s}}(\mathfrak{B})$. Consider the following matrix $T \in \mathcal{S}_q^{\mathrm{s}}(\mathfrak{B}) \otimes \mathrm{End}_{\mathbb{C}}(V_2)$ (here $V_2$ is just considered as a vector space):
\begin{center}
%% Accompanies image file 'bigonT.pdf' (pdf, eps, ps)
\begingroup%
  \makeatletter%
  \providecommand\color[2][]{%
    \errmessage{(Inkscape) Color is used for the text in Inkscape, but the package 'color.sty' is not loaded}%
    \renewcommand\color[2][]{}%
  }%
  \providecommand\transparent[1]{%
    \errmessage{(Inkscape) Transparency is used (non-zero) for the text in Inkscape, but the package 'transparent.sty' is not loaded}%
    \renewcommand\transparent[1]{}%
  }%
  \providecommand\rotatebox[2]{#2}%
  \newcommand*\fsize{\dimexpr\f@size pt\relax}%
  \newcommand*\lineheight[1]{\fontsize{\fsize}{#1\fsize}\selectfont}%
  \ifx\svgwidth\undefined%
    \setlength{\unitlength}{312.9716642bp}%
    \ifx\svgscale\undefined%
      \relax%
    \else%
      \setlength{\unitlength}{\unitlength * \real{\svgscale}}%
    \fi%
  \else%
    \setlength{\unitlength}{\svgwidth}%
  \fi%
  \global\let\svgwidth\undefined%
  \global\let\svgscale\undefined%
  \makeatother%
  \begin{picture}(1,0.17672217)%
    \lineheight{1}%
    \setlength\tabcolsep{0pt}%
    \put(0,0){\includegraphics[width=\unitlength,page=1]{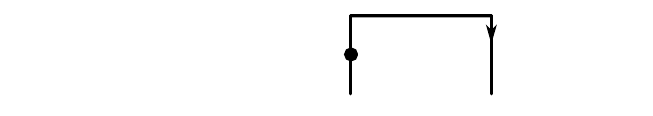}}%
    \put(0.6345541,0.16156224){\color[rgb]{0,0,0}\makebox(0,0)[lt]{\lineheight{1.25}\smash{\begin{tabular}[t]{l}$s$\end{tabular}}}}%
    \put(0.63858904,0.00293261){\color[rgb]{0,0,0}\makebox(0,0)[lt]{\lineheight{1.25}\smash{\begin{tabular}[t]{l}$t$\end{tabular}}}}%
    \put(0,0){\includegraphics[width=\unitlength,page=2]{bigonT.pdf}}%
    \put(-0.00165629,0.08511358){\color[rgb]{0,0,0}\makebox(0,0)[lt]{\lineheight{1.25}\smash{\begin{tabular}[t]{l}$T = \begin{pmatrix}T^-_- & T^-_+\\T^+_- & T^+_+ \end{pmatrix}$, $\quad$ with  $\: T^s_t =\:$\end{tabular}}}}%
  \end{picture}%
\endgroup%

\end{center}
($s, t \in \{\pm\}$ are states). First, observe that
\begin{center}
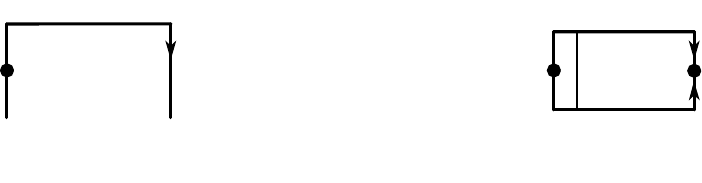
\end{center}

\medskip

\noindent where $X$ is any tangle. For the first equality we used $\mathrm{id} = (\varepsilon \otimes \mathrm{id}) \circ \Delta_e$ where $e$ is the dashed arc, and for the second we used the product in $\mathcal{S}_q^{\mathrm{s}}(\mathfrak{B})$ and \eqref{RTbigone}. Hence we see that the elements $T^s_t$ generate $\mathcal{S}_q^{\mathrm{s}}(\mathfrak{B})$ as an algebra. Similarly, we have:
\begin{center}
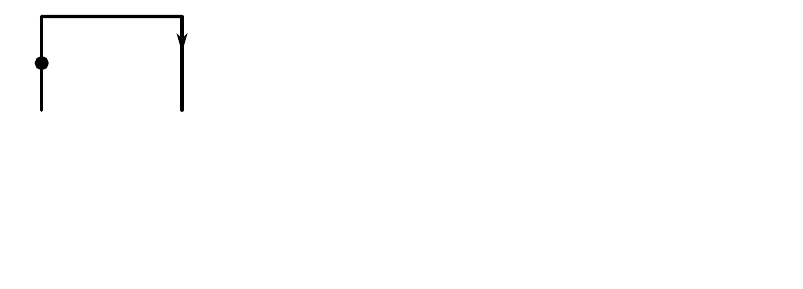
\end{center}
Hence we obtain the relation $RT_1T_2 = T_2T_1R$, where as usual $T_1, T_2$ denote the canonical embeddings of $T \in \mathcal{S}_q^{\mathrm{s}}(\mathfrak{B}) \otimes \mathrm{End}_{\mathbb{C}}(V_2)$ into $\mathcal{S}_q^{\mathrm{s}}(\mathfrak{B}) \otimes \mathrm{End}_{\mathbb{C}}(V_2)^{\otimes 2}$. Finallly:
\begin{center}
%% Accompanies image file 'detqBigon.pdf' (pdf, eps, ps)
\begingroup%
  \makeatletter%
  \providecommand\color[2][]{%
    \errmessage{(Inkscape) Color is used for the text in Inkscape, but the package 'color.sty' is not loaded}%
    \renewcommand\color[2][]{}%
  }%
  \providecommand\transparent[1]{%
    \errmessage{(Inkscape) Transparency is used (non-zero) for the text in Inkscape, but the package 'transparent.sty' is not loaded}%
    \renewcommand\transparent[1]{}%
  }%
  \providecommand\rotatebox[2]{#2}%
  \newcommand*\fsize{\dimexpr\f@size pt\relax}%
  \newcommand*\lineheight[1]{\fontsize{\fsize}{#1\fsize}\selectfont}%
  \ifx\svgwidth\undefined%
    \setlength{\unitlength}{394.29454176bp}%
    \ifx\svgscale\undefined%
      \relax%
    \else%
      \setlength{\unitlength}{\unitlength * \real{\svgscale}}%
    \fi%
  \else%
    \setlength{\unitlength}{\svgwidth}%
  \fi%
  \global\let\svgwidth\undefined%
  \global\let\svgscale\undefined%
  \makeatother%
  \begin{picture}(1,0.15146319)%
    \lineheight{1}%
    \setlength\tabcolsep{0pt}%
    \put(0,0){\includegraphics[width=\unitlength,page=1]{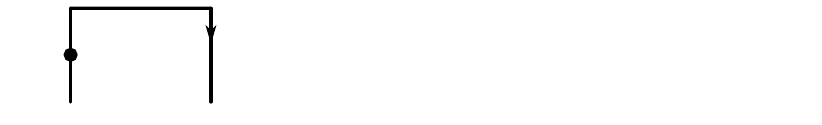}}%
    \put(0.12940026,0.00331394){\color[rgb]{0,0,0}\makebox(0,0)[lt]{\lineheight{1.25}\smash{\begin{tabular}[t]{l}$s$\end{tabular}}}}%
    \put(0,0){\includegraphics[width=\unitlength,page=2]{detqBigon.pdf}}%
    \put(0.19525898,0.00232776){\color[rgb]{0,0,0}\makebox(0,0)[lt]{\lineheight{1.25}\smash{\begin{tabular}[t]{l}$t$\end{tabular}}}}%
    \put(0,0){\includegraphics[width=\unitlength,page=3]{detqBigon.pdf}}%
    \put(0.28265078,0.07988839){\color[rgb]{0,0,0}\makebox(0,0)[lt]{\lineheight{1.25}\smash{\begin{tabular}[t]{l}$= \:\displaystyle \sum_{ij}$\end{tabular}}}}%
    \put(0,0){\includegraphics[width=\unitlength,page=4]{detqBigon.pdf}}%
    \put(0.37426049,0.07855798){\color[rgb]{0,0,0}\makebox(0,0)[lt]{\lineheight{1.25}\smash{\begin{tabular}[t]{l}$\varepsilon$\end{tabular}}}}%
    \put(0,0){\includegraphics[width=\unitlength,page=5]{detqBigon.pdf}}%
    \put(0.47067957,0.02065412){\color[rgb]{0,0,0}\makebox(0,0)[lt]{\lineheight{1.25}\smash{\begin{tabular}[t]{l}$i$\end{tabular}}}}%
    \put(0,0){\includegraphics[width=\unitlength,page=6]{detqBigon.pdf}}%
    \put(0.52452931,0.02014816){\color[rgb]{0,0,0}\makebox(0,0)[lt]{\lineheight{1.25}\smash{\begin{tabular}[t]{l}$j$\end{tabular}}}}%
    \put(0,0){\includegraphics[width=\unitlength,page=7]{detqBigon.pdf}}%
    \put(0.6878208,0.13292864){\color[rgb]{0,0,0}\makebox(0,0)[lt]{\lineheight{1.25}\smash{\begin{tabular}[t]{l}$i$\end{tabular}}}}%
    \put(0.74521109,0.13286617){\color[rgb]{0,0,0}\makebox(0,0)[lt]{\lineheight{1.25}\smash{\begin{tabular}[t]{l}$j$\end{tabular}}}}%
    \put(0,0){\includegraphics[width=\unitlength,page=8]{detqBigon.pdf}}%
    \put(0.84077935,0.0751989){\color[rgb]{0,0,0}\makebox(0,0)[lt]{\lineheight{1.25}\smash{\begin{tabular}[t]{l}$= \: \displaystyle \sum_{ij} \cap_{ij} T^i_s T^j_t $\end{tabular}}}}%
    \put(-0.00131468,0.07892764){\color[rgb]{0,0,0}\makebox(0,0)[lt]{\lineheight{1.25}\smash{\begin{tabular}[t]{l}$\cap_{st} = $\end{tabular}}}}%
    \put(0.68101833,0.02233526){\color[rgb]{0,0,0}\makebox(0,0)[lt]{\lineheight{1.25}\smash{\begin{tabular}[t]{l}$s$\end{tabular}}}}%
    \put(0,0){\includegraphics[width=\unitlength,page=9]{detqBigon.pdf}}%
    \put(0.74687708,0.02134908){\color[rgb]{0,0,0}\makebox(0,0)[lt]{\lineheight{1.25}\smash{\begin{tabular}[t]{l}$t$\end{tabular}}}}%
  \end{picture}%
\endgroup%

\end{center}
and it is easy to compute that this implies $\mathrm{det}_q(T) = T^-_-T^+_+ - q^{-2}T^-_+T^+_- = 1$. Hence the coefficients of $T \in \mathcal{S}_q^{\mathrm{s}}(\mathfrak{B}) \otimes \mathrm{End}_{\mathbb{C}}(V_2)$ satisfy exactly the relations of the matrix coefficients of the fundamental representation $V_2$ of $U_{q^2}$. In \cite[\S 3.3]{CL} a less usual presentation of $\mathcal{O}_{q^2}$ is obtained using a similar method and it is shown that these algebras are indeed isomorphic: $\mathcal{S}_q^{\mathrm{s}}(\mathfrak{B}) \cong \mathcal{O}_{q^2}$. In the sequel we identify these two algebras (\textit{i.e.} we identify $T$ with $\overset{V_2}{T}$), so that $\mathcal{S}_q^{\mathrm{s}}(\mathfrak{S})$ becomes a $\mathcal{O}_{q^2}$-comodule-algebra for any punctured bordered surface $\mathfrak{S}$.

\smallskip

Let $\mathbf{T}^{\mathrm{s}} \subset \boldsymbol{\Sigma} = (\Sigma_{g,n}^{\mathrm{o},\bullet}) \times [0,1]$ be an unoriented, uncolored stated $\partial\boldsymbol{\Sigma}$-tangle, with states $s_1, \ldots, s_k \in \{\pm\}$. Forget the states of $\mathbf{T}^{\mathrm{s}}$ and on each strand put an arbitrary orientation and the color $V_2$. This gives an oriented colored $\partial\boldsymbol{\Sigma}$-tangle $\mathbf{T}$ and by Definition \ref{defHol}, $\mathrm{hol}(\mathbf{T}) \in \mathcal{L}_{0,1}(H) \otimes V_2^{a_1} \otimes \ldots \otimes V_2^{a_k}$ where $a_i \in \{ \downarrow, \uparrow \}$. Let $o^{\downarrow} = \mathrm{id} : V_2 \to V_2$ and $o^{\uparrow} = D : V_2^* \overset{\sim}{\to} V_2$. We define a map
\begin{equation}\label{formuleHolSgn}
\fonc{\mathrm{hol}^{\mathrm{s}}}{ \{ \text{stated}\ \partial \boldsymbol{\Sigma}\text{-tangles} \} }{\mathcal{L}_{g,n}(U_{q^2})}{\mathbf{T}^{\mathrm{s}}}{\bigr(\mathrm{id} \otimes (v^{s_1} \circ o^{a_1}) \otimes \ldots \otimes (v^{s_k} \circ o^{a_k})\bigl)\bigl(\mathrm{hol}(\mathbf{T})\bigr)}
\end{equation}
where $\bigl(v^-, v^+\bigr)$ is dual to the canonical basis $\bigl( v_-, v_+ \bigr)$ of $V_2$ (recall that we use the special pivotal element $g$ discussed above to evaluate the diagram of $\mathrm{hol}(\mathbf{T})$). This map is obviously extended to formal linear combinations of $\partial\boldsymbol{\Sigma}$-tangles.

\smallskip

$\mathrm{hol}^{\mathrm{s}}$ can also be described as follows: consider the (unoriented and uncolored) diagram built like in Definition \ref{defHol} associated to $\mathbf{T}^{\mathrm{s}}$ (states are not used in the diagram). Then evaluate this diagram with the unoriented graphical calculus defined in \eqref{evalNonOriente}; the result is an element of $\mathcal{L}_{g,n}(U_{q^2}) \otimes V_2^{\otimes k}$ and $\mathrm{hol}^{\mathrm{s}}(\mathbf{T}^{\mathrm{s}})$ is the component $(s_1, \ldots, s_k)$ of that tensor. In particular, we see that the formula in \eqref{formuleHolSgn} does not depend on the arbitrary orientations choosen to obtain $\mathbf{T}$.

\begin{theorem}\label{thIsoSkeinLgn}
The map $\mathrm{hol}^{\mathrm{s}}$ descends to an isomorphism of $\mathcal{O}_{q^2}$-comodule-algebras:
\[ \mathcal{S}^{\mathrm{s}}_q(\Sigma_{g,n}^{\mathrm{o},\bullet}) \cong \mathcal{L}_{g,n}(U_{q^2}). \]
\end{theorem}
\begin{proof}
We first show that $\mathrm{hol}^{\mathrm{s}}$ is well-defined, \textit{i.e.} that it preserves the defining relations of $\mathcal{S}^{\mathrm{s}}_q(\Sigma_{g,n}^{\mathrm{o},\bullet})$. It is clear that $\mathrm{hol}^{\mathrm{s}}$ is compatible with the Kauffman bracket skein relations \eqref{kauffmanUq2} since so does the unoriented graphical calculus for $U_{q^2}$. For the boundary relations \eqref{boundaryRels}, we have
\begin{center}
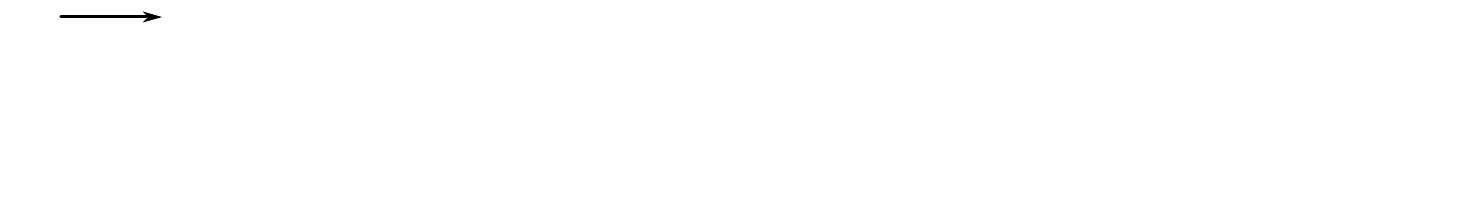
\end{center}
for all $s,t \in \{\pm\}$, as desired. 
\\\noindent The fact that $\mathrm{hol}^{\mathrm{s}}$ is an algebra morphism follows immediately from Theorem \ref{wilsonStack}. To show that it is an isomorphism, we will use the following facts:
\begin{align*}
\mathcal{S}^{\mathrm{s}}_q(\Sigma_{g,n}^{\mathrm{o},\bullet}) &\cong \mathcal{S}^{\mathrm{s}}_q(\Sigma_{1,0}^{\mathrm{o},\bullet})^{\widetilde{\otimes} g} \, \widetilde{\otimes} \, \mathcal{S}^{\mathrm{s}}_q(\Sigma_{0,1}^{\mathrm{o},\bullet})^{\widetilde{\otimes} n}, \\
\mathcal{L}_{g,n}(U_{q^2}) &\cong \mathcal{L}_{1,0}(U_{q^2})^{\widetilde{\otimes} g} \, \widetilde{\otimes} \, \mathcal{L}_{0,1}(U_{q^2})^{\widetilde{\otimes} n}
\end{align*}
where $\widetilde{\otimes}$ is the braided tensor product of algebras in the category of left $\mathcal{O}_{q^2}$-comodules. The first isomorphism follows from the glueing theorem for stated skein algebras \cite[Th. 4.13]{CL}. The second isomorphism is a well-known and general fact about the algebras $\mathcal{L}_{g,n}(H)$ (\cite[\S 3.2]{AS2}, also see \cite[\S 3.2]{Fai18c}). Hence to show that $\mathrm{hol}^{\mathrm{s}}$ is an isomorphism, it is enough to show it for the surfaces $\Sigma_{0,1}^{\mathrm{o},\bullet}$ and $\Sigma_{1,0}^{\mathrm{o},\bullet}$, which is done in the two next subsections. The fact that $\mathrm{hol}^{\mathrm{s}}$ is a morphism of $\mathcal{O}_{q^2}$-comodules also follows from the corresponding fact for these two building blocks.
\end{proof}

\begin{remark}
I learned from the referee that T. L\^e announced independently in some conferences (in particular Paris, october 2019) that stated skein algebras are isomorphic to $\mathcal{L}_{g,n}(U_{q^2})$, with a sketch of the proof. This result is also mentionned in his preprint with Yu \cite{LY}, which appeared shortly after the present work.
\end{remark}

\indent Let $\mathbf{U}^{(b_i)}{^s_t}, \mathbf{U}^{(a_i)}{^s_t}, \mathbf{U}^{(m_i)}{^s_t}$  (with $s,t \in \{\pm\}$) be the $\partial\boldsymbol{\Sigma}$-tangles naturally associated to the generators $b_i, a_i,m_j$ of $\pi_1(\Sigma_{g,n}^{\mathrm{o}})$ represented in Figure \ref{surfaceGN}:
\begin{equation}\label{genSgn}
%% Accompanies image file 'genSgn.pdf' (pdf, eps, ps)
\begingroup%
  \makeatletter%
  \providecommand\color[2][]{%
    \errmessage{(Inkscape) Color is used for the text in Inkscape, but the package 'color.sty' is not loaded}%
    \renewcommand\color[2][]{}%
  }%
  \providecommand\transparent[1]{%
    \errmessage{(Inkscape) Transparency is used (non-zero) for the text in Inkscape, but the package 'transparent.sty' is not loaded}%
    \renewcommand\transparent[1]{}%
  }%
  \providecommand\rotatebox[2]{#2}%
  \newcommand*\fsize{\dimexpr\f@size pt\relax}%
  \newcommand*\lineheight[1]{\fontsize{\fsize}{#1\fsize}\selectfont}%
  \ifx\svgwidth\undefined%
    \setlength{\unitlength}{368.8386049bp}%
    \ifx\svgscale\undefined%
      \relax%
    \else%
      \setlength{\unitlength}{\unitlength * \real{\svgscale}}%
    \fi%
  \else%
    \setlength{\unitlength}{\svgwidth}%
  \fi%
  \global\let\svgwidth\undefined%
  \global\let\svgscale\undefined%
  \makeatother%
  \begin{picture}(1,0.26459984)%
    \lineheight{1}%
    \setlength\tabcolsep{0pt}%
    \put(0,0){\includegraphics[width=\unitlength,page=1]{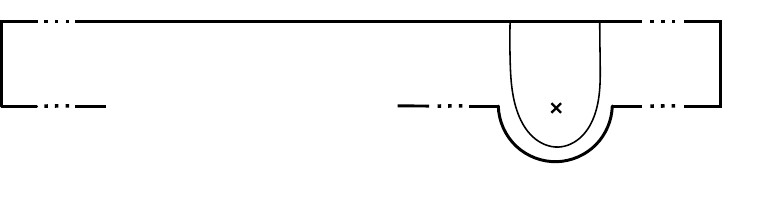}}%
    \put(0.79287569,0.17185623){\color[rgb]{0,0,0}\makebox(0,0)[lt]{\lineheight{1.25}\smash{\begin{tabular}[t]{l}$\mathbf{U}^{(m_j)}{^s_t}$\end{tabular}}}}%
    \put(0.49868812,0.1812448){\color[rgb]{0,0,0}\makebox(0,0)[lt]{\lineheight{1.25}\smash{\begin{tabular}[t]{l}$\mathbf{U}^{(a_i)}{^s_t}$\end{tabular}}}}%
    \put(0.65984591,0.24854552){\color[rgb]{0,0,0}\makebox(0,0)[lt]{\lineheight{1.25}\smash{\begin{tabular}[t]{l}$s$\end{tabular}}}}%
    \put(0.7761316,0.24854552){\color[rgb]{0,0,0}\makebox(0,0)[lt]{\lineheight{1.25}\smash{\begin{tabular}[t]{l}$t$\end{tabular}}}}%
    \put(0,0){\includegraphics[width=\unitlength,page=2]{genSgn.pdf}}%
    \put(0.39492399,0.18029911){\color[rgb]{0,0,0}\makebox(0,0)[lt]{\lineheight{1.25}\smash{\begin{tabular}[t]{l}$\mathbf{U}^{(b_i)}{^s_t}$\end{tabular}}}}%
    \put(0,0){\includegraphics[width=\unitlength,page=3]{genSgn.pdf}}%
    \put(0.15327966,0.24776196){\color[rgb]{0,0,0}\makebox(0,0)[lt]{\lineheight{1.25}\smash{\begin{tabular}[t]{l}$s$\end{tabular}}}}%
    \put(0.38081157,0.24626814){\color[rgb]{0,0,0}\makebox(0,0)[lt]{\lineheight{1.25}\smash{\begin{tabular}[t]{l}$t$\end{tabular}}}}%
    \put(0,0){\includegraphics[width=\unitlength,page=4]{genSgn.pdf}}%
    \put(0.48796383,0.24643371){\color[rgb]{0,0,0}\makebox(0,0)[lt]{\lineheight{1.25}\smash{\begin{tabular}[t]{l}$t$\end{tabular}}}}%
    \put(0.26901528,0.24805858){\color[rgb]{0,0,0}\makebox(0,0)[lt]{\lineheight{1.25}\smash{\begin{tabular}[t]{l}$s$\end{tabular}}}}%
  \end{picture}%
\endgroup%

\end{equation}
These elements can be arranged as matrices $\mathbf{U}^{(b_i)}, \mathbf{U}^{(a_i)}, \mathbf{U}^{(m_j)} \in \mathcal{S}^{\mathrm{s}}_q(\Sigma_{g,n}^{\mathrm{o},\bullet}) \otimes \mathrm{End}_{\mathbb{C}}(V_2)$, and by definition it holds
\begin{equation}\label{holSurMatricesGen}
\mathrm{hol}^{\mathrm{s}}(\mathbf{U}^{(b_i)}) = B(i) \, ^t \!D, \:\:\:\:\:\:\:\: \mathrm{hol}^{\mathrm{s}}(\mathbf{U}^{(a_i)}) = A(i) \, ^t \!D, \:\:\:\:\:\:\:\: \mathrm{hol}^{\mathrm{s}}(\mathbf{U}^{(m_j)}) = M(j) \, ^t \!D 
\end{equation}
where obviously $\mathrm{hol}^{\mathrm{s}}(\mathbf{U}^{(b_i)})^s_t = \mathrm{hol}^{\mathrm{s}}(\mathbf{U}^{(b_i)}{^s_t})$ \textit{etc}.  Thanks to Lemma \ref{GensRelsLgnUq2}, we obtain a presentation by generators and relations of $\mathcal{S}^{\mathrm{s}}_q(\Sigma_{g,n}^{\mathrm{o},\bullet})$:

\begin{corollary}\label{presentationStatedSkein}
The algebra $\mathcal{S}^{\mathrm{s}}_q(\Sigma_{g,n}^{\mathrm{o},\bullet})$ is generated by the coefficients of the matrices $\mathbf{U}^{(b_i)}, \mathbf{U}^{(a_i)}$ for $1 \leq i \leq g$ and $\mathbf{U}^{(m_i)}$ for $g+1 \leq i \leq g+n$, modulo the relations obtained by replacing $B(i)$ with $\mathbf{U}^{(b_i)} (^t \!D)^{-1}$, $A(i)$ with $\mathbf{U}^{(a_i)} (^t \!D)^{-1}$ and $M(i)$ with $\mathbf{U}^{(m_i)} (^t \!D)^{-1}$ in the relations of Lemma \ref{GensRelsLgnUq2}.
\end{corollary}

\indent The (usual) skein algebra $\mathcal{S}_q(\Sigma_{g,n}^{\mathrm{o}})$ is the subalgebra of $\mathcal{S}^{\mathrm{s}}_q(\Sigma_{g,n}^{\mathrm{o},\bullet})$ consisting of linear combinations of closed links (tangles without boundary points and hence without states). On this subalgebra, $\mathrm{hol}^{\mathrm{s}}$ is simply the Wilson loop map $W$ with all the strands colored by $\chi^+_2$, the character of $V_2$. By \cite[Th. 10]{BFKB} (see also \cite[Th. 8.4]{BaR} for the statement of this result in the formalism used here), the map $W$ provides an isomorphism from $\mathcal{S}_q(\Sigma_{g,n}^{\mathrm{o}})$ to the subalgebra of invariant elements $\mathcal{L}_{g,n}^{\mathrm{inv}}(U_{q^2})$. Hence:

\begin{corollary}\label{coroInvariantsStated}
The (usual) skein algebra $\mathcal{S}_q(\Sigma_{g,n}^{\mathrm{o}})$ is exactly the subalgebra of invariant elements of the $\mathcal{O}_{q^2}$-comodule-algebra $\mathcal{S}^{\mathrm{s}}_q(\Sigma_{g,n}^{\mathrm{o},\bullet})$.
\end{corollary}

\subsection{Proof for $\Sigma_{0,1}^{\mathrm{o},\bullet}$}
Let
\[ M = \overset{V_2}{M} = 
\begin{pmatrix}
M^-_- & M^-_+\\
M^+_- & M^+_+
\end{pmatrix} \in \mathcal{L}_{0,1}(U_{q^2}) \otimes \mathrm{End}_{\mathbb{C}}(V_2).
\]
Then by Lemma \ref{GensRelsLgnUq2}, $\mathcal{L}_{0,1}(U_{q^2})$ is generated by the $4$ coefficients of the matrix $M$, modulo the following relations:
\begin{equation}\label{presentationL01Uq2}
 R M_1 R_{21} M_2 = M_2 R M_1 R_{21}, \:\:\:\:\:\:\:\: \mathrm{qdet}(M) = 1
\end{equation}
where by definition the quantum determinant is $\mathrm{qdet}(M) = M^-_-M^+_+ - q^4M^-_+M^+_-$.

\smallskip

\indent For $l \geq 0$ and $s_1, \ldots, s_l, t_1, \ldots, t_l \in \{\pm\}$, we define the following elements of $\mathcal{S}^{\mathrm{s}}_q(\Sigma_{0,1}^{\mathrm{o},\bullet})$:
\begin{center}
%% Accompanies image file 'matriceY.pdf' (pdf, eps, ps)
\begingroup%
  \makeatletter%
  \providecommand\color[2][]{%
    \errmessage{(Inkscape) Color is used for the text in Inkscape, but the package 'color.sty' is not loaded}%
    \renewcommand\color[2][]{}%
  }%
  \providecommand\transparent[1]{%
    \errmessage{(Inkscape) Transparency is used (non-zero) for the text in Inkscape, but the package 'transparent.sty' is not loaded}%
    \renewcommand\transparent[1]{}%
  }%
  \providecommand\rotatebox[2]{#2}%
  \newcommand*\fsize{\dimexpr\f@size pt\relax}%
  \newcommand*\lineheight[1]{\fontsize{\fsize}{#1\fsize}\selectfont}%
  \ifx\svgwidth\undefined%
    \setlength{\unitlength}{177.91399733bp}%
    \ifx\svgscale\undefined%
      \relax%
    \else%
      \setlength{\unitlength}{\unitlength * \real{\svgscale}}%
    \fi%
  \else%
    \setlength{\unitlength}{\svgwidth}%
  \fi%
  \global\let\svgwidth\undefined%
  \global\let\svgscale\undefined%
  \makeatother%
  \begin{picture}(1,0.39426455)%
    \lineheight{1}%
    \setlength\tabcolsep{0pt}%
    \put(0,0){\includegraphics[width=\unitlength,page=1]{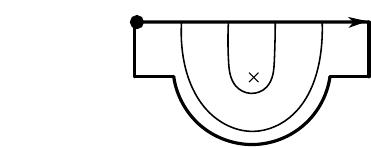}}%
    \put(0.45932036,0.36520712){\color[rgb]{0,0,0}\makebox(0,0)[lt]{\lineheight{1.25}\smash{\begin{tabular}[t]{l}$s_1$\end{tabular}}}}%
    \put(0.53621522,0.22425118){\color[rgb]{0,0,0}\makebox(0,0)[lt]{\lineheight{1.25}\smash{\begin{tabular}[t]{l}...\end{tabular}}}}%
    \put(0.77513628,0.22425118){\color[rgb]{0,0,0}\makebox(0,0)[lt]{\lineheight{1.25}\smash{\begin{tabular}[t]{l}...\end{tabular}}}}%
    \put(0.6038525,0.36520724){\color[rgb]{0,0,0}\makebox(0,0)[lt]{\lineheight{1.25}\smash{\begin{tabular}[t]{l}$s_l$\end{tabular}}}}%
    \put(0.73107102,0.36520712){\color[rgb]{0,0,0}\makebox(0,0)[lt]{\lineheight{1.25}\smash{\begin{tabular}[t]{l}$t_l$\end{tabular}}}}%
    \put(0.85678394,0.36520724){\color[rgb]{0,0,0}\makebox(0,0)[lt]{\lineheight{1.25}\smash{\begin{tabular}[t]{l}$t_1$\end{tabular}}}}%
    \put(-0.00317464,0.22573764){\color[rgb]{0,0,0}\makebox(0,0)[lt]{\lineheight{1.25}\smash{\begin{tabular}[t]{l}$\mathbf{U}(l)^{s_1 \ldots s_l}_{t_1 \ldots t_l} = $\end{tabular}}}}%
  \end{picture}%
\endgroup%

\end{center}
(for $l=0$ this is $\mathbf{U}(0)=1$, namely the empty diagram). Note that for a fixed $l$, these elements can be arranged into a tensor $\mathbf{U}(l) \in \mathcal{S}^{\mathrm{s}}_q(\Sigma_{0,1}^{\mathrm{o},\bullet}) \otimes \mathrm{End}_{\mathbb{C}}(V_2)^{\otimes l}$. For simplicity we denote $\mathbf{U} = \mathbf{U}(1)$ and $\mathbf{V} = \mathbf{U}(2)$; in a matrix form these tensors are written as
\[
\mathbf{U} = \mathbf{U}(1)= 
\begin{pmatrix}
\mathbf{U}^-_- & \mathbf{U}^-_+\\
\mathbf{U}^+_- & \mathbf{U}^+_+
\end{pmatrix}, \:\:\:\:\:\:
\mathbf{V}= \mathbf{U}(2) =
\begin{pmatrix}
\mathbf{V}^{--}_{--} & \mathbf{V}^{--}_{-+} & \mathbf{V}^{--}_{+-} & \mathbf{V}^{--}_{++}\\
\mathbf{V}^{-+}_{--} & \mathbf{V}^{-+}_{-+} & \mathbf{V}^{-+}_{+-} & \mathbf{V}^{-+}_{++}\\
\mathbf{V}^{+-}_{--} & \mathbf{V}^{+-}_{-+} & \mathbf{V}^{+-}_{+-} & \mathbf{V}^{+-}_{++}\\
\mathbf{V}^{++}_{--} & \mathbf{V}^{++}_{-+} & \mathbf{V}^{++}_{+-} & \mathbf{V}^{++}_{++}
\end{pmatrix}.
\]
\indent As explained previously, we use the $(\mathcal{S}_q^{\mathrm{s}}(\mathfrak{B}) = \mathcal{O}_{q^2})$-comodule-algebra structure of stated skein algebras together with $\mathrm{id} = (\varepsilon \otimes \mathrm{id}) \circ \Delta_e$ and \eqref{RTbigone} to derive equalities graphically. Here we apply this to the dashed arc $e$ depicted below.
\begin{equation}\label{ruseStated}
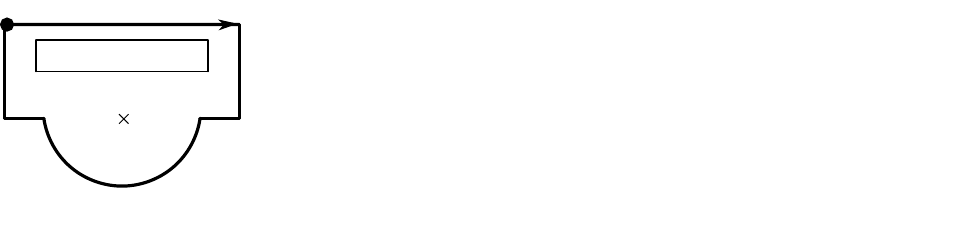
\end{equation}
where the summation indices are all in $\{ \pm \}$. This equality implies in particular that the collection of elements $\mathbf{U}(l)^{i_1 \ldots i_l}_{j_1 \ldots j_l}$ (with $l \geq 0$ and $i_{\eta}, j_{\eta} \in \{ \pm \}$ for each $\eta$) spans $\mathcal{S}^{\mathrm{s}}_q(\Sigma_{0,1}^{\mathrm{o},\bullet})$ as a $\mathbb{Z}[q^{\pm 1/2}]$-module.

\smallskip

\indent Since $\mathcal{L}_{0,1}(U_{q^2})$ is isomorphic to the braided dual of $U_{q^2}$ \cite[\S 4]{BaR}, we can deduce the following lemma from \cite[Prop. 4.17]{CL}, where it is shown that $\mathcal{S}^{\mathrm{s}}_q(\Sigma_{0,1}^{\mathrm{o},\bullet})$ is isomorphic to the braided dual. Our proof is however different.
\begin{lemma}\label{lemmeInjectionL01S01}
1. We have a morphism $j : \mathcal{L}_{0,1}(U_{q^2}) \to \mathcal{S}^{\mathrm{s}}_q(\Sigma_{0,1}^{\mathrm{o},\bullet})$ of $\mathcal{O}_{q^2}$-comodule-algebras defined by $j(M) = \mathbf{U} (^t \! D)^{-1}$. Explicitly:
\[  
\begin{pmatrix}
j(M^-_-) & j(M^-_+) \\
j(M^+_-) & j(M^+_+)
\end{pmatrix}
=
\begin{pmatrix}
q^{-5/2}\mathbf{U}^-_+  & -q^{-1/2}\mathbf{U}^-_- \\
q^{-5/2}\mathbf{U}^+_+ & -q^{-1/2}\mathbf{U}^+_-
\end{pmatrix}.
\]
2. The morphism $j$ is surjective.
\end{lemma}
\begin{proof}
1. We will show that the relations \eqref{presentationL01Uq2} defining $\mathcal{L}_{0,1}(U_{q^2})$ are satisfied. This is entirely based on formal matrix computations. Using the product in $\mathcal{S}^{\mathrm{s}}_q(\Sigma_{0,1}^{\mathrm{o},\bullet})$ and \eqref{ruseStated}, we obtain:
\begin{center}
%% Accompanies image file 'produitS01.pdf' (pdf, eps, ps)
\begingroup%
  \makeatletter%
  \providecommand\color[2][]{%
    \errmessage{(Inkscape) Color is used for the text in Inkscape, but the package 'color.sty' is not loaded}%
    \renewcommand\color[2][]{}%
  }%
  \providecommand\transparent[1]{%
    \errmessage{(Inkscape) Transparency is used (non-zero) for the text in Inkscape, but the package 'transparent.sty' is not loaded}%
    \renewcommand\transparent[1]{}%
  }%
  \providecommand\rotatebox[2]{#2}%
  \newcommand*\fsize{\dimexpr\f@size pt\relax}%
  \newcommand*\lineheight[1]{\fontsize{\fsize}{#1\fsize}\selectfont}%
  \ifx\svgwidth\undefined%
    \setlength{\unitlength}{668.0966576bp}%
    \ifx\svgscale\undefined%
      \relax%
    \else%
      \setlength{\unitlength}{\unitlength * \real{\svgscale}}%
    \fi%
  \else%
    \setlength{\unitlength}{\svgwidth}%
  \fi%
  \global\let\svgwidth\undefined%
  \global\let\svgscale\undefined%
  \makeatother%
  \begin{picture}(1,0.14523442)%
    \lineheight{1}%
    \setlength\tabcolsep{0pt}%
    \put(0,0){\includegraphics[width=\unitlength,page=1]{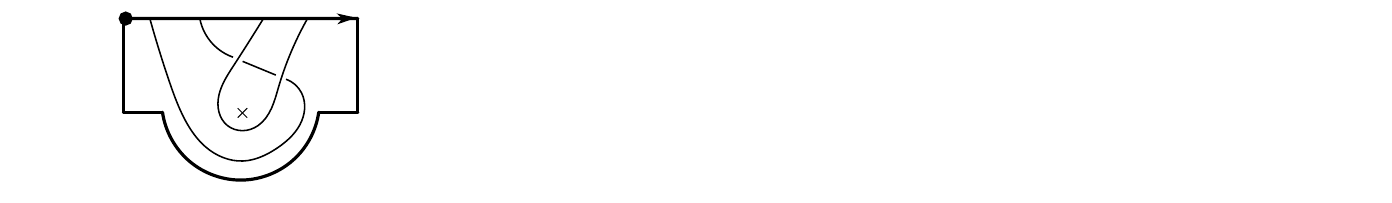}}%
    \put(0.09985646,0.13813271){\color[rgb]{0,0,0}\makebox(0,0)[lt]{\lineheight{1.25}\smash{\begin{tabular}[t]{l}$s_1$\end{tabular}}}}%
    \put(0.18179226,0.13780068){\color[rgb]{0,0,0}\makebox(0,0)[lt]{\lineheight{1.25}\smash{\begin{tabular}[t]{l}$s_2$\end{tabular}}}}%
    \put(0.13544843,0.13799706){\color[rgb]{0,0,0}\makebox(0,0)[lt]{\lineheight{1.25}\smash{\begin{tabular}[t]{l}$t_1$\end{tabular}}}}%
    \put(0.21380096,0.13780068){\color[rgb]{0,0,0}\makebox(0,0)[lt]{\lineheight{1.25}\smash{\begin{tabular}[t]{l}$t_2$\end{tabular}}}}%
    \put(-0.00077589,0.08761608){\color[rgb]{0,0,0}\makebox(0,0)[lt]{\lineheight{1.25}\smash{\begin{tabular}[t]{l}$U^{s_1}_{t_1}U^{s_2}_{t_2} =$\end{tabular}}}}%
    \put(0.26960436,0.08759592){\color[rgb]{0,0,0}\makebox(0,0)[lt]{\lineheight{1.25}\smash{\begin{tabular}[t]{l}$\displaystyle= \sum_{ijkl} \mathbf{V}^{ij}_{lk} \, Z$\end{tabular}}}}%
    \put(0,0){\includegraphics[width=\unitlength,page=2]{produitS01.pdf}}%
    \put(0.47957108,0.0580871){\color[rgb]{0,0,0}\makebox(0,0)[lt]{\lineheight{1.25}\smash{\begin{tabular}[t]{l}$_{ijkl}$\end{tabular}}}}%
    \put(0.47886204,0.12390931){\color[rgb]{0,0,0}\makebox(0,0)[lt]{\lineheight{1.25}\smash{\begin{tabular}[t]{l}$_{s_1t_1s_2t_2}$\end{tabular}}}}%
    \put(0.53170444,0.08785575){\color[rgb]{0,0,0}\makebox(0,0)[lt]{\lineheight{1.25}\smash{\begin{tabular}[t]{l}$\displaystyle= \sum_{jklm} \mathbf{V}^{s_1j}_{lk}R^{s_2t_1}_{jm}R^{t_2m}_{kl}$\end{tabular}}}}%
    \put(0.26868376,0.01616528){\color[rgb]{0,0,0}\makebox(0,0)[lt]{\lineheight{1.25}\smash{\begin{tabular}[t]{l}$\displaystyle= \sum_{jklm} (a_{\beta})^{s_2}_j \mathbf{V}^{s_1j}_{lk} \,(^t b_{\alpha})^l_m \,(^t b_{\beta})^m_{t_1}\, (^t \! a_{\alpha})^k_{t_2}.$\end{tabular}}}}%
  \end{picture}%
\endgroup%

\end{center}
where we wrote $R = a_{\eta} \otimes b_{\eta}$ with implicit summation and implicit representation on the fundamental representation $V_2$ (\textit{i.e.} $R$ is the matrix in \eqref{RMatriceV2}). Hence
\[ \mathbf{U}_1 \mathbf{U}_2 = (a_{\beta})_2 \mathbf{V} \, (^t b_{\alpha})_1 \, (^t b_{\beta})_1 \, ^t\!(a_{\alpha})_2 \]
where as usual the subscript $1,2$ indicate the two canonical embeddings of $\mathcal{S}^{\mathrm{s}}_q(\Sigma_{0,1}^{\mathrm{o},\bullet}) \otimes \mathrm{End}_{\mathbb{C}}(V_2)$ in $\mathcal{S}^{\mathrm{s}}_q(\Sigma_{0,1}^{\mathrm{o},\bullet}) \otimes \mathrm{End}_{\mathbb{C}}(V_2)^{\otimes 2}$. Now we compute:
\begin{align*}
\mathbf{U}_1 \mathbf{U}_2 &= (a_{\beta})_2 \, \mathbf{V} \, D_1^{-1} \, S^{-1}(b_{\alpha})_1 \, S^{-1}(b_{\beta})_1 \, D_1 \, D_2^{-1} \,  S^{-1}(a_{\alpha})_2 \, D_2\\
& = (a_{\beta})_2 \, \mathbf{V} \, D_1^{-1} \, D_2^{-1} \, (b_{\alpha})_1 \, S^{-1}(b_{\beta})_1 \, (a_{\alpha})_2 \, D_1 \, D_2\\
&= (a_{\beta})_2 \, \mathbf{V} \, D_1^{-1} \, D_2^{-1} \, (b_{\alpha})_1 \, g_1^{-1} \, S(b_{\beta})_1 \, g_1 \, g_2^{-1} \, S^2(a_{\alpha})_2 \, g_2 \, D_1 \, D_2\\
&= (a_{\beta})_2 \, \mathbf{V} \, D_1^{-1} \, g_1^{-1} \, D_2^{-1} \, g_2^{-1} \, S^2(b_{\alpha})_1 \, S(b_{\beta})_1 \, S^2(a_{\alpha})_2 \, g_1 \, D_1 \, g_2 \, D_2\\
&= (a_{\beta})_2 \, \mathbf{V} \, (\exposantGauche{t}{\!D})_1^{-1} \, (\exposantGauche{t}{\!D})_2^{-1} \, R_{21} \, S(b_{\beta})_1 \, (\exposantGauche{t}{\!D})_1 \, (\exposantGauche{t}{\!D})_2.
\end{align*}
For the first equality we used \eqref{ruseFondaDual}, for the second we used $(S \otimes S)(R) = R$ and obvious commutation relations between tensors, for the third and the fourth we used \eqref{pivot} (here $g$ implicitly means $\overset{V_2}{g}$) and obvious commutation relations between tensors, and for the fifth we used \eqref{propKD}. Thanks to the relation $a_{\gamma}a_{\beta} \otimes S(b_{\beta})b_{\gamma} = 1 \otimes 1$ we can invert this equality, which gives
\begin{equation}\label{fusionS01}
\mathbf{V}  (\exposantGauche{t}{\!D})_1^{-1} (\exposantGauche{t}{\!D})_2^{-1} = \bigl( \mathbf{U} (^t \! D)^{-1} \bigr)_1 R_{21}  \bigl( \mathbf{U} (^t \! D)^{-1} \bigr)_2 R_{21}^{-1}.
\end{equation}
Now since
\begin{center}
%% Accompanies image file 'eqRefS01.pdf' (pdf, eps, ps)
\begingroup%
  \makeatletter%
  \providecommand\color[2][]{%
    \errmessage{(Inkscape) Color is used for the text in Inkscape, but the package 'color.sty' is not loaded}%
    \renewcommand\color[2][]{}%
  }%
  \providecommand\transparent[1]{%
    \errmessage{(Inkscape) Transparency is used (non-zero) for the text in Inkscape, but the package 'transparent.sty' is not loaded}%
    \renewcommand\transparent[1]{}%
  }%
  \providecommand\rotatebox[2]{#2}%
  \newcommand*\fsize{\dimexpr\f@size pt\relax}%
  \newcommand*\lineheight[1]{\fontsize{\fsize}{#1\fsize}\selectfont}%
  \ifx\svgwidth\undefined%
    \setlength{\unitlength}{258.0613308bp}%
    \ifx\svgscale\undefined%
      \relax%
    \else%
      \setlength{\unitlength}{\unitlength * \real{\svgscale}}%
    \fi%
  \else%
    \setlength{\unitlength}{\svgwidth}%
  \fi%
  \global\let\svgwidth\undefined%
  \global\let\svgscale\undefined%
  \makeatother%
  \begin{picture}(1,0.30839194)%
    \lineheight{1}%
    \setlength\tabcolsep{0pt}%
    \put(0,0){\includegraphics[width=\unitlength,page=1]{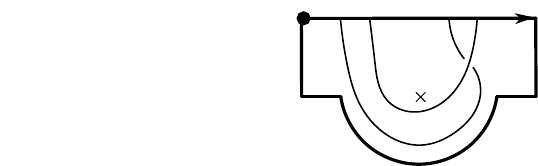}}%
    \put(0.61829661,0.28918299){\color[rgb]{0,0,0}\makebox(0,0)[lt]{\lineheight{1.25}\smash{\begin{tabular}[t]{l}$s_1$\end{tabular}}}}%
    \put(0.82188046,0.29000628){\color[rgb]{0,0,0}\makebox(0,0)[lt]{\lineheight{1.25}\smash{\begin{tabular}[t]{l}$t_1$\end{tabular}}}}%
    \put(0.67300954,0.28883177){\color[rgb]{0,0,0}\makebox(0,0)[lt]{\lineheight{1.25}\smash{\begin{tabular}[t]{l}$s_2$\end{tabular}}}}%
    \put(0.88419745,0.29000628){\color[rgb]{0,0,0}\makebox(0,0)[lt]{\lineheight{1.25}\smash{\begin{tabular}[t]{l}$t_2$\end{tabular}}}}%
    \put(0,0){\includegraphics[width=\unitlength,page=2]{eqRefS01.pdf}}%
    \put(0.06504703,0.28814505){\color[rgb]{0,0,0}\makebox(0,0)[lt]{\lineheight{1.25}\smash{\begin{tabular}[t]{l}$s_1$\end{tabular}}}}%
    \put(0.26914994,0.28896828){\color[rgb]{0,0,0}\makebox(0,0)[lt]{\lineheight{1.25}\smash{\begin{tabular}[t]{l}$t_1$\end{tabular}}}}%
    \put(0.11975996,0.28779382){\color[rgb]{0,0,0}\makebox(0,0)[lt]{\lineheight{1.25}\smash{\begin{tabular}[t]{l}$s_2$\end{tabular}}}}%
    \put(0.33146693,0.28896828){\color[rgb]{0,0,0}\makebox(0,0)[lt]{\lineheight{1.25}\smash{\begin{tabular}[t]{l}$t_2$\end{tabular}}}}%
    \put(0,0){\includegraphics[width=\unitlength,page=3]{eqRefS01.pdf}}%
    \put(0.48915933,0.19166978){\color[rgb]{0,0,0}\makebox(0,0)[lt]{\lineheight{1.25}\smash{\begin{tabular}[t]{l}$=$\end{tabular}}}}%
  \end{picture}%
\endgroup%

\end{center}
we obtain thanks to relation \eqref{ruseStated} that $\sum_{ij} R^{s_2s_1}_{ij} \mathbf{V}^{ij}_{t_2t_1} = \sum_{ij} \mathbf{V}^{s_1s_2}_{ji} R^{t_2t_1}_{ij}$, that is $R_{21} \mathbf{V}_{21} = \mathbf{V} (\exposantGauche{t}{b_{\alpha}})_1 (\exposantGauche{t}{a_{\alpha}})_2$. Then, relations  \eqref{ruseFondaDual} and $(S \otimes S)(R) = R$ yields $R_{21} \mathbf{V}_{21}D_2^{-1}D_1^{-1} = \mathbf{V} D_1^{-1}D_2^{-1} R_{21}$. Finally, we multiply this equality by $g_1^{-1}g_2^{-1}$ and we use the relations $Rg_1g_2 = g_1g_2R$ and \eqref{propKD}:
\[ R_{21} \mathbf{V}_{21} (\exposantGauche{t}{\!D})_2^{-1} (\exposantGauche{t}{\!D})_1^{-1} =  \mathbf{V} (\exposantGauche{t}{\!D})_1^{-1} (\exposantGauche{t}{\!D})_2^{-1} R_{21}. \]
This last equality together with \eqref{fusionS01} gives the first equality of \eqref{presentationL01Uq2}:
\[ R \, \bigl( \mathbf{U} (^t \! D)^{-1} \bigr)_1 \, R_{21} \, \bigl( \mathbf{U} (^t \! D)^{-1} \bigr)_2 = \bigl( \mathbf{U} (^t \! D)^{-1} \bigr)_2 \, R \, \bigl( \mathbf{U} (^t \! D)^{-1} \bigr)_1 \, R_{21}. \]
In order to obtain the quantum determinant relation, observe that
\begin{center}
%% Accompanies image file 'qdetS01.pdf' (pdf, eps, ps)
\begingroup%
  \makeatletter%
  \providecommand\color[2][]{%
    \errmessage{(Inkscape) Color is used for the text in Inkscape, but the package 'color.sty' is not loaded}%
    \renewcommand\color[2][]{}%
  }%
  \providecommand\transparent[1]{%
    \errmessage{(Inkscape) Transparency is used (non-zero) for the text in Inkscape, but the package 'transparent.sty' is not loaded}%
    \renewcommand\transparent[1]{}%
  }%
  \providecommand\rotatebox[2]{#2}%
  \newcommand*\fsize{\dimexpr\f@size pt\relax}%
  \newcommand*\lineheight[1]{\fontsize{\fsize}{#1\fsize}\selectfont}%
  \ifx\svgwidth\undefined%
    \setlength{\unitlength}{256.45414334bp}%
    \ifx\svgscale\undefined%
      \relax%
    \else%
      \setlength{\unitlength}{\unitlength * \real{\svgscale}}%
    \fi%
  \else%
    \setlength{\unitlength}{\svgwidth}%
  \fi%
  \global\let\svgwidth\undefined%
  \global\let\svgscale\undefined%
  \makeatother%
  \begin{picture}(1,0.26644255)%
    \lineheight{1}%
    \setlength\tabcolsep{0pt}%
    \put(0.49272963,0.15534556){\color[rgb]{0,0,0}\makebox(0,0)[lt]{\lineheight{1.25}\smash{\begin{tabular}[t]{l}$=$\end{tabular}}}}%
    \put(0,0){\includegraphics[width=\unitlength,page=1]{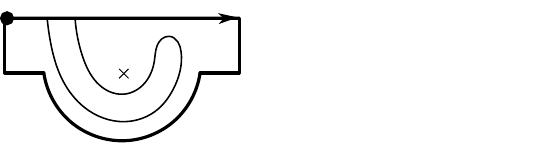}}%
    \put(0.07366821,0.24794166){\color[rgb]{0,0,0}\makebox(0,0)[lt]{\lineheight{1.25}\smash{\begin{tabular}[t]{l}$s$\end{tabular}}}}%
    \put(0.12872403,0.24758824){\color[rgb]{0,0,0}\makebox(0,0)[lt]{\lineheight{1.25}\smash{\begin{tabular}[t]{l}$t$\end{tabular}}}}%
    \put(0,0){\includegraphics[width=\unitlength,page=2]{qdetS01.pdf}}%
    \put(0.6622822,0.24633961){\color[rgb]{0,0,0}\makebox(0,0)[lt]{\lineheight{1.25}\smash{\begin{tabular}[t]{l}$s$\end{tabular}}}}%
    \put(0.71733801,0.24598619){\color[rgb]{0,0,0}\makebox(0,0)[lt]{\lineheight{1.25}\smash{\begin{tabular}[t]{l}$t$\end{tabular}}}}%
  \end{picture}%
\endgroup%

\end{center}
and then \eqref{ruseStated} gives $\sum_{ij}\mathbf{V}^{st}_{ij} \cap_{ji} = \cup^{st}$, which can be rewritten as $\mathbf{V} \, \exposantGauche{t}{\!(P\cap)} = \cup$, where $P^{ij}_{kl} = \delta^i_l\delta^j_k$ is the flip tensor. Moreover, a computation reveals that $\exposantGauche{t}{\!(P\cap)} = (\exposantGauche{t}{\!D})_1^{-1} (\exposantGauche{t}{\!D})_2^{-1} \cup$. Hence, by \eqref{fusionS01} we have 
\[ \bigl( \mathbf{U} (^t \! D)^{-1} \bigr)_1 R_{21}  \bigl( \mathbf{U} (^t \! D)^{-1} \bigr)_2 R_{21}^{-1} \cup = \cup. \]
Expanding this matrix equality in components gives only one new relation, which is precisely $\mathrm{qdet}\bigl( \mathbf{U} (^t \! D)^{-1} \bigr) = 1$. Hence $j$ is a morphism of algebras. Let us show that it is a morphism of $\mathcal{O}_{q^2}$-comodules (recall that we identify $\mathcal{S}^{\mathrm{s}}_q(\mathfrak{B})$ with $\mathcal{O}_{q^2}$). By definition of the coaction of $\mathcal{O}_{q^2}$ on $\mathcal{S}^{\mathrm{s}}_q(\Sigma_{0,1}^{\mathrm{o},\bullet})$, we have
\begin{center}
%% Accompanies image file 'coactionS01.pdf' (pdf, eps, ps)
\begingroup%
  \makeatletter%
  \providecommand\color[2][]{%
    \errmessage{(Inkscape) Color is used for the text in Inkscape, but the package 'color.sty' is not loaded}%
    \renewcommand\color[2][]{}%
  }%
  \providecommand\transparent[1]{%
    \errmessage{(Inkscape) Transparency is used (non-zero) for the text in Inkscape, but the package 'transparent.sty' is not loaded}%
    \renewcommand\transparent[1]{}%
  }%
  \providecommand\rotatebox[2]{#2}%
  \newcommand*\fsize{\dimexpr\f@size pt\relax}%
  \newcommand*\lineheight[1]{\fontsize{\fsize}{#1\fsize}\selectfont}%
  \ifx\svgwidth\undefined%
    \setlength{\unitlength}{411.8112632bp}%
    \ifx\svgscale\undefined%
      \relax%
    \else%
      \setlength{\unitlength}{\unitlength * \real{\svgscale}}%
    \fi%
  \else%
    \setlength{\unitlength}{\svgwidth}%
  \fi%
  \global\let\svgwidth\undefined%
  \global\let\svgscale\undefined%
  \makeatother%
  \begin{picture}(1,0.18385076)%
    \lineheight{1}%
    \setlength\tabcolsep{0pt}%
    \put(0,0){\includegraphics[width=\unitlength,page=1]{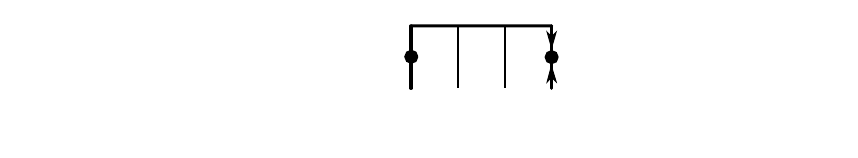}}%
    \put(0.52747336,0.16218703){\color[rgb]{0,0,0}\makebox(0,0)[lt]{\lineheight{1.25}\smash{\begin{tabular}[t]{l}$s$\end{tabular}}}}%
    \put(0.58242251,0.16212722){\color[rgb]{0,0,0}\makebox(0,0)[lt]{\lineheight{1.25}\smash{\begin{tabular}[t]{l}$t$\end{tabular}}}}%
    \put(0.52946886,0.05583789){\color[rgb]{0,0,0}\makebox(0,0)[lt]{\lineheight{1.25}\smash{\begin{tabular}[t]{l}$i$\end{tabular}}}}%
    \put(0,0){\includegraphics[width=\unitlength,page=2]{coactionS01.pdf}}%
    \put(0.59252625,0.05489366){\color[rgb]{0,0,0}\makebox(0,0)[lt]{\lineheight{1.25}\smash{\begin{tabular}[t]{l}$j$\end{tabular}}}}%
    \put(0,0){\includegraphics[width=\unitlength,page=3]{coactionS01.pdf}}%
    \put(0.09476423,0.16382851){\color[rgb]{0,0,0}\makebox(0,0)[lt]{\lineheight{1.25}\smash{\begin{tabular}[t]{l}$s$\end{tabular}}}}%
    \put(0.17429811,0.16416306){\color[rgb]{0,0,0}\makebox(0,0)[lt]{\lineheight{1.25}\smash{\begin{tabular}[t]{l}$t$\end{tabular}}}}%
    \put(0.20419174,0.10725946){\color[rgb]{0,0,0}\makebox(0,0)[lt]{\lineheight{1.25}\smash{\begin{tabular}[t]{l}$e$\end{tabular}}}}%
    \put(0,0){\includegraphics[width=\unitlength,page=4]{coactionS01.pdf}}%
    \put(0.33333445,0.12767423){\color[rgb]{0,0,0}\makebox(0,0)[lt]{\lineheight{1.25}\smash{\begin{tabular}[t]{l}$\Delta_e$\end{tabular}}}}%
    \put(0,0){\includegraphics[width=\unitlength,page=5]{coactionS01.pdf}}%
    \put(0.81414736,0.16577982){\color[rgb]{0,0,0}\makebox(0,0)[lt]{\lineheight{1.25}\smash{\begin{tabular}[t]{l}$i$\end{tabular}}}}%
    \put(0.89563252,0.16611437){\color[rgb]{0,0,0}\makebox(0,0)[lt]{\lineheight{1.25}\smash{\begin{tabular}[t]{l}$j$\end{tabular}}}}%
    \put(0.67176557,0.11298823){\color[rgb]{0,0,0}\makebox(0,0)[lt]{\lineheight{1.25}\smash{\begin{tabular}[t]{l}$\otimes$\end{tabular}}}}%
    \put(0.41289968,0.11205064){\color[rgb]{0,0,0}\makebox(0,0)[lt]{\lineheight{1.25}\smash{\begin{tabular}[t]{l}$\displaystyle\sum_{ij}$\end{tabular}}}}%
  \end{picture}%
\endgroup%

\end{center}
Hence $\Delta_e(\mathbf{U}^s_t) = \sum_{ij} T^s_iT^t_j \otimes \mathbf{U}^i_j$. Let us identify $\mathbf{U}^k_l$ (resp. $T^k_l$) with $1 \otimes \mathbf{U}^k_l$ (resp. $T^k_l  \otimes 1$); then we get $\Delta_e(\mathbf{U}) = T\, \mathbf{U} \, ^tT$. By \eqref{ruseFondaDual} and \eqref{propKD}, it holds $\exposantGauche{t}{T} = D^{-1} \, S^{-1}(T) \, D = D^{-1}g^{-1} S(T) gD = \exposantGauche{t}{\! D}^{-1} \, S(T) \, \exposantGauche{t}{\! D}$ and hence
\[ \Delta_e \circ j(M) = \Delta_e(\mathbf{U}) \, \exposantGauche{t}{\! D}^{-1} = T\, \mathbf{U} \, ^tT \, \exposantGauche{t}{\! D}^{-1} = T\, \big(\mathbf{U} \, \exposantGauche{t}{\! D}^{-1}\big) \, S(T) = (j \otimes \mathrm{id}) \circ \Omega(M)  \]
where $\Omega$ is the coaction \eqref{coactionH} on $\mathcal{L}_{0,1}(U_{q^2})$.
\\2. Observe that
\begin{center}
%% Accompanies image file 'generateurS01.pdf' (pdf, eps, ps)
\begingroup%
  \makeatletter%
  \providecommand\color[2][]{%
    \errmessage{(Inkscape) Color is used for the text in Inkscape, but the package 'color.sty' is not loaded}%
    \renewcommand\color[2][]{}%
  }%
  \providecommand\transparent[1]{%
    \errmessage{(Inkscape) Transparency is used (non-zero) for the text in Inkscape, but the package 'transparent.sty' is not loaded}%
    \renewcommand\transparent[1]{}%
  }%
  \providecommand\rotatebox[2]{#2}%
  \newcommand*\fsize{\dimexpr\f@size pt\relax}%
  \newcommand*\lineheight[1]{\fontsize{\fsize}{#1\fsize}\selectfont}%
  \ifx\svgwidth\undefined%
    \setlength{\unitlength}{187.02106563bp}%
    \ifx\svgscale\undefined%
      \relax%
    \else%
      \setlength{\unitlength}{\unitlength * \real{\svgscale}}%
    \fi%
  \else%
    \setlength{\unitlength}{\svgwidth}%
  \fi%
  \global\let\svgwidth\undefined%
  \global\let\svgscale\undefined%
  \makeatother%
  \begin{picture}(1,0.45527069)%
    \lineheight{1}%
    \setlength\tabcolsep{0pt}%
    \put(0,0){\includegraphics[width=\unitlength,page=1]{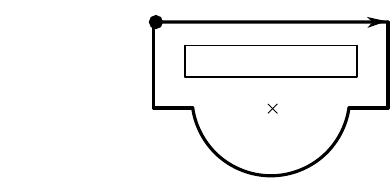}}%
    \put(0.68121361,0.28045127){\color[rgb]{0,0,0}\makebox(0,0)[lt]{\lineheight{1.25}\smash{\begin{tabular}[t]{l}$\sigma$\end{tabular}}}}%
    \put(0,0){\includegraphics[width=\unitlength,page=2]{generateurS01.pdf}}%
    \put(0.48564914,0.42762823){\color[rgb]{0,0,0}\makebox(0,0)[lt]{\lineheight{1.25}\smash{\begin{tabular}[t]{l}$s_1$\end{tabular}}}}%
    \put(0.55879939,0.21333115){\color[rgb]{0,0,0}\makebox(0,0)[lt]{\lineheight{1.25}\smash{\begin{tabular}[t]{l}...\end{tabular}}}}%
    \put(0.78608609,0.21333115){\color[rgb]{0,0,0}\makebox(0,0)[lt]{\lineheight{1.25}\smash{\begin{tabular}[t]{l}...\end{tabular}}}}%
    \put(0.67253976,0.36340108){\color[rgb]{0,0,0}\makebox(0,0)[lt]{\lineheight{1.25}\smash{\begin{tabular}[t]{l}...\end{tabular}}}}%
    \put(0,0){\includegraphics[width=\unitlength,page=3]{generateurS01.pdf}}%
    \put(0.5644218,0.42762834){\color[rgb]{0,0,0}\makebox(0,0)[lt]{\lineheight{1.25}\smash{\begin{tabular}[t]{l}$t_1$\end{tabular}}}}%
    \put(0.80646863,0.42762823){\color[rgb]{0,0,0}\makebox(0,0)[lt]{\lineheight{1.25}\smash{\begin{tabular}[t]{l}$s_l$\end{tabular}}}}%
    \put(0.88524135,0.42762834){\color[rgb]{0,0,0}\makebox(0,0)[lt]{\lineheight{1.25}\smash{\begin{tabular}[t]{l}$t_l$\end{tabular}}}}%
    \put(-0.00170092,0.27776332){\color[rgb]{0,0,0}\makebox(0,0)[lt]{\lineheight{1.25}\smash{\begin{tabular}[t]{l}$\mathbf{U}^{s_1}_{t_1} \ldots \mathbf{U}^{s_l}_{t_l} = $\end{tabular}}}}%
  \end{picture}%
\endgroup%

\end{center}
where the handle contains a bunch of $l$ parallel strands and $\sigma \in B_{2n}$ is some braid (it can be computed explicitly but we do not need it). By \eqref{ruseStated}, it follows that 
\[ \mathbf{U}^{s_1}_{t_1} \ldots \mathbf{U}^{s_l}_{t_l} = \sum_{\substack{i_1, \ldots, i_k \\i'_1, \ldots, i'_k}} Z(\sigma)_{i_1 \ldots i_l \, i'_1 \ldots i'_l}^{s_1 \, t_1 \ldots s_l \, t_l} \, \mathbf{U}(l)^{i_1 \ldots i_l}_{i'_1 \ldots i'_l}. \]
But $\sigma$ is a braid, so that $Z(\sigma)$ is an invertible matrix, and hence we can invert these relations and write any element $\mathbf{U}(l)^{i_1 \ldots i_l}_{i'_1 \ldots i'_l}$ as a linear combination of the monomials $\mathbf{U}^{s_1}_{t_1} \ldots \mathbf{U}^{s_l}_{t_l}$. The claim is proved since, as already observed, the collection of elements $\mathbf{U}(l)^{i_1 \ldots i_l}_{i'_1 \ldots i'_l}$ spans $\mathcal{S}^{\mathrm{s}}_q(\Sigma_{0,1}^{\mathrm{o},\bullet})$ as a $\mathbb{Z}[q^{\pm 1/2}]$-module.
\end{proof}

\begin{lemma}\label{isoS01L01}
The map $\mathrm{hol}^{\mathrm{s}} : \mathcal{S}_q^{\mathrm{s}}(\Sigma_{0,1}^{\mathrm{o},\bullet}) \to \mathcal{L}_{0,1}(U_{q^2})$ is an isomorphism of $\mathcal{O}_{q^2}$-comodule-algebras.
\end{lemma}
\begin{proof}
Thanks to \eqref{holSurMatricesGen} we have $\mathrm{hol}^{\mathrm{s}} \circ j = \mathrm{id}$. It follows that $j$ is injective and it is surjective by Lemma \ref{lemmeInjectionL01S01}. Hence it is an isomorphism of $\mathcal{O}_{q^2}$-comodule-algebras, with inverse $\mathrm{hol}^{\mathrm{s}}$.
\end{proof}

\subsection{Proof for $\Sigma_{1,0}^{\mathrm{o},\bullet}$}
Let \[
B = \overset{V_2}{B} = 
\begin{pmatrix}
B^-_- & B^-_+\\
B^+_- & B^+_+
\end{pmatrix}, \:\:
A = \overset{V_2}{A} =
\begin{pmatrix}
A^-_- & A^-_+\\
A^+_- & A^+_+
\end{pmatrix} \:\: \in \mathcal{L}_{1,0}(U_{q^2}) \otimes \mathrm{End}_{\mathbb{C}}(V_2).
\]
Then by Lemma \ref{GensRelsLgnUq2}, $\mathcal{L}_{1,0}(U_{q^2})$ is generated by the $8$ coefficients of the matrices $B$ and $A$, modulo the following relations:
\begin{equation}\label{presentationL10Uq2}
\begin{split}
& R B_1 R_{21} B_2 = B_2 R B_1 R_{21}, \:\:\:\:\:\:\:\:  B^-_-B^+_+ - q^4B^-_+B^+_- = 1\\
& R A_1 R_{21} A_2 = A_2 R A_1 R_{21}, \:\:\:\:\:\:\:\: A^-_-A^+_+ - q^4A^-_+A^+_- = 1\\
& R B_1 R_{21} A_2 = A_2 R B_1 R^{-1}.
\end{split}
\end{equation}
\indent As previously, for $l \geq 0$ and $s_1, \ldots, s_l, t_1, \ldots, t_l \in \{\pm\}$ we define the following elements of $\mathcal{S}_q^{\mathrm{s}}(\Sigma_{1,0}^{\mathrm{o},\bullet})$:
\begin{center}
%% Accompanies image file 'matricesTore.pdf' (pdf, eps, ps)
\begingroup%
  \makeatletter%
  \providecommand\color[2][]{%
    \errmessage{(Inkscape) Color is used for the text in Inkscape, but the package 'color.sty' is not loaded}%
    \renewcommand\color[2][]{}%
  }%
  \providecommand\transparent[1]{%
    \errmessage{(Inkscape) Transparency is used (non-zero) for the text in Inkscape, but the package 'transparent.sty' is not loaded}%
    \renewcommand\transparent[1]{}%
  }%
  \providecommand\rotatebox[2]{#2}%
  \newcommand*\fsize{\dimexpr\f@size pt\relax}%
  \newcommand*\lineheight[1]{\fontsize{\fsize}{#1\fsize}\selectfont}%
  \ifx\svgwidth\undefined%
    \setlength{\unitlength}{504.69674593bp}%
    \ifx\svgscale\undefined%
      \relax%
    \else%
      \setlength{\unitlength}{\unitlength * \real{\svgscale}}%
    \fi%
  \else%
    \setlength{\unitlength}{\svgwidth}%
  \fi%
  \global\let\svgwidth\undefined%
  \global\let\svgscale\undefined%
  \makeatother%
  \begin{picture}(1,0.18108823)%
    \lineheight{1}%
    \setlength\tabcolsep{0pt}%
    \put(0,0){\includegraphics[width=\unitlength,page=1]{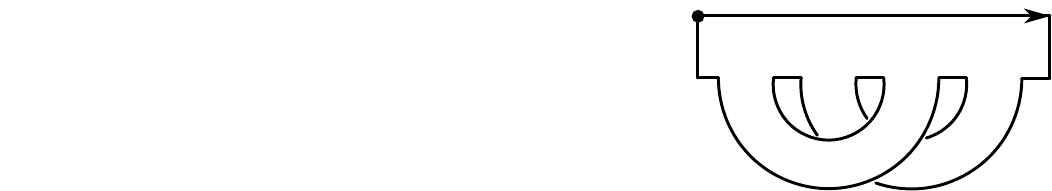}}%
    \put(-0.00064125,0.11966177){\color[rgb]{0,0,0}\makebox(0,0)[lt]{\lineheight{1.25}\smash{\begin{tabular}[t]{l}$\mathbf{U}(l)^{(b)}{^{s_1 \ldots s_l}_{t_1 \ldots t_l}}=$\end{tabular}}}}%
    \put(0,0){\includegraphics[width=\unitlength,page=2]{matricesTore.pdf}}%
    \put(0.52690321,0.1191739){\color[rgb]{0,0,0}\makebox(0,0)[lt]{\lineheight{1.25}\smash{\begin{tabular}[t]{l}$\mathbf{U}(l)^{(a)}{^{s_1 \ldots s_l}_{t_1 \ldots t_l}}=$\end{tabular}}}}%
    \put(0,0){\includegraphics[width=\unitlength,page=3]{matricesTore.pdf}}%
    \put(0.95816153,0.12543066){\color[rgb]{0,0,0}\rotatebox{-180}{\makebox(0,0)[lt]{\lineheight{1.25}\smash{\begin{tabular}[t]{l}...\end{tabular}}}}}%
    \put(0.7961727,0.12539852){\color[rgb]{0,0,0}\rotatebox{-180}{\makebox(0,0)[lt]{\lineheight{1.25}\smash{\begin{tabular}[t]{l}...\end{tabular}}}}}%
    \put(0.7643082,0.17445441){\color[rgb]{0,0,0}\makebox(0,0)[lt]{\lineheight{1.25}\smash{\begin{tabular}[t]{l}$s_1$\end{tabular}}}}%
    \put(0.79811055,0.1744672){\color[rgb]{0,0,0}\makebox(0,0)[lt]{\lineheight{1.25}\smash{\begin{tabular}[t]{l}$s_l$\end{tabular}}}}%
    \put(0.92995379,0.17469474){\color[rgb]{0,0,0}\makebox(0,0)[lt]{\lineheight{1.25}\smash{\begin{tabular}[t]{l}$t_l$\end{tabular}}}}%
    \put(0.95765789,0.17461373){\color[rgb]{0,0,0}\makebox(0,0)[lt]{\lineheight{1.25}\smash{\begin{tabular}[t]{l}$t_1$\end{tabular}}}}%
    \put(0.17212615,0.17401915){\color[rgb]{0,0,0}\makebox(0,0)[lt]{\lineheight{1.25}\smash{\begin{tabular}[t]{l}$s_1$\end{tabular}}}}%
    \put(0.20592846,0.17403193){\color[rgb]{0,0,0}\makebox(0,0)[lt]{\lineheight{1.25}\smash{\begin{tabular}[t]{l}$s_l$\end{tabular}}}}%
    \put(0.33139195,0.17238305){\color[rgb]{0,0,0}\makebox(0,0)[lt]{\lineheight{1.25}\smash{\begin{tabular}[t]{l}$t_l$\end{tabular}}}}%
    \put(0.36031572,0.17220823){\color[rgb]{0,0,0}\makebox(0,0)[lt]{\lineheight{1.25}\smash{\begin{tabular}[t]{l}$t_1$\end{tabular}}}}%
    \put(0,0){\includegraphics[width=\unitlength,page=4]{matricesTore.pdf}}%
    \put(0.20178903,0.12816383){\color[rgb]{0,0,0}\rotatebox{-180}{\makebox(0,0)[lt]{\lineheight{1.25}\smash{\begin{tabular}[t]{l}...\end{tabular}}}}}%
    \put(0.36313056,0.12816383){\color[rgb]{0,0,0}\rotatebox{-180}{\makebox(0,0)[lt]{\lineheight{1.25}\smash{\begin{tabular}[t]{l}...\end{tabular}}}}}%
  \end{picture}%
\endgroup%

\end{center}
These elements can be arranged into tensors $\mathbf{U}(l)^{(b)}, \mathbf{U}(l)^{(a)} \in \mathcal{S}_q^{\mathrm{s}}(\Sigma_{1,0}^{\mathrm{o},\bullet}) \otimes \mathrm{End}_{\mathbb{C}}(V_2)^{\otimes l}$. For simplicity we denote $\mathbf{U}^{(b)} = \mathbf{U}(1)^{(b)}, \, \mathbf{U}^{(a)} = \mathbf{U}(1)^{(a)}$. We have the analogue of \eqref{ruseStated}:
\begin{equation}\label{decoupeTore}
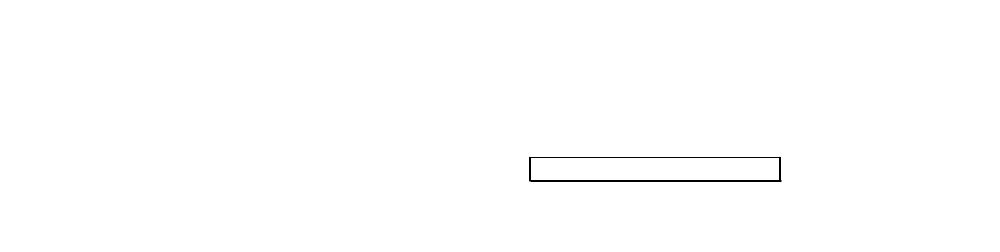
\end{equation}
where the sum is over all the indices $i_{\eta}, i'_{\eta}, j_{\eta}, j'_{\eta}$, each one ranging in $\{ \pm \}$. To obtain this relation, one applies as before the equality $\mathrm{id} = (\varepsilon \otimes \mathrm{id}) \circ \Delta_e$ and \eqref{RTbigone}  to the dashed arc $e$ in the figure below, which gives a bigon and another copy of $\Sigma_{1,0}^{\mathrm{o},\bullet}$:
\begin{center}
%% Accompanies image file 'ruseTore.pdf' (pdf, eps, ps)
\begingroup%
  \makeatletter%
  \providecommand\color[2][]{%
    \errmessage{(Inkscape) Color is used for the text in Inkscape, but the package 'color.sty' is not loaded}%
    \renewcommand\color[2][]{}%
  }%
  \providecommand\transparent[1]{%
    \errmessage{(Inkscape) Transparency is used (non-zero) for the text in Inkscape, but the package 'transparent.sty' is not loaded}%
    \renewcommand\transparent[1]{}%
  }%
  \providecommand\rotatebox[2]{#2}%
  \newcommand*\fsize{\dimexpr\f@size pt\relax}%
  \newcommand*\lineheight[1]{\fontsize{\fsize}{#1\fsize}\selectfont}%
  \ifx\svgwidth\undefined%
    \setlength{\unitlength}{198.65193233bp}%
    \ifx\svgscale\undefined%
      \relax%
    \else%
      \setlength{\unitlength}{\unitlength * \real{\svgscale}}%
    \fi%
  \else%
    \setlength{\unitlength}{\svgwidth}%
  \fi%
  \global\let\svgwidth\undefined%
  \global\let\svgscale\undefined%
  \makeatother%
  \begin{picture}(1,0.72439789)%
    \lineheight{1}%
    \setlength\tabcolsep{0pt}%
    \put(0,0){\includegraphics[width=\unitlength,page=1]{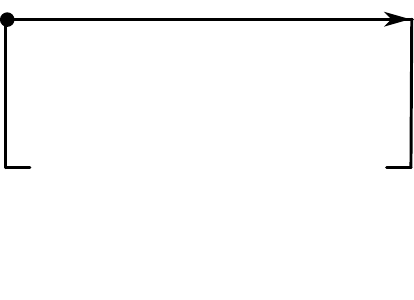}}%
    \put(0.16773787,0.64373985){\color[rgb]{0,0,0}\makebox(0,0)[lt]{\lineheight{1.25}\smash{\begin{tabular}[t]{l}...\end{tabular}}}}%
    \put(0.77158561,0.64364372){\color[rgb]{0,0,0}\makebox(0,0)[lt]{\lineheight{1.25}\smash{\begin{tabular}[t]{l}...\end{tabular}}}}%
    \put(0.13393738,0.70237608){\color[rgb]{0,0,0}\makebox(0,0)[lt]{\lineheight{1.25}\smash{\begin{tabular}[t]{l}$s_1$\end{tabular}}}}%
    \put(0.85001365,0.70527465){\color[rgb]{0,0,0}\makebox(0,0)[lt]{\lineheight{1.25}\smash{\begin{tabular}[t]{l}$s_k$\end{tabular}}}}%
    \put(0,0){\includegraphics[width=\unitlength,page=2]{ruseTore.pdf}}%
    \put(0.49735973,0.56430001){\color[rgb]{0,0,0}\makebox(0,0)[lt]{\lineheight{1.25}\smash{\begin{tabular}[t]{l}$T$\end{tabular}}}}%
    \put(0,0){\includegraphics[width=\unitlength,page=3]{ruseTore.pdf}}%
    \put(0.10234182,0.4906995){\color[rgb]{0,0,0}\makebox(0,0)[lt]{\lineheight{1.25}\smash{\begin{tabular}[t]{l}...\end{tabular}}}}%
    \put(0.8385543,0.49204792){\color[rgb]{0,0,0}\makebox(0,0)[lt]{\lineheight{1.25}\smash{\begin{tabular}[t]{l}...\end{tabular}}}}%
    \put(0,0){\includegraphics[width=\unitlength,page=4]{ruseTore.pdf}}%
  \end{picture}%
\endgroup%

\end{center}
where we force the apparition of the product $\mathbf{U}(l)^{(b)}{^{i_1 \ldots i_l}_{j_1 \ldots j_l}} \mathbf{U}(m)^{(a)}{^{i'_1 \ldots i'_m}_{j'_1 \ldots j'_m}}$ using isotopy; this explains why we have $T^{\mathrm{br}}$ instead of $T$ in the right-hand side of \eqref{decoupeTore}. In particular, \eqref{decoupeTore} implies that the collection of elements $\mathbf{U}(l)^{(b)}{^{i_1 \ldots i_l}_{j_1 \ldots j_l}} \mathbf{U}(m)^{(a)}{^{i'_1 \ldots i'_m}_{j'_1 \ldots j'_m}}$ (with $l,m \geq 0$ and $i_{\eta}, j_{\eta}, i'_{\eta}, j'_{\eta} \in \{ \pm \}$ for each $\eta$) spans $\mathcal{S}^{\mathrm{s}}_q(\Sigma_{1,0}^{\mathrm{o},\bullet})$ as a $\mathbb{Z}[q^{\pm 1/2}]$-module.

\smallskip

\indent The next result is the exact analogue of Lemma \ref{lemmeInjectionL01S01}:

\begin{lemma}\label{lemmeInjectionL10S10}
1. We have a morphism $j : \mathcal{L}_{1,0}(U_{q^2}) \to \mathcal{S}^{\mathrm{s}}_q(\Sigma_{1,0}^{\mathrm{o},\bullet})$ of $\mathcal{O}_{q^2}$-comodule-algebras defined by $j(B) = \mathbf{U}^{(b)} (^t \! D)^{-1}$ and $j(A) = \mathbf{U}^{(a)} (^t \! D)^{-1}$. Explicitly:
\[  
\begin{pmatrix}
j(B^-_-) & j(B^-_+) \\
j(B^+_-) & j(B^+_+)
\end{pmatrix}
=
\begin{pmatrix}
q^{-5/2}\mathbf{U}^{(b)}{^-_+}  & -q^{-1/2}\mathbf{U}^{(b)}{^-_-} \\
q^{-5/2}\mathbf{U}^{(b)}{^+_+} & -q^{-1/2}\mathbf{U}^{(b)}{^+_-}
\end{pmatrix}, \:\:\:\:\:\:\:
\begin{pmatrix}
j(A^-_-) & j(A^-_+) \\
j(A^+_-) & j(A^+_+)
\end{pmatrix}
=
\begin{pmatrix}
q^{-5/2}\mathbf{U}^{(a)}{^-_+}  & -q^{-1/2}\mathbf{U}^{(a)}{^-_-} \\
q^{-5/2}\mathbf{U}^{(a)}{^+_+} & -q^{-1/2}\mathbf{U}^{(a)}{^+_-}
\end{pmatrix}.
\]
2. The morphism $j$ is surjective.
\end{lemma}
\begin{proof}
1. The proof is again entirely based on formal matrix computations. It is clear that the two first lines of relations in \eqref{presentationL10Uq2} are satisfied, since it suffices to repeat the proof of Lemma \ref{lemmeInjectionL01S01} in each handle with the matrices $\mathbf{U}^{(b)}, \mathbf{U}^{(a)}$ instead of $\mathbf{U}$. To show the exchange relation (third line of \eqref{presentationL10Uq2}), observe that by \eqref{decoupeTore}:
\begin{center}
%% Accompanies image file 'EchangeS10.pdf' (pdf, eps, ps)
\begingroup%
  \makeatletter%
  \providecommand\color[2][]{%
    \errmessage{(Inkscape) Color is used for the text in Inkscape, but the package 'color.sty' is not loaded}%
    \renewcommand\color[2][]{}%
  }%
  \providecommand\transparent[1]{%
    \errmessage{(Inkscape) Transparency is used (non-zero) for the text in Inkscape, but the package 'transparent.sty' is not loaded}%
    \renewcommand\transparent[1]{}%
  }%
  \providecommand\rotatebox[2]{#2}%
  \newcommand*\fsize{\dimexpr\f@size pt\relax}%
  \newcommand*\lineheight[1]{\fontsize{\fsize}{#1\fsize}\selectfont}%
  \ifx\svgwidth\undefined%
    \setlength{\unitlength}{455.95161859bp}%
    \ifx\svgscale\undefined%
      \relax%
    \else%
      \setlength{\unitlength}{\unitlength * \real{\svgscale}}%
    \fi%
  \else%
    \setlength{\unitlength}{\svgwidth}%
  \fi%
  \global\let\svgwidth\undefined%
  \global\let\svgscale\undefined%
  \makeatother%
  \begin{picture}(1,0.19518288)%
    \lineheight{1}%
    \setlength\tabcolsep{0pt}%
    \put(0,0){\includegraphics[width=\unitlength,page=1]{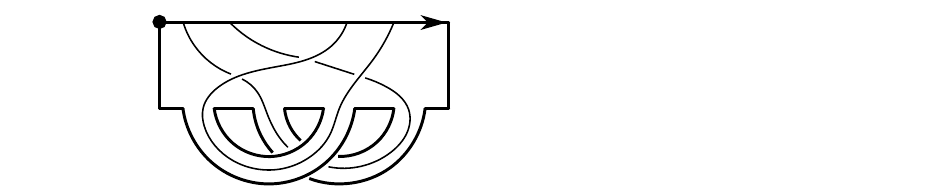}}%
    \put(0.18459987,0.18129877){\color[rgb]{0,0,0}\makebox(0,0)[lt]{\lineheight{1.25}\smash{\begin{tabular}[t]{l}$s_1$\end{tabular}}}}%
    \put(0.23819102,0.18168204){\color[rgb]{0,0,0}\makebox(0,0)[lt]{\lineheight{1.25}\smash{\begin{tabular}[t]{l}$t_1$\end{tabular}}}}%
    \put(0.35715066,0.18087913){\color[rgb]{0,0,0}\makebox(0,0)[lt]{\lineheight{1.25}\smash{\begin{tabular}[t]{l}$s_2$\end{tabular}}}}%
    \put(0.41074181,0.1812624){\color[rgb]{0,0,0}\makebox(0,0)[lt]{\lineheight{1.25}\smash{\begin{tabular}[t]{l}$t_2$\end{tabular}}}}%
    \put(0.48801479,0.09547944){\color[rgb]{0,0,0}\makebox(0,0)[lt]{\lineheight{1.25}\smash{\begin{tabular}[t]{l}$\displaystyle= \sum_{ijkl} \mathbf{U}^{(b)}{^i_j}\,\mathbf{U}^{(a)}{^k_l}\,Z$\end{tabular}}}}%
    \put(-0.00109496,0.09403165){\color[rgb]{0,0,0}\makebox(0,0)[lt]{\lineheight{1.25}\smash{\begin{tabular}[t]{l}$\mathbf{U}^{(a)}{^{s_1}_{t_1}} \mathbf{U}^{(b)}{^{s_2}_{t_2}} =$\end{tabular}}}}%
    \put(0,0){\includegraphics[width=\unitlength,page=2]{EchangeS10.pdf}}%
    \put(0.87681649,0.14941812){\color[rgb]{0,0,0}\makebox(0,0)[lt]{\lineheight{1.25}\smash{\begin{tabular}[t]{l}$_{s_1t_1s_2t_2}$\end{tabular}}}}%
    \put(0.87617864,0.05042959){\color[rgb]{0,0,0}\makebox(0,0)[lt]{\lineheight{1.25}\smash{\begin{tabular}[t]{l}$_{ijkl}$\end{tabular}}}}%
  \end{picture}%
\endgroup%

\end{center}
The evaluation of the braid gives the following expression:
\[ \mathbf{U}^{(a)}_1 \mathbf{U}^{(b)}_2 = \Big(a_{\delta} \, a_{\gamma} \, \mathbf{U}^{(b)} \, ^tb_{\alpha} \, ^ta_{\beta}\Big)_{\!2} \, \Big( b_{\gamma} \, S(a_{\alpha}) \, \mathbf{U}^{(a)} \, {^tb_{\beta}} \, {^tb_{\delta}} \Big)_{\!1} \]
with $R = a_{\eta} \otimes b_{\eta}$ is the matrix in \eqref{RMatriceV2}. Then as in the proof of Lemma \ref{lemmeInjectionL01S01}, a computation left to the reader based on \eqref{ruseFondaDual} and \eqref{propKD} shows that this equality is equivalent to
\[ R \, \bigl( \mathbf{U}^{(b)} (^t \! D)^{-1} \bigr)_{\!1} \, R_{21} \, \bigl( \mathbf{U}^{(a)} (^t \! D)^{-1} \bigr)_{\!2} = \bigl( \mathbf{U}^{(a)} (^t \! D)^{-1} \bigr)_{\!2} \, R \,  \bigl( \mathbf{U}^{(b)} (^t \! D)^{-1} \bigr)_{\!1 } \, R^{-1} \]
as desired.
\\2. By repeating the proof of the second claim of Lemma \ref{lemmeInjectionL01S01} in each handle, we know that we can write any element $\mathbf{U}(l)^{(b)}{^{i_1 \ldots i_l}_{j_1 \ldots j_l}}$ (resp. $\mathbf{U}(m)^{(a)}{^{i'_1 \ldots i'_m}_{j'_1 \ldots j'_m}}$) as a polynomial in the $4$ elements $\mathbf{U}^{(b)}{^s_t}$ (resp. $\mathbf{U}^{(a)}{^{s'}_{t'}}$), with $s,t,s',t' \in \{ \pm \}$. Hence any element $\mathbf{U}(l)^{(b)}{^{i_1 \ldots i_l}_{j_1 \ldots j_l}}\mathbf{U}(m)^{(a)}{^{i'_1 \ldots i'_m}_{j'_1 \ldots j'_m}}$ can be written as a polynomial in the $8$ elements $\mathbf{U}^{(b)}{^s_t}$, $\mathbf{U}^{(a)}{^{s'}_{t'}}$. The claim is proved since, as already observed, this collection of elements spans $\mathcal{S}^{\mathrm{s}}_q(\Sigma_{1,0}^{\mathrm{o},\bullet})$ as a $\mathbb{Z}[q^{\pm 1/2}]$-module.
\end{proof}

\begin{lemma}
The map $\mathrm{hol}^{\mathrm{s}} : \mathcal{S}_q^{\mathrm{s}}(\Sigma_{1,0}^{\mathrm{o},\bullet}) \to \mathcal{L}_{1,0}(U_{q^2})$ is an isomorphism of $\mathcal{O}_{q^2}$-comodule-algebras.
\end{lemma}
\begin{proof}
Completely similar to the proof of Lemma \ref{isoS01L01}.
\end{proof}

\section{Geometric interpretation of the vacuum representation of $\mathcal{L}_{g,0}(H)$ on $\mathcal{L}_{0,g}(H)$}\label{sectionVacuum}
In this section we take $n=0$. Let $\mathcal{L}_{g,0}^A(H)$ be the subalgebra of $\mathcal{L}_{g,0}(H)$ generated by all the coefficients $\overset{I}{A}(i)^k_l$ of the matrices $\overset{I}{A}(i)$ (for all $i$ and all the finite dimensional $H$-modules $I$). We define a right action $\triangleleft$ of $\mathcal{L}^A_{g,0}(H)$ on $\mathbb{C}$ by 
\[ \forall \, i, \:\:\: 1 \triangleleft \overset{I}{A}(i) = \mathbb{I}_{\dim(I)} \:\:\:\:\:\:\:\:\: (\textit{i.e.} \:\: 1 \triangleleft \overset{I}{A}(i)^k_l = \delta^k_l \: ). \]
It is immediate that $\triangleleft$ is compatible with the relations \eqref{relationFusion}, \eqref{echangeLgn} of Definition \ref{defLgn}. Hence we have a representation $\mathbb{C}_{\mathrm{vac}}$ of $\mathcal{L}^A_{g,0}(H)$. $1$ is called a vacuum vector or a cyclic vector, often denoted $\Omega$ or $|0\rangle$ (see for instance \cite[Thm. 21]{AS}, but with different conventions than here). Then we define the right vacuum representation of $\mathcal{L}_{g,0}(H)$ to be the induced representation
\[ \mathbb{C}_{\mathrm{vac}} \underset{\mathcal{L}^A_{g,0}(H)}{\otimes} \mathcal{L}_{g,0}(H) \]
where $\mathcal{L}^A_{g,0}(H)$ acts by left multiplication on $\mathcal{L}_{g,0}(H)$. Thanks to the defining relations of $\mathcal{L}_{g,0}(H)$ (Definition \ref{defLgn}), it is clear that any element of $\mathcal{L}_{g,0}(H)$ can be written as a linear combination of some products of the form
\[ \overset{J_1}{A}(1)^{k_1}_{l_1} \ldots \overset{J_g}{A}(g)^{k_g}_{l_g} \, \overset{I_1}{B}(1)^{i_1}_{j_1} \ldots \overset{I_g}{B}(g)^{i_g}_{j_g}. \]
Hence, any element in the vacuum representation is a linear combination of vectors of the form 
\[ 1 \triangleleft \overset{I_1}{B}(1)^{i_1}_{j_1} \ldots \overset{I_g}{B}(g)^{i_g}_{j_g}. \]
Similarly, any element of $\mathcal{L}_{0,g}(H)$ can be written as a linear combination of some products of the form $\overset{I_1}{M}(1)^{i_1}_{j_1} \ldots \overset{I_g}{M}(g)^{i_g}_{j_g}$, and we have an isomorphism of vectors spaces
\[ \flecheIso{ \mathbb{C}_{\mathrm{vac}} \underset{\mathcal{L}^A_{g,0}(H)}{\otimes} \mathcal{L}_{g,0}(H) }{ \mathcal{L}_{0,g}(H) }{ 1 \triangleleft \overset{I_1}{B}(1)^{i_1}_{j_1} \ldots \overset{I_g}{B}(g)^{i_g}_{j_g} }{ \overset{I_1}{M}(1)^{i_1}_{j_1} \ldots \overset{I_g}{M}(g)^{i_g}_{j_g}. } \]
By identification of these two spaces we obtain the (right) vacuum representation of $\mathcal{L}_{g,0}(H)$ on $\mathcal{L}_{0,g}(H)$; we again denote the action by $\triangleleft$.

\begin{remark}
We consider a right representation because we want to relate it to the stacking representation of $\mathbb{C}\mathscr{T}(\Sigma_{g,0}^{\mathrm{o},\bullet})$ on $\mathbb{C}\mathscr{T}(\Sigma_{0,g}^{\mathrm{o},\bullet})$, which is a right representation due to our convention on the stack product (Definition \ref{stackProduct}).
\end{remark}

\begin{lemma}\label{formulesVacuum}
The vacuum representation of $\mathcal{L}_{g,0}(H)$ on $\mathcal{L}_{0,g}(H)$ is explicitly given by 
\begin{center}
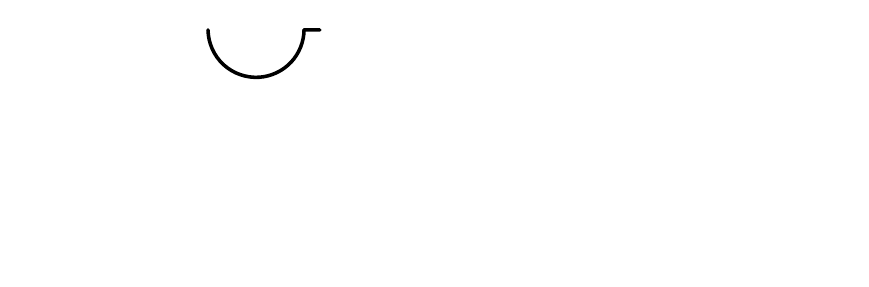
\end{center}
and
\begin{center}
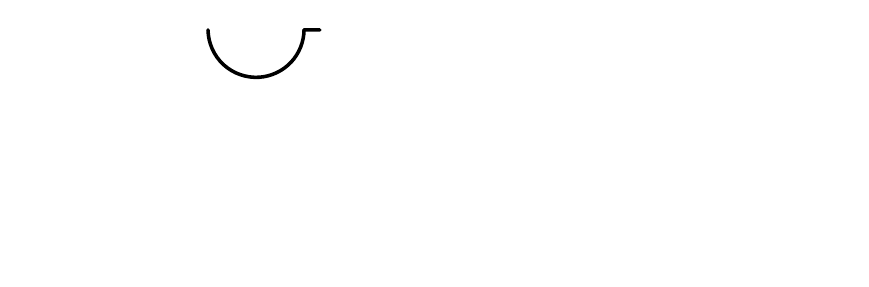
\end{center}
The left-hand side of the first equality means of course $\big(\overset{I_1}{M}(1)^i_k \ldots \overset{I_g}{M}(g)^l_m \triangleleft \overset{J}{B}(j)^n_o\big)\otimes u_i \otimes u^k \otimes \ldots \otimes v_l \otimes v^m \otimes w_n \otimes w^o$, and similarly for the second.
\end{lemma}
\begin{proof}
These are just the outcomes of graphical computations. For instance, below is the computation for the action of $\overset{J}{A}(1)$ in the case $g=2$. Other cases are treated similarly.
\begin{center}
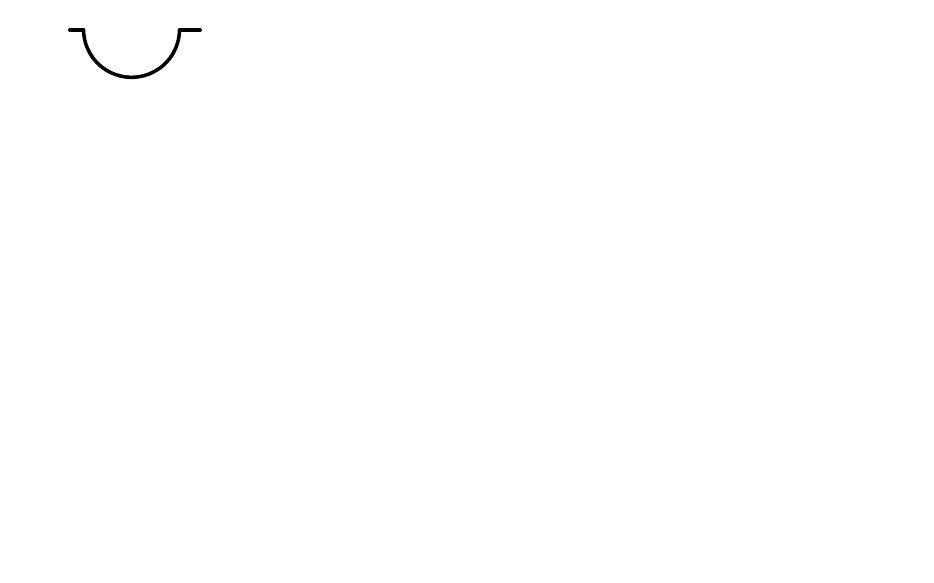
\end{center}
We used \eqref{dessinEchangeLgn}, \eqref{dessinEchangeL10} and the definition of the vacuum representation.
\end{proof}

\begin{remark}\label{remarqueHeisenberg}
There is a left representation of $\mathcal{L}_{g,0}(H)$ on $(H^{\circ})^{\otimes g}$ which comes from the facts that there exists \cite{alekseev} a morphism of algebras $\mathcal{L}_{g,0}(H) \to \mathcal{H}\big( \mathcal{O}(H) \big)^{\otimes g}$ (the Heisenberg double of the dual Hopf algebra $\mathcal{O}(H)$) and that there is a natural representation of $\mathcal{H}\big( \mathcal{O}(H) \big)$ on $H^{\circ}$. When $H$ is finite dimensional and factorizable, the morphism $\mathcal{L}_{g,0}(H) \to \mathcal{H}\big( \mathcal{O}(H) \big)^{\otimes g}$ is an isomorphism \cite[\S 3.3]{Fai18c}, and it is well-known that when $H$ is finite dimensional $\mathcal{H}\big( \mathcal{O}(H) \big) \cong \mathrm{End}_{\mathbb{C}}(H^*)$. Hence under these assumptions, $\mathcal{L}_{g,0}(H)$ has only one indecomposable left representation, namely $(H^*)^{\otimes g}$, which is necessarily dual to the vacuum (right) representation on $\mathcal{L}_{0,g}(H)$. Note however that writing down an explicit intertwiner is not obvious (one has to use certain elements defined in \cite{Fai18c} which implement the action of the mapping class group). I do not know if these two representations are isomorphic when $H$ is infinite dimensional.
\end{remark}

\indent Recall that we denote by $\mathscr{T}(\Sigma_{g,n}^{\mathrm{o},\bullet})$ the set of isotopy classes of $H$-colored $\partial\boldsymbol{\Sigma}$-tangles, with $\boldsymbol{\Sigma} = \Sigma_{g,n}^{\mathrm{o},\bullet} \times [0,1]$, and by $\mathbb{C}\mathscr{T}(\Sigma_{g,n}^{\mathrm{o},\bullet})$ the $\mathbb{C}$-vector space with basis $\mathscr{T}(\Sigma_{g,n}^{\mathrm{o},\bullet})$. Then $\mathbb{C}\mathscr{T}(\Sigma_{g,n}^{\mathrm{o},\bullet})$ is an algebra for the stack product. There is a right representation of $\mathbb{C}\mathscr{T}(\Sigma_{g,0}^{\mathrm{o},\bullet})$ on $\mathbb{C}\mathscr{T}(\Sigma_{0,g}^{\mathrm{o},\bullet})$ obtained by stacking $\Sigma_{g,0}^{\mathrm{o},\bullet} \times [0,1]$ atop $\Sigma_{0,g}^{\mathrm{o},\bullet} \times [0,1]$; we denote it by $\blacktriangleleft$. Note that it is a right representation due to our convention for the stack product (Definition \ref{stackProduct}).

\smallskip

\indent For $x \otimes v \in \mathcal{L}_{0,g}(H) \otimes V$ and $y \otimes w \in \mathcal{L}_{g,0}(H) \otimes W$, where $V,W$ are finite dimensional $H$-modules, we define
\[ (x \otimes v) \triangleleft (y \otimes w) = (x \triangleleft y) \otimes v \otimes w \in \mathcal{L}_{0,g}(H) \otimes V \otimes W. \]
We also denote by $\mathrm{hol}_{0,g}$ and $\mathrm{hol}_{g,0}$ the holonomy maps for $\mathbb{C}\mathscr{T}(\Sigma_{0,g}^{\mathrm{o},\bullet})$ and $\mathbb{C}\mathscr{T}(\Sigma_{g,0}^{\mathrm{o},\bullet})$ respectively (see Definition \ref{defHol}). The following result relates the representations $\triangleleft$ and $\blacktriangleleft$.

\begin{theorem}\label{theoremRepVacuum}
For $\mathbf{S} \in \mathbb{C}\mathscr{T}(\Sigma_{0,g}^{\mathrm{o},\bullet})$ and $\mathbf{T} \in \mathbb{C}\mathscr{T}(\Sigma_{g,0}^{\mathrm{o},\bullet})$ it holds
\[ \mathrm{hol}_{0,g}(\mathbf{S}) \triangleleft \mathrm{hol}_{g,0}(\mathbf{T}) = \mathrm{hol}_{0,g}(\mathbf{S} \blacktriangleleft \mathbf{T}). \]
\end{theorem}
\begin{proof}
The representation $\blacktriangleleft$ is depicted in Figure \ref{vacuumEnPile} for the case $g=2$ (which is completely representative of the general situation). For simplicity, we omitted the orientations and the colorings in this figure; moreover even if we draw only one strand in each handle for $\mathbf{S}$ and $\mathbf{T}$, this represents in fact bunches of parallel strands. To obtain $\mathbf{S} \blacktriangleleft \mathbf{T}$, use the embeddings of the fat graphs in the more intuitive views of the surfaces $\Sigma_{2,0}^{\mathrm{o},\bullet}$ and $\Sigma_{0,2}^{\mathrm{o},\bullet}$ (see Figure \ref{grapheSurSurface}); it is more easy to perform the stacking from this viewpoint. The theorem follows from the comparison between Figure \ref{vacuumEnPile} and the diagrammatic formulas in Lemma \ref{formulesVacuum}. Note that in Figure \ref{vacuumEnPile} we have presented $\mathbf{S} \blacktriangleleft \mathbf{T}$ in such a way that the comparison is immediate; in particular the two crossings just below the coupon $T$ correspond to the crossings which appear when we apply $\mathrm{hol}_{2,0}$ to $\mathbf{T}$ (Definition \ref{defHol}).
\begin{figure}
\centering
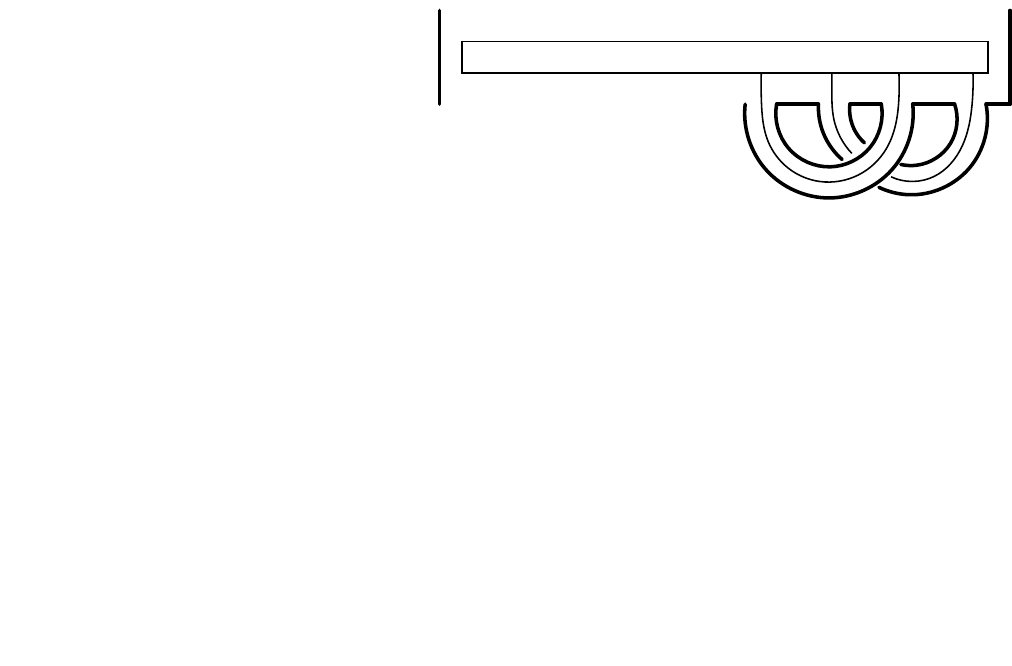
\caption{The stacking representation of $\mathbb{C}\mathscr{T}(\Sigma_{2,0}^{\mathrm{o},\bullet})$ on $\mathbb{C}\mathscr{T}(\Sigma_{0,2}^{\mathrm{o},\bullet})$.}
\label{vacuumEnPile}
\end{figure}
\end{proof}

\begin{corollary}
The subspace $\mathcal{L}_{0,g}^{\mathrm{inv}}(H)$ is stable under the action of $\mathcal{L}_{g,0}^{\mathrm{inv}}(H)$: 
\[ x \in \mathcal{L}_{0,g}^{\mathrm{inv}}(H), \: y \in \mathcal{L}_{g,0}^{\mathrm{inv}}(H) \:\:\: \Rightarrow \:\:\: x \triangleleft y \in \mathcal{L}_{0,g}^{\mathrm{inv}}(H). \]
\end{corollary}
\begin{proof}
Thanks to the fifth property in Proposition \ref{propertiesWf} write  $x = W^{f_1}(L_1)$ (resp.  $y=W^{f_2}(L_2)$), where the based link $L_1$ (resp. $L_2$) has $k$ (resp. $l$) basepoints and $f_1 \in \mathrm{Inv}_k(H)$ (resp. $f_2 \in \mathrm{Inv}_l(H)$). Then
\[ x \triangleleft y = W^{f_1}(L_1) \triangleleft W^{f_2}(L_2) = W^{f_1 \otimes f_2}(L_1 \blacktriangleleft L_2).\]
But it is clear that $f_1 \otimes f_2 \in \mathrm{Inv}_{k+l}(H)$ and thus by the first property in Proposition \ref{propertiesWf}, this element is invariant.
\end{proof}

\indent Recall that $\mathscr{L}(\Sigma_{g,0}^{\mathrm{o}})$ is the set of (isotopy classes of) $\mathrm{SLF}(H)$-colored framed links in $\Sigma_{g,0}^{\mathrm{o}} \times [0,1]$, and similarly for $\mathscr{L}(\Sigma_{g,0})$. The canonical embedding $j \times \mathrm{id} : \Sigma_{g,0}^{\mathrm{o}} \times [0,1] \to \Sigma_{g,0} \times [0,1]$ induces a surjective (non injective) linear map $\pi : \mathbb{C}\mathscr{L}(\Sigma_{g,0}^{\mathrm{o}}) \to \mathbb{C}\mathscr{L}(\Sigma_{g,0})$. Also recall the generalized Wilson loop map $W$ defined in \eqref{defWilsonGen}.

\begin{proposition}\label{propRepSurfaceFermee}
Let $x \in \mathcal{L}_{0,g}^{\mathrm{inv}}(H)$ and $L_1, L_2 \in \mathbb{C}\mathscr{L}(\Sigma_{g,0}^{\mathrm{o}})$. It holds
\[ \pi(L_1) = \pi(L_2) \:\:\: \Rightarrow \:\:\: x \triangleleft W(L_1) = x \triangleleft W(L_2). \]
Hence for $L \in \mathbb{C}\mathscr{L}(\Sigma_{g,0})$, $x \triangleleft W(L)$ is well-defined.
\end{proposition}
\begin{proof}
Consider the $\mathrm{SLF}(H)$-colored links $L, L^{\partial}$ depicted below:
\begin{center}
%% Accompanies image file 'dessinPreuveRecolleDisque.pdf' (pdf, eps, ps)
\begingroup%
  \makeatletter%
  \providecommand\color[2][]{%
    \errmessage{(Inkscape) Color is used for the text in Inkscape, but the package 'color.sty' is not loaded}%
    \renewcommand\color[2][]{}%
  }%
  \providecommand\transparent[1]{%
    \errmessage{(Inkscape) Transparency is used (non-zero) for the text in Inkscape, but the package 'transparent.sty' is not loaded}%
    \renewcommand\transparent[1]{}%
  }%
  \providecommand\rotatebox[2]{#2}%
  \newcommand*\fsize{\dimexpr\f@size pt\relax}%
  \newcommand*\lineheight[1]{\fontsize{\fsize}{#1\fsize}\selectfont}%
  \ifx\svgwidth\undefined%
    \setlength{\unitlength}{324.65341434bp}%
    \ifx\svgscale\undefined%
      \relax%
    \else%
      \setlength{\unitlength}{\unitlength * \real{\svgscale}}%
    \fi%
  \else%
    \setlength{\unitlength}{\svgwidth}%
  \fi%
  \global\let\svgwidth\undefined%
  \global\let\svgscale\undefined%
  \makeatother%
  \begin{picture}(1,0.28339819)%
    \lineheight{1}%
    \setlength\tabcolsep{0pt}%
    \put(0,0){\includegraphics[width=\unitlength,page=1]{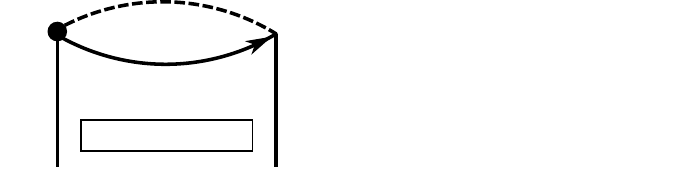}}%
    \put(0.23504747,0.0687546){\color[rgb]{0,0,0}\makebox(0,0)[lt]{\lineheight{1.25}\smash{\begin{tabular}[t]{l}$T$\end{tabular}}}}%
    \put(0,0){\includegraphics[width=\unitlength,page=2]{dessinPreuveRecolleDisque.pdf}}%
    \put(0.82413714,0.06875462){\color[rgb]{0,0,0}\makebox(0,0)[lt]{\lineheight{1.25}\smash{\begin{tabular}[t]{l}$T$\end{tabular}}}}%
    \put(0,0){\includegraphics[width=\unitlength,page=3]{dessinPreuveRecolleDisque.pdf}}%
    \put(-0.00192097,0.10550753){\color[rgb]{0,0,0}\makebox(0,0)[lt]{\lineheight{1.25}\smash{\begin{tabular}[t]{l}$L =$\end{tabular}}}}%
    \put(0.57157816,0.10585536){\color[rgb]{0,0,0}\makebox(0,0)[lt]{\lineheight{1.25}\smash{\begin{tabular}[t]{l}$L^{\partial} =$\end{tabular}}}}%
  \end{picture}%
\endgroup%

\end{center}
The dots mean the remaining of $\Sigma_{g,0}^{\mathrm{o}} \times [0,1]$ (see Figure \ref{grapheSurSurface}), where the links $L, L^{\partial}$ are equal. It is clear that two links $L_1,L_2$ satisfy $\pi(L_1) = \pi(L_2)$ if we can pass from one to another by this transformation. Thanks to the fifth property in Proposition \ref{propertiesWf}, write  $x = W^f(K)$ where the uncolored based link $K$ has $k$ basepoints and $f \in \mathrm{Inv}_k(H)$. We use the Hennings formulation of holonomy (see \S\ref{Sectionhennings}). Let $L_b$ and $L_b^{\partial}$ be uncolored based links such that $[L_b]=L$ and $[L_b^{\partial}]=L^{\partial}$. The diagrams (as in \eqref{defHolHennings}) associated to $K \blacktriangleleft L_b$ and $K \blacktriangleleft L^{\partial}_b$ are related as follows:
\begin{center}
%% Accompanies image file 'diagrammePreuveRecolleDisque.pdf' (pdf, eps, ps)
\begingroup%
  \makeatletter%
  \providecommand\color[2][]{%
    \errmessage{(Inkscape) Color is used for the text in Inkscape, but the package 'color.sty' is not loaded}%
    \renewcommand\color[2][]{}%
  }%
  \providecommand\transparent[1]{%
    \errmessage{(Inkscape) Transparency is used (non-zero) for the text in Inkscape, but the package 'transparent.sty' is not loaded}%
    \renewcommand\transparent[1]{}%
  }%
  \providecommand\rotatebox[2]{#2}%
  \newcommand*\fsize{\dimexpr\f@size pt\relax}%
  \newcommand*\lineheight[1]{\fontsize{\fsize}{#1\fsize}\selectfont}%
  \ifx\svgwidth\undefined%
    \setlength{\unitlength}{487.50001882bp}%
    \ifx\svgscale\undefined%
      \relax%
    \else%
      \setlength{\unitlength}{\unitlength * \real{\svgscale}}%
    \fi%
  \else%
    \setlength{\unitlength}{\svgwidth}%
  \fi%
  \global\let\svgwidth\undefined%
  \global\let\svgscale\undefined%
  \makeatother%
  \begin{picture}(1,0.15043177)%
    \lineheight{1}%
    \setlength\tabcolsep{0pt}%
    \put(0,0){\includegraphics[width=\unitlength,page=1]{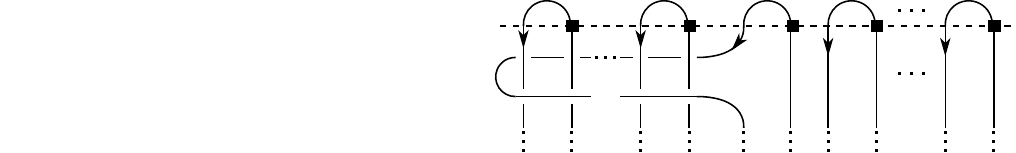}}%
    \put(0.56511418,0.13444421){\color[rgb]{0,0,0}\makebox(0,0)[lt]{\lineheight{1.25}\smash{\begin{tabular}[t]{l}$1$\end{tabular}}}}%
    \put(0.68058277,0.13482881){\color[rgb]{0,0,0}\makebox(0,0)[lt]{\lineheight{1.25}\smash{\begin{tabular}[t]{l}$k$\end{tabular}}}}%
    \put(0.78242964,0.13529543){\color[rgb]{0,0,0}\makebox(0,0)[lt]{\lineheight{1.25}\smash{\begin{tabular}[t]{l}$1'$\end{tabular}}}}%
    \put(0.86496192,0.13467262){\color[rgb]{0,0,0}\makebox(0,0)[lt]{\lineheight{1.25}\smash{\begin{tabular}[t]{l}$2'$\end{tabular}}}}%
    \put(0.98027227,0.13638289){\color[rgb]{0,0,0}\makebox(0,0)[lt]{\lineheight{1.25}\smash{\begin{tabular}[t]{l}$l'$\end{tabular}}}}%
    \put(0,0){\includegraphics[width=\unitlength,page=2]{diagrammePreuveRecolleDisque.pdf}}%
    \put(0.07265425,0.13467261){\color[rgb]{0,0,0}\makebox(0,0)[lt]{\lineheight{1.25}\smash{\begin{tabular}[t]{l}$1$\end{tabular}}}}%
    \put(0.1879646,0.13638289){\color[rgb]{0,0,0}\makebox(0,0)[lt]{\lineheight{1.25}\smash{\begin{tabular}[t]{l}$k$\end{tabular}}}}%
    \put(0,0){\includegraphics[width=\unitlength,page=3]{diagrammePreuveRecolleDisque.pdf}}%
    \put(0.28034655,0.13467261){\color[rgb]{0,0,0}\makebox(0,0)[lt]{\lineheight{1.25}\smash{\begin{tabular}[t]{l}$1'$\end{tabular}}}}%
    \put(0.39565687,0.13638289){\color[rgb]{0,0,0}\makebox(0,0)[lt]{\lineheight{1.25}\smash{\begin{tabular}[t]{l}$l'$\end{tabular}}}}%
  \end{picture}%
\endgroup%

\end{center}
This represents a neighborhood of the dotted line, and the remaining of the diagrams are equal. We number the basepoints of $L_b, L_b^{\partial}$ with primes to distinguish them from the basepoints of $K$. To avoid cumbersome computations, take $k=2$ (this is completely representative of the general situation). Write $\mathrm{Hen}(K \blacktriangleleft L_b) =Z_0 \otimes Z_1 \otimes Z_2 \otimes Z'_{1} \otimes \ldots \otimes Z'_{l} \in \mathcal{L}_{g,n}(H) \otimes H^{\otimes (2+l)}$, with implicit summation. Let $a_{i(1)} \otimes b_{i(1)}, \ldots, a_{i(8)} \otimes b_{i(8)}$ be $8 = 4k$ copies of $R = a_i \otimes b_i$ (one for each crossing in the diagram above), then
\begin{align*}
\mathrm{Hen}(K \blacktriangleleft L^{\partial}_b) =&Z_0 \otimes gS(a_{i(4)})b_{i(5)}g^{-1}Z_1b_{i(6)}S^2(a_{i(3)}) \otimes gS(a_{i(2)}) b_{i(7)}g^{-1}Z_2b_{i(8)}S^2(a_{i(1)}) \\
&\otimes gb_{i(1)}b_{i(2)}b_{i(3)}b_{i(4)}S(a_{i(5)})a_{i(6)}S(a_{i(7)})a_{i(8)}g^{-1}Z'_{1} \otimes Z_{2}' \otimes \ldots \otimes Z'_{l}\\
=&Z_0 \otimes S^2\big(a_{i(4)} b_{i(5)}\big)Z_1S\big(a_{i(3)}b_{i(6)}\big) \otimes S^2\big(a_{i(2)} b_{i(7)})Z_2 S\big(a_{i(1)}b_{i(8)}\big) \\
&\otimes gS^{-1}\big(b_{i(4)}b_{i(3)}b_{i(2)}b_{i(1)}\big)S\big(a_{i(8)} a_{i(7)} a_{i(6)} a_{i(5)}\big) g^{-1}Z'_{1} \otimes Z_{2}' \otimes \ldots \otimes Z'_{l}\\
=&Z_0 \otimes S^2\big(a_i^{(1)} b_j^{(1)}\big)Z_1S\big(a_i^{(2)}b_j^{(2)}\big) \otimes S^2\big(a_i^{(3)} b_j^{(3)})Z_2 S\big(a_i^{(4)}b_j^{(4)}\big) \\
&\otimes gS^{-1}(b_i)S(a_j) g^{-1}Z'_{1} \otimes Z_{2}' \otimes \ldots \otimes Z'_{l} 
\end{align*}
For the second equality we used the properties of $S$, the formula $(S\otimes S)(R)=R$ and \eqref{pivot}, and for the third equality we used an iteration of the formulas $(\Delta \otimes \mathrm{id})(R)=R_{13}R_{23}$, $(\mathrm{id} \otimes \Delta)(R)=R_{13}R_{12}$. Hence, denoting by $\varphi_1, \ldots, \varphi_l$ the colors of the components of $L, L^{\partial}$, we have:
\begin{align*}
x \triangleleft W(L^{\partial}) =& W^f(K) \triangleleft W^{\varphi_1 \otimes \ldots \otimes \varphi_l}(L^{\partial}_b)  = W^{f \otimes \varphi_1 \otimes \ldots \otimes \varphi_l}(K \blacktriangleleft L^{\partial}_b)\\
=& Z_0 \, f\bigg(S^2\big(a_i^{(1)} b_j^{(1)}\big)Z_1S\big(a_i^{(2)}b_j^{(2)}\big) \otimes S^2\big(a_i^{(3)} b_j^{(3)})Z_2 S\big(a_i^{(4)}b_j^{(4)}\big)\bigg) \\
& \varphi_1\big(gS^{-1}(b_i)S(a_j) g^{-1}Z'_{1}\big) \varphi_2(Z_{2}') \ldots \varphi_l(Z'_{l})\\
=& Z_0 \, f\big(Z_1 \otimes Z_2\big) \varphi_1(Z'_{1}) \ldots \varphi_l(Z'_{l})\\
=& W^{f \otimes \varphi_1 \otimes \ldots \otimes \varphi_l}(K \blacktriangleleft L_b) = W^f(K) \triangleleft W^{\varphi_1 \otimes \ldots \otimes \varphi_l}(L_b) = x \triangleleft W(L).
\end{align*}
For the second and sixth equalities we used Theorem \ref{theoremRepVacuum} and for the fourth equality we used that $f \in \mathrm{Inv}_k(H)$ and that $(\varepsilon \otimes \mathrm{id})(R) = (\mathrm{id} \otimes \varepsilon)(R) = 1$.
\end{proof}
\noindent This proposition shows that we have a representation of $\mathbb{C}\mathscr{L}(\Sigma_{g,0})$ (and not just of $\mathbb{C}\mathscr{L}(\Sigma_{g,0} \!\setminus\! D)$) on $\mathcal{L}_{0,g}^{\mathrm{inv}}(H)$. Hence in some sense, the disc $D$ is glued back \textit{via} this process and we can deal with the closed surface $\Sigma_{g,0}$.

\smallskip

\indent Let us specialize these results to $H = U_{q^2} = U_{q^2}(\mathfrak{sl}_2)$, where they descend to the skein algebra (quotient of $\mathbb{C}\mathscr{L}(\Sigma_{g,0})$ by skein relations). As already said (before Corollary \ref{coroInvariantsStated}), in that case we can without loss of generality color all the link components by $\chi^+_2$ and we have two isomorphisms 
\begin{equation}\label{identificationInvariantsSkein}
 W_{g,0} : \mathcal{S}_q(\Sigma_{g,0}^{\mathrm{o}}) \overset{\sim}{\to} \mathcal{L}_{g,0}^{\mathrm{inv}}(U_{q^2}), \:\:\:\:\:W_{0,g} : \mathcal{S}_q(\Sigma_{0,g}^{\mathrm{o}}) \overset{\sim}{\to} \mathcal{L}_{0,g}^{\mathrm{inv}}(U_{q^2}). 
\end{equation}
Moreover, there is a natural representation of $\mathcal{S}_q(\Sigma_{g,0})$ on $\mathcal{S}_q(\Sigma_{0,g}^{\mathrm{o}}) = \mathcal{S}_q(\Sigma_{0,g+1})$ (skein module of a handlebody) obtained by stacking, we still denote it by $\blacktriangleleft$.

\begin{corollary}\label{corollaireRepSkein}
Under the identifications \eqref{identificationInvariantsSkein}, the representation $\triangleleft$ of $\mathcal{L}_{g,0}^{\mathrm{inv}}(U_{q^2})$ on $\mathcal{L}_{0,g}^{\mathrm{inv}}(U_{q^2})$ is the stacking representation $\blacktriangleleft$ of $\mathcal{S}_q(\Sigma_{g,0})$ on $\mathcal{S}_q(\Sigma_{0,g}^{\mathrm{o}})$. In other words:
\[ W_{0,g}(L_1) \triangleleft W_{g,0}(L_2) = W_{0,g}(L_1 \blacktriangleleft L_2). \]
\end{corollary}

\section{A remark on roots of unity}\label{sectionRepSkeinAlg}
\indent Corollary \ref{corollaireRepSkein} is a bit disappointing because it tells us that the representation of $\mathcal{S}_q(\Sigma_{g,0})$ produced by combinatorial quantization is just equivalent to the obvious representation on the skein module of a handlebody. This changes dramatically if we replace the formal variable $q$ by a root of unity $\epsilon$ and we will now describe a construction producing much less obvious representations of $\mathcal{S}_{\epsilon}(\Sigma_{g,0})$, which are moreover finite dimensional.

\smallskip

\indent Let $\epsilon = e^{i\pi/2p}$ be a $4p$-th root of unity ($p \geq 2$), and let $U_{\epsilon^2} $ be the specialization of $U_{q^2}(\mathfrak{sl}_2)$ to $\epsilon^2 = e^{i\pi/p}$. Then we can consider the restricted quantum group of $\mathfrak{sl}_2$, denoted by $\overline{U}_{\! \epsilon^2} = \overline{U}_{\! \epsilon^2}(\mathfrak{sl}_2)$, which is the quotient of $U_{\epsilon^2}$ by
\[ E^{p} = F^{p} = 0, \:\:\:\:\: K^{2p}=1. \]

\indent The algebra $\mathcal{L}_{g,n}(U_{\epsilon^2})$ does not exist since $U_{\epsilon^2}$ is not braided. Nevertheless, we can consider the specialization $\mathcal{L}_{g,n}(U_{q^2})_{\epsilon^2}$ of $\mathcal{L}_{g,n}(U_{q^2})$ at $q^2 = \epsilon^2$ (this specialization is studied in great detail in \cite{BaR} for $g=0$). Moreover, it is possible to define $\mathcal{L}_{g,n}(\overline{U}_{\!\epsilon^2})$ (see \cite[\S 6]{Fai18} for full details), and it turns out it is the quotient of $\mathcal{L}_{g,n}(U_{q^2})_{\epsilon^2}$ by
\[ (X(i)^-_+)^{p} = (X(i)^+_-)^{p} = 0, \:\:\:\:\: (X(i)^+_+)^{2p} = 1 \]
where as usual $X(i)$ is $A(i)$ or $B(i)$ if $1 \leq i \leq g$ and is $M(i)$ if $g+1 \leq i \leq g+n$ (recall Lemma \ref{GensRelsLgnUq2}, from which we take back the notations). This together with the relations of Lemma \ref{GensRelsLgnUq2} gives a presentation by generators and relations of $\mathcal{L}_{g,n}(\overline{U}_{\!\epsilon^2})$. The dimension of $\mathcal{L}_{g,n}(\overline{U}_{\!\epsilon^2})$ is $(2p^3)^{2g+n}$.

\smallskip

\indent Let $W_{\epsilon} : \mathcal{S}_{\epsilon}(\Sigma_{g,n}^{\mathrm{o}}) \to \mathcal{L}_{g,n}^{\mathrm{inv}}(U_{q^2})_{\epsilon^2}$ be the specialization of the Wilson loop map (we recall that all the strands in $\mathcal{S}_{\epsilon}(\Sigma_{g,n}^{\mathrm{o}})$ are implicitly colored by the character $\chi^+_2$ of the fundamental representation $V_2$) and let $\overline{W}$ be the Wilson loop map for $H = \overline{U}_{\!\epsilon^2}$ (with source the $\mathrm{SLF}(\overline{U}_{\! \epsilon^2})$-colored links and with values in $\mathcal{L}_{g,n}^{\mathrm{inv}}(\overline{U}_{\!\epsilon^2})$, according to Definition \ref{defWilsonLoopMap}); note that we choose $g = K^{p+1}$ for the pivotal element in $\overline{U}_{\!\epsilon^2}$. The representation $V_2$ descends to $\overline{U}_{\!\epsilon^2}$, so that $\chi^+_2 \in \mathrm{SLF}(\overline{U}_{\! \epsilon^2})$.  If we restrict $\overline{W}$ to $\mathcal{S}_{\epsilon}(\Sigma_{g,n}^{\mathrm{o}})$ (\textit{i.e.} we restrict $\overline{W}$ to the links colored by $\chi^+_2$), then $\overline{W}$ factors through $W_{\epsilon}$:
\[
\xymatrix{
\mathcal{S}_{\epsilon}(\Sigma_{g,n}^{\mathrm{o}})  \ar[r]^{\overline{W}} \ar[d]_{W_{\epsilon}} & \mathcal{L}_{g,n}^{\mathrm{inv}}(\overline{U}_{\!\epsilon^2}) \\  
\mathcal{L}_{g,n}^{\mathrm{inv}}(U_{q^2})_{\epsilon^2} \ar[ru]_{\mathrm{pr}} &
}
\]
where $\mathrm{pr}$ is the canonical projection. The restriction of $\mathrm{pr}$ to $\mathcal{L}_{g,n}^{\mathrm{inv}}(U_{q^2})_{\epsilon^2}$ (specialization of $\mathcal{L}_{g,n}^{\mathrm{inv}}(U_{q^2})$ at ${\epsilon^2}$) takes values in $\mathcal{L}_{g,n}^{\mathrm{inv}}(\overline{U}_{\!\epsilon^2})$ (but note that this restriction is not surjective). By the results of the previous section, we obtain an action of $L \in \mathcal{S}_{\epsilon}(\Sigma_{g,0})$ on $x \in \mathcal{L}_{0,g}^{\mathrm{inv}}(\overline{U}_{\!\epsilon^2})$ by
\[ x \triangleleft \overline{W}(L) \]
and this produces a finite dimensional right representation of $\mathcal{S}_{\epsilon}(\Sigma_{g,0})$.

\smallskip

\indent In the sequel we describe explicitly this representation in the case of the torus $\Sigma_{1,0}$. Consider the following curves in $\Sigma_{0,1}^{\mathrm{o}} \times \{0\} \subset \Sigma_{0,1}^{\mathrm{o}} \times [0,1]$ and in $\Sigma_{1,0} \times \{0\} \subset \Sigma_{1,0} \times [0,1]$ respectively:
\begin{center}
%% Accompanies image file 'anneauEtTore.pdf' (pdf, eps, ps)
\begingroup%
  \makeatletter%
  \providecommand\color[2][]{%
    \errmessage{(Inkscape) Color is used for the text in Inkscape, but the package 'color.sty' is not loaded}%
    \renewcommand\color[2][]{}%
  }%
  \providecommand\transparent[1]{%
    \errmessage{(Inkscape) Transparency is used (non-zero) for the text in Inkscape, but the package 'transparent.sty' is not loaded}%
    \renewcommand\transparent[1]{}%
  }%
  \providecommand\rotatebox[2]{#2}%
  \newcommand*\fsize{\dimexpr\f@size pt\relax}%
  \newcommand*\lineheight[1]{\fontsize{\fsize}{#1\fsize}\selectfont}%
  \ifx\svgwidth\undefined%
    \setlength{\unitlength}{229.77916697bp}%
    \ifx\svgscale\undefined%
      \relax%
    \else%
      \setlength{\unitlength}{\unitlength * \real{\svgscale}}%
    \fi%
  \else%
    \setlength{\unitlength}{\svgwidth}%
  \fi%
  \global\let\svgwidth\undefined%
  \global\let\svgscale\undefined%
  \makeatother%
  \begin{picture}(1,0.29993334)%
    \lineheight{1}%
    \setlength\tabcolsep{0pt}%
    \put(0,0){\includegraphics[width=\unitlength,page=1]{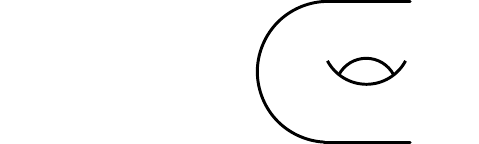}}%
    \put(0.88523485,0.22751224){\color[rgb]{0,0,0}\makebox(0,0)[lt]{\lineheight{1.25}\smash{\begin{tabular}[t]{l}$b$\end{tabular}}}}%
    \put(0,0){\includegraphics[width=\unitlength,page=2]{anneauEtTore.pdf}}%
    \put(0.70521325,0.08458243){\color[rgb]{0,0,0}\makebox(0,0)[lt]{\lineheight{1.25}\smash{\begin{tabular}[t]{l}$a$\end{tabular}}}}%
    \put(0,0){\includegraphics[width=\unitlength,page=3]{anneauEtTore.pdf}}%
    \put(0.13376461,0.24841372){\color[rgb]{0,0,0}\makebox(0,0)[lt]{\lineheight{1.25}\smash{\begin{tabular}[t]{l}$m$\end{tabular}}}}%
    \put(0,0){\includegraphics[width=\unitlength,page=4]{anneauEtTore.pdf}}%
  \end{picture}%
\endgroup%

\end{center}
\noindent Since $\overline{U}_{\! \epsilon^2}$ is factorizable, we know by \cite[Lem. 3.9]{Fai18} that we have an isomorphism of algebras
\begin{equation}\label{isoInvSLF}
\mathrm{SLF}(\overline{U}_{\! \epsilon^2}) \overset{\sim}{\to} \mathcal{L}_{0,1}^{\mathrm{inv}}(\overline{U}_{\! \epsilon^2}), \:\:\:\:\:\varphi \mapsto \overline{W}^{\varphi}(m).
\end{equation}
\noindent where $\overline{W}^{\varphi}(m)$ is the value of $\overline{W}$ on $m$ colored by $\varphi$. Recall that $\mathrm{SLF}(\overline{U}_{\! \epsilon^2})$ is a commutative algebra, for the usual product of linear forms $\varphi\psi = (\varphi \otimes \psi) \circ \Delta$ (this is true for any finite dimensional factorizable ribbon Hopf algebra $H$; actually for such $H$ one has $\mathrm{SLF}(H) \cong \mathcal{Z}(H) \cong \mathcal{L}_{0,1}^{\mathrm{inv}}(H)$).

\smallskip

\indent According to \cite{BP}, $\mathcal{S}_{\epsilon}(\Sigma_{1,0})$ is generated by the links $a$ and $b$. Hence, to describe the representation it suffices to compute the action of $\overline{W}(a) = \overline{W}^{\chi^+_2}(a)$ and $\overline{W}(b) = \overline{W}^{\chi^+_2}(b)$ on $\overline{W}^{\varphi}(m)$. Recall that the Casimir element is $\displaystyle C = FE + \frac{\epsilon^2 K + \epsilon^{-2}K^{-1}}{(\epsilon^2-\epsilon^{-2})^2} \in \mathcal{Z}(\overline{U}_{\!\epsilon^2})$.
\begin{lemma}
For all $\varphi \in \mathrm{SLF}(H)$, it holds
\[ \overline{W}^{\varphi}(m) \triangleleft \overline{W}(a) = \overline{W}^{\varphi(?c)}(m), \:\:\:\:\:\:\: \overline{W}^{\varphi}(m) \triangleleft \overline{W}(b) =  \overline{W}^{\varphi \chi^+_2}(m) \]
where $c = (\epsilon^2-\epsilon^{-2})^2C$, $\varphi(?c)$ is defined by $h \mapsto \varphi(hc)$ and $\varphi \chi^+_2$ is the product of the linear forms $\varphi$ and $\chi^+_2$.
\end{lemma}
\begin{proof}
Thanks to Lemma \ref{formulesVacuum}, we have
\begin{center}
%% Accompanies image file 'actionWA.pdf' (pdf, eps, ps)
\begingroup%
  \makeatletter%
  \providecommand\color[2][]{%
    \errmessage{(Inkscape) Color is used for the text in Inkscape, but the package 'color.sty' is not loaded}%
    \renewcommand\color[2][]{}%
  }%
  \providecommand\transparent[1]{%
    \errmessage{(Inkscape) Transparency is used (non-zero) for the text in Inkscape, but the package 'transparent.sty' is not loaded}%
    \renewcommand\transparent[1]{}%
  }%
  \providecommand\rotatebox[2]{#2}%
  \newcommand*\fsize{\dimexpr\f@size pt\relax}%
  \newcommand*\lineheight[1]{\fontsize{\fsize}{#1\fsize}\selectfont}%
  \ifx\svgwidth\undefined%
    \setlength{\unitlength}{395.52448061bp}%
    \ifx\svgscale\undefined%
      \relax%
    \else%
      \setlength{\unitlength}{\unitlength * \real{\svgscale}}%
    \fi%
  \else%
    \setlength{\unitlength}{\svgwidth}%
  \fi%
  \global\let\svgwidth\undefined%
  \global\let\svgscale\undefined%
  \makeatother%
  \begin{picture}(1,0.18069025)%
    \lineheight{1}%
    \setlength\tabcolsep{0pt}%
    \put(0,0){\includegraphics[width=\unitlength,page=1]{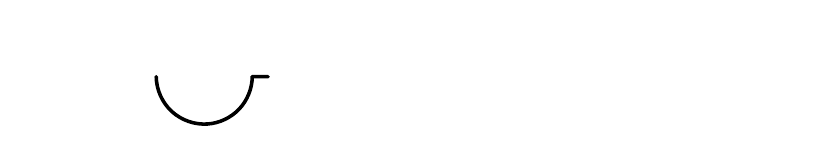}}%
    \put(0.28918858,0.01046593){\color[rgb]{0,0,0}\makebox(0,0)[lt]{\lineheight{1.25}\smash{\begin{tabular}[t]{l}$\overset{I}{M}$\end{tabular}}}}%
    \put(0,0){\includegraphics[width=\unitlength,page=2]{actionWA.pdf}}%
    \put(0.33518498,0.07960646){\color[rgb]{0,0,0}\makebox(0,0)[lt]{\lineheight{1.25}\smash{\begin{tabular}[t]{l}$\triangleleft$\end{tabular}}}}%
    \put(0.52590949,0.07949915){\color[rgb]{0,0,0}\makebox(0,0)[lt]{\lineheight{1.25}\smash{\begin{tabular}[t]{l}$=$\end{tabular}}}}%
    \put(0,0){\includegraphics[width=\unitlength,page=3]{actionWA.pdf}}%
    \put(0.6739328,0.00276896){\color[rgb]{0,0,0}\makebox(0,0)[lt]{\lineheight{1.25}\smash{\begin{tabular}[t]{l}$\overset{I}{M}$\end{tabular}}}}%
    \put(0,0){\includegraphics[width=\unitlength,page=4]{actionWA.pdf}}%
    \put(0.49305566,0.01285166){\color[rgb]{0,0,0}\makebox(0,0)[lt]{\lineheight{1.25}\smash{\begin{tabular}[t]{l}$\overset{V_2}{A}$\end{tabular}}}}%
    \put(0,0){\includegraphics[width=\unitlength,page=5]{actionWA.pdf}}%
    \put(0.18902978,0.10540729){\color[rgb]{0,0,0}\makebox(0,0)[lt]{\lineheight{1.25}\smash{\begin{tabular}[t]{l}$I$\end{tabular}}}}%
    \put(0.37526737,0.11075699){\color[rgb]{0,0,0}\makebox(0,0)[lt]{\lineheight{1.25}\smash{\begin{tabular}[t]{l}$V_2$\end{tabular}}}}%
    \put(0,0){\includegraphics[width=\unitlength,page=6]{actionWA.pdf}}%
    \put(0.56492479,0.15828071){\color[rgb]{0,0,0}\makebox(0,0)[lt]{\lineheight{1.25}\smash{\begin{tabular}[t]{l}$I$\end{tabular}}}}%
    \put(0.7025924,0.12251823){\color[rgb]{0,0,0}\makebox(0,0)[lt]{\lineheight{1.25}\smash{\begin{tabular}[t]{l}$V_2$\end{tabular}}}}%
    \put(0,0){\includegraphics[width=\unitlength,page=7]{actionWA.pdf}}%
    \put(-0.00054728,0.08171326){\color[rgb]{0,0,0}\makebox(0,0)[lt]{\lineheight{1.25}\smash{\begin{tabular}[t]{l}$\overset{I}{M} \triangleleft \overline{W}(a) =$\end{tabular}}}}%
    \put(0.72710416,0.08233531){\color[rgb]{0,0,0}\makebox(0,0)[lt]{\lineheight{1.25}\smash{\begin{tabular}[t]{l}$= \overset{I}{M} \, (\chi^+_2 \otimes \overset{I}{T})\big((g \otimes 1)RR'\big)$\end{tabular}}}}%
  \end{picture}%
\endgroup%

\end{center}
where the last equality is obtained by a straightforward computation. Using the formula for $R$ \cite[\S 4]{FGST}, we get that $(\chi^+_2 \otimes \mathrm{id})\big( (g \otimes 1) RR' \big) = c$, and hence $\overset{I}{M} \triangleleft \overline{W}(a) = \overset{I}{M}\overset{I}{c}$. Thus if we let $\varphi = \sum_I \mathrm{tr}\big( \Lambda_I \overset{I}{T} \big)$,
\[ \overline{W}^{\varphi}(m) \triangleleft \overline{W}(a) = \sum_I \mathrm{tr}\big( \Lambda_I \overset{I}{g} \overset{I}{M} \triangleleft \overline{W}(a) \big) = \sum_I \mathrm{tr}\big( \Lambda_I \overset{I}{g} \overset{I}{M}\overset{I}{c} \big) = \overline{W}^{\varphi(?c)}(m). \]
For $b$, it is clear by definition of $\triangleleft$ that $\overset{I}{M} \triangleleft \overline{W}(b) = \overset{I}{M} \overline{W}^{\chi^+_2\!}(m)$, which implies $\overline{W}^{\varphi}(m) \triangleleft \overline{W}(b) = \overline{W}^{\varphi}(m) \overline{W}^{\chi^+_2\!}(m)$. But due to \eqref{isoInvSLF} it holds $\overline{W}^{\varphi}(m) \overline{W}^{\chi^+_2\!}(m) = \overline{W}^{\varphi \chi^+_2\!}(m)$.
\end{proof}

\indent To simplify notations, we identify $\mathrm{SLF}(\overline{U}_{\! \epsilon^2})$ and $\mathcal{L}_{0,1}^{\mathrm{inv}}(\overline{U}_{\! \epsilon^2})$ through the isomorphism \eqref{isoInvSLF}, so that we obtain a representation of $\mathcal{S}_{\epsilon}(\Sigma_{1,0})$ on $\mathrm{SLF}(\overline{U}_{\! \epsilon^2})$. Recall that $\dim\big(\mathrm{SLF}(\overline{U}_{\! \epsilon^2})\big) = 3p-1$, and we have the GTA basis (\cite{GT}, \cite{arike}, here we use notations from \cite{F}):
\[ \chi^+_1, \: \chi^-_1, \: \ldots, \: \chi^+_p, \: \chi^-_p, \: G_1, \: \ldots, \: G_{p-1}. \]
The $\chi^{\pm}_s$ are the characters of the simple modules $\mathcal{X}^{\pm}(s)$ (in particular $\mathcal{X}^+(2) = V_2$) and the $G_s$ are certain linear combinations of matrix coefficients of projective modules. It is easy to compute the action of $a$ ($\varphi \triangleleft \overline{W}(a) = \varphi(?c)$) on these elements:
\[ \chi^{\alpha}_s \triangleleft \overline{W}(a) = -\alpha (\epsilon^{2s} + \epsilon^{-2s})\chi^{\alpha}_s, \:\:\:\:\:\: G_s \triangleleft \overline{W}(a) = -(\epsilon^{2s} + \epsilon^{-2s})G_s - (\epsilon^2 - \epsilon^{-2})^2 (\chi^+_s + \chi^-_{p-s})  \]
where $\alpha \in \{\pm\}$. The action of $b$ ($\varphi \triangleleft \overline{W}(b) = \varphi \chi^+_2$) is given by the multiplication rules in the GTA basis (\cite{GT}, also see \cite{F} for an elementary proof):
\[ \begin{array}{lll}
\displaystyle \chi^{\alpha}_1 \triangleleft \overline{W}(b) = \chi^{\alpha}_2, & \displaystyle \chi^{\alpha}_s \triangleleft \overline{W}(b) = \chi^{\alpha}_{s-1} + \chi^{\alpha}_{s+1}, & \displaystyle \chi^{\alpha}_p \triangleleft \overline{W}(b) = 2\chi^{\alpha}_{p-1} + 2\chi^{-\alpha}_{1},\\[5pt]
\displaystyle G_1 \triangleleft \overline{W}(b) = [2]G_2, & \displaystyle G_s \triangleleft \overline{W}(b) = \frac{[s-1]}{[s]}G_{s-1} + \frac{[s+1]}{[s]}G_{s+1}, & \displaystyle G_{p-1} \triangleleft \overline{W}(b) = [2]G_{p-2}.
\end{array}\]
where $\displaystyle [k] = \frac{\epsilon^{2k} - \epsilon^{-2k}}{\epsilon^2 - \epsilon^{-2}}$.

\smallskip

Let 
\[ \mathcal{V}_1 = \mathrm{vect}\big( \chi^+_s + \chi^-_{p-s}, \chi^+_p, \chi^-_p \big)_{1 \leq s \leq p-1}, \:\:\:\: \mathcal{V}_2 = \mathrm{vect}\big( \chi^+_s \big)_{1 \leq s \leq p-1}, \:\:\:\: \mathcal{V}_3 = \mathrm{vect}\big( G_s \big)_{1 \leq s \leq p-1} \]
so that $\mathrm{SLF}(\overline{U}_{\! \epsilon^2}) = \mathcal{V}_1 \oplus \mathcal{V}_2 \oplus \mathcal{V}_3$. Note that $\mathcal{V}_1$ is the subspace generated by the characters of the projective modules. From the formulas above, we see that the structure of $\mathrm{SLF}(\overline{U}_{\! \epsilon^2})$ under the action of $\mathcal{S}_{\epsilon}(\Sigma_{1,0})$ has the following shape:
\[
\xymatrix{
\mathcal{V}_2  \ar[rd]_{\overline{W}(b)} & &  \mathcal{V}_3 \ar[ld]^{\overline{W}(a), \, \overline{W}(b)}\\  
 &\mathcal{V}_1 &
}
\]
By this diagram we mean that $\mathcal{V}_1$ is a submodule, that $\mathcal{V}_2 \triangleleft W(a) \subset \mathcal{V}_2$, $\mathcal{V}_2 \triangleleft W(b) \subset \mathcal{V}_1 + \mathcal{V}_2$ \textit{etc}. Let
\[ J_1 = \mathcal{V}_1, \:\:\:\:\: J_2 = \mathcal{V}_1 + \mathcal{V}_2, \:\:\:\:\: J_3 = \mathcal{V}_1 + \mathcal{V}_2 + \mathcal{V}_3. \]

\begin{proposition}\label{structureRepSkein}
$J_1 \subset J_2 \subset J_3$ is a composition series for $\mathrm{SLF}(\overline{U}_{\! \epsilon^2})$ under the action of $\mathcal{S}_{\epsilon}(\Sigma_{1,0})$. Moreover, the composition factor $J_2/J_1$ is isomorphic to $\mathcal{S}_{\epsilon}^{\mathrm{red}}(\Sigma_{0,1}^{\mathrm{o}})$, the reduced skein module of the annulus (which is a $\mathcal{S}_{\epsilon}(\Sigma_{1,0})$-module by stacking and reducing).
\end{proposition}
\begin{proof}
We must show that $J_1, J_2/J_1, J_3/J_2$ are irreducible. Recall that the center of $\overline{U}_{\! \epsilon^2}$ contains $p+1$ primitive orthogonal idempotents $e_0, \ldots, e_p$, associated to the blocks of $\overline{U}_{\! \epsilon^2}$ \cite[Prop. 4.4.4]{FGST}. They can be expressed as polynomials of $c$ \cite[Prop. D.1.1]{FGST}: $e_0 = P_0(c), \ldots, e_p = P_p(c)$. Moreover, it is easy to see that
\[ \chi^+_s(?e_t) = \delta_{s,t}\chi^+_s, \:\:\:\:\: \chi^-_{s}(?e_t) = \delta_{p-s,t}\chi^-_{s}, \:\:\:\:\: G_s(?e_t) = \delta_{s,t}G_s. \]
Let us show that $J_1$ is irreducible. Let $0 \neq S \subset J_1$ be a submodule, and let $0 \neq \varphi = \lambda_0 \chi^-_p + \sum_s \lambda_s(\chi^+_s + \chi^-_{p-s}) + \lambda_p \chi^+_p \in S$. At least one of the coefficients, say $\lambda_s$, is non-zero. Then 
\[ \varphi \triangleleft P_s(\overline{W}(a)) = \varphi(?P_s(c)) = \varphi(?e_s) = \lambda_s(\chi^+_s + \chi^-_{p-s})\]
and thus $\chi^+_s + \chi^-_{p-s} \in S$. Assume for instance $s=1$. We have $(\chi^+_1 + \chi^-_{p-1}) \triangleleft \overline{W}(b) = (\chi^+_2 + \chi^-_{p-2}) + \chi^-_p$, and 
\begin{align*}
\big((\chi^+_2 + \chi^-_{p-2}) + \chi^-_p\big) \triangleleft P_2(\overline{W}(a)) &= \big((\chi^+_2 + \chi^-_{p-2}) + \chi^-_p\big)(?e_2) = \chi^+_2 + \chi^-_{p-2},\\
\big((\chi^+_2 + \chi^-_{p-2}) + \chi^-_p\big) \triangleleft P_0(\overline{W}(a)) &= \big((\chi^+_2 + \chi^-_{p-2}) + \chi^-_p\big)(?e_0) = \chi^-_p
\end{align*}
so that $\chi^+_2 + \chi^-_{p-2}, \chi^-_p \in S$. Continuing like this, one shows step by step that all the basis vectors are in $S$, and thus $S=J_1$ as desired. The proofs for $J_2/J_1$ and $J_3/J_2$ are similar.
\\For the last claim, recall that $\mathcal{S}_{\epsilon}^{\mathrm{red}}(\Sigma_{0,1}^{\mathrm{o}})$ is the quotient (in some sense) of $\mathcal{S}_{\epsilon}(\Sigma_{0,1}^{\mathrm{o}})$ by the $(p\!-\!1)$-th Jones-Wenzl idempotent (see \cite[\S 6.5]{costantino} for a survey) and that it is generated as a vector space by the ``closures'' $\mathrm{cl}(f_n)$ of the Jones-Wenzl idempotents $f_n$ for $0 \leq n \leq p-2$:
\begin{center}
%% Accompanies image file 'closure.pdf' (pdf, eps, ps)
\begingroup%
  \makeatletter%
  \providecommand\color[2][]{%
    \errmessage{(Inkscape) Color is used for the text in Inkscape, but the package 'color.sty' is not loaded}%
    \renewcommand\color[2][]{}%
  }%
  \providecommand\transparent[1]{%
    \errmessage{(Inkscape) Transparency is used (non-zero) for the text in Inkscape, but the package 'transparent.sty' is not loaded}%
    \renewcommand\transparent[1]{}%
  }%
  \providecommand\rotatebox[2]{#2}%
  \newcommand*\fsize{\dimexpr\f@size pt\relax}%
  \newcommand*\lineheight[1]{\fontsize{\fsize}{#1\fsize}\selectfont}%
  \ifx\svgwidth\undefined%
    \setlength{\unitlength}{125.02380804bp}%
    \ifx\svgscale\undefined%
      \relax%
    \else%
      \setlength{\unitlength}{\unitlength * \real{\svgscale}}%
    \fi%
  \else%
    \setlength{\unitlength}{\svgwidth}%
  \fi%
  \global\let\svgwidth\undefined%
  \global\let\svgscale\undefined%
  \makeatother%
  \begin{picture}(1,0.59019134)%
    \lineheight{1}%
    \setlength\tabcolsep{0pt}%
    \put(-0.00164496,0.28835496){\color[rgb]{0,0,0}\makebox(0,0)[lt]{\lineheight{1.25}\smash{\begin{tabular}[t]{l}$\mathrm{cl}(f_n) =$\end{tabular}}}}%
    \put(0,0){\includegraphics[width=\unitlength,page=1]{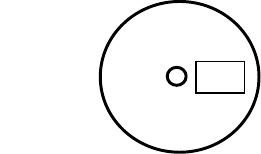}}%
    \put(0.80280541,0.27175387){\color[rgb]{0,0,0}\makebox(0,0)[lt]{\lineheight{1.25}\smash{\begin{tabular}[t]{l}$f_n$\end{tabular}}}}%
    \put(0,0){\includegraphics[width=\unitlength,page=2]{closure.pdf}}%
  \end{picture}%
\endgroup%

\end{center}
We have $\overline{W}\big( \mathrm{cl}(f_{s-1}) \big) = \overline{W}^{\chi^+_{s}}(m)$, so that $\overline{W}\big( \mathrm{cl}(f_{p-1}) \big) + J_1 = \overline{W}^{\chi^+_{p}}(m) + J_1 = 0$ (where $x+J_1$ means the class of $x$ modulo $J_1$) and $\overline{W} : \mathcal{S}_{\epsilon}^{\mathrm{red}}(\Sigma_{0,1}^{\mathrm{o}}) \to J_2/J_1$ is a well-defined map. Due to Theorem \ref{theoremRepVacuum} we see that
\[ \overline{W}\big(\mathrm{cl}(f_{s-1}) \blacktriangleleft L\big) + J_1 = \overline{W}\big(\mathrm{cl}(f_{s-1})\big) \triangleleft \overline{W}(L) + J_1 = \overline{W}^{\chi^+_{s}}(m) \triangleleft \overline{W}(L) + J_1 \]
where $L \in \mathcal{S}_{\epsilon}(\Sigma_{1,0})$ and $\blacktriangleleft$ is the stacking and reducing representation of $\mathcal{S}_{\epsilon}(\Sigma_{1,0})$ on $\mathcal{S}_{\epsilon}^{\mathrm{red}}(\Sigma_{0,1}^{\mathrm{o}})$. Thus $\overline{W} : \mathcal{S}_{\epsilon}^{\mathrm{red}}(\Sigma_{0,1}^{\mathrm{o}}) \to J_2/J_1$ is an isomorphism of $\mathcal{S}_{\epsilon}(\Sigma_{1,0})$-modules.
\end{proof}

\begin{remark}
In \cite[\S 6.5]{these}, we studied the left representation of $\mathcal{S}_{\epsilon}(\Sigma_{1,0})$ on the space mentionned in Remark \ref{remarqueHeisenberg} instead of on the vacuum representation space. Then the claim in \cite[Prop. 6.5.5]{these}, even if it is in agreement with the one in Proposition \ref{structureRepSkein}, does not make sense since we compare a left representation with the representation of $\mathcal{S}_{\epsilon}(\Sigma_{1,0})$ on $\mathcal{S}_{\epsilon}^{\mathrm{red}}(\Sigma_{0,1}^{\mathrm{o}})$ which is right due to our convention for the stack product (Definition \ref{stackProduct}); we did not realize this subtlety at that time. To make sense of it, one must take the dual representation, and apply the intertwiner mentionned in Remark \ref{remarqueHeisenberg}.
\end{remark}

\indent It is very difficult to generalize to higher genus such an explicit description of the representation of $\mathcal{S}_{\epsilon}(\Sigma_{g,0})$ on $\mathcal{L}_{0,g}^{\mathrm{inv}}(\overline{U}_{\!\epsilon^2})$. Indeed, by Remark \ref{baseInvariants}, finding a basis of $\mathcal{L}_{0,g}^{\mathrm{inv}}(\overline{U}_{\! \epsilon^2})$ is equivalent to finding a basis of $\mathrm{Inv}_{2g}(\overline{U}_{\! \epsilon^2})$ (multilinear forms invariant under the iterated coadjoint action), which is a difficult problem. Moreover, at roots of unity, generating sets of $\mathcal{S}_{\epsilon^2}(\Sigma_{g,0})$ are not known. However, we see from the previous result that the natural representation of $\mathcal{S}_{\epsilon^2}(\Sigma_{g,0})$ on $\mathcal{S}^{\mathrm{red}}_{\epsilon^2}(\Sigma_{0,g}^{\mathrm{o}})$ (which is known to be irreducible \cite{BW}) will be a composition factor of the representation of $\mathcal{S}_{\epsilon^2}(\Sigma_{g,0})$ on $\mathcal{L}_{0,g}^{\mathrm{inv}}(\overline{U}_{\!\epsilon^2})$ in a non-trivial way. Finally, the socle of this representation might be a generalization of the ideal formed by the characters of projective modules (which is the socle in the case of the torus, see $\mathcal{V}_1$ above).

\appendix
\section{Appendix: Proof of Theorem \ref{wilsonStack}}\label{appendixPreuveStack}
The proof is purely diagrammatic but requires some preliminaries. The interested reader might find it relevant to study the proof on an example with small values of $g,n$, like $g=2, n=3$. We begin with some notations:
\begin{center}
%% Accompanies image file 'tresse_1_0.pdf' (pdf, eps, ps)
\begingroup%
  \makeatletter%
  \providecommand\color[2][]{%
    \errmessage{(Inkscape) Color is used for the text in Inkscape, but the package 'color.sty' is not loaded}%
    \renewcommand\color[2][]{}%
  }%
  \providecommand\transparent[1]{%
    \errmessage{(Inkscape) Transparency is used (non-zero) for the text in Inkscape, but the package 'transparent.sty' is not loaded}%
    \renewcommand\transparent[1]{}%
  }%
  \providecommand\rotatebox[2]{#2}%
  \newcommand*\fsize{\dimexpr\f@size pt\relax}%
  \newcommand*\lineheight[1]{\fontsize{\fsize}{#1\fsize}\selectfont}%
  \ifx\svgwidth\undefined%
    \setlength{\unitlength}{433.24863599bp}%
    \ifx\svgscale\undefined%
      \relax%
    \else%
      \setlength{\unitlength}{\unitlength * \real{\svgscale}}%
    \fi%
  \else%
    \setlength{\unitlength}{\svgwidth}%
  \fi%
  \global\let\svgwidth\undefined%
  \global\let\svgscale\undefined%
  \makeatother%
  \begin{picture}(1,0.20079189)%
    \lineheight{1}%
    \setlength\tabcolsep{0pt}%
    \put(0,0){\includegraphics[width=\unitlength,page=1]{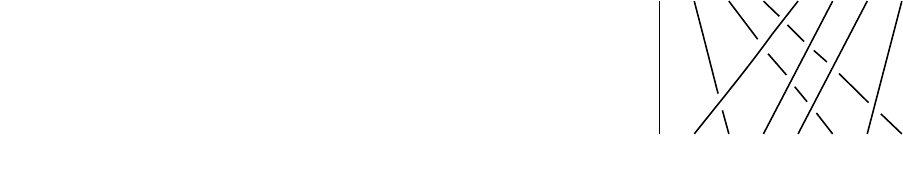}}%
    \put(0.64291618,0.11019271){\color[rgb]{0,0,0}\makebox(0,0)[lt]{\lineheight{1.25}\smash{\begin{tabular}[t]{l}$\beta_{1,0} =$\end{tabular}}}}%
    \put(0,0){\includegraphics[width=\unitlength,page=2]{tresse_1_0.pdf}}%
    \put(0.36539282,0.1086915){\color[rgb]{0,0,0}\makebox(0,0)[lt]{\lineheight{1.25}\smash{\begin{tabular}[t]{l}$\beta_{0,1} =$\end{tabular}}}}%
    \put(0,0){\includegraphics[width=\unitlength,page=3]{tresse_1_0.pdf}}%
    \put(-0.00278928,0.10808383){\color[rgb]{0,0,0}\makebox(0,0)[lt]{\lineheight{1.25}\smash{\begin{tabular}[t]{l}$c_{i,j} =$\end{tabular}}}}%
    \put(0,0){\includegraphics[width=\unitlength,page=4]{tresse_1_0.pdf}}%
    \put(0.06660494,0.00432537){\color[rgb]{0,0,0}\makebox(0,0)[lt]{\lineheight{1.25}\smash{\begin{tabular}[t]{l}$i$ strands\end{tabular}}}}%
    \put(0.19808621,0.00432817){\color[rgb]{0,0,0}\makebox(0,0)[lt]{\lineheight{1.25}\smash{\begin{tabular}[t]{l}$j$ strands\end{tabular}}}}%
  \end{picture}%
\endgroup%

\end{center}

\smallskip

\indent Note that up to using coupons we can assume that all the strands are positively oriented when they go through an handle and also that the bunch of strands in each handle is reduced to only one strand (colored by the tensor product of the colors). Hence we can assume the following form for $\mathbf{T}_1$:
\begin{center}
%% Accompanies image file 'boundaryTanglePreuveStack.pdf' (pdf, eps, ps)
\begingroup%
  \makeatletter%
  \providecommand\color[2][]{%
    \errmessage{(Inkscape) Color is used for the text in Inkscape, but the package 'color.sty' is not loaded}%
    \renewcommand\color[2][]{}%
  }%
  \providecommand\transparent[1]{%
    \errmessage{(Inkscape) Transparency is used (non-zero) for the text in Inkscape, but the package 'transparent.sty' is not loaded}%
    \renewcommand\transparent[1]{}%
  }%
  \providecommand\rotatebox[2]{#2}%
  \newcommand*\fsize{\dimexpr\f@size pt\relax}%
  \newcommand*\lineheight[1]{\fontsize{\fsize}{#1\fsize}\selectfont}%
  \ifx\svgwidth\undefined%
    \setlength{\unitlength}{482.99372168bp}%
    \ifx\svgscale\undefined%
      \relax%
    \else%
      \setlength{\unitlength}{\unitlength * \real{\svgscale}}%
    \fi%
  \else%
    \setlength{\unitlength}{\svgwidth}%
  \fi%
  \global\let\svgwidth\undefined%
  \global\let\svgscale\undefined%
  \makeatother%
  \begin{picture}(1,0.20750645)%
    \lineheight{1}%
    \setlength\tabcolsep{0pt}%
    \put(0,0){\includegraphics[width=\unitlength,page=1]{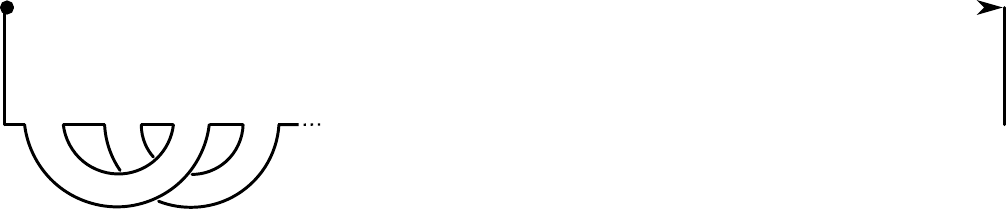}}%
    \put(0.62923982,0.09648827){\color[rgb]{0,0,0}\makebox(0,0)[lt]{\lineheight{1.25}\smash{\begin{tabular}[t]{l}{\footnotesize $I_{2g+1}$}\end{tabular}}}}%
    \put(0,0){\includegraphics[width=\unitlength,page=2]{boundaryTanglePreuveStack.pdf}}%
    \put(0.01194613,0.09421409){\color[rgb]{0,0,0}\makebox(0,0)[lt]{\lineheight{1.25}\smash{\begin{tabular}[t]{l}$I_1$\end{tabular}}}}%
    \put(0,0){\includegraphics[width=\unitlength,page=3]{boundaryTanglePreuveStack.pdf}}%
    \put(0.09573712,0.09343645){\color[rgb]{0,0,0}\makebox(0,0)[lt]{\lineheight{1.25}\smash{\begin{tabular}[t]{l}$I_2$\end{tabular}}}}%
    \put(0,0){\includegraphics[width=\unitlength,page=4]{boundaryTanglePreuveStack.pdf}}%
    \put(0.81753997,0.09648827){\color[rgb]{0,0,0}\makebox(0,0)[lt]{\lineheight{1.25}\smash{\begin{tabular}[t]{l}{\footnotesize $I_{2g+n}$}\end{tabular}}}}%
    \put(0,0){\includegraphics[width=\unitlength,page=5]{boundaryTanglePreuveStack.pdf}}%
    \put(0.47098025,0.13381054){\color[rgb]{0,0,0}\makebox(0,0)[lt]{\lineheight{1.25}\smash{\begin{tabular}[t]{l}$T_1$\end{tabular}}}}%
    \put(0,0){\includegraphics[width=\unitlength,page=6]{boundaryTanglePreuveStack.pdf}}%
    \put(0.81392093,0.1728931){\color[rgb]{0,0,0}\makebox(0,0)[lt]{\lineheight{1.25}\smash{\begin{tabular}[t]{l}$V_k$\end{tabular}}}}%
    \put(0.16558507,0.17284044){\color[rgb]{0,0,0}\makebox(0,0)[lt]{\lineheight{1.25}\smash{\begin{tabular}[t]{l}$V_1$\end{tabular}}}}%
    \put(0,0){\includegraphics[width=\unitlength,page=7]{boundaryTanglePreuveStack.pdf}}%
    \put(0.30164664,0.09558663){\color[rgb]{0,0,0}\makebox(0,0)[lt]{\lineheight{1.25}\smash{\begin{tabular}[t]{l}$I_{2g-1}$\end{tabular}}}}%
    \put(0,0){\includegraphics[width=\unitlength,page=8]{boundaryTanglePreuveStack.pdf}}%
    \put(0.40524102,0.09382859){\color[rgb]{0,0,0}\makebox(0,0)[lt]{\lineheight{1.25}\smash{\begin{tabular}[t]{l}$I_{2g}$\end{tabular}}}}%
    \put(0,0){\includegraphics[width=\unitlength,page=9]{boundaryTanglePreuveStack.pdf}}%
  \end{picture}%
\endgroup%

\end{center}
where the tangle $T_1$ is allowed to contain coupons. Similarly we can assume such a form for $\mathbf{T}_2$, but with colors $J_1, \ldots, J_{2g+n}$ instead of $I_1, \ldots, I_{2g+n}$ and $W_1, \ldots, W_l$ instead of $V_1, \ldots, V_k$ and with some tangle $T_2$ instead of $T_1$. The product $\mathbf{T}_1 \ast \mathbf{T}_2$ then looks as follows:
\begin{center}
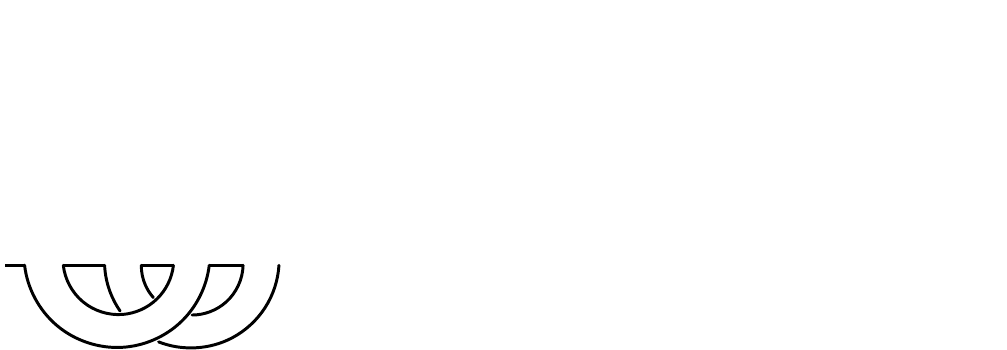
\end{center}
where the braid $\mathrm{ex}_{g,n}$ (or rather its diagram) is constructed as follows. Let $[0,1]^3$ with coordinates $(x,y,z)$ and let $\pi_x$ be the projection on the first coordinate. Let $a_i, c_j$ be points in $[0,1] \times \{0\} \times \{0\}$ and let $b_i, d_j$ be points in $[0,1] \times \{0\} \times \{1\}$, where $1 \leq i \leq g$ and $1 \leq j \leq n$, such that $\pi_x$ has strictly increasing values on the following sequence:
\[ \bigl( a_1, b_1, a_2, b_2, \ldots, a_g, b_g, c_1, d_1, c_2, d_2, \ldots, c_n, d_n \bigr). \]
Similarly, let $\bar a_i, \bar c_j$ be points in $[0,1] \times \{1\} \times \{0\}$ and $\bar b_i, \bar d_j$ be points in $[0,1] \times \{1\} \times \{1\}$, where $1 \leq i \leq g$ and $1 \leq j \leq n$, such that $\pi_x$ has strictly increasing values on the following sequence:
\[ \bigl( \bar a_1, \ldots, \bar a_g, \bar c_1, \ldots, \bar c_n, \bar b_1, \ldots, \bar b_g, \bar d_1, \ldots, \bar d_n \bigr). \]
Join the points $a_i$ and $\bar a_i$ (resp. $b_i$ and $\bar b_i$, $c_j$ and $\bar c_j$, $d_j$ and $\bar d_j$) by straight lines and project the result on the $(x,y)$-plane, this gives a braid diagram (we can always arrange the points $a_i,b_i,c_j,d_j$ and $\bar a_i, \bar b_i, \bar c_j, \bar d_j$ so that the projection on $(x,y)$ has only simple or double points and we get a well-defined diagram). Finally, for each $1 \leq i \leq g$, replace the strand joining the points $a_i$ and $\bar a_i$ (resp. $b_i$ and $\bar b_i$) by $4$ parallel strands ($4$-cabling). Similarly, for each $1 \leq j \leq n$, replace the strand joining the points $c_j$ and $\bar c_j$ (resp. $d_j$ and $\bar d_j$) by $2$ parallel strands ($2$-cabling). This gives a braid diagram of $\mathrm{ex}_{g,n}$ in $[0,1] \times [0,1]$.

\begin{lemma}\label{lemmaStack}
1. The following equalities hold for all $1 \leq i \leq g$ and $g+1 \leq j \leq g+n$:
\begin{center}
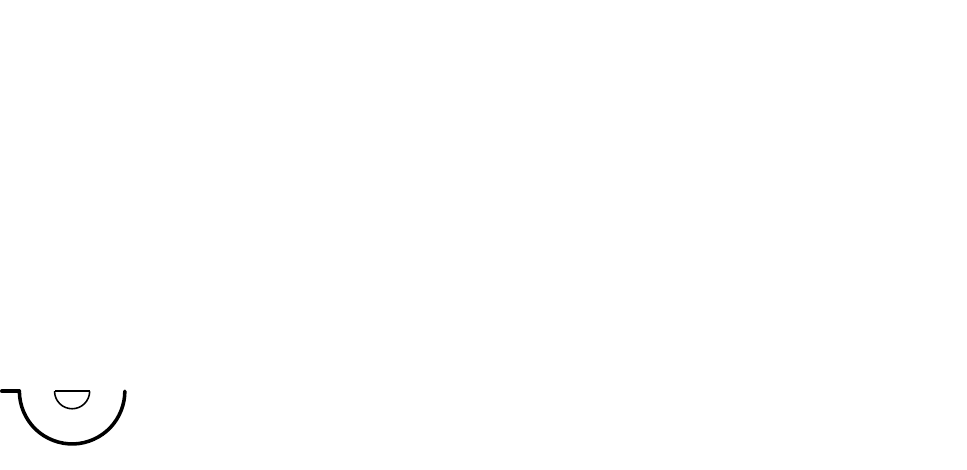
\end{center}
2. Let $\alpha_1, \ldots, \alpha_{g}$ and $\beta_1, \ldots, \beta_g$ be braids on $4$ strands and let $\gamma_1, \ldots, \gamma_n$ and $\delta_1, \ldots, \delta_n$ be braids on $2$ strands. It holds
\begin{align*}
&\mathrm{ex}_{g,n} \circ (\alpha_1 \otimes \beta_1 \otimes \alpha_2 \otimes \beta_2 \otimes \ldots \otimes \alpha_g \otimes \beta_g \otimes \gamma_1 \otimes \delta_1 \otimes \gamma_2 \otimes \delta_2 \otimes \ldots \otimes \gamma_n \otimes \delta_n)\\
= \: &(\alpha_1 \otimes \ldots \otimes \alpha_g \otimes \gamma_1 \otimes \ldots \otimes \gamma_n \otimes \beta_1 \otimes \ldots \otimes \beta_g \otimes \delta_1 \otimes \ldots \otimes \delta_n) \circ \mathrm{ex}_{g,n}.
\end{align*}
3. For $k, l \geq 1$, define 
\begin{align*}
\Gamma_{k,l} = \: &\bigl( \mathrm{id}_k^{\otimes (l-1)} \otimes c_{k,k} \otimes \mathrm{id}_k^{\otimes (l-1)} \bigr) \circ 
\bigl( \mathrm{id}_k^{\otimes (l-2)} \otimes c_{k,k}^{\otimes 2} \otimes \mathrm{id}_k^{\otimes (l-2)} \bigr) \circ \ldots\\
& \ldots \circ \bigl( \mathrm{id}_k^{\otimes 2} \otimes c_{k,k}^{\otimes (l-2)} \otimes \mathrm{id}_k^{\otimes 2} \bigr) \circ \bigl( \mathrm{id}_k \otimes c_{k,k}^{\otimes (l-1)} \otimes \mathrm{id}_k \bigr)
\end{align*}
where $\mathrm{id}_k = \mathrm{id}^{\otimes k}$ is the identity braid on $k$ strands; then
\[ \mathrm{ex}_{g,n} = \bigl(\mathrm{id}^{\otimes 4g} \otimes c_{4g,2n} \otimes \mathrm{id}^{\otimes 2n}\bigr) \circ \bigl( \Gamma_{4,g} \otimes \Gamma_{2,n} \bigr). \]
\end{lemma}
\begin{proof}
1. The first equality is readily equivalent to \eqref{dessinRelationFusion}. The second equality is the outcome of a diagrammatic computation displayed in Figure \ref{preuveHolStackTore} and based on \eqref{dessinRelationFusion} and \eqref{dessinEchangeL10}.
\\2. This is obvious by definition of the braid $\mathrm{ex}_{g,n}$.
\\3. Let $a'_i, c'_j$ be points in $[0,1] \times \{\frac{1}{2}\} \times \{0\}$ and let $b'_i, d'_j$ be points in $[0,1] \times \{\frac{1}{2}\} \times \{1\}$, where $1 \leq i \leq g$ and $1 \leq j \leq n$, such that $\pi_x$ has strictly increasing values on the following sequence:
\[ \bigl( a'_1, \ldots, a'_g, b'_1, \ldots, b'_g, c'_1, \ldots, c'_n, d'_1, \ldots, d'_n \bigr). \]
Join the points $a_i$ and $a'_i$ (resp. $b_i$ and $b'_i$, $c_j$ and $c'_j$, $d_j$ and $d'_j$) by straight lines and project the result on the $(x,y)$-plane, this gives a braid diagram. For each $1 \leq i \leq g$, replace the strand joining the points $a_i$ and $a'_i$ (resp. $b_i$ and $b'_i$) by $4$ parallel strands ($4$-cabling); for each $1 \leq j \leq n$, replace the strand joining the points $c_j$ and $c'_j$ (resp. $d_j$ and $d'_j$) by $2$ parallel strands ($2$-cabling). This defines a braid diagram $B_1$ in $[0,1] \times [0,\frac{1}{2}]$.
\\\noindent Similarly, join the points $a'_i$ and $\bar a_i$ (resp. $b'_i$ and $\bar b_i$, $c'_j$ and $\bar c_j$, $d'_j$ and $\bar d_j$) by straight lines and project the result on the $(x,y)$-plane, this gives a braid diagram. For each $1 \leq i \leq g$, replace the strand joining the points $a'_i$ and $\bar a_i$ (resp. $b'_i$ and $\bar b_i$) by $4$ parallel strands ($4$-cabling); for each $1 \leq j \leq n$, replace the strand joining the points $c'_j$ and $\bar c_j$ (resp. $d'_j$ and $\bar d_j$) by $2$ parallel strands ($2$-cabling). This defines a braid diagram $B_2$ in $[0,1] \times [\frac{1}{2},1]$.
\\\noindent Then it is not difficult to see that $\mathrm{ex}_{g,n} = B_2 \circ B_1$, $B_1 = \Gamma_{4,g} \otimes \Gamma_{2,n}$, $B_2 = \mathrm{id}^{\otimes 4g} \otimes c_{4g,2n} \otimes \mathrm{id}^{\otimes 2n}$ (it is helpful to draw some examples for small values of $g,n$).

\begin{figure}[h]
\centering
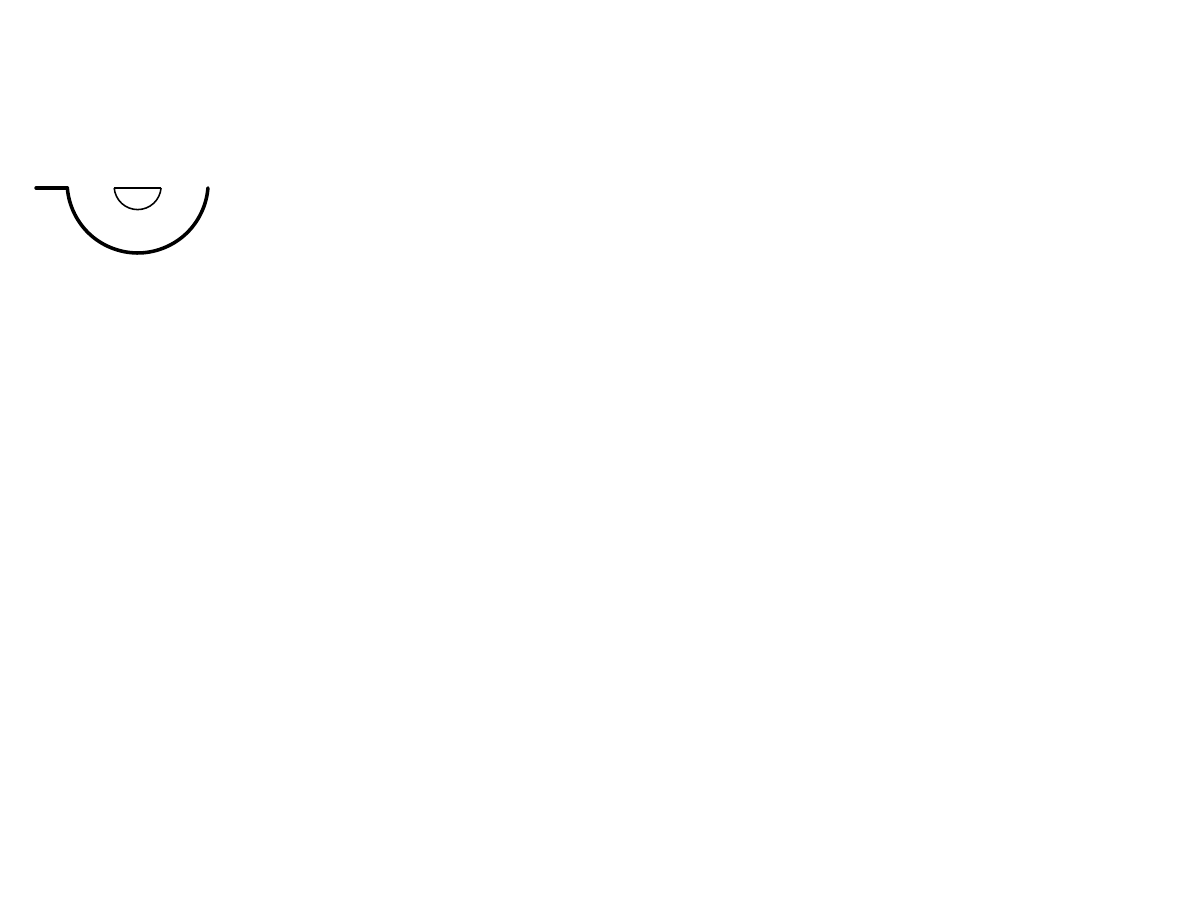
\caption{Proof of the second equality in Lemma \ref{lemmaStack}.1.}
\label{preuveHolStackTore}
\end{figure}
\end{proof}

The proof of Theorem \ref{wilsonStack} is now reduced to the diagrammatic computation displayed in Figure \ref{preuveStackGeneral} on page \pageref{preuveStackGeneral}, as we now explain. First note that by definition, the first diagram in this figure represents $\mathrm{hol}(\mathbf{T}_1 \ast \mathbf{T}_2)$ while the last diagram represents $\mathrm{hol}(\mathbf{T}_1) \odot \mathrm{hol}(\mathbf{T}_2)$. The first equality is obtained by applying the first item in Lemma \ref{lemmaStack} (1.) while the second equality is obtained thanks to the second and third items. For the third equality, observe first that due to the exchange relation \eqref{dessinEchangeLgn} we have that for $i < j$:
\begin{center}
%% Accompanies image file 'echangeBABA.pdf' (pdf, eps, ps)
\begingroup%
  \makeatletter%
  \providecommand\color[2][]{%
    \errmessage{(Inkscape) Color is used for the text in Inkscape, but the package 'color.sty' is not loaded}%
    \renewcommand\color[2][]{}%
  }%
  \providecommand\transparent[1]{%
    \errmessage{(Inkscape) Transparency is used (non-zero) for the text in Inkscape, but the package 'transparent.sty' is not loaded}%
    \renewcommand\transparent[1]{}%
  }%
  \providecommand\rotatebox[2]{#2}%
  \newcommand*\fsize{\dimexpr\f@size pt\relax}%
  \newcommand*\lineheight[1]{\fontsize{\fsize}{#1\fsize}\selectfont}%
  \ifx\svgwidth\undefined%
    \setlength{\unitlength}{318.85276682bp}%
    \ifx\svgscale\undefined%
      \relax%
    \else%
      \setlength{\unitlength}{\unitlength * \real{\svgscale}}%
    \fi%
  \else%
    \setlength{\unitlength}{\svgwidth}%
  \fi%
  \global\let\svgwidth\undefined%
  \global\let\svgscale\undefined%
  \makeatother%
  \begin{picture}(1,0.25894095)%
    \lineheight{1}%
    \setlength\tabcolsep{0pt}%
    \put(0,0){\includegraphics[width=\unitlength,page=1]{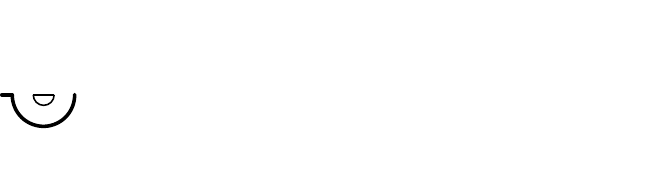}}%
    \put(0.03762925,0.00329696){\color[rgb]{0,0,0}\makebox(0,0)[lt]{\lineheight{1.25}\smash{\begin{tabular}[t]{l}$\overset{J_{2i-1}}{B}\!\!(i)$\end{tabular}}}}%
    \put(0,0){\includegraphics[width=\unitlength,page=2]{echangeBABA.pdf}}%
    \put(0.16020037,0.00338659){\color[rgb]{0,0,0}\makebox(0,0)[lt]{\lineheight{1.25}\smash{\begin{tabular}[t]{l}$\overset{J_{2i}}{A}(i)$\end{tabular}}}}%
    \put(0,0){\includegraphics[width=\unitlength,page=3]{echangeBABA.pdf}}%
    \put(0.25966961,0.00359661){\color[rgb]{0,0,0}\makebox(0,0)[lt]{\lineheight{1.25}\smash{\begin{tabular}[t]{l}$\overset{I_{2j-1}}{B}\!\!(j)$\end{tabular}}}}%
    \put(0,0){\includegraphics[width=\unitlength,page=4]{echangeBABA.pdf}}%
    \put(0.3809807,0.00389636){\color[rgb]{0,0,0}\makebox(0,0)[lt]{\lineheight{1.25}\smash{\begin{tabular}[t]{l}$\overset{I_{2j}}{A}(j)$\end{tabular}}}}%
    \put(0,0){\includegraphics[width=\unitlength,page=5]{echangeBABA.pdf}}%
    \put(0.48352812,0.10776558){\color[rgb]{0,0,0}\makebox(0,0)[lt]{\lineheight{1.25}\smash{\begin{tabular}[t]{l}$=$\end{tabular}}}}%
    \put(0.20680814,0.16818596){\color[rgb]{0,0,0}\makebox(0,0)[lt]{\lineheight{1.25}\smash{\begin{tabular}[t]{l}$c_{4,4}$\end{tabular}}}}%
    \put(0,0){\includegraphics[width=\unitlength,page=6]{echangeBABA.pdf}}%
    \put(0.57488862,0.00329696){\color[rgb]{0,0,0}\makebox(0,0)[lt]{\lineheight{1.25}\smash{\begin{tabular}[t]{l}$\overset{I_{2j-1}}{B}\!\!(j)$\end{tabular}}}}%
    \put(0,0){\includegraphics[width=\unitlength,page=7]{echangeBABA.pdf}}%
    \put(0.69745977,0.00338656){\color[rgb]{0,0,0}\makebox(0,0)[lt]{\lineheight{1.25}\smash{\begin{tabular}[t]{l}$\overset{I_{2j}}{A}(j)$\end{tabular}}}}%
    \put(0,0){\includegraphics[width=\unitlength,page=8]{echangeBABA.pdf}}%
    \put(0.79692898,0.00359659){\color[rgb]{0,0,0}\makebox(0,0)[lt]{\lineheight{1.25}\smash{\begin{tabular}[t]{l}$\overset{J_{2i-1}}{B}\!\!(i)$\end{tabular}}}}%
    \put(0,0){\includegraphics[width=\unitlength,page=9]{echangeBABA.pdf}}%
    \put(0.91824008,0.00389634){\color[rgb]{0,0,0}\makebox(0,0)[lt]{\lineheight{1.25}\smash{\begin{tabular}[t]{l}$\overset{J_{2i}}{A}(i)$\end{tabular}}}}%
    \put(0,0){\includegraphics[width=\unitlength,page=10]{echangeBABA.pdf}}%
  \end{picture}%
\endgroup%

\end{center}
Using repeatedly this identity, we transform the tensor product of the handles labelled by
\[ \overset{I_1}{B}(1), \overset{I_2}{A}(1), \overset{J_1}{B}(1), \overset{J_2}{A}(1), \overset{I_3}{B}(2), \overset{I_4}{A}(2), \overset{J_3}{B}(2), \overset{J_4}{A}(2), \ldots, \overset{I_{2g-1}}{B}\!\!(g), \overset{I_{2g}}{A}(g), \overset{J_{2g-1}}{B}\!\!(g), \overset{J_{2g}}{A}(g) \]
into the tensor product of the handles labelled by
\[ \overset{I_1}{B}(1), \overset{I_2}{A}(1), \overset{I_3}{B}(2), \overset{I_4}{A}(2), \ldots, \overset{I_{2g-1}}{B}\!\!(g), \overset{I_{2g}}{A}(g), \overset{J_1}{B}(1), \overset{J_2}{A}(1), \overset{J_3}{B}(2), \overset{J_4}{A}(2), \ldots, \overset{J_{2g-1}}{B}\!\!(g), \overset{J_{2g}}{A}(g). \]
Moreover, this manipulation removes exactly the braid $\Gamma_{4,g}$ in the diagram. Similarly, using repeatedly relation \eqref{dessinEchangeLgn}, we transform the tensor product of the handles labelled by
\[ \overset{I_{2g+1}}{M}\!\!(g\!+\!1), \overset{J_{2g+1}}{M}\!\!(g\!+\!1), \overset{I_{2g+2}}{M}\!\!(g\!+\!2), \overset{J_{2g+2}}{M}\!\!(g\!+\!2), \ldots, \overset{I_{2g+n}}{M}\!\!(g\!+\!n), \overset{J_{2g+n}}{M}\!\!(g\!+\!n) \]
into the tensor product of the handles labelled by
\[ \overset{I_{2g+1}}{M}\!\!(g\!+\!1), \overset{I_{2g+2}}{M}\!\!(g\!+\!2), \ldots, \overset{I_{2g+n}}{M}\!\!(g\!+\!n), \overset{J_{2g+1}}{M}\!\!(g\!+\!1), \overset{J_{2g+2}}{M}\!\!(g\!+\!2), \ldots, \overset{J_{2g+n}}{M}\!\!(g\!+\!n) \]
and this manipulation removes exactly the braid $\Gamma_{2,n}$ in the diagram. For the fourth equality, we simply use relation \eqref{dessinEchangeLgn} several times to transform the tensor product of the handles labelled by
\[ \overset{J_1}{B}(1), \overset{J_2}{A}(1), \overset{J_3}{B}(2), \overset{J_4}{A}(2), \ldots, \overset{J_{2g-1}}{B}\!\!(g), \overset{J_{2g}}{A}(g), \overset{I_{2g+1}}{M}\!\!(g\!+\!1), \overset{I_{2g+2}}{M}\!\!(g\!+\!2), \ldots, \overset{I_{2g+n}}{M}\!\!(g\!+\!n) \]
into the tensor product of the handles labelled by
\[ \overset{I_{2g+1}}{M}\!\!(g\!+\!1), \overset{I_{2g+2}}{M}\!\!(g\!+\!2), \ldots, \overset{I_{2g+n}}{M}\!\!(g\!+\!n), \overset{J_1}{B}(1), \overset{J_2}{A}(1), \overset{J_3}{B}(2), \overset{J_4}{A}(2), \ldots, \overset{J_{2g-1}}{B}\!\!(g), \overset{J_{2g}}{A}(g) \]
and this manipulation removes exactly the braid $c_{4g,2n}$ in the diagram.

\begin{figure}[h]
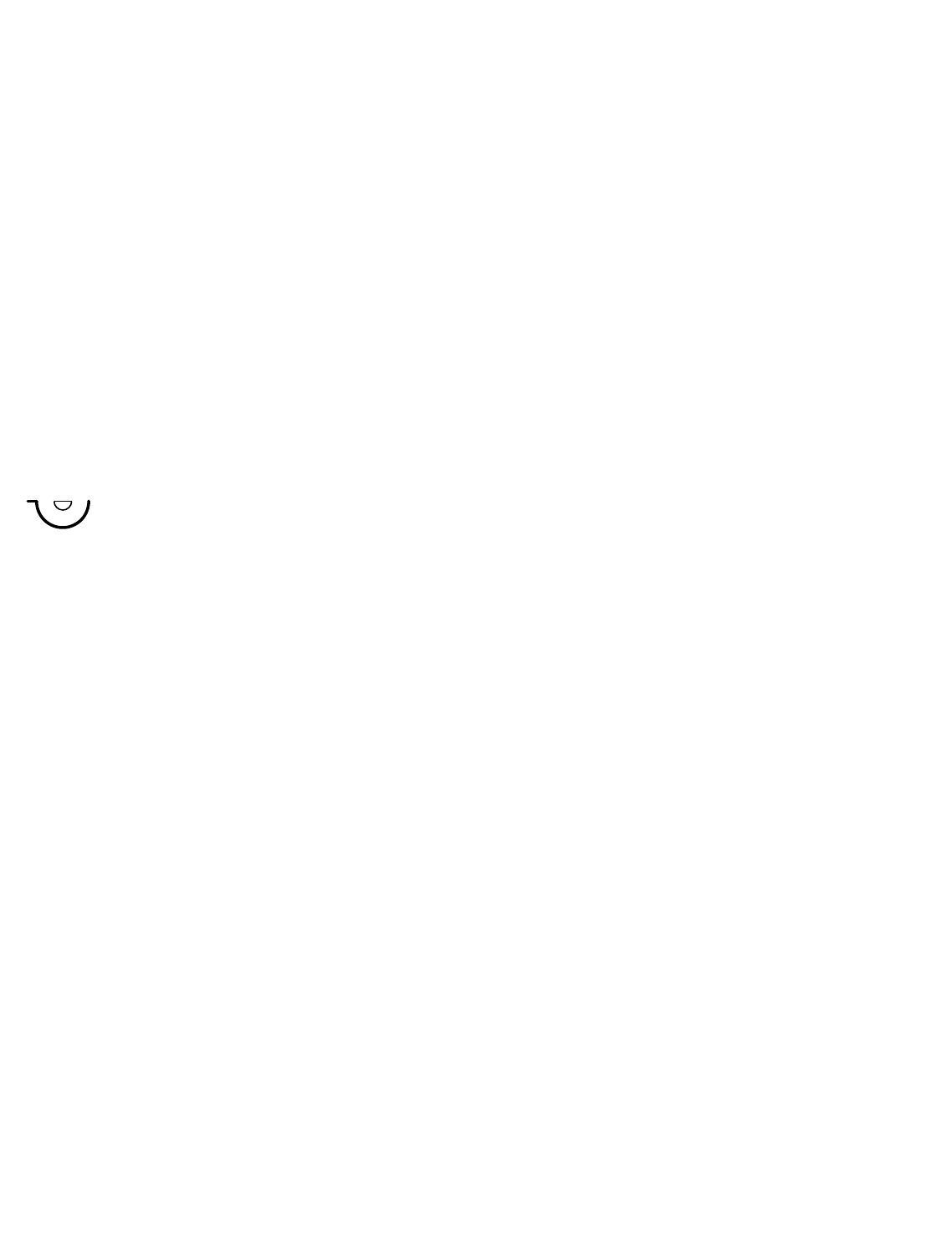
\caption{Proof of $\mathrm{hol}(\mathbf{T}_1 \ast \mathbf{T}_2) = \mathrm{hol}(\mathbf{T}_1) \odot \mathrm{hol}(\mathbf{T}_2)$.}
\label{preuveStackGeneral}
\end{figure}

\end{document}